\documentclass[11pt]{amsart}
\usepackage{amsmath,amsfonts,amsthm,amssymb,url,graphicx,algorithm,tikz,hyperref}
\usepackage{enumerate}
\usepackage{geometry}
\usepackage{xcolor}
\usepackage{manfnt}
\usepackage{tikz}
\usetikzlibrary{graphs, graphs.standard,math,positioning}

\DeclareFontFamily{OT1}{rsfs}{}
\DeclareFontShape{OT1}{rsfs}{n}{it}{<-> rsfs10}{}
\DeclareMathAlphabet{\mathscr}{OT1}{rsfs}{n}{it}
\DeclareMathOperator{\Tr}{Tr}
\DeclareMathOperator{\Prob}{Prob}

\DeclareMathOperator{\mo}{\,mod}
\DeclareMathOperator{\Ad}{Ad}
\DeclareMathOperator{\st}{s.t.}
\DeclareMathOperator{\rank}{rank}
\DeclareMathOperator{\Span}{Span}
\newtheorem{theorem}{Theorem}[section]
\newtheorem{lemma}[theorem]{Lemma}
\newtheorem{prop}[theorem]{Proposition}

\newtheorem{corollary}[theorem]{Corollary}
\newtheorem*{main}{Main Theorem}
\numberwithin{equation}{section}

\newcommand{\newp}[1]{\textcolor{purple}{#1}}

\usepackage[skins]{tcolorbox}
\usepackage{lipsum}

\newenvironment{danger}{\medbreak\noindent\hangindent=2pc\hangafter=-2%
  \clubpenalty=10000%
  \hbox to0pt{\hskip-\hangindent\dbend\hfill}\small\ignorespaces}%
  {\medbreak\par}

\usepackage{verbatim}

\title{Expansion, divisibility and parity}
\author{Harald Andr\'es Helfgott}
\address{Harald A. Helfgott, IMJ-PRG, UMR 7586,
 58 avenue de France, B\^{a}timent S. Germain, case 7012,
 75013 Paris CEDEX 13, France;
 Mathematisches Institut,
 Georg-August Universit\"{a}t G\"{o}ttingen, Bunsenstra{\ss}e 3-5,
 D-37073 G\"{o}ttingen, Deutschland}
\email{harald.helfgott@gmail.com}
\author{Maksym Radziwi{\l}{\l}}
\address{Maksym Radziwi{\l}{\l}, Department of Mathematics,
  California Institute of Technology,
  1200 E California Blvd,
  Pasadena, CA, 91125 \\
        USA}
\email{maksym.radziwill@gmail.com}
\begin{document}


\begin{abstract}
    Let $\mathbf{P} \subset [H_0,H]$ be a set of primes,
    where $\log H_0 \geq (\log H)^{2/3 + \epsilon}$.
    Let $\mathscr{L} = \sum_{p \in \mathbf{P}} 1/p$.
    Let $N$ be such that $\log H \leq (\log N)^{1/2-\epsilon}$.
     We show there exists a subset $\mathscr{X} \subset (N, 2N]$
    of density close to $1$ such that all the eigenvalues of the linear operator
  $$
 (A_{|\mathscr{X}} f)(n) = \sum_{\substack{p \in \mathbf{P} : p | n \\ n, n \pm p \in \mathscr{X}}} f(n \pm p) \; - \sum_{\substack{p \in \mathbf{P} \\ n, n \pm p \in \mathscr{X}}} \frac{f(n \pm p)}{p} 
  $$
    are  $O(\sqrt{\mathscr{L}})$. This bound is optimal up to a constant factor.
    In other words, we prove that a graph describing divisibility by
    primes
    is a strong local expander almost everywhere, and indeed within
    a constant factor of being ``locally Ramanujan'' (a.e.).
         
     Specializing to $f(n) = \lambda(n)$ with $\lambda(n)$ the Liouville function, and using an estimate by Matom\"aki, Radziwi{\l}{\l} and Tao on the
     average of $\lambda(n)$ in short intervals, 
  we derive that
\[
    \frac{1}{\log x} \sum_{n\leq x} \frac{\lambda(n) \lambda(n+1)}{n} =
    O\Big(\frac{1}{\sqrt{\log \log x}}\Big),\]
    improving on a result of Tao's.
  We also prove that $\sum_{N<n\leq 2 N} \lambda(n) \lambda(n+1)=o(N)$
  at almost all scales
  with a similar error term,
    improving on a result by Tao and Ter\"av\"ainen.
    (Tao and Tao-Ter\"av\"ainen followed a different approach, based on
    entropy, not expansion; significantly, we can take a much
     larger value of $H$, and thus consider many more primes.)
    
    We can also prove sharper results with ease. Thus, for instance, we can show that
    \[\frac{1}{\log x}
    \mathop{\sum_{n\leq x}}_{|\Omega(n)-\Omega(n+1)|\leq s(x)}
    \frac{\lambda(n) \lambda(n+1)}{n} =
    o\Big(\frac{s(x)}{\sqrt{\log \log x}}\Big)\]
    for any $s(x)$ tending to $\infty$ as $x\to \infty$,
    where $\Omega(n)$ is the number of prime divisors of $n$, considered
     with multiplicity. We
     also show that, for example,
     for $S_{N,k}$ the set of $N<n\leq 2 N$ such that $\Omega(n) = k$,
     and any fixed value of $k$ with
    $k = \log \log N + O(\sqrt{\log \log N})$, 
    the average of $\lambda(n+1)$ over $S_{N,k}$ is $o(1)$
    at almost all
    scales.
\end{abstract}
\maketitle
\thispagestyle{empty}
\section{Introduction}

\subsection{Motivation: averages of multiplicative functions.}
Let $\lambda:\mathbb{Z}_{>0}\to \mathbb{C}$ be the Liouville function, i.e., the
completely multiplicative function such that $\lambda(p) = -1$ for all primes $p$. 
%
  The fact that
  \[\lim_{x\to \infty} \frac{1}{x} \sum_{n\leq x} \lambda(n) = 0\]
  is already non-trivial, being equivalent to the Prime Number Theorem.
  Establishing
  \begin{equation}\label{eq:chowla2}
    \lim_{x\to\infty} \frac{1}{x} \sum_{n\leq x} \lambda(n) \lambda(n+1) = 0
  \end{equation}
  may be viewed as a multiplicative analogue of the twin prime conjecture;
  it remains
 a very hard open problem. Together with higher-degree analogues, it
 is a conjecture ascribed to Chowla \cite[Ch.~8, 57]{zbMATH03216354} --
 one of a family of central conjectures in analytic number theory
   that go under the conceptual umbrella of the {\em parity problem}, which,
   plainly put,
    states $\lambda$ is very hard to deal with, and is indeed the point
   at which many standard tools break.
 


  In \cite{MR3569059} Tao established a weak version of \eqref{eq:chowla2}:
  \begin{equation}\label{eq:oldlogchowla}
    \frac{1}{\log x}
    \sum_{n\leq x} \frac{\lambda(n) \lambda(n+1)}{n} \to 0\end{equation}
    as $x\to \infty$.  
    Tao's proof depends on a result of Matom\"aki-Radziwi{\l\l}
    \cite{MR3488742}
    on short averages of the Liouville function (in a slightly stronger version proven in \cite{MR3435814}) and on the \textit{entropy decrement method} introduced to analytic number theory by Tao. 

    Tao observes in \cite{MR3569059} that \eqref{eq:oldlogchowla}
    reduces to showing that
\begin{equation} \label{eq:tao}
\sum_{n \leq x} \lambda(n) \Big ( \sum_{\substack{p | n \\ p \in \mathbf{P}}} \lambda(n + p) \Big ) = o( x \mathscr{L} )
\end{equation}
where $\mathbf{P}$ is a subset of the primes and $\mathscr{L} := \sum_{p \in \mathbf{P}} 1/p$.

\subsection{A prime divisibility graph.}
We are led to consider a graph $\Gamma = (V, E)$
whose set of vertices $V$ is
\begin{equation}\label{eq:hahn1}
  V = \mathbf{N} := \{n\in \mathbb{Z}: N<n\leq 2 N\},\end{equation}
and whose set of edges $E$ is given by
\begin{equation}\label{eq:hahn2}
  E = \{(n,n+\sigma p): \sigma = \pm 1,\; p\in \mathbf{P},\; p|n,\; n\in V,\; n+\sigma p \in V\}.\end{equation}
Graphs essentially equivalent to this one were also discussed in \cite{MR3513734}.
Tao remarks
  that ``some sort of expander graph property''
  may hold for $\Gamma$ ``or for some
  closely related graph'' \cite[\S 4]{MR3569059}.

The notion of an {\em expander graph} is usually defined for
  regular graphs, that is, graphs where every vertex has the same degree
  $d$.
(Of course, our graph $\Gamma$ is not regular; $\mathscr{L}$ is its
  {\em average}  degree.)
  Define the adjacency operator $\Ad$ of a graph $(V,E)$ as follows:
  for  $f:V\to \mathbb{C}$, let $\Ad f:V\to \mathbb{C}$ be given by
  $\Ad f(v) = \sum_{w: \{v,w\}\in E} f(w)$. A regular graph of degree $d$
  is called a (two-sided)
   {\em $\epsilon$-expander} if the eigenvalues of
   the restriction of $\Ad$ to functions
   orthogonal to the constant eigenvector (that is, functions
   $f:V\to \mathbb{C}$ of average $0$) all have absolute value
   $\leq (1-\epsilon) d$. It is not hard to show that a graph is an
   $\epsilon$-expander if
   and only if the outcome of a random walk on it of length
   $(C/\epsilon) \log |V|$ (where
   $C$ is a large constant) has an outcome extremely close to the uniform distribution.

It is clear that, in the case of $\Gamma$, one should hope for a local
kind of expansion -- 
walks of moderate length $k$ carry one only so far in $\Gamma$,
namely, no farther than $k H$, where $H = \max \mathbf{P}$ -- and even
that only almost
everywhere; for instance, random walks starting at a prime $n$ go nowhere.
Conversely, if we show that, starting at a typical $n$, a random walk converges quickly to a distribution spread around $n$, then we should be
able to reduce \eqref{eq:tao} to the statement proved by \cite{MR3488742},
namely, $\lambda$ has average close to $0$ on most short intervals.
  


Tao does not establish an expansion property for $\Gamma$. In fact he
remarks that
``Unfortunately we were unable to establish such an
    expansion property, as the edges in the graph [\dots] do not seem
    to be either random enough or structured enough for standard methods of
    establishing expansion to work''. He circumvents this obstacle 
    with the \textit{entropy decrement method}, which allows
    one to show that there {\em exists} an $H$ such that
\begin{equation}  \label{eq:as}
\mathop{\sum_{p | n}}_{p \in \mathbf{P}\cap [H/2,H]} \!\!\!\!\!\! f(n + p)\quad \approx\sum_{p \in \mathbf{P}\cap [H/2,H]}\!\! \frac{f(n + p)}{p},
\end{equation}
holds for most $n$,
with $H \leq \log^{o(1)} N$.
This constraint on $H$ is understood now as a significant obstacle in establishing the full logarithmic Sarnak conjecture, since local Fourier uniformity is known at the scale $H = N^{\epsilon}$ \cite{MRTTZ} 
  -- with an eventual strengthening
to  $\exp(\log^{\theta} N)$ for some $\theta < 1$ being likely --
     but appears far out of each for $H = \log^{o(1)} N$. See \S \ref{subs:adan} for more comments on this matter.

\subsection{Main result.}
     The object of this paper is to 
     establish a strong form of expansion for an operator closely related
     to the adjacency operator $\Ad_\Gamma$ of $\Gamma$.
     The statement is equivalent to a strong form of ``local expansion
     almost everywhere'' in the sense sketched above in terms of random walks.


  One immediate issue is that $\Gamma$ is not
  regular, and thus does not have a constant eigenvector. Define
  $\Gamma'$ to have the same vertex set $V = \mathbf{N}$ as $\Gamma$,
  and edges $(n,n+\sigma p)$ with weight $1/p$ for all $p\in \mathbf{P}$ with
  $n, n + \sigma p\in V$, regardless of whether $p|n$. We will work
  with the difference of the adjacency operators of $\Gamma$ and $\Gamma'$:
  \begin{equation}\label{eq:defA}
    A = \Ad_\Gamma - \Ad_{{\Gamma}'}.
    \end{equation}
  Explicitly, $A$ is the linear operator taking any function $f:V\to \mathbb{C}$
  to the function $Af:V\to \mathbb{C}$ given by
  \begin{equation}\label{eq:alterno}
    Af(n) = \mathop{\mathop{\sum_{p\in \mathbf{P},\; p|n}}_{\sigma = \pm 1}}_{
    n + \sigma p \in V} f(n+\sigma p) -
  \mathop{\mathop{\sum_{p\in \mathbf{P}}}_{\sigma = \pm 1}}_{
    n + \sigma p \in V} \frac{f(n+\sigma p)}{p}.
  \end{equation}
  Given a subset $\mathscr{X}\subset V$, we can also define the restriction
  $A|_{\mathscr{X}}$ to be the linear operator on functions $f:V\to \mathbb{C}$
  taking $f$ to $A|_{\mathscr{X}} f := (A(f|_{\mathscr{X}}))|_{\mathscr{X}}$.
  We will prove expansion for $A|_{\mathscr{X}}$, where $\mathscr{X}$ is
  almost all of $V$.
  \begin{main}
    Let the operator $A$ be as above, with
     $V=\mathbf{N}=\{N+1,\dotsc,2 N\}$ and
    $H_0,H,N\geq 1$
    such that $H_0\leq H$ 
    and $\log H_0 \geq (\log H)^{2/3} (\log \log H)^{2}$,
    and $\mathbf{P}\subset [H_0,H]$ a set of primes such that
    $\mathscr{L} = \sum_{p\in \mathbf{P}} 1/p \geq e$ and
    $\log H \leq \sqrt{(\log N)/\mathscr{L}}$.
    

    Then, for any $1\leq K\leq (\log N)/(\mathscr{L} (\log H)^2)$,
    there is a subset $\mathscr{X}\subset \mathbf{N}$
    with       $|\mathbf{N}\setminus \mathscr{X}|\ll
    N e^{-K \mathscr{L} \log K} + N/\sqrt{H_0}$
    such that every eigenvalue of $A|_{\mathscr{X}}$ is
    \begin{equation}\label{eq:mainthm}O\left(\sqrt{K \mathscr{L}}\right),
    \end{equation}
    where the implied constants are absolute.
  \end{main}
  Here $O(\sqrt{\mathscr{L}})$ is in the order of the strongest expansion
  property one may have (``Ramanujan graphs''). We will later see (confirming our previous discussion) that
  the restriction $A|_{\mathscr{X}}$ to $\mathscr{X}$ is in fact necessary.
  Some will recognize that the bound on $|V\setminus \mathscr{X}|$
  is of the same order as the number of integers in $V=\mathbf{N}$ with
  more than $K \mathscr{L}$ factors in $\mathbf{P}$. Those integers have to be
  excluded from $V$.
The set $V \setminus \mathscr{X}$ will consist of them,
  together with a smaller number of other integers.

  The following are immediate consequences (one might almost
  call them restatements) of the main theorem. 
  We define the $\ell^p$ norm on functions $f:\mathbf{N}\to \mathbb{C}$ by
  $$|f|_p = \left( \frac{1}{N}
  \sum_{n\in \mathbf{N}} |f(n)|^p\right)^{1/p}.$$
  \begin{corollary}\label{cor:maic}
    Let $\mathbf{N}$, $\mathbf{P}$ and $\mathscr{L}$ be as above, with
    $H_0$, $H$, $N$ and $\mathscr{L}$ satisfying the same conditions
    as in the Main Theorem.
    Let $f,g:\mathbf{N}\to \mathbb{C}$
    satisfy $|f|_2, |g|_2\leq 1$
    and $|f|_4, |g|_4 \leq e^{C \mathscr{L}}$ for some $C>0$. Then
    \begin{equation}\label{eq:corncob}
      \frac{1}{N \mathscr{L}} \left|
      \sum_{n\in \mathbf{N}} \sum_{\sigma=\pm 1} \sum_{p\in \mathbf{P},\; p|n}
      f(n) \overline{g(n+\sigma p)} -
      \sum_{n\in \mathbf{N}} \sum_{\sigma=\pm 1}
      \sum_{p\in \mathbf{P}} \frac{f(n) \overline{g(n+\sigma p)}}{p}\right| =
    O\left(\frac{1}{\sqrt{\mathscr{L}}}\right),\end{equation}
    where the implied constant depends only on $C$. 
  \end{corollary}

\begin{corollary}\label{cor:kukuruc}
    Let $\mathbf{N}$, $\mathbf{P}$ and $\mathscr{L}$ be as above, with
    $H_0$, $H$, $N$ and $\mathscr{L}$ satisfying the same conditions
    as in the Main Theorem.
    Let $f:\mathbf{N}\to \mathbb{C}$ satisfy $|f|_2\leq 1$
    and $|f|_4 \leq e^{C \mathscr{L}}$ for some $C>0$. Then
    \begin{equation}\label{eq:choclo}
      \frac{1}{N} \sum_{n\in \mathbf{N}} \left| 
      \sum_{\sigma = \pm 1}
      \sum_{p \in \mathbf{P}, p|n}
      f(n + \sigma p) -
      \sum_{\sigma = \pm 1}
      \sum_{p \in \mathbf{P}} \frac{f(n + \sigma p)}{p}\right|^2 = O\left( \mathscr{L}\right),
      \end{equation}
    where the implied constant depends only on $C$. 
  \end{corollary}
  
 Bounds \eqref{eq:corncob} and \eqref{eq:choclo} are sharp up to a constant factor, as is shown by the example $f(n)=g(n)=1$.

 We can also easily deduce from the main theorem versions of it without
 a parameter $H_0$, so that $\mathbf{P}$ can be any set of primes $p\leq H$.
 The following is one such version.
 \begin{corollary}\label{cor:sansH0}
    Let the operator $A$ be as above, with $N$ and $H\geq 16$
    such that $\log H \leq \sqrt{(\log N)/\mathscr{L}}$,
    $V=\mathbf{N}=\{N+1,N+2,\dotsc,2 N\}$ and $\mathbf{P}$
    a set of primes in $[1,H]$ with 
    $\mathscr{L} = \sum_{p\in \mathbf{P}} 1/p \geq 1$.
      Then, for any $1\leq K\leq (\log N)/(\mathscr{L} (\log H)^2)$,
          there is a subset $\mathscr{X}\subset V$ 
            such that every eigenvalue of $A|_{\mathscr{X}}$ is
            \begin{equation}\label{eq:slorse}O\left(\sqrt{K \mathscr{L}}
              \cdot \log \log \log H\right),
            \end{equation}
            and  $|\mathbf{N}\setminus \mathscr{X}|\ll e^{-
            \sqrt{K \mathscr{L}} \log K} N$,
    where the implied constants are absolute.
   \end{corollary}
  \subsection{Consequences: correlations of arithmetical functions}
    
\subsubsection{General statement}
  When we apply Corollary \ref{cor:maic} to functions $f$, $g$ for which
  the values $f(n)$, $g(n\pm p)$ can be easily related to $f(n/p)$ and
  $g(n/p\pm 1)$ when $p|n$, we obtain an estimate on the correlations
  $\sum_{x<n\leq 2 x} f(n) g(n+1)$ of 
  $f(n)$ and $g(n+1)$. In particular, when $f(n)$ and $g(n)$ depend only 
  the number $\Omega(n)$ of prime divisors of $n$ considered with multiplicity,
  we obtain a straightforward general statement.

  \begin{corollary}\label{cor:cruxio}
     Let $\mathbf{N}$, $\mathbf{P}$ and $\mathscr{L}$ be as above, with
    $H_0$, $H$, $N$ and $\mathscr{L}$ satisfying the same conditions
     as in the Main Theorem.
     Let $F_1,F_2:\mathbb{Z}_{\geq 0}\to \mathbb{C}$ be such
     that $|F_i|_\infty\leq 1$ and $F_i$ has support on $S_i\subset \mathbb{Z}_{\geq 0}$ for $i=1,2$. Then
     \begin{equation}\label{eq:startrek}\begin{aligned}
     \frac{1}{N \mathscr{L}}
     \sum_{p\in \mathbf{P}}
     \sum_{\frac{N}{p}<n\leq \frac{2 N}{p}} &F_1(\Omega(n))
     \overline{F_2(\Omega(n+1))} =
     \frac{1}{N \mathscr{L}} \sum_{n\in \mathbf{N}}
     F_1(\Omega(n)-1) \sum_{p\in \mathbf{P}}
     \frac{\overline{F_2(\Omega(n+p)-1)}}{p}
     \\ &+ O\left(\min\left(
\frac{1}{\sqrt{\mathscr{L}}},
     \frac{\sqrt{s_1 s_2}}{
       \sqrt{\log \log N}} \cdot
     \left(\frac{1}{\sqrt{\mathscr{L}}} +
     \frac{\log \log \log N}{\mathscr{L}}
     \right)
     \right)\right),
     \end{aligned}\end{equation}
     where $s_i = \min(|S_i|,\sqrt{\log \log N})$ for $i=1,2$.
\end{corollary}

It is clear, or will soon be clear,
     why we would like
     to have an estimate on the left side of \eqref{eq:startrek}:
     it directly implies estimates on expressions of the form
     $\sum_{x/w< n\leq x} \frac{F_1(\Omega(n)) \overline{F_2(\Omega(n+1))}}{n},$
     as well as 
     estimates on $\sum_{x<n\leq 2 x} F_1(\Omega(n))
     \overline{F_2(\Omega(n+1))}$ valid ``at almost all scales''. Let us see
     how to read the right side of \eqref{eq:startrek}.

     The first expression on the right side of \eqref{eq:startrek}
     generally becomes available from the moment that we know
     how to estimate $\sum_{n\in \mathbf{N}}
     F_1(\Omega(n)) \sum_{|h|\leq H} \overline{F_2(\Omega(n+h))}$.
     We can usually prove this kind of estimate thanks to
     \cite{MR3488742} and all that has followed.

     The bound $O(1/\sqrt{\mathscr{L}})$ on the error term should be compared
     to the trivial 
     bound on the left side of \eqref{eq:startrek}, viz., $O(1)$.
     The second bound inside $\min$ in \eqref{eq:startrek}
     can be compared to 
     \begin{equation}\label{eq:raravoz}
              O\left(\frac{s_1 s_2}{\log \log N}\right),\end{equation}
     which follows from 
     sieve theory.
     When  $\mathscr{L}\gg \log \log N$, we see that
the second bound in \eqref{eq:startrek} is stronger 
     than the bound in \eqref{eq:raravoz} if either
     $s_1$ or $s_2$ goes to infinity, however slowly, as $N\to \infty$.


\subsubsection{Consequences on parity}
    We may apply Cor.~\ref{cor:maic} as is to the function $f(n)=\lambda(n)$.
    We will bound the second sum on the right of \eqref{eq:corncob}
    using \cite[Thm.~1.3]{MR3435814} (essentially as in
    \cite[Lemmas 3.4--3.5]{MR3569059}). We then obtain that
    \[\frac{1}{N \mathscr{L}} \sum_{p\in \mathbf{P}}
    \sum_{\frac{N}{p} < n\leq \frac{2 N}{p}} \lambda(n) \lambda(n+1) =
    O\left(\frac{1}{\mathscr{L}^{1/2}}\right).\]

    The following two consequences are almost immediate.
    The first concerns what has been called a {\em logarithmic average}.
    
    \begin{corollary}\label{cor:newlogchowla}
      For any $e<w\leq x$ such that $w\to\infty$ as $x\to \infty$, 
      \[\frac{1}{\log w} \sum_{\frac{x}{w}\leq n\leq x} \frac{\lambda(n) \lambda(n+1)}{n} = O\left(\frac{1}{\sqrt{\log \log w}}\right).\]
    \end{corollary}
    In this situation, somewhat more general than that of
    (\ref{eq:oldlogchowla}), \cite{MR3569059} also gives
    a result, which has been made explicit as
    $O(1/\min(\log \log \log w, \log \log \log \log x)^c)$, with $c$
    at least $1/5$ \cite[Thm.~5.1]{HelfUbis}.
    The improved argument in \cite{zbMATH07081585} can probably be made to give
    $O(1/\min(\log \log w, \log \log \log x)^c)$ for some
    $0<c<1/3$.
    
    We can in fact give a stronger statement than
    Cor.~\ref{cor:newlogchowla},
    proving that Chowla's conjecture in degree $2$ holds at almost all scales.
        A statement like the one we are about to give, but with $o(1)$ instead of
    $O(1/\sqrt{\log \log w})$, first appeared in
    \cite[Thm.~1.7]{zbMATH07141311}.

     \begin{corollary}[Chowla at almost all scales]\label{cor:newalmostchowla}
       Write $S(x) = (1/x) \sum_{x<n\leq 2 x} \lambda(n) \lambda(n+1)$.
       Then, for any $e<w\leq x$ such that $w\to \infty$ as $x\to \infty$,
       \begin{equation}\label{eq:newalmostchowla}
         \frac{1}{\log w} \int_{x/w}^x |S(t)| \frac{dt}{t} =
         O\left(\frac{1}{\sqrt{\log \log w}}\right).\end{equation}
     \end{corollary}

    The fact Corollaries \ref{cor:newlogchowla}--\ref{eq:newalmostchowla}
    are stronger than
    what results from \cite{MR3569059},
    \cite{zbMATH07081585} or \cite{zbMATH07141311}
    should be unsurprising, as the overall
    strategy is more direct.
    It is clearly implied in 
    \cite[\S 4]{MR3569059} that a proof of expansion is the natural road;
     that naturality is part of our motivation.

    We can also restrict $\Omega(n)$ or $\Omega(n+1)$ to given ranges
    of values, or even to fixed values, and still get cancellation.
             \begin{corollary}\label{cor:lobster}
       For any $e<w\leq x$ such that
       $w\geq \exp((\log x)^\epsilon)$ with $\epsilon>0$
       and any intervals $I_1=I_1(x)\subset \mathbb{Z}_{>0}$,
       $I_2 = I_2(x) \subset \mathbb{Z}_{>0}$,
       \[\frac{1}{\log w}
       \mathop{\mathop{\sum_{\frac{x}{w} <n\leq x}}_{\Omega(n) \in I_1}}_{
         \Omega(n+1)\in I_2}
       \frac{\lambda(n) \lambda(n+1)}{n}
       = O_\epsilon\left(\frac{\sqrt{s_1 s_2}}{\log \log x}\right),
       \]
       where $s_i = s_i(x) = \min(|I_i|,\sqrt{\log \log x})$ for $i=1,2$.
           \end{corollary}
Thus we see that, for instance,
     \begin{equation}\label{eq:amador1}\frac{1}{\log w}
     \mathop{\sum_{\frac{x}{w} < n\leq x}}_{\Omega(n) = k_x}
     \frac{\lambda(n+1)}{n} = O_\epsilon\left(\frac{1}{(\log \log x)^{3/4}}
     \right), \end{equation} and so we get cancellation 
     for any $k_x \in (\log \log x - C \sqrt{\log \log x},
     \log \log x + C \sqrt{\log \log x})$, i.e., any value
     of $\Omega(n)$ outside the tails of the distribution.
     We also obtain immediately that, for any $s(x)\geq 1$,
\begin{equation}\label{eq:amador2}\frac{1}{\log w}
     \mathop{\sum_{\frac{x}{w} < n\leq x}}_{|\Omega(n)-\Omega(n+1)|\leq s(x)}
\!\!\!\!\!\!\!\!     \frac{\lambda(n) \lambda(n+1)}{n} =
     O_\epsilon\left(\frac{\sqrt{s(x)}}{\sqrt{\log \log x}}\right)
,\end{equation}
and so we get cancellation provided that
$s(x)\to \infty$ as $x\to \infty$. (The easy bound is
of course $O(s(x)/\sqrt{\log \log x})$.)

It does not seem too difficult to prove an analogue of
Cor.~\ref{cor:lobster} giving a result at almost all scales, like
Cor.~\ref{cor:newalmostchowla}, rather than as an average over
$(x/w,x]$. We will prove a special case.

     \begin{corollary}\label{cor:almav}
       Let $S_k(x) = (1/x) \sum_{x<n\leq 2 x: \Omega(n)=k}
      \lambda(n+1)$.
          For any $e<w\leq x$ such that
       $w\geq \exp((\log x)^\epsilon)$ with $\epsilon>0$,
       \begin{equation}\label{eq:almav}
         \frac{1}{\log w} \int_{x/w}^x |S(t)| \frac{dt}{t} =
         O\left(\frac{\epsilon^{-1/2}}{(\log \log x)^{3/4}}\right).\end{equation}
     \end{corollary}
     Again, this bound is $o(\text{trivial bound})$ when
     $k$ is outside the tails of the distribution of $\Omega(n)$. It is still unknown
     whether, e.g., $\lambda(p+1)=1$ (or $=-1$) for infinitely many $p$.



     It should be clear that the above corollaries are meant in part as
     paradigmatic examples. Analogous statements can be easily derived for
     multiplicative functions other than $\lambda$.
    
      \subsection{Strategy of proof}
      We will have to define our set $\mathscr{X}$ so as to exclude (among others)
      all integers $n$
      whose average number of prime divisors in $[H_0,H]$ is well
        above average; otherwise, even the
       apparently trivial bound $O(\mathscr{L})$ in
        (\ref{eq:mainthm}) would not hold.

       The next step (\S \ref{sec:eigentrace})
       is familiar: we want to show that, if our operator $A$ has
        a large eigenvalue $\lambda$, then all even powers of $A$ have
        large trace (and so do even powers of $A|_{\mathscr{X}}$). The usual procedure,
        in many contexts, is to show that $\lambda$ has high multiplicity
        $M$, implying that
        $\Tr A^{2 k} = \sum_i \lambda_i^{2 k} \geq M \lambda^{2 k}$
        for $\lambda_i$ the eigenvalues of $A$.
        We are not quite in that case, but we will be able to show that,
        either we can define $\mathscr{X}$ so that $A|_{\mathscr{X}}$ has no large eigenvalues
        (thus establishing the main theorem), or
        there are many orthogonal vectors $v_i$ such that
        $\langle v_i, A|_{\mathscr{X}} v_i\rangle$ is large. The reason here is
        that our graph is rather ``local'', i.e.,
        the edges of our graph are short; the way
        we are defining $\mathscr{X}$ also helps. As a result, we obtain that, if the main theorem does not hold, then
        $\Tr (A|_{\mathscr{X}})^{2 k}$ is large.

        Of course, $\Tr (A|_{\mathscr{X}})^{2 k}$ can also be expressed as
        a sum over closed
        walks of length $2 k$ in the graph $(\Gamma\cup \Gamma')|_{\mathscr{X}}$,
        much like $(\Ad_\Gamma)^{2 k}$ is the number of closed walks of length
        $2 k$ in $\Gamma$. In effect, edges of $\Gamma \cup \Gamma'$
        have the weight $1-1/p$ (if they are edges of both $\Gamma$ and
        $\Gamma'$) or $-1/p$ (if they are only edges of $\Gamma$). It is
        not hard to see that the resulting cancellation implies that,
        in $\Tr A^{2 k}$, we can consider only walks where every edge length $p$
        that appears, appears at least twice; the total contribution
        of all other walks is very nearly $0$. The question is how to
        show that roughly the same happens when we consider
        $\Tr (A|_{\mathscr{X}})^{2 k}$.

        In this matter,
        the effect of excluding from $\mathscr{X}$ all
        integers $n$ with many prime divisors in $I$
        is significant, but can be dealt with by an essentially standard
        application of the Kubilius model (\S \ref{subs:kubilius}),
        together with multiple contour
        integration (\S \ref{subs:cancap}).
        Cancellation is not as total as before, but only
        those walks
        where few edge lengths $p$ appear only once survive.
        
        We will also decide to exclude from $\mathscr{X}$ all $n$ that can
be the starting point of
walks in which the first prime edge length reappears after a relatively short
but non-trivial
sequence of steps.
(Here ``trivial'' means ``reducing to an empty word'', where, for walks,
``reducing'' means essentially the same as for words.)
        It will be easy to show that such $n$ are fairly rare. What will
        be harder is showing that they are well-distributed in arithmetic
        progressions. We will do so by constructing a combinatorial sieve
        excluding certain congruence classes to {\em composite} moduli
        (\S \ref{se:abstr}--\ref{subs:rota}), and then
        using it as an enveloping sieve. Using composite moduli in inputs
        to a sieve is not traditional. The alert reader will
        realize that a difficulty arises
         due to the fact that two distinct sets of conditions can have the
        same logical conjunction. We address it in \ref{subs:rota}
        by cancellation in a combinatorial context, using
        what amounts to Rota's cross-cut theorem.

        We apply the sieve above to our context in \S \ref{sec:sievrecur}.
        Doing so will involve {\em sieve graphs} (\S \ref{subs:contrert})
        and a careful use of the notion of redundancy to ensure that we
        have enough independent conditions to bound our error terms well.

        In the end, in \S \ref{sec:returns},
        everything reduces to what is essentially the following
        problem: we are to bound the number of closed walks of even length
        $\leq 2 k$ on $\Gamma$ such that almost every
        prime length $p$ that appears,
        appears at least twice, but never in close succession (unless
        the appearances are consecutive, or consecutive after reduction).

        Actually, the main result in \S \ref{sec:returns} also involves
        some other walks, of special kinds. In \S \ref{sec:praeludium},
        we show that their contribution is very small. Doing so requires
        some very basic tools from geometry of numbers and linear
        algebra. The procedure will not be particularly difficult, as it
        will be easy to show that we have enough linearly independent
        divisibility conditions (so to speak) constraining our variables
        for the total number of possibilities to be small.

        We come to the heart of the proof
        (\S \ref{sec:main}). We are to count closed walks where almost
        every prime length $p$ appears at least twice; moreover, any two
        appearances of $p$ impose a divisibility condition on the sum
        of the steps taken between those two appearances (or else
        the walk must pay a ``fine'') -- and the succession of those steps must be either
        fairly long or trivial.
        
        Every walk
        may be said to have a {\em shape}, meaning the pair
        $(\sim,\vec{\sigma})$, where $\sim$ is the equivalence relation
        on indices defined by when primes repeat
        ($i\sim j \Leftrightarrow p_i=p_j$), and $\vec{\sigma}$ is
        the vector whose entries are the signs of the steps in the walk.
        When we consider all walks
        of a given shape, we see that their steps $p_i$ are variables
        satisfying a large system of divisibility conditions.
        Part of the difficulty is of course that the same variables
        $p_i$ appear as divisors and in the sums in the dividends.
        There is also the issue of ensuring that we have enough
        {\em independent} conditions. Once we manage to solve
        these two issues, we will be done, by some simple geometry of numbers.

        It will be enough to find within the matrix describing our system
        a submatrix of large rank such that the set of row indices
        and the set of column indices are disjoint.
        In order to show that such a submatrix exists, we
         define a new graph
         $\mathscr{G}=\mathscr{G}_{\sim,\vec{\sigma}}$,
         corresponding to closed walks of a given shape $(\sim,\vec{\sigma})$,
        and show that the rank we just mentioned can be bounded from below
        in terms of number of leaves of an arbitrary
         spanning tree of $\mathscr{G}$. 
         We will then use a standard result in graph theory showing that there
         exists a spanning tree with many leaves, provided that
         our tree has enough vertices of degree $\geq 3$. If $\mathscr{G}$
        has few such vertices, we bound the contribution of the walks
        corresponding to $\mathscr{G}$ in a different way, showing that
        walk shapes inducing graphs $\mathscr{G}$ with few vertices of degree
        $\geq 3$ are rare.
        If $\mathscr{G}$ has enough such vertices, then, after
        an additional argument to take care of the fact that our
        walk need not be reduced, we obtain that we do have a system of
        large rank, and we are done.
        
        What follows (\S \ref{sec:conseq}) is applications. It is here
        (\S \ref{subs:divcond}--\ref{subs:antoro})
        that we use the kind of result
        on averages in short intervals that originated
        in \cite{MR3488742}.
        
      \subsection{Notation}  
      By
      $f(x) = O(g(x))$ we mean, as is usual, that there exists a constant $C>0$
  such that $|f(x)|\leq C g(x)$ for all large enough $x$. We write
  $f(x) = O_{\eta,\epsilon}(g(x))$ (say) if the implied constant $C$ may depend
  on $\eta$ and $\epsilon$, but depends on nothing else. Analogously,
  $f(x) = o(g(x))$ means that $\lim_{x\to \infty} f(x)/g(x) = 0$
  (and $g(x)>0$ for $x$ large enough). We use $f(x) \ll g(x)$ to mean
  $f(x) = O(g(x))$. Lastly, we write
  $f(n)=O^*(g(n))$ if $|f(n)|\leq g(n)$.

  Given an equivalence relation $\sim$ on a set $I$, we write $[i]$
  for the equivalence class of an element $i\in I$. Given a collection
  $\mathscr{S}$
  of subsets of a set $X$, we write $\bigcup \mathscr{S}$ for the union of all
  sets in $\mathscr{S}$, and $\bigcap \mathscr{S}$ for the intersection of all
  sets in $\mathscr{S}$. By convention, if $\mathscr{S}=\emptyset$, then
  $\bigcap \mathscr{S} = X$. Write $\mathscr{S}^\cap$ for
   $\{\bigcap \mathscr{S}': \mathscr{S}'\subset \mathscr{S}\}$.

  Unless we state the contrary, a {\em graph}, for us, will be an
  undirected, finite graph with no multiple edges and no loops.

As is usual, we write $\omega(n)$
for the number of prime divisors of a number $n$,
and $\Omega(n)$ for the number of its prime divisors counted
with multiplicity.
  Given a set of primes $\mathbf{P}$, we define $\omega_{\mathbf{P}}(n)$
  to be the number of prime divisors of $n$ that lie in $\mathbf{P}$.
  We write $v_p$ for the $p$-adic valuation: for $n$ a non-zero integer,
  $v_p(n)$ equals the greatest $k$ such that $p^k|n$.
  We adopt the notation
  $\mathfrak{q}(P)$ for the modulus of an arithmetic progression $P$.
  By convention, $\mathfrak{q}(\emptyset)=0$ and $\omega(0)=\infty$.

Given a function $W:[0,\infty)\to \mathbb{C}$, we denote by
  $\widetilde{W}(s)$  the {\em Laplace transform}
    \[\widetilde{W}(s) = \int_{0}^{\infty} W(x) e^{-x s} dx.\]

   Given a proposition $P$, we
      define $\mathbf{1}_P$ 
      to be $0$ if $P$ is false and $1$ if $P$ is true. We also
    write $\mathbf{1}(P)$ as a synonym of $\mathbf{1}_P$.

    A {\em singleton} is an equivalence class with exactly one element.
  \subsection{Acknowledgments}
  H.~A.~Helfgott was supported by his European Research Council Consolidator grant
  (Grant ID: 648329; codename GRANT) and by his Humboldt professorship.
   M.~Radziwi{\l}{\l}
  was supported by a Sloan Fellowship and NSF grant DMS-1902063.
  The authors also
  thank MSRI (Berkeley) and AIM (San Jos\'e) as well as their home
  institutions for hosting them during visits. They are grateful to several
  colleagues who gave them helpful answers and references, mostly on
  MathOverflow: Yves Cornulier, Hailong Dao, R. van Dobben de Bruyn,
Shmuel Friedland,  Oleksiy Klurman,
  Dimitris Koukoulopoulos,
  Achim Krause, Lek-Heng Lim,
  Michael Magee, Brendan McKay, Anton Mellit, Ryan O'Donnell,
  Fedor Petrov, Federico Poloni, Geoff Robinson,
  Will Sawin, Ilya Shkredov,
  Lior Silberman, G\'erald Tenenbaum, Adrian Ubis, André Uschmajew
  and Gjerji Zaimi, and pseudonymous users 
  BS., MTyson, user174768, user174996, vidyarthi and 2734364041, among others. They would also like to thank
  Kaisa Matom\"aki for early discussions and later helpful remarks.
  H.~A.~Helfgott would also like to express his deep appreciation
  to those graduate students and postdocs
  at G\"ottingen who, during the COVID-19 pandemic, attended 
  two semester-long virtual
  lecture courses he gave on the proof as it was still taking shape.

   \section{Eigenvalue bounds from trace bounds}\label{sec:eigentrace}

We will work with the space of functions $f:\mathbf{N}\to \mathbb{C}$,
    with the inner product
\begin{equation}\label{eq:fuenvi}
\langle f,g \rangle = \frac{1}{N} \sum_{n\in \mathbf{N}} f(n) \overline{g(n)},
\end{equation}
where $\mathbf{N}\subset \mathbb{Z}$ is a segment
$\{N_0+1,\dotsc,N_0+N\}$. (We say that the {\em length} $|\mathbf{N}|$ of
$\mathbf{N}$ is $N$.)

We would like to show that, for a certain kind of real symmetric operator $A$,
either there is a large set $\mathscr{X}$ such that the restriction
$A|_{\mathscr{X}}$ has no large eigenvalues, or the traces $\Tr A^{2 k}$
of powers $A^{2 k}$
are large. 

\begin{lemma}\label{lem:restrictor}
  Let $A=(a_{i,j})_{i,j\in \mathbf{N}}$ be a real symmetric matrix,
  where $\mathbf{N}\subset \mathbb{Z}$ is a segment
  of length $N$. Assume that, for some
  $H, L>0$,
  \begin{enumerate}[(i)]
  \item\label{it:restr1} $\sum_{j\in \mathbf{N}} |a_{i,j}|\leq L$ for every $i\in \mathbf{N}$,
  \item\label{it:restr2} $a_{i,j}=0$ whenever $|i-j|>H$.
  \end{enumerate}
  Let $f:\mathbf{N}\to \mathbb{C}$ with $|f|_2^2=1$ and
  $|\langle f,A f\rangle|\geq \alpha$.
  
  Then there is a
  segment $I \subset \mathbf{N}$
   of length $\leq 4 \lceil L /\alpha\rceil H$ such
that
\[|\langle f|_I, A (f|_I)\rangle|\geq \frac{\alpha}{2} \left|f|_I\right|^2.\]
\end{lemma}
Here we are seeing $f$ as a vector, and so $A$ defines a
linear operator taking $f$ to $A f$.
\begin{proof}
  Consider a partition of $\mathbf{N}$
  inducing an equivalence relation $\sim$. Define the linear operator
  $A_\sim$ on functions $g:\mathbf{N}\to \mathbb{C}$ by
the matrix $(a_{i,j}^\sim)_{i,j\in \mathbf{N}}$ with $a_{i,j}^\sim = a_{i,j}$ when $i\sim j$
and $a_{i,j}^\sim=0$ otherwise.
Write $\partial_\sim$ for the set of $n\in \mathbf{N}$ such that there
is an $m\in \mathbf{N}$ with $n\not\sim m$ and $a_{n,m}\ne 0$.
Then
  \begin{equation}\label{eq:ipanema}\begin{aligned}
  |\langle g, (A_\sim-A) g\rangle| &= \left|\mathop{\sum_{i,j\in \mathbf{N}}}_{
    i\not\sim j} a_{i,j} g(i) \overline{g(j)}\right| \leq
  \mathop{\sum_{i,j\in \mathbf{N}}}_{i\not\sim j}
  \frac{|g(i)|^2+|g(j)|^2}{2} |a_{i,j}|\leq L \sum_{n\in \partial_\sim} |g(n)|^2.\end{aligned}\end{equation}

  Define $C = 2 \lceil L/\alpha\rceil$.
  Let $a\in \{1,3,\dotsc, 2 C - 1\}$.
  Consider the partition of
  $\mathbf{N}$ into segments of the form
    \begin{equation}\begin{aligned}\label{eq:partomime}
        &\{N + 1, N + 2,\dotsc,N + a H\},\\
        &\{N + (a + 2 C j) H + 1,\dotsc, N+ (a + 2 C (j+1) ) H\}\;\;\;\;
        \text{for $0\leq j\leq m-1$,}
        \\&
      \{N + (a + 2 C m) H +1,\dotsc, 2 N\},
    \end{aligned}\end{equation}
    where all segments are of length $2 C H$,
    except possibly for the first one, which is shorter,
    and the last one, which may be shorter.
    (Here $m$ is the unique integer for which the segment lengths can have
    these properties.)
    
     For the corresponding equivalence relation $\sim_a$,
    $$\partial_{\sim_a} \subset \partial_{a} := \bigcup_{j=0}^{m}\;
    (N + (a + 2 C j) H +\{-(H - 1), -(H - 2), \ldots, H\}).$$
    It is clear that $\partial_{1},\partial_{3},\dotsc,
    \partial_{2 C - 1}$ are disjoint. Hence, by
    pigeonhole, there is an
    $a\in \{1,3,\dotsc,2 C - 1\}$ such that $\sum_{n\in \partial_a} |f(n)|^2 \leq 1/C$. We choose that $a$ and work with the corresponding partition
  $P$ defined in  (\ref{eq:partomime}).
    Then, by \eqref{eq:ipanema},
    $|\langle f,A_\sim f\rangle|\geq
|\langle f,Af\rangle|- |\langle f,(A_\sim - A) f\rangle| \geq 
    \alpha - L/C \geq \alpha/2$.

    It remains to show that
    $\left|\langle f|_I, A f|_I\rangle\right| \geq \frac{\alpha}{2}
    \left|f|_I\right|_2^2
    $ for some $I\in P$. If not, then
    \[|\langle f,A_\sim f\rangle| =
    \sum_{I\in P} \left|\langle f|_I, A f|_I\rangle\right| < \frac{\alpha}{2}
    \sum_{I\in P} \left| f_I\right|_2^2 = \frac{\alpha}{2}
    \sum_{n\in \mathbf{N}} |f(n)|^2 = \frac{\alpha}{2},\]
    giving us a contradiction to $|\langle f,A_\sim f\rangle| \geq \alpha/2$.
\end{proof}

\begin{prop}\label{prop:isoeig}
  Let $\mathbf{N}\subset \mathbb{Z}$ be a segment of length $N$.
  Let $A=(a_{i,j})_{i,j\in \mathbf{N}}$ be a real symmetric matrix
  obeying conditions \eqref{it:restr1}--\eqref{it:restr2} in Lemma
  \ref{lem:restrictor} for some $H,L>0$.
  
  Then, for any $\alpha,\epsilon>0$, either
  \begin{enumerate}[(i)]
  \item\label{it:hru1} there is a subset
    $\mathscr{E}\subset \mathbf{N}$ with $|\mathscr{E}|\leq \epsilon N$ such
    that every eigenvalue of
      $A|_{\mathbf{N}\setminus \mathscr{E}}$ has absolute value $\leq \alpha$, or
    \item\label{it:hru2} there are $\geq \beta N /H$ orthogonal
      functions $g:\mathbf{N}\to \mathbb{C}$ with $|g|_2^2=1$ and
      $|\langle g, A g\rangle|\geq \alpha/2$, where
      $\beta = \epsilon \alpha/8 L$.
      \end{enumerate}
\end{prop}
\begin{proof}
  For every $f:\mathbf{N}\to \mathbb{C}$ with $|f|_2^2=1$,
  support in $X$ and
  $\left|\langle f,A f\rangle\right|\geq \alpha > 0$, Lemma
  \ref{lem:restrictor} gives us a 
  segment $I\subset \mathbf{N}$
   of length $\leq \ell = 4 \lceil L/\alpha\rceil H$
such
that
$|\langle f|_I, A f|_I\rangle|\geq \frac{\alpha}{2} \left|f|_I\right|^2$.
We can assume that $L>\alpha$, as otherwise conclusion \eqref{it:hru1}
holds with $\mathscr{E}=\emptyset$. Thus $\ell \leq 8 L H/\alpha$.

Let $\mathbf{I}$ be the set of all intervals $I \subset \mathbf{N}$ of length
$\leq \ell$ for which there exists a $g : \mathbf{N} \rightarrow \mathbb{C}$
such that $|\langle g_{|I}, A g_{|I} \rangle| > \frac{\alpha}{2} |g_{|I}|^2$. Let
also $\mathscr{E} = \bigcup_{I \in \mathbf{I}} I.$
Then, for any $f:\mathbf{N}\to \mathbb{C}$ with
$|f|_2^2=1$ and support on $\mathbf{N}\setminus \mathscr{E}$,
we know that $|\langle f, A f\rangle|<\alpha$ (or else we would get
a contradiction by invoking Lemma \ref{lem:restrictor}).
It is easy to see that we can choose a subset $\mathbf{I}'\subset
\mathbf{I}$ consisting of
$|\mathbf{I}'|\geq |\mathscr{E}|/2 \lfloor \ell\rfloor$ disjoint intervals.
For each $I\in \mathbf{I}'$, there exists, by definition, a function $g$ supported on $I$ with $|g|_2^2=1$ and $|\langle g,A g\rangle|\geq \alpha/2$.
Functions $g$ corresponding to different $I\in \mathbf{I}'$ are obviously orthogonal to each other.

Thus, for any $\epsilon>0$, we know that either $|\mathscr{E}|\leq \epsilon N$, or there are $> \epsilon N/2 \lfloor \ell\rfloor >
\epsilon \alpha N/ 16 L H$
orthogonal functions $g$ with $|g|_2^2=1$ and $|\langle g,A g\rangle| \geq \alpha/2$.
\end{proof}

\begin{lemma}\label{lem:hoelder}
  Let $A$ be a real symmetric operator. Then, for any vector $v$
  with $|v|_2 = 1$ and any positive integer $k$,
\[\langle v, A^{2k} v\rangle\geq
  \left|\langle v, A v\rangle \right|^{2 k}.\]
\end{lemma}
\begin{proof}
  Let $v_i$ be an orthonormal basis of eigenvectors of $A$, so that
  $A v_i = \lambda_i v_i$. Then
  \[\langle v, A v\rangle = \sum_i \lambda_i \cdot |\langle v,v_i\rangle|^2.\]
  Because $|v|_2=1$, 
  $\sum_i |\langle v, v_i\rangle|^2=1$. Therefore, by H\"older's inequality,
  \[|\langle v, A v\rangle|^{2 k} = \left| \sum_i \lambda_i
  |\langle v, v_i\rangle|^2 \right|^{2 k} \leq
  \sum_i \lambda_i^{2 k} |\langle v,v_i\rangle|^2 = \langle v, A^{2 k} v\rangle.
  \]
\end{proof}

We can come to our objective in this section, namely, to show that, for
certain kinds of symmetric operators $A$ on functions
$f:\mathbf{N}\to \mathbb{R}$, if we can 
bound a trace $\Tr A^{2 k}$, then we know that
there exists some small $\mathscr{E}\subset \mathbf{N}$ such that
all eigenvalues of $A|_{\mathbf{N}\setminus \mathscr{E}}$ are small.

\begin{prop}\label{prop:rabmento}
    Let $\mathbf{N}\subset \mathbb{Z}$ be a segment of length $N$.
  Let $A=(a_{i,j})_{i,j\in \mathbf{N}}$ be a real symmetric matrix
  obeying \eqref{it:restr1}--\eqref{it:restr2} in Lemma
  \ref{lem:restrictor} for some $H,L>0$.
  

  Assume that, for some $k\geq 1$ and $\alpha,\epsilon>0$,
  \begin{equation}\label{eq:lilinput}
    \Tr A^{2 k} < \frac{\epsilon \alpha}{8 L H}
  \left(\frac{\alpha}{2}\right)^{2 k} N\end{equation}
  Then there is an $\mathscr{E}\subset \mathbf{N}$ with
  $|\mathscr{E}|\leq \epsilon N$
  such that every eigenvalue of $A|_{\mathbf{N}\setminus \mathscr{E}}$ has absolute value
  $\leq \alpha$.
\end{prop}
\begin{proof}
  Apply Prop.~\ref{prop:isoeig}. If conclusion (\ref{it:hru1}) there holds,
  we are done. Assume, then, that conclusion (\ref{it:hru2}) holds;
  for each $g$ in that conclusion, $|\langle g, A g\rangle| \geq
  \alpha/2$. Applying Lemma \ref{lem:hoelder}, we see
  we get that $|\langle g,A^{2 k} g\rangle|\geq (\alpha/2)^{2 k}$.
  Since $A^2$ has full non-negative real spectrum, it follows that
  \[\Tr A^{2 k} \geq \frac{\epsilon \alpha}{8 L}
  \left(\frac{\alpha}{2}\right)^{2 k} \frac{N}{H},\]
  contradicting our assumption.
\end{proof}

Our intention is to apply Prop.~\ref{prop:rabmento}, not quite
to the operator
$A = \Ad_\Gamma - \Ad_{\Gamma'}$ 
from the introduction, but
to an operator $A|_{X}$, where  $X\subset \mathbf{N}$
is such that $\mathbf{N}\setminus X$ is small and every element $n\in X$ has $\leq K\mathscr{L}$ prime divisors. Conditions
\eqref{it:restr1}--\eqref{it:restr2} in Lemma
\ref{lem:restrictor} will then hold with $L = (K + 1) \mathscr{L}$.

{\bf Remark.} Excluding integers with many prime divisors from our
set of vertices
is not just a convenience but a necessity. If $n_0\in \mathbf{N}$ has
$> K \mathscr{L}$ prime divisors, then, assuming $H\leq N/2$, for
$A$ as in the introduction,
  \[\langle 1_{n=n_0}, A^2 1_{n=n_0}\rangle 
  \geq \mathop{\sum_{p\in \mathbf{P}}}_{p|n}
  \left(1 - \frac{2}{p} + \frac{1}{p^2}\right) > (K-2) \mathscr{L},\]
  and so $A$ must have an eigenvalue $\lambda$ with
  $|\lambda|> \sqrt{(K-2) \mathscr{L}}$.

  {\bf Remark.} Let us discuss briefly the relative strength
  of Prop.~\ref{prop:rabmento}. If we had not done the work in this
  section, we would have to prove that $\Tr A^{2 k} < \alpha^{2 k}$
  and that would almost certainly be out of reach: we would have to work
  with $k$ larger than $\log N$, and then it would seem that we would have
  to consider moduli beyond $H^{2 \log N}$, which is much larger than $N$.
  On the other hand, if we could prove Prop.~\ref{prop:rabmento}
  with a much weaker assumption, with $H_0$ instead of $H$ in
  \eqref{eq:lilinput}, then our remaining work would be much easier; we would
  just need to gain a factor of $H_0$ or two from a single congruence condition, or
  two.
  Of course such a weak assumption does not seem realistic, since the locality
  of our graph is at scale roughly $H$, not $H_0$.
  
\section{Variations on inclusion-exclusion}

We will need some preparatory work. In \S \ref{se:abstr}, we will show
how the basic framework of a combinatorial sieve can be set out abstractly.
We shall later use that framework to sieve by conditions more general than
those usual in sieve theory.

To be able to sieve by congruence conditions
to composite moduli -- rather than prime moduli, as is traditional in sieve
theory -- we will need a way to have some control on the total contribution
made when distinct sets of conditions have the same conjunction
(e.g., $6|n \wedge 5|n$ is equivalent to $2|n \wedge 15|n$). The key
will be given by a simple consequence of Rota's cross-cut theorem
(\S \ref{subs:rota}).

We will find the Kubilius model \cite{zbMATH03216361} to be a useful
abstraction for our analytic work. We will derive the slight variant that
we need (\S \ref{subs:kubilius}), based, as is the usual version, on
the fundamental lemma of sieve theory.

\subsection{An abstract combinatorial sieve} \label{se:abstr}
Let $\mathbf{Q}$ be a finite set of propositions that an arbitrary integer may or may not fulfill. 
  Given a proposition $Q \in \mathbf{Q}$, and an integer $n$, we declare $Q(n)$
      to be true if $n$ fulfills $Q$ and false if it does not.
  Denote by $\mathbf{Q}(n) \in 2^{\mathbf{Q}}$ the set $\{Q\in \mathbf{Q}: Q(n)\; \text{is true}\}$,
  i.e., the set of propositions in $\mathbf{Q}$ fulfilled by $n$.
    Define $1_\emptyset:2^\mathbf{Q}\to \mathbb{R}$ by setting, for any $\mathbf{S} \subset \mathbf{Q}$, 
    \[1_\emptyset(\mathbf{S}) = \begin{cases} 1 &\text{if $\mathbf{S}$ empty,}\\
      0 &\text{otherwise.}\end{cases}\]
    We will be particularly interested in the case $\mathbf{S} = \mathbf{Q}(n)$ with $n$ varying. 
    By inclusion-exclusion,
    \[1_\emptyset(\mathbf{Q}(n)) = \sum_{\mathbf{T}\subset \mathbf{Q}(n)} (-1)^{|\mathbf{T}|}.\]
    In applications, this formula might not be particularly useful, as there are
    $2^{|\mathbf{Q}(n)|}$ conditions
    to probe, a number that is potentially exceedingly large. One may say
    that the main idea of sieve theory is to provide approximate
    versions of this formula with far fewer terms. Sieves are usually
    stated for propositions $Q$ such that $Q(n)$ is of the form
    ``$n$ is congruent to $a \pmod{p}$'',
    but, as we are about to see, one may study sieves
    in much more general terms.

    For any given function $g:2^\mathbf{Q}\to \{0,1\}$ we define an associated
    function $g^{\star}: 2^{\mathbf{Q}} \to \mathbb{Z}$ by 
    \begin{equation}\label{eq:deff}
      g^{\star}(\mathbf{S}) = \sum_{\mathbf{T}\subset \mathbf{S}} g(\mathbf{T}) (-1)^{|\mathbf{T}|}. \end{equation} Notice that $g^{\star \star} \equiv g$. 
    In analogy with combinatorial sieves (e.g.,
    \cite[\S 6.2]{zbMATH02239783}), our aim will be to choose a
    $g$ with relatively small support, and such that $g^{\star}(\mathbf{Q}(n))$ 
    approximates $1_\emptyset(\mathbf{Q}(n))$ well, in the sense that
    \[\sum_{n\in \mathbf{N}} \left|1_\emptyset(\mathbf{Q}(n)) - g^{\star}(\mathbf{Q}(n))\right|\]
    is small.
    The following identity is fundamental in this respect. 

    \begin{lemma}\label{lem:whithe}
  Let $g:2^\mathbf{Q}\to \{0,1\}$, where $\mathbf{Q}$ is a finite set of conditions that
  an arbitrary integer may or may not fulfill. Assume $g(\emptyset)=1$,
  and define $g^{\star}:2^\mathbf{Q}\to \mathbb{Z}$
  as in (\ref{eq:deff}). 
  Choose a total ordering for $\mathbf{Q}$. Then
  \begin{equation}\label{eq:dadardo}
    1_\emptyset(\mathbf{Q}(n)) = g^{\star}(\mathbf{Q}(n)) + 
\sum_{\substack{\emptyset \neq \mathbf{S}\subset \mathbf{Q} \\ Q(n), \forall
        Q\in \mathbf{S} \\ Q<\min(\mathbf{S}) \Rightarrow \neg Q(n)}}
(-1)^{|\mathbf{S}|} (g(\mathbf{S}\setminus \{\min(\mathbf{S})\}) - g(\mathbf{S})).
\end{equation}  
\end{lemma}
    \begin{proof}
    We can assume $g(\emptyset)=1$, and so $g^{\star}(\emptyset)=1$.
    Tautologically,
    \begin{align*}
      g^{\star}(\mathbf{Q}(n)) = \mathop{\mathop{\sum_{\mathbf{T}\subset \mathbf{Q}}}_{Q(n)\; \forall
      Q\in \mathbf{T}}}_{\neg Q(n)\; \forall
      Q\in \mathbf{Q}\setminus \mathbf{T}} g^{\star}(\mathbf{T}) =
    1_\emptyset(\mathbf{Q}(n)) 
   +\sum_{\substack{\mathbf{T}\subset \mathbf{Q}, \mathbf{T} \ne \emptyset \\ Q(n), \forall
       Q\in \mathbf{T} \\ \neg Q(n), \forall Q\in \mathbf{Q}\setminus \mathbf{T}}} g^{\star}(\mathbf{T}).    \end{align*}

    Now
    \begin{align}
      \sum_{\substack{\mathbf{T}\subset \mathbf{Q}, \mathbf{T}\ne
\emptyset\\ Q(n), \forall
      Q\in \mathbf{T} \\ \neg Q(n), \forall
      Q\in \mathbf{Q}\setminus \mathbf{T}}} g^{\star}(\mathbf{T}) &=
    \sum_{Q_0\in \mathbf{Q}}
    \mathop{\mathop{\sum_{\mathbf{S}\subset \mathbf{Q}:
          (Q>Q_0 \forall Q\in \mathbf{S})
      }}_{
      Q(n)\; \forall
      Q\in \{Q_0\}\cup \mathbf{S}}}_{\neg Q(n)\; \forall
      Q\in \mathbf{Q}\setminus (\{Q_0\}\cup \mathbf{S})}
    g^{\star}(\{Q_0\}\cup \mathbf{S})\\
    &= \label{eq:cilinea}    \sum_{Q_0\in \mathbf{Q}}
    \mathop{\mathop{\sum_{\mathbf{S}\subset \mathbf{Q}:
              (Q>Q_0 \forall Q\in \mathbf{S})}}_{
      Q(n)\; \forall
      Q\in \{Q_0\}\cup \mathbf{S}}}_{\neg Q(n)\; \forall
      Q\in \mathbf{Q}\setminus (\{Q_0\}\cup \mathbf{S})}
    \sum_{\mathbf{U}\subset \mathbf{S}} (-1)^{|\mathbf{U}|} (g(\mathbf{U}) - g(\{Q_0\} \cup \mathbf{U})),
    \end{align}
    where we use the definition of $g^{\star}$ in  (\ref{eq:deff}).
    Changing the order of summation, we see that the expression in
    \eqref{eq:cilinea} equals
    \begin{align*}
      &\sum_{\substack{Q_0 \in \mathbf{Q}}}
      \sum_{\substack{\mathbf{U} \subset \mathbf{Q} \\
Q>Q_0 \forall Q\in \mathbf{U} \\ Q(n)\, \forall Q \in \mathbf{U}}}
    (-1)^{|\mathbf{U}|} (g(\mathbf{U}) - g(\{Q_0\} \cup \mathbf{U}))
    \sum_{\substack{\mathbf{S}: \mathbf{U} \subset \mathbf{S} \subset \mathbf{Q}\\ Q>Q_0 \forall Q \in \mathbf{S} \\
      Q(n) \forall
      Q\in \{Q_0\}\cup \mathbf{S} \\ \neg Q(n) \forall
      Q\in \mathbf{Q}\setminus (\{Q_0\}\cup \mathbf{S})}} 1.
    \end{align*}
    The inner sum over $\mathbf{S}$ is simply equal to the indicator function of the event $Q_0 = \min \{Q\in \mathbf{Q}: Q(n)\}$ (or,
    equivalently, $Q_0(n)\wedge (Q<Q_0\Rightarrow \neg Q(n))$).
    Therefore, the above equals
    \begin{align*}
      \sum_{\substack{Q_0 \in \mathbf{Q}:\, Q_0(n) \\ Q < Q_0 \Rightarrow \neg Q(n)}}
      \sum_{\substack{\mathbf{U} \subset \mathbf{Q} \\ Q(n)\, \forall Q \in \mathbf{U}\\
      Q_0\not\in \mathbf{U}}}
            (-1)^{|\mathbf{U}|} ( g(\mathbf{U}) - g(\{Q_0\} \cup \mathbf{U})).
    \end{align*}
    Letting $\mathbf{S} = \{Q_0\}\cup \mathbf{U}$, we obtain our
    conclusion \eqref{eq:dadardo}.
        \end{proof}

    The idea is to choose $g$ so that (a) the sum in (\ref{eq:dadardo}) is small on average, (b) the support of $g$ is moderate. Then (\ref{eq:dadardo}) furnishes an approximation to
$1_\emptyset\circ \mathbf{Q}(n)$ by a linear combination of a moderate number of simpler 
functions $1_{\{n: Q(n) \forall Q\in \mathbf{T}\}}$, where $\mathbf{T}\subset \mathbf{Q}$ is such
that $g(\mathbf{T})=1$.

    \subsection{Sieving by composite moduli}\label{subs:rota}

    We will now work with a finite collection $\mathscr{Q}$ of arithmetic
    progressions. We can apply our framework from \ref{se:abstr}
    to the set $\mathbf{Q}$ of propositions $n\in P$ for $P\in \mathscr{Q}$,
    and thus obtain an approximation to $1_{n\not\in P\;\forall P\in \mathscr{Q}}$.
    The approximation\footnote{Since we are not just trying to
      count $\{n\in P\;\forall P\in \mathscr{Q}\}\cap \mathbf{N}$
      for some interval $\mathbf{N}$, we can say that we view our sieve
      as an {\em enveloping sieve}. The term goes back to
      \cite{zbMATH03512236} and \cite{zbMATH00871718}; the strategy has
      become more
      widely known thanks to Goldston-Y{\i}ld{\i}r{\i}m and Green-Tao.}
    will be given
    in terms of a moderate number
    of sets $\{n: n\in P\; \forall P\in \mathscr{T}\} = \bigcap \mathscr{T}$
    for $\mathscr{T}\subset \mathscr{Q}$. The sets $\bigcap \mathscr{T}$
    are of course just arithmetic
    progressions.

    What is non-traditional here is that the moduli
    $\mathfrak{q}(P)$ of the arithmetic
      progressions $P\in \mathscr{Q}$ are not necessarily prime.
      One problem we will
      have to address is that we can have two, or very many, subsets $S$
      of $\mathscr{Q}$ with the same intersection $\bigcap S$
      (e.g., $6\mathbb{Z}\cap 5\mathbb{Z} = 2 \mathbb{Z}\cap 15\mathbb{Z}$,
      to repeat an example). This is not a problem that appears for prime moduli.
The following lemma, a simple consequence of Rota's cross-cut theorem
      \cite[Thm.~3]{zbMATH03195424}, will be key to addressing the difficulty.
    
    \begin{lemma}\label{lem:crosscut}
      Let $\mathscr{Q}$ be a collection of subsets of a finite set $X$.
      Then,
      $$
      \left|\mathop{\sum_{\mathscr{S} \subset \mathscr{Q}}}_{
          \bigcup \mathscr{S} = X} (-1)^{|\mathscr{S}|} \right| \leq 2^{|X|}
      $$
    \end{lemma}
    Let us give a self-contained proof.
    \begin{proof}
      We can rewrite the above sum as
      $$
      \sum_{\mathscr{S} \subset \mathscr{Q}} (-1)^{|\mathscr{S}|} \sum_{\bigcup \mathscr{S} \subset W \subset X} (-1)^{|W|} = \sum_{W \subset X} (-1)^{|W|} \sum_{\substack{\mathscr{S} \subset \mathscr{Q} \\ \bigcup \mathscr{S} \subset W}} (-1)^{|\mathscr{S}|}
      $$
      Again by inclusion-exclusion, the inner sum is equal to $1$ if there exists no
      element of $\mathscr{Q}$ contained in $W$,  and $0$ otherwise.
      Either way, the
      absolute value of the inner sum is $\leq 1$. Since the outer sum
      has at most $2^{|X|}$ terms, the claim follows. 
    \end{proof}

\begin{prop} \label{prop:sieve}
      Let $\mathscr{Q}$ be a finite collection of distinct arithmetic progressions
      in $\mathbb{Z}$ with square-free moduli. Let $\mathfrak{D}$ be a non-empty subset of
      $\mathscr{Q}^\cap = \{\bigcap \mathscr{S}:
      \mathscr{S}\subset \mathscr{Q}\}$ with $\emptyset\notin \mathfrak{D}$.
      Assume $\mathfrak{D}$ is closed under containment, i.e.,
      if $S \in \mathfrak{D}$ and $S \subset S'$ with $S' \in \mathscr{Q}^{\cap}$ then $S' \in \mathfrak{D}$. 
      Then
      \begin{equation}\label{eq:prilev}
        \begin{aligned} 1_{n\not\in P\;\forall P\in \mathscr{Q}}
      &= F_{\mathscr{Q},\mathfrak{D}}(n) + 
          O^*\left(\sum_{R\in \partial \mathfrak{D}}
          2^{\omega(\mathfrak{q}(R))} 1_{n\in R}\right) \\ &=
          F_{\mathscr{Q},\mathfrak{D}}(n) +
O^*\left(\sum_{R\in \partial_{\textrm{out}} \mathfrak{D}}
          3^{\omega(\mathfrak{q}(R))} 1_{n\in R}\right),
                \end{aligned}\end{equation}
      where
      \begin{equation}\label{eq:ishi2}
        F_{\mathscr{Q},\mathfrak{D}}(n) = \mathop{\sum_{\mathscr{S}\subset \mathscr{Q}}}_{
      \bigcap \mathscr{S} \in \mathfrak{D}}
      (-1)^{|\mathscr{S}|}  1_{n\in \bigcap \mathscr{S}} =
\sum_{R\in \mathfrak{D}} c_R 1_{n\in R}
\end{equation}
      for some $c_R\in \mathbb{R}$ with $|c_R|\leq 2^{\omega(\mathfrak{q}(R))}$, and
\[\partial \mathfrak{D} = \{R\in \mathfrak{D}:
      \exists P\in \mathscr{Q}\, \st P\cap R \not\in \mathfrak{D}\},\]
     \[\partial_\textrm{out} \mathfrak{D} = \{D\in \mathscr{Q}^\cap \setminus
\mathfrak{D} : \exists P\in \mathscr{Q}, R \in \mathfrak{D}\, \text{s.t.}\, D = P \cap R \}.\]
The term $3^{\omega(\mathfrak{q}(R))}$ in (\ref{eq:prilev}) can be
replaced by
$2^{\omega(\mathfrak{q}(R))} (3/2)^{\max_{P\in \mathscr{Q}} \omega(\mathfrak{q}(P))}$.
\end{prop}
\begin{proof}
      Let $\mathbf{Q}$ be the set of
      propositions $n\in P$ for $P\in \mathscr{Q}$,
      and impose an arbitrary total ordering on $\mathbf{Q}$.
      Define a function $g$ on subsets $\mathbf{S}\subset \mathbf{Q}$ by
      \[g(\mathbf{S}) = \begin{cases} 1 &\text{if 
          $\bigcap \mathscr{S}\in \mathfrak{D}$,}\\
        0 &\text{otherwise},\end{cases}\]
      where $\mathscr{S}$ is the subset of $\mathscr{Q}$ corresponding to
      $\mathbf{S}$. Then, by Lemma \ref{lem:whithe},
      \begin{equation}\label{eq:ranon}
        \begin{aligned} 1_{n\not\in P\;\forall P\in \mathscr{Q}}
      &= \mathop{\sum_{\mathscr{S}\subset \mathscr{Q}}}_{
      \bigcap \mathscr{S} \in \mathfrak{D}}
      (-1)^{|\mathscr{S}|}  1_{n\in \bigcap \mathscr{S}} \;+
      \mathop{\mathop{\sum_{\mathscr{S}\subset \mathscr{Q}}}_{
          \bigcap \mathscr{S} \not\in \mathfrak{D}}}_{\bigcap
        (\mathscr{S}\setminus \min(\mathscr{S})) \in \mathfrak{D}}
      (-1)^{|\mathscr{S}|} 1_{n\in \bigcap \mathscr{S}}
      1_{n\notin \bigcup \{P\in \mathscr{Q}: P<\min(\mathscr{S})\}}.
      \end{aligned}\end{equation}
      
      Let us examine the contribution of a given
      $R\in \mathfrak{D}$ (that is, of all $\mathscr{S}\subset \mathfrak{D}$
      with $\bigcap \mathscr{S} = R$)
      to the first term in \eqref{eq:ranon}.
      Let $\mathscr{Q}_{R}$ be the collection of all
      $P\in \mathscr{Q}$ containing $R$. An element of
      $\mathscr{Q}_{R}$ is determined by its modulus, which
      divides the modulus $\mathfrak{q}(R)$ of $R$. Thus,
      $\mathscr{Q}_{R}$ induces a collection
      $\mathcal{Q}_R$ of subsets of
      $X_R = \{p: p|\mathfrak{q}(R)\}$.
      A subcollection of $\mathscr{Q}_R$ has intersection $R$
      if and only if the corresponding subcollection of $\mathcal{Q}_R$
      has $X_R$ as its union. Hence
      \[ \mathop{\sum_{\mathscr{S}\subset \mathscr{Q}}}_{
      \bigcap \mathscr{S} = R} (-1)^{|\mathscr{S}|}
      = \mathop{\sum_{\mathscr{S}\subset \mathscr{Q}_{R}}}_{
        \bigcap \mathscr{S} = R} (-1)^{|\mathscr{S}|}
      = \mathop{\sum_{\mathcal{S}\subset \mathcal{Q}_R}}_{
        \bigcup \mathcal{S} = X_R} (-1)^{|\mathcal{S}|},
          \]
          and, by Lemma \ref{lem:crosscut},
    \[\left|\mathop{\sum_{\mathcal{S}\subset \mathcal{Q}_R}}_{
      \bigcup \mathcal{S} = X_R} (-1)^{|\mathcal{S}|}\right|
    \leq 2^{|X_R|} = 2^{\omega(\mathfrak{q}(R))}
          \]
          Hence, the second equality in (\ref{eq:ishi2}) is valid.

          The second term in \eqref{eq:ranon} equals
          \begin{equation}\label{eq:abalon}\sum_{M\in \mathscr{Q}}
          \mathop{\mathop{\sum_{\mathscr{S}\subset
                \{P\in \mathscr{Q}: P > M\}}}_{
          \bigcap \mathscr{S}\in \mathfrak{D}}}_{M\cap \bigcap
            \mathscr{S} \not\in \mathfrak{D}}
          (-1)^{|\mathscr{S}|} 1_{n\in M\cap \bigcap \mathscr{S}}
          1_{n\notin \bigcup \{P\in \mathscr{Q}: P<M\}}.
          \end{equation}
          We can bound the contribution of given $M\in \mathscr{Q}$,
          $R\in \mathfrak{D}$ (with $M\cap R\not\in \mathfrak{D}$)
          proceeding as we just did, 
          replacing
          $\mathscr{Q}_R$ by the collection 
          $\mathscr{Q}_{R,>M}$ of all $P\in \mathscr{Q}$ containing $R$
          and satisfying $P>M$. We obtain
          \[\left|\mathop{\sum_{\mathscr{S}\subset
              \{P\in \mathscr{Q}: P > M\}}}_{\bigcap \mathscr{S} = R}
          (-1)^{|\mathscr{S}|}\right|\leq 2^{\omega(\mathfrak{q}(R))}.\]
          Hence, the expression in \eqref{eq:abalon} is
          \begin{equation}\label{eq:jocoso}
            O^*\left(\sum_{M\in \mathscr{Q}}
          \mathop{\sum_{R\in \mathfrak{D}}}_{M\cap R\notin \mathfrak{D}}
          2^{\omega(\mathfrak{q}(R))}
          1_{n\in M\cap R} 1_{n\notin \bigcup \{P\in \mathscr{Q}:
            P<M\}}\right).\end{equation}
          Given any $R\in \mathfrak{D}$ and any $n\in \mathbb{Z}$,
          at most one $M\in \mathscr{Q}$ can give us a non-zero term
          in (\ref{eq:jocoso}): if two $M,M'\in \mathscr{Q}$ with
          $M'<M$ did, then we would have $n\in M'\cap R$,
          and so $n\in M'$, giving us a contradiction to
          $n\notin \bigcup \{P\in \mathscr{Q}: P<M\}$. Thus,
          (\ref{eq:jocoso}) is bounded by
          \[O^*\left(\sum_{R\in \partial \mathfrak{D}}
          2^{\omega(\mathfrak{q}(R))} 1_{n\in R}\right).\]

          To obtain our other bound on \eqref{eq:jocoso}, we proceed
          as follows. Let $M\in \mathscr{Q}$, $R\in \mathfrak{D}$ with
          $D = M \cap R \not \in \mathfrak{D}$ be such that their
          contribution to \eqref{eq:jocoso} is non-zero. Then
          $n\in D$, and, for any $P < M$ with $P \in \mathscr{Q}$,
          we must have $n \not \in P$. Thus,
          there can be no $P \in \mathscr{Q}$ with $D \subset P$ and $P < M$;
          otherwise we would conclude that $n \not \in P$ and hence also $n \not \in D$, since $D \subset P$. In other words, $M$ has to equal $\min\{ P \in \mathscr{Q} : D \subset P\}$.

          Therefore, once $D \not \in \mathfrak{D}$ is fixed, $M$ is determined; it remains to count the number of distinct $R \in \mathfrak{D}$ for which $D = M \cap R$. Necessarily, $D \subset R$ and hence $\mathfrak{q}(R) | \mathfrak{q}(D)$. Now, $\mathfrak{q}(R)$ uniquely determines $R$, since $R\supset D$ and $D$ is fixed. By
          $$
          \sum_{\mathfrak{q}(R) | \mathfrak{q}(D)} 2^{\omega(\mathfrak{q}(R))} = 3^{\omega(\mathfrak{q}(D))},
          $$
          we conclude that the expression in \eqref{eq:jocoso} is bounded by
          \[O^*\left(
\sum_{D\in \partial_{\textrm{out}} \mathfrak{D}} 
3^{\omega(\mathfrak{q}(D))} 1_{n\in D}\right).\]

We can do a little better by noticing that $D = R \cap M$ implies the stronger condition that $[\mathfrak{q}(R), \mathfrak{q}(M)] = \mathfrak{q}(D)$. Since $\mathfrak{q}(R), \mathfrak{q}(D)$ and $\mathfrak{q}(M)$ are all square-free, we obtain that $(\mathfrak{q}(D) / \mathfrak{q}(M)) | \mathfrak{q}(R)
| \mathfrak{q}(D)$. As before, knowing $\mathfrak{q}(R)$ determines $R$, since $R \supset D$. Finally,
          $$
          \sum_{\substack{\frac{\mathfrak{q}(D)}{\mathfrak{q}(M)} | \mathfrak{q}(R) | \mathfrak{q}(D)}} \!\!\!\!\!\!\!\!2^{\omega(\mathfrak{q}(R))} = 2^{\omega(\mathfrak{q}(D) / \mathfrak{q}(M))} 3^{\omega(\mathfrak{q}(M))} = 2^{\omega(\mathfrak{q}(D))} (3/2)^{\omega(\mathfrak{q}(M))},
          $$
          and so, given that
$M=\{P\in \mathscr{Q}:  D\subset P\}$, we can bound
          the expression in 
          (\ref{eq:jocoso}) by
          \[O^*\left(
\sum_{D\in \partial_{\textrm{out}} \mathfrak{D}} 
2^{\omega(\mathfrak{q}(D))} (3/2)^{\mathfrak{q}(\min \{P\in \mathscr{Q}:
  D\subset P\})} 1_{n\in D}\right).\]
    \end{proof}

\subsection{The Kubilius model}\label{subs:kubilius}
In this subsection, we will set out
a slight generalization of the usual Kubilius model
(cf. \cite[Ch.~3]{zbMATH03670537}, \cite{zbMATH03216361}). A
generalization that is
``multi-dimensional'' in the same sense as ours can already be found
in \cite{zbMATH03216361}), but it would give weaker results in our context.\footnote{G. Tenenbaum kindly pointed out a third alternative in private
  communication. Our more primitive procedure still seems to give a somewhat
  better result in the end.}

Let $a + q \mathbb{Z}$ be an arithmetic progression with $q$ square-free.
Let $\alpha_i\in \mathbb{Z}$ and $\mathcal{P}_i\subset \mathbf{P} \backslash \{p \in \mathbf{P} : p | q\}$
be given for $1 \leq i \leq \ell$.
Write $\boldsymbol{\mathcal{P}} = (\mathcal{P}_1, \ldots, \mathcal{P}_{\ell})$.
For any $n$,
consider the subsets
$$
E_{a,q,N}(p; \boldsymbol{\mathcal{P}}, \delta_1(p), \ldots, \delta_{\ell}(p))  \subset \mathbb{Z}
$$
consisting of integers $n \equiv a \pmod{q}$ with $N < n \leq 2N$ for which $p | n + \alpha_{i}$ if $\delta_i(p) = 1$, $p \nmid n + \alpha_i$ if $\delta_i(p) = 0$, and $(n + \alpha_i, \prod_{p \in \mathcal{P}_i} p) = 1$
for all $1 \leq i \leq \ell$. The conditions $n \equiv a \pmod{q}$ and $(n + \alpha_i, \prod_{p \in \mathcal{P}_i} p) = 1$ are equivalent to requiring that $\delta_i(p) = \mathbf{1}_{p | a + \beta_i}$ for all $p | q$ and $\delta_i(p) = 0$ for all $p \in \mathcal{P}_i$. Since $\mathcal{P}_i$ does not contain primes $p | q$ these conditions are always consistent. 

We are interested in understanding the $\sigma$-algebra $\mathfrak{B}$ generated by
$$
E_{a,q,N}(p; \boldsymbol{\mathcal{P}}, \delta_1(p), \ldots, \delta_{\ell}(p))
$$
as $p$ ranges over $\mathbf{P}$ and each $\delta_i(p)$ ranges over
$\{0,1\}$. 
Notice that, for any fixed $p \in \mathbf{P}$,
\begin{equation} \label{eq:disjoint}
  \{ n \in (a + q\mathbb{Z}) \cap (N, 2N]:
(n + \alpha_i, \prod_{p \in \mathcal{P}_i} p) = 1
    \ \forall 1 \leq i \leq \ell \}
   = \bigcup E_{a,q,N}(p; \boldsymbol{\mathcal{P}}, \delta_1(p), \ldots, \delta_{\ell}(p))
\end{equation}
where
 $\delta_i(p)$ ranges over $\{0,1\}$ while obeying the following constraints: 
  \begin{itemize}
  \item If $ p | \alpha_i - \alpha_j$, then  $\delta_i = \delta_j$
  \item If $\delta_i(p) = 1$ and $ p \nmid \alpha_i - \alpha_j$, then $\delta_j = 0$.
  \item If $p | q$, then $\delta_i(p) = \mathbf{1}_{p | a + \alpha_i}$.
  \item If $p \in \mathcal{P}_i$, then $\delta_i(p) = 0$.
  \end{itemize}
  These constraints are forced by natural divisibility relationships between shifts $n + \alpha_i$. For instance if $p | \alpha_i - \alpha_j$ then $p | n + \alpha_i$ if and only if $p | n + \alpha_j$. We will call these constraints on the values $\delta_i(p)$, \textit{consistency constraints}.

The union in \eqref{eq:disjoint} is disjoint. Therefore every $A \in \mathfrak{B}$ can be written as a disjoint union of sets of the form
\begin{equation} \label{eq:set}
\bigcap_{p \in \mathbf{P}} E_{a,q,N}(p; \boldsymbol{\mathcal{P}}, \delta_1(p), \ldots, \delta_{\ell}(p))
\end{equation}
with $\delta_i(p)$ obeying the constraints laid out in the bullet points above.
Hence, to compare $\mathfrak{B}$ with an appropriate probabilistic model,
we should give an asymptotic estimate for
the number of $n \in (N, 2N]$ belonging to \eqref{eq:set} and compare the resulting main term with an appropriate probabilistic model. 
  We do as much
  in Lemma \ref{lem:kubiliusmain} below.

  In order to establish that
  Lemma, we first recall a version of the Fundamental lemma of
  sieve theory. It is really what is under the hood of the
    Kubilius model.
  
  \begin{lemma} \label{le:fundamentallemma}
    Let $\mathscr{P}$ be a set of primes $p\leq z$.
    Let $a_n$ be a sequence of real numbers. Suppose that,
    for every square-free $d\leq D$ such that $p | d \Rightarrow p \in \mathscr{P}$,
    $$
    \sum_{\substack{d | n}} a_n = g(d) \mathcal{M} + R_d
    $$
    with $g$ a multiplicative function such that, for all $2 \leq w \leq z$, 
    \begin{equation}\label{eq:shaibel}
    \prod_{\substack{w \leq p < z \\ p \in \mathscr{P}}} (1 - g(p))^{-1} \leq K \Big ( \frac{\log z}{\log w} \Big )^{\kappa},
    \end{equation}
    where $K > 1$ is a constant.
  Let $s=\log D/\log z$, and assume $s>9\kappa+1$. Then
    $$
  \sum_{\substack{p | n \Rightarrow p \not \in \mathscr{P}}} a_n = \prod_{p \in \mathscr{P}} \Big (1 - g(p) \Big )^{-1} \cdot \Big ( 1 + O(e^{9 \kappa - s} K^{10}) \Big )
  \cdot \mathcal{M} + O \Big ( \sum_{\substack{ p | d \Rightarrow p \in \mathscr{P} \\ d \leq D}} \mu^2(d) |R_d| \Big ).
    $$
  \end{lemma}
  \begin{proof}
    This is \cite[Theorem 6.9]{MR2647984}.
  \end{proof}
  It is easy to see that, if $g(p)\leq \ell/p$ and $g(p)<1$
  for all $p\in \mathscr{P}$,
  then condition \eqref{eq:shaibel} is met with
  $K = e^{O(\ell)}$ and $\kappa = \ell$:  for $w\geq 2 \ell$,
  by a Taylor expansion and a standard estimate,
  \[- \frac{1}{\ell}
  \sum_{w\leq p<z} \log \left(1 - \frac{\ell}{p}\right)
  = \sum_{w\leq p<z} \frac{1}{p} + \sum_{j=2}^\infty \frac{\ell^{j-1}}{j p^j}
  = \log \log z - \log \log w + O(1);\]
  setting $w$ lower than $2 \ell$ increases the right side of \eqref{eq:shaibel}
  without making the left side much larger, since
  $\prod_{p\leq 2 \ell} (1-(p-1)/p)^{-1} = \prod_{p\leq 2 \ell} p = e^{O(\ell)}$.

\begin{lemma} \label{lem:kubiliusmain}
  Let $a + q \mathbb{Z}$ be an arithmetic progression with $q$ square-free.
  Let $\mathbf{P}$ be a set of primes $\leq H$.
Let $\alpha_i\in \mathbb{Z}$ 
be given for $1 \leq i \leq \ell$.
Let $\mathcal{P}_i \subset \mathbf{P} \backslash \{p \in \mathbf{P} : p | q\}$ be given for $1 \leq i \leq \ell$. 
Let $Z_{p}^{(i)}$ be random variables for
$p\in \mathbf{P}$, $1\leq i \leq \ell$, with the following properties:
\begin{enumerate}
  \item $Z_{p}^{(i)}$ and $Z_{r}^{(j)}$ are independent for all distinct primes $p \neq r$ and $0 \leq i, j \leq \ell$,
\item if $p|q$, then $Z_p^{(i)}=1$ for $i$ such that $p|\alpha_i + a$ and
  $Z_p^{(i)}=0$ otherwise,
\item For 
  $p\nmid q$,
  $
  \mathbb{P}(Z_p^{(i)} = 1) = 1/p = 1 - \mathbb{P}(Z_p^{(i)}= 0),
  $
\item If $p | \alpha_i - \alpha_j$, then $Z_p^{(i)} = Z_p^{(j)}$.
  If $p \nmid \alpha_i - \alpha_j$ and $Z_p^{(i)} = 1$,
  then $Z_p^{(j)} = 0$.
\end{enumerate}
  
  Then, given $N$ such that
$\log N \geq C \ell \log H$, where $C$ is an absolute constant,
    and given for every $p \in \mathbf{P}$ and $1 \leq i \leq \ell$
 the values $\delta_i(p) \in \{0,1\}$ obeying the consistency constraints enumerated above, we have, 
  \begin{align} \label{eq:obj}
   & \frac{1}{N}  \sum_{\substack{n \in (a+q\mathbb{Z}) \cap (N, 2N] \\ (n + \alpha_i, \prod_{p \in \mathcal{P}_i} p) = 1\,  \forall i \leq \ell
    }
    } \mathbf{1}_{v_p(n + \alpha_i) > 0 \iff \delta_i(p) = 1 \ \forall p \in \mathbf{P} \ \forall 1 \leq i \leq \ell} 
      \\ \nonumber & = \frac{1}{q} \cdot \mathbb{P} \Big ( Z_p^{(i)} = \delta_i(p) \ \forall p \in \mathbf{P} \ \forall 1\leq i \leq \ell \Big ) \cdot \Big (1 + O \Big (e^{- \frac{\log N}{2 \log H}}\Big)\Big) + O\left(\frac{1}{N^{1/3}}\right).
  \end{align}
\end{lemma}
Here, as usual, $v_p$ is the $p$-adic valuation.
Notice furthermore that the condition that for all $1 \leq i \leq \ell$ we have $(n + \alpha_i, \prod_{p \in \mathcal{P}_i} p) = 1$ is redundant since it is already contained in the fact that the $\delta_i(p)$ obey the consistency constraints and $v_p(n + \alpha_i) > 0 \iff \delta_i(p) = 1$. 

\begin{proof}
  We assume without loss of generality that, for each $p$, the values
  $\delta_i=\delta_i(p)$ obey the consistency conditions enumerated below equation \eqref{eq:disjoint}; otherwise \eqref{eq:obj} is an empty sum.
  
  For $p | q$, there is only one valid choice of $\delta_i(p)$:
  the condition $n \in a + q \mathbb{Z}$
  enforces that $v_p(n + \alpha_i) > 0 \iff \delta_i(p) = 1$ for $p | q$.
  For $p \nmid q$, we can enforce that $v_p(n + \alpha_i) > 0$ for all $i$ with $\delta_i(p) = 1$ by
  requiring that $n + \alpha_i$ be divisible by $f_i = \prod_{p \nmid q} p^{\delta_i(p)}$ for all $1 \leq i \leq \ell$.
  Notice also that, for any $p \nmid q$ for which there exists an $i$ such that $\delta_i(p) = 1$,
  the divisibility of all the $n + \alpha_j$ by $p$ is
  determined: 
  $p | n + \alpha_j$ for all $j$ such that $p | \alpha_j - \alpha_i$ and $p \nmid n + \alpha_j$ for all $j$ such that $p \nmid \alpha_j - \alpha_i$. Thus, to fully
  express the condition $v_p(n + \alpha_i) \iff \delta_i(p) = 1$, it remains to ensure that
  $p\nmid \prod_{1\leq i \leq \ell} (n + \alpha_i)$ whenever $\delta_i(p) = 0$ for all $1\leq i \leq \ell$. 

  With these remarks in mind, we can re-write \eqref{eq:obj} as
  \begin{equation} \label{eq:obj2}
    \frac{1}{N} \mathop{\sum_{n \in (a+q\mathbb{Z}) \cap (N, 2N]}}_{
        f_i | n + \alpha_i} \mathbf{1}_{p|\prod_{1 \leq i \leq \ell} (n + \alpha_i)
        \Rightarrow p\not\in \mathscr{P}},
  \end{equation}
where $\mathscr{P} =   
 \{p \in \mathbf{P}: p \nmid q [f_1, \ldots, f_{\ell}]\}$.
 Thus, to apply the fundamental lemma of sieve theory we need to understand
    \begin{align}\label{eq:ultrasiete}
      \frac{1}{N} &
\mathop{\sum_{n \in (a+q\mathbb{Z}) \cap (N, 2N]}}_{
        f_i | n + \alpha_i} 
    \mathbf{1}_{d | (n + \alpha_1 ) \cdots (n + \alpha_{\ell})}
    \end{align}
 for square-free $d$ such that all prime factors of $d$ are
   in $\mathbf{P}$ but do not divide $q [f_1, \ldots, f_{\ell}]$.
   Since the summand in \eqref{eq:ultrasiete} is periodic, we can evaluate
   the sum by splitting it into arithmetic progressions to modulus $q [f_1, \ldots, f_{\ell}] d$. Thus we get
    \begin{align}\label{eq:cuyes}
      &  \frac{1}{q [f_1, \ldots, f_{\ell}] d}
      \sum_{\substack{ x \mo q [f_1, \ldots, f_{\ell}] d \\ x \equiv a \mo q \\ f_i | x + \alpha_i}} \mathbf{1}_{d | (x + \alpha_1) \ldots (x + \alpha_{\ell})} + O \Big ( \frac{\ell^{\omega(d)}}{N} \Big ), 
    \end{align}
    since there are $\leq \ell^{\omega(d)}$ non-zero terms in the
    sum in \eqref{eq:cuyes}: the congruence class of $x
    \mo q [f_1,\dotsc, f_\ell]$ is determined by the sum's conditions,
    and $d|(x+\alpha_1) \dotsc (x+\alpha_\ell)$ implies that, for
    each $p|d$, $x$ must be congruent to one of $-\alpha_1,\dotsc, -\alpha_\ell \mo p$.
        Furthermore, since $d$ and $q [f_1, \ldots, f_{\ell}]$ are coprime,
        we can write $x = a d + b q [f_1,\ldots,f_{\ell}]$ for some
        $a \mo q [f_1, \ldots, f_{\ell}]$ and some $b \mo d$, and then the sum
        in \eqref{eq:cuyes} splits as follows:
    $$
\Big ( \frac{1}{q [f_1, \ldots, f_{\ell}]}
\mathop{\mathop{\sum_{x \mo q [ f_1, \ldots, f_{\ell}]}}_{ x \equiv a \mo q}}_{
  f_i | x + \alpha_i} 1 \Big ) \cdot \Big ( \frac{1}{d} \sum_{x \mo {d}} \mathbf{1}_{d | (x + \alpha_1) \cdots (x + \alpha_{\ell})} \Big )
    $$
    (We have also made the change of variables $x d \rightarrow x$ in the first sum and $x q [ f_1, \ldots, f_{\ell}] \rightarrow x$ in the second sum; both are allowed because $(d, q [ f_1, \ldots, f_{\ell}]) = 1$.) Since $q$ is coprime to $[f_1, \ldots, f_{\ell}]$, we can further factor the above as
    $$
    \frac{1}{q} \Big ( \frac{1}{[f_1, \ldots, f_{\ell}]} \sum_{\substack{x \mo {[f_1, \ldots, f_{\ell}]} \\ f_i | x + \alpha_i}} 1 \Big ) \cdot \frac{\varrho(d)}{d}
    $$
    where $\varrho(d)$ counts the number of solutions to $d | (x + \alpha_1) \ldots (x + \alpha_{\ell})$ for $x\in \mathbb{Z}/d\mathbb{Z}$.
    By the Chinese remainder theorem,  $\varrho(d)$ is multiplicative.
    It is clear that $\varrho(p)\leq \ell$ for all $p$; we may also
    assume that $\varrho(p)<p$ for all $p\in \mathscr{P}$, as otherwise
    the sum in \ref{eq:obj2} is $0$.    
    Applying Lemma \ref{le:fundamentallemma} with $z = H$, $\kappa = \ell$,
    $K = e^{O(\ell)}$ and $D = N^{3/5}$, 
    we obtain that \eqref{eq:obj2} is equal to  
    \begin{equation}\label{eq:senso}\begin{aligned}
        \frac{1}{q}  \prod_{p | [f_1, \ldots, f_{\ell}]} \Big ( \frac{1}{p}
        \mathop{\sum_{x \mod p}}_{p^{\delta_i(p)} | x + \alpha_i} 1 \Big) \cdot
        \prod_{p\in \mathscr{P}} & \Big ( 1 - \frac{\varrho(p)}{p} \Big )  \cdot \Big (1 + O \Big ( \exp \Big ( - \frac{\log N}{2 \log H} \Big ) \Big ) \\ & + O \left( \frac{D (2 \log H)^{\ell}}{N}\right),
    \end{aligned}\end{equation}
    where we are using the assumption $\log N \geq C \ell \log H$ to
    ensure that $e^{9 \kappa} K^{10} \leq e^{\frac{\log N}{10 \log H}}$, and where the error term $D (2 \log H)^{\ell}$ arises from noticing that
    $$\begin{aligned}
    \sum_{\substack{d \leq D \\ p | d \Rightarrow p \in \mathbf{P}}} \mu^2(d)
    \ell^{\omega(d)} &\leq D \sum_{\substack{d \leq D \\ p | d \Rightarrow p \in \mathbf{P}}} \mu^2(d) \frac{\ell^{\omega(d)}}{d}\\ &\leq D \prod_{p \in \mathbf{P}} \Big ( 1 + \frac{\ell}{p} \Big ) \leq
    D \prod_{p \in \mathbf{P}} \Big ( 1 + \frac{1}{p} \Big )^\ell
\leq D (2 \log H)^{\ell},\end{aligned}$$
since $\prod_{p\leq H} (1+1/p)\leq 2 \log H$ by, say, 
\cite[Thm.~5, Cor.]{MR0137689} and a check for $H=2$.
Now, $2 \log H \leq H^{3/4}$ for all $H\geq 1$, and so,
again by $\log N\geq C \ell \log H$,
\[D (2 \log H)^{\ell} \leq N^{\frac{3}{5}} H^{\frac{3}{4} \ell} \leq N^{2/3}\]
provided that $C$ is taken to be sufficiently large (but fixed). 

Finally, it suffices to reinterpret the main term in \eqref{eq:senso}
probabilistically by noticing that, for each $p \nmid q$ for which there is an $i$ such that $\delta_i(p) = 1$,
    $$
    \mathbb{P}\Big (Z_p^{(i)} = \delta_i(p) \ \forall i \leq \ell \Big ) = \frac{1}{p} = \frac{1}{p} \sum_{\substack{x \pmod{p} \\ p^{\delta_i(p)} | x + \alpha_i}} 1 
    $$
    provided that $\delta_i(p)$ obey the consistency conditions. 
    Finally, for $p \nmid q$ such that $\delta_i(p) = 0$ for all $i \leq \ell$, 
    $$
    \mathbb{P} \Big ( Z_{p}^{(i)} = 0 \ \forall i \leq \ell \Big ) = 1 - \frac{1}{p} \sum_{x \mo p} \mathbf{1}_{p | (x + \alpha_1) \cdots (x + \alpha_{\ell})} =
    1 - \frac{\rho(p)}{p}
    $$
    Since the variables $Z_p$ are independent, we conclude that the main term of \eqref{eq:senso} can be re-written as
    $$
    \frac{1}{q} \mathbb{P} \Big ( Z_p^{(i)} = \delta_i(p) \ \forall i \leq \ell \ \forall p \in \mathbf{P} \Big ). 
    $$
\end{proof}

Of course, for Lemma \ref{lem:kubiliusmain} to be truly useful, the main
term has to be smaller than the remainder term.
In general, when estimating
the difference between reality and the Kubilius model for arbitrary
events, one bounds rare events using a simple sieve or trivially, rather
than by means of the Fundamental lemma. In our case, the fact that the
sieve dimension $\kappa = \ell$ is relatively large would complicate
the task of obtaining a good total bound, since there could be very many rare events.
Fortunately, the fact that
we will be working with bounds on the number of prime factors
of $n+\alpha_i$ helps us.

\begin{lemma}\label{lem:kubiliusprax}
  Let $\mathbf{P}$ be a set of primes $\leq H$.
  Let $\alpha_i\in \mathbb{Z}$ be given for
  $1\leq i\leq \ell$.
  Let $F$ be a function from $\{\textbf{true}, \textbf{false}\}^{\mathbf{k}\times \mathbf{P}}$ to $\mathbb{C}$ with $|F|_\infty \leq 1$, where
  $\mathbf{k}=\{1,2,\dotsc,\ell\}$.
  Assume that $F(\{v_{i,p}\})=0$ if there is an $i\in \mathbf{k}$
  such that $\# \{ p \in \mathbf{P} : v_{i,p} = \textbf{true} \} > L$, where $L\in \mathbf{Z}_{\geq 0}$.
  
  Let $a + q \mathbb{Z}$ be an arithmetic progression, and let
  $Z_p^{(i)}$ be random variables as in Lemma \ref{lem:kubiliusmain}. 
  Then there is an absolute constant $C_0$ such that,
  given $N$ for which $\log N\geq C_0 \ell \log H$,
  \[\begin{aligned}\frac{1}{N}  \sum_{n \in (a+q\mathbb{Z}) \cap (N, 2N]}
    &F\Big(\{p|n+\alpha_i\}_{i\in \mathbf{k}, p\in \mathbf{P}}\Big)
    = \frac{1}{q} \mathbb{E}\left(F\Big(\{Z_p^{(i)} = 1\}_{i\in \mathbf{k}, p\in \mathbf{P}}\Big)\right)\\
    &+ \frac{1}{q} \mathbb{E}\left(\left|F\Big(\{Z_p^{(i)} = 1 \}_{i\in \mathbf{k}, p\in \mathbf{P}}\Big)\right|\right)
    \cdot O \left(e^{- \frac{\log N}{2 \log H}}\right)
    +  O\left(\frac{H^{\ell L}}{L!^{\ell} N^{1/3}}\right),
  \end{aligned}
  \]
  where the implied constants are absolute.
\end{lemma}
\begin{proof}
  Apply  Lemma \ref{lem:kubiliusmain} for each choice of
  $\{\delta_i(p)\}_{i\in \mathbf{k}, p\in \mathbf{P}}$ such that
  $\sum_{p\in \mathbf{P}} \delta_i(p) \leq L$. There are
  $\binom{|\mathbf{P}|}{L}^{\ell} \leq (|P|^L/L!)^{\ell} \leq
  H^{L \ell}/L!^\ell$ such choices altogether.
\end{proof}

\section{Using a sieve to prevent early recurrence}\label{sec:sievrecur}

We will define a set $Y_\ell\subset \mathbf{N}$ 
with the property that no walk on $\Gamma|_{Y_\ell}$ of length $\leq \ell$
can have the same prime $p \in \mathbf{P}$ occur as an edge length more than once,
except for consecutive occurrences, which will be allowed.

It will be natural to approximate membership in $Y_\ell$ by means of
a sieve for composite moduli. When we later apply this approximation,
we will have a main term and an error term, and we will have to know how
to bound both. The arguments for proving these two bounds are closely related.
We will find it convenient to express them in terms of a new graph --
a {\em sieve graph} that expresses the logical conjunction of several
congruence conditions.

\subsection{Sieve approximation to \texorpdfstring{$Y_{\ell}$}{Yl}}\label{subs:sievappr}

We define $Y_\ell = Y_{\ell,\mathbf{P}}$ to be the set of all integers $n$ except
for those for which there are distinct
primes $p_1, \ldots, p_{l} \in \mathbf{P}$ and signs $\sigma_1, \ldots, \sigma_{l} \in \{-1,1\}$ with $1 \leq l < \ell$ such that
\begin{equation} \label{eq:first}
p_1 | n, p_2 | n + \sigma_1 p_1, \ldots, p_l | n + \sigma_1 p_1 + \ldots + \sigma_{l - 1} p_{l - 1},
\end{equation}
there are no $1\leq i<j<i'\leq l$ such that $p_j\ne p_i=p_{i'}$ and no
$1\leq i < l$ such that $p_{i+1}=p_i$ and $\sigma_{i+1}=-\sigma_i$,
and, in addition, at least one of the following consequences holds:
\begin{itemize}
\item there exists a prime $p_0 \in \mathbf{P}$ distinct from $p_1, \ldots, p_{l}$ such that
  \begin{equation} \label{eq:second}
  p_0 | n \text{ and } p_0 | n + \sigma_1 p_1 + \ldots + \sigma_l p_l
  \end{equation}
\item we have \begin{equation} \label{eq:third}
  \sigma_1 p_1 + \ldots + \sigma_{l} p_l = 0 .
  \end{equation}
\end{itemize}
We let $p_0 = \infty$ in the event that \eqref{eq:third} holds.

The set of integers $n$ that obey conditions \eqref{eq:first} and either condition \eqref{eq:second} or \eqref{eq:third} forms an arithmetic progression to modulus $[p_0, p_1, \ldots, p_l]$ or $[p_1, \ldots, p_l]$ (thus square-free),
unless it is empty.
Let $\mathscr{W}_{\ell,\mathbf{P}}$ denote the set consisting of all such arithmetic progressions with $1 \leq l < \ell$. The event $n \in Y_{\ell}$ is then equivalent to 
$$
n \not \in \bigcup_{P \in \mathscr{W}_{\ell,\mathbf{P}}} P.
$$
Likewise, for any integers $\beta_1, \ldots \beta_{2k}$, ensuring that
$$
n + \beta_i \in Y_{\ell}
$$
for all $1 \leq i \leq 2k$ is equivalent to asking that
$$
n \not \in \bigcup_{1 \leq i \leq 2k} \bigcup_{P \in \mathscr{W}_{\ell,\mathbf{P}}} (P - \beta_i).
$$
(Here $P-\beta_i$ is just a displacement; if
$P=\{n\equiv a \mo q\}$, then $P-\beta_i = \{n\equiv a-\beta_i \mo q\}$.)
Given $\boldsymbol{\beta} = (\beta_1, \ldots, \beta_{2k})$, we denote
by $\mathscr{W}_{\ell,\mathbf{P}}(\boldsymbol{\beta})$ the set consisting of all
$P - \beta_i$ with $P \in \mathscr{W}_{\ell,\mathbf{P}}$ and $1 \leq i \leq 2k$. 


\begin{prop} \label{prop:sievorino}
  Let $Y_{\ell,\mathbf{P}}$, $W_{\ell,\mathbf{P}}$ be as above for 
   $\mathbf{P}$ a set of primes and some $\ell\geq 1$. Let 
  $\boldsymbol{\beta} = (\beta_1, \ldots, \beta_{2k})$,
  $\beta_i\in \mathbb{Z}$. Then, for any $m\geq 1$,
  \begin{equation}\label{eq:dalesto}\begin{aligned}
    \mathbf{1}_{n + \beta_i  \in Y_{\ell}\; \forall 1 \leq i \leq 2k}
    &= \mathop{\sum_{R\in \mathscr{Q}^\cap}}_{\omega(\mathfrak{q}(R))\leq m}
    c_R 1_{n\in R}
    + O^*\left(2^m 3^\ell
    \mathop{\sum_{R\in \mathscr{Q}^\cap}}_{m<\omega(\mathfrak{q}(R))\leq m+\ell}
    1_{n\in R}\right)\end{aligned}\end{equation}
  for $\mathscr{Q} = W_{\ell,\mathbf{P}}(\boldsymbol{\beta})$ and
    some $c_R\in \mathbb{R}$ with $|c_R|\leq 2^{\omega(\mathfrak{q}(R))}$.
\end{prop}


We recall our convention $\omega(\mathfrak{q}(\emptyset)) = \infty$.
It implies that
both sums on the right-hand side
of \eqref{eq:dalesto} are over non-empty sets $R$.

\begin{proof}
  Apply Prop.~\ref{prop:sieve} with $\mathfrak{D}$ the set of all non-empty $R\in \mathscr{Q}^\cap$
  such that the modulus of $R$ has $\leq m$ prime factors.
\end{proof}

\subsection{Controlling the error term}\label{subs:contrert}


For us, a {\em sieve graph} will be a directed graph
consisting of:
\begin{enumerate}[(i)]
\item a marked path of length $2 k$, called the {\em horizontal path};
  we may picture it as starting at the ``leftmost'' vertex and ending at
  the ``rightmost'' vertex;
\item {\em threads} of length $\leq \ell$, of two kinds:
  \begin{enumerate}[(a)]
  \item a {\em closed thread} is a cycle containing some vertex
    of the horizontal path, and otherwise disjoint from it;
  \item an {\em open thread} is a path that 
    has an endpoint at some vertex of the horizontal path, and is
    otherwise disjoint from it;
    \end{enumerate}
\item for each open thread and each of the two endpoints of that thread,
  an edge containing that endpoint as its tail,
  with its head being a vertex of degree $1$ (i.e., not contained
  in any other edge).
  These two edges at the thread's endpoints are
  marked as associated to the thread; they will be called the
  thread's {\em witnesses}, and count as part of the thread.
\end{enumerate}
Any two distinct threads are disjoint from each other except for possibly sharing a vertex
on the horizontal path. We recall that a {\em path} does not self-intersect.
(Contrast with a {\em closed path}, such as a cycle: a closed path self-intersects
at its origin.)

We present below a picture of a sieve graph, where the blue path corresponds
to the horizontal path, and the witness edges of each open path are highlighted in red.


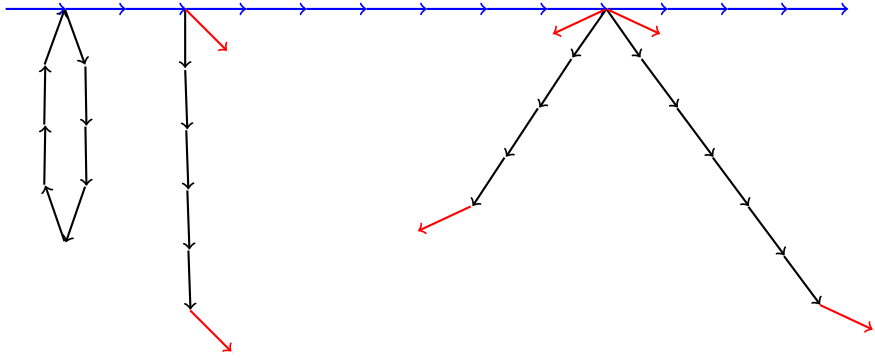
\begin{center}
\begin{tikzpicture}[thick,->,scale=0.8, fill=black!50]
  \tikzstyle{vertex}=[circle,fill=black!50,minimum size=0pt, inner sep=0pt]
  

          \foreach \x in {0,1,...,14}
             \node[vertex] (G-\x) at (\x,0) {};

         \node[vertex] (C-1) at ([shift=(-70:1)]{G-1}) {};
         \draw (G-1) -> (C-1);
         \node[vertex] (C-2) at ([shift=(-90:1)]{C-1}) {};
         \node[vertex] (C-3) at ([shift=(-90:1)]{C-2}) {};
         \node[vertex] (C-4) at ([shift=(-110:1)]{C-3}) {};
         \node[vertex] (C-5) at ([shift=(110:1)]{C-4}) {};
         \node[vertex] (C-6) at ([shift=(90:1)]{C-5}) {};
         \node[vertex] (C-7) at ([shift=(90:1)]{C-6}) {};
         \draw (C-7) -> (G-1);
         \foreach \x in {1,...,6}
              {
                \tikzmath{\y=\x+1;}
                \draw (C-\x) -> (C-\y);
              };
              
         \foreach \x in {0,1,...,13}
              {
                \tikzmath{\y=\x+1;}
                \draw[blue] (G-\x) -> (G-\y);
              };

        \node[vertex] (O-1) at ([shift=(-90:1)]{G-3}) {};
        \draw (G-3) -> (O-1);
        \foreach \x in {2,...,5}
              {
                \tikzmath{\y=\x-1;}
                \node[vertex] (O-\x) at ([shift=(-90:1)]{O-\y}) {};
              }
        \foreach \x in {1,...,4}
              {
                \tikzmath{\y=\x+1;}
                \draw (O-\x) -> (O-\y);
              };
        \node[vertex] (W-1) at ([shift=(-45:1)]{G-3}) {};
        \node[vertex] (W-2) at ([shift=(-45:1)]{O-5}) {};
        \draw[red] (G-3) -> (W-1);
        \draw[red] (O-5) -> (W-2);

        \node[vertex] (OT-1) at ([shift=(-125:1)]{G-10}) {};
        \draw (G-10) -> (OT-1);
        \foreach \x in {2,...,4}
              {
                \tikzmath{\y=\x-1;}
                \node[vertex] (OT-\x) at ([shift=(-125:1)]{OT-\y}) {};
              }
        \foreach \x in {1,...,3}
              {
                \tikzmath{\y=\x+1;}
                \draw (OT-\x) -> (OT-\y);
              };
        \node[vertex] (WT-1) at ([shift=(-155:1)]{G-10}) {};
        \node[vertex] (WT-2) at ([shift=(-155:1)]{OT-4}) {};
        \draw[red] (G-10) -> (WT-1);
        \draw[red] (OT-4) -> (WT-2);

        \node[vertex] (OX-1) at ([shift=(-55:1)]{G-10}) {};
        \draw (G-10) -> (OX-1);
        \foreach \x in {2,...,6}
              {
                \tikzmath{\y=\x-1;}
                \node[vertex] (OX-\x) at ([shift=(-55:1)]{OX-\y}) {};
              }
        \foreach \x in {1,...,5}
              {
                \tikzmath{\y=\x+1;}
                \draw (OX-\x) -> (OX-\y);
              };
        \node[vertex] (WX-1) at ([shift=(-25:1)]{G-10}) {};
        \node[vertex] (WX-2) at ([shift=(-25:1)]{OX-6}) {};
        \draw[red] (G-10) -> (WX-1);
        \draw[red] (OX-6) -> (WX-2);
        
        \end{tikzpicture}
\end{center}






Write $E(G)$ for the set of edges in a graph $G$.
We identify the edges of the horizontal path with
the elements of $\mathbf{k} = \{1,2,\dotsc,2 k\}$.

We will work with pairs $(G,\sim)$, where $G$ is a sieve graph and $\sim$
an equivalence relation on the edges of $G$ such that, in any thread,
the set of edges in a given equivalence class form a (possibly empty)
connected subgraph, except in the case of witnesses, which are equivalent to
each other but to no other edges in their thread.
We say $(G,\sim)$ is {\em non-redundant} if, for every thread,
the thread contains at least one edge $x$ (possibly a witness) 
whose equivalence class $[x]$ contains no edge in any other thread.
The {\em cost} $\kappa(G,\sim)$ of $(G,\sim)$ is the number of equivalence
classes that contain at least one edge (possibly a witness) in some thread.
Denote by $\mathbf{W}_{k,\ell,m}$ the set of non-redundant pairs $(G,\sim)$
as above
with parameters $k$, $\ell$ and cost $\kappa(G,\sim)=m$. It is clear that
$\mathbf{W}_{k,\ell,m}$ is finite, since any element of $\mathbf{W}_{k,\ell,m}$
contains at most $m$ threads.

It should be clear that we see a pair $(G,\sim)$ as a template,
leading to divisibility conditions when we assign a prime to each equivalence
classes of $\sim$. Let us be precise.
For $(G,\sim)$ a pair as above, $\mathbf{l}\subset \mathbf{k}$,
$\vec{\sigma} \in \{-1,1\}^{E(G)}$ and a
$(p_{[x]})$ a tuple consisting of a prime $p_{[x]}$ for each equivalence
class $[x]$ of $\sim$, we say 
$((G,\sim),\mathbf{l},\vec{\sigma},(p_{[x]}))$ is {\em valid} if
\begin{enumerate}[(i)]
\item\label{it:sigracon1} whenever two edges $e_1=\{v_1,v_1'\}$, $e_2 = \{v_2,v_2'\}$
in $G$ but not in $\mathbf{k}\setminus \mathbf{l}$ are in the same equivalence
class $[x]=[e_1]=[e_2]$, we have that, for any walk from $v_1$ to $v_2$,
$p_{[x]}$ divides the sum
$\sum_{y\;\text{on}\;\phi} \sigma_y p_{[y]}$, where $y$ ranges over the edges
of the walk;
\item\label{it:sigracon2} the sum $\sum_{y\;\text{on}\;\phi} \sigma_y p_{[y]}$ over the edges $y$ in any closed path $\phi$ is
  $0$, but the sum $\sum_{y\;\text{on}\;\phi} \sigma_y p_{[y]}$ over the edges $y$ in an open thread $\phi$ (excluding the witnesses) is always non-zero;
\item\label{it:sigracon3} if two adjacent edges $e_1$, $e_2$ in a thread are equivalent, their
  signs $\sigma_{e_1}$, $\sigma_{e_2}$ are the same.
\end{enumerate}
In conditions \eqref{it:sigracon1} and \eqref{it:sigracon2} and in all of what
follows, a sum $\sum_{y\;\text{on}\;\phi}
\sigma_y p_{[y]}$ over the edges $y$ of a walk
is understood as follows: if the walk traverses an edge $y$
contrary to its direction, its sign $\sigma_y$ is flipped.

\begin{lemma}\label{lem:cagliostro}
  Let $\mathbf{N} = (N,2N]\cap \mathbb{Z}$ and $\mathbf{k}=\{1,2,\dotsc,2 k\}$.
  Let $\mathbf{P}$ be a set of primes $\leq H$. Define
  $\mathscr{W}_{\ell,\mathbf{P}}(\boldsymbol{\beta})$ and
  $\mathbf{W}_{k,\ell,m}$ as above. Given $\vec{p} = (p_1,p_2,\dotsc,p_{2 k})$
  and $\vec{\sigma}\in \{-1,1\}^{2 k}$,
  let $\boldsymbol{\beta}(\vec{p},\vec{\sigma}) = (\beta_1,\dotsc,\beta_{2 k})$
  with $\beta_i = \sum_{j=1}^i \sigma_j p_j$. Then, for $m\geq 0$ with
  $H^m\leq N$,
        \begin{equation}\label{eq:biden}  \sum_{\mathbf{l}\subset \mathbf{k}}
        \sum_{\vec{\sigma} \in \{-1,1\}^{2 k}}
  \sum_{p_i\in \mathbf{P}\; \forall i\in \mathbf{k}}
  \prod_{i\in \mathbf{k}\setminus \mathbf{l}} \frac{1}{p_i}
  \mathop{\sum_{R\in \mathscr{W}_{\ell,\mathbf{P}}(\boldsymbol{\beta}(\vec{p}, \vec{\sigma}))^\cap}}_{R\ne \emptyset,\; \omega(\mathfrak{q}(R)) = m}
  \left|\{n\in R\cap \mathbf{N}: p_i|n+\beta_i\; \forall i\in
  \mathbf{l}\}\right|
        \end{equation}
  is at most $2 N$ times
  \begin{equation}\label{eq:badvat}
    \sum_{\mathbf{l}\subset \mathbf{k}}
  \sum_{(G,\sim)\in \mathbf{W}_{k,\ell,m}} \sum_{\vec{\sigma} \in \{-1,1\}^{E(G)}}
  \mathop{\sum_{p_{[x]}\in \mathbf{P}}}_{\text{$((G,\sim),\mathbf{l},\vec{\sigma},(p_{[x]}))$ valid}} \prod_{i\in \mathbf{k}\setminus \mathbf{l}} \frac{1}{p_{[i]}}
  \prod_{[x]\not\subset \mathbf{k}\setminus \mathbf{l}} \frac{1}{p_{[x]}}.
  \end{equation}
\end{lemma}
It is actually unnecessary to
assign signs $\sigma$ to witnesses; this upper bound is loose.
\begin{proof}
  The moduli of the progressions $R\in
  \mathscr{W}_{\ell,\mathbf{P}}(\boldsymbol{\beta}(\vec{p}, \vec{\sigma}))^\cap$ are square-free.
  Thus $\omega(\mathfrak{q}(R))=m$ implies
  $\mathfrak{q}(R)\leq H^m\leq N$.
  Hence, the innermost sum in (\ref{eq:biden}) is
  \[\leq
  \mathop{\sum_{R\in \mathscr{W}_{\ell,\mathbf{P}}(\boldsymbol{\beta}(\vec{p}, \vec{\sigma}))^\cap}}_{\omega(\mathfrak{q}(R)) = m} \left(\frac{N}{\mathfrak{q}(R
    \cap R_{\vec{p},\mathbf{l}})} + 1\right) \;\leq
    2 N \mathop{\sum_{R\in \mathscr{W}_{\ell,\mathbf{P}}(\boldsymbol{\beta}(\vec{p}, \vec{\sigma}))^\cap}}_{\omega(\mathfrak{q}(R)) = m} \frac{1}{\mathfrak{q}(R
      \cap R_{\vec{p},\mathbf{l}})} 
    ,\]
    where $R_{\vec{p},\mathbf{l}}$ is
    $\{n: p_i|n+\beta_i\; \forall i\in \mathbf{l}\}$, which is
    either an arithmetic progression or empty
    (in which case we recall that, by convention,
    $\mathfrak{q}(\emptyset) = \infty$).

      An element
      $R\in\mathscr{W}_{\ell,\mathbf{P}}(\boldsymbol{\beta}(\vec{p}, \vec{\sigma}))^\cap$
      induces a sieve graph $G$ in a natural way:
      an element of 
      $\mathscr{W}_{\ell,\mathbf{P}}(\boldsymbol{\beta}(\vec{p}, \vec{\sigma}))$ is defined
      by a set of conditions \eqref{eq:first}--\eqref{eq:third} inducing a
      thread at some vertex of the horizontal path
      of length $2 k$. (The thread is open when
      \eqref{eq:second} holds, and closed
      when \eqref{eq:third} holds) Moreover,
      for any tuple $\vec{p}=(p_i)_{i\in \mathbf{k}}$ with $p_i\in \mathbf{P}$,
      the pair $(R,\vec{p})$ induces an equivalence relation $\sim$ on the
      edges of $G$, with two edges being defined as equivalent when they are
      labeled with the same prime. The prime labels are given by
      $p_i$ for edges in the horizontal path; for other edges, they are
      given by conditions \eqref{eq:first}--\eqref{eq:third}.
      (In particular, given an open thread, we associate the label $p_0$ in \eqref{eq:second} to the thread's witnesses.) 
      We may employ the notation $p_{[x]}$ for the label of any edge in an
      equivalence class $[x]$.
      It is clear that we can assume that
      $(G,\sim)$ is non-redundant: if it is not, there is some thread
      such that every edge in it
      is equivalent to some edge in some other thread,
      and so we may remove our thread without changing $R$
      (unless $R$ was $\emptyset$);
      we repeat this step until $(G,\sim)$ is indeed non-redundant.

      It is clear that
      $\mathfrak{q}(R\cap R_{\vec{p},\mathbf{l}})$ equals
      $\prod_{[x]\not\subset \mathbf{k}\setminus \mathbf{l}} p_{[x]}$ unless
      $R\cap R_{\vec{p},\mathbf{l}}=\emptyset$.
      It is also clear that any tuple
      $((G,\sim),\mathbf{l},\vec{\sigma},(p_{[x]}))$ constructed as above
      is valid, provided that $R\cap R_{\vec{p},\mathbf{l}}\ne \emptyset$:
      conditions \eqref{it:sigracon2} and \eqref{it:sigracon2}
      defining a valid sieve graph are immediate,
      and, to verify condition \eqref{it:sigracon1}
      (namely, that
      $p|\sum_{y\;\text{on}\; \phi} \sigma_y p_{[y]}$ for any walk $\phi$
      between two vertices $v_1$, $v_2$ contained in edges
      $e_1,e_2\not\in \mathbf{k}\setminus \mathbf{l}$
      with $e_1\sim e_2$) it is enough to note
      that, for $n\in R\cap R_{\vec{p},\mathbf{l}}$
      and $p = p_{[e_1]} = p_{[e_2]}$, we have both
      $p = p_{[e_1]}| (n + \sum_{y\;\text{on}\; \phi_1} \sigma_y p_{[y]})$
      for any walk $\phi_1$ from the ``origin'' (i.e., the leftmost vertex
      in the horizontal path) to $v_1$, and also
      $p = p_{[e_2]}| (n + \sum_{y\;\text{on}\; \phi_2} \sigma_y p_{[y]})$    
      for the concatenation $\phi_2$ of the walk $\phi_1$ and our walk $\phi$,
      since that concatenation is a walk from the origin to $v_2$;
      it then obviously follows that $p|\sum_{y\;\text{on}\; \phi} \sigma_y p_{[y]}$.
\end{proof}

\begin{lemma}\label{lem:mandrake}
  Let $\mathbf{l}\subset \mathbf{k}=\{1,2,\dotsc, 2k\}$,
  $(G,\sim)\in \mathbf{W}_{k,\ell,m}$
  and $\vec{\sigma} \in \{-1,1\}^{E(G)}$.
  Let $\mathbf{P}\subset [H_0,H]$ be a set of primes, where $H\geq H_0\geq
  \max(e^2,\ell)$. Let $\mathscr{L} = \sum_{p\in \mathbf{P}} 1/p$.
  Then \[\mathop{\sum_{p_{[x]}\in \mathbf{P}}}_{
    \text{$((G,\sim),\mathbf{l},\vec{\sigma},(p_{[x]}))$ valid}} \prod_{i\in \mathbf{k}\setminus \mathbf{l}} \frac{1}{p_{[i]}}
  \prod_{[x]\not\subset \mathbf{k}\setminus \mathbf{l}} \frac{1}{p_{[x]}}\leq
\mathscr{L}^{s-r} \left(\frac{\log H}{H_0}\right)^r
  ,\]
where 
$r$ is the number
  of threads in $G$ and $s$ is the number of equivalence classes of
  $\sim$. 
\end{lemma}
\begin{proof}
  The idea is that every thread in $G$ binds a different variable
  $p_{[x]}$, thus subtracting $r$ degrees of freedom in total. 

  Since $(G,\sim)$ is non-redundant, we may choose, for each thread,
  an equivalence class containing an edge in a thread (possibly a witness)
  whose equivalence class $[x]$ contains no edge in any other thread.
  We allow the primes associated to all other equivalence classes to
  range freely over $\mathbf{P}$. Let them take some values that will
  henceforth be fixed, and let us see which values the primes
  $p_{[x]}$ can take, where $[x]$ goes over the equivalence classes we chose,
  one per thread.

  Let us consider closed threads first. There, there is one edge $x$
  we chose; all other edges have already been assigned values. Since
  the sum 
$\sum_y \sigma_y p_{[y]}$ over the closed thread must be $0$,
  we see that $p_{[x]}$ is fully determined. Thus, the term $1/p_{[x]}$
  contributes a factor $\leq 1/H_0$, since $p_{[x]}\geq H_0$.

  Now consider open threads where we have chosen the witnesses as our
  class $[x]$. All other edges have already been assigned values, and
  so their sum $\sum_y \sigma_y p_{[y]}$ already has a fixed integer
  value $d$, which we know not to be $0$. Since
  $|d|\leq \ell H \leq H_0^{\Delta + 1}$, where $\Delta = (\log H)/(\log H_0)$, 
  we see that $d$ has $\leq \Delta+1$ prime divisors $\geq H_0$;
  $p_{[x]}$ must equal one of them, and so the sum of $1/p_{[x]}$
over $p_{[x]}$ contributes a factor $\leq (\Delta+1)/H_0$.

Lastly, consider open threads in which we have chosen 
an edge $x$ other than the witnesses. The sum
$\sum_y \sigma_y p_{[y]}$ over the thread must be $\equiv 0$ modulo $p_{[z]}$,
where $[z]$ is the witnesses' class. Thus, the congruence class modulo $p_{[z]}$
of  $p_{[x]}$ is determined; call it $a \mo p_{[z]}$, say. We see then that
the sum of $1/p_{[x]}$ over $p_{[x]}$ is at most\footnote{It goes almost without saying that
  one could obtain a more precise upper bound using
  the Brun-Titchmarsh inequality, or even just restricting $m$ to be odd.}
\begin{equation}\label{eq:ramor}
  \mathop{\sum_{H_0\leq p\leq H}}_{p\equiv a \mo p_{[z]}} \frac{1}{p} \leq
  \sum_{n\leq H/H_0} \frac{1}{n H_0}
  \leq \frac{\log \frac{e H}{H_0}}{H_0} \leq \frac{\log H}{H_0},\end{equation}
since $p_{[z]}\geq H_0\geq e$.

Note that $\log H = \Delta \log H_0 \geq 2\Delta\geq \Delta+1$.
Hence, each prime from a chosen class $[x]$ contributes a factor
$\leq (\log H)/H_0$, and each of the other primes contributes
a factor $\mathscr{L}$. (Actually, primes $p_{[i]}$ such that
$i\in \mathbf{k}\setminus \mathbf{l}$ but
$[i]\not\subset \mathbf{k}\setminus \mathbf{l}$ contribute even less, but
we choose to forego that effect.)
\end{proof}

\begin{prop}\label{prop:errorterm}
  Let $\mathbf{N} = (N,2N]\cap \mathbb{Z}$ and
    $\mathbf{k}=\{1,2,\dotsc,2 k\}$.
    Let $\mathbf{P}\subset [H_0,H]$ be a set of primes where $H\geq H_0\geq e^2$.
    Let $\ell\leq H_0$. 
    Given $\vec{p} = (p_1,p_2,\dotsc,p_{2 k})$
  and $\vec{\sigma}\in \{-1,1\}^{2 k}$,
  let $\boldsymbol{\beta}(\vec{p},\vec{\sigma}) = (\beta_1,\dotsc,\beta_{2 k})$
  with $\beta_i = \sum_{j=1}^i \sigma_j p_j$.
  Define
      $\mathscr{W}_{\ell,\mathbf{P}}(\boldsymbol{\beta})$
      as before.
  Write $\mathscr{L}$ for $\sum_{p\in \mathbf{P}} 1/p$.
  
  Then, for $m\geq 3$ such that $H^m\leq N$
  and $H_0 \geq 8 (2 k + 1) (\log H) (2 m)^{\ell+1}/\mathscr{L}$,
    \begin{equation}\label{eq:biden2}  \begin{aligned}
      \sum_{\mathbf{l}\subset \mathbf{k}}
    &\sum_{\vec{\sigma} \in \{-1,1\}^{2 k}}
    \sum_{p_i\in \mathbf{P}\; \forall i\in \mathbf{k}}
    \prod_{i\in \mathbf{k}\setminus \mathbf{l}} \frac{1}{p_i}
    \mathop{\sum_{R\in \mathscr{W}_{\ell,\mathbf{P}}(\boldsymbol{\beta}(\vec{p}, \vec{\sigma}))^\cap}}_{R\ne \emptyset,\; \omega(\mathfrak{q}(R)) = m}
    \left|\{n\in R\cap \mathbf{N}: p_i|n+\beta_i\; \forall i\in
    \mathbf{l}\}\right|\\
 &\leq
N\cdot 4  (2 e)^m (4 k + 2 m)^{2 k}
  \mathscr{L}^{2 k + m}  \left(\frac{(8 k + 4) 
\log H}{\mathscr{L} H_0}\right)^{\frac{m}{\ell+1}}.\end{aligned}\end{equation}
\end{prop}
\begin{proof}
  We apply Lemmas \ref{lem:cagliostro} and \ref{lem:mandrake}, and obtain
  that the first line in \eqref{eq:biden2} is
\begin{equation}\label{eq:blug} \leq  2 N\cdot \sum_{\mathbf{l}\subset \mathbf{k}}
\sum_{(G,\sim)\in \mathbf{W}_{k,\ell,m}} \sum_{\vec{\sigma} \in \{-1,1\}^{E(G)}}
\mathscr{L}^{s(\sim)-r(G)}
\left(\frac{\log H}{H_0}\right)^{r(G)},\end{equation}
where $s(\sim)$ is the number of equivalence classes of $\sim$ and
$r(G)$ is the number of threads in $G$. The sum over
$\mathbf{l}\subset \mathbf{k}$ has $2^{2 k}$ terms.
Since the threads are of length $\ell$, we see that
$|E(G)|\leq 2 k + r (\ell + 2)$ and $s(\sim)\leq 2 k + m$,
where $m$ is the cost of $(G,\sim)$.
(We may irrelevantly remark that we could
work with a bound of $\leq 2 k + r \ell$ in the first case,
since witnesses should not have to carry signs.)
It remains to count how many
$(G,\sim)\in \mathbf{W}_{k,\ell,m}$ there are with given $r = r(G)$
and $\kappa(G,\sim)=m$. Here $m$ is given and $m/(\ell+1)\leq r\leq m$:
$r\geq m/(\ell+1)$ because a thread (counting witnesses) cannot exhaust more
than $\ell+1$ equivalence classes, and $r\leq m$
since
each thread (counting witnesses)
must have at least one equivalence class that does not appear in other
threads (counting witnesses).

Let
$1\leq l_1,\dotsc,l_r\leq \ell$ be the lengths of the threads.
There are $s(\sim)^{2k} \leq (2k + m)^{2k}$ ways of assigning
equivalence classes of $\sim$ to the $2k$ edges of the horizontal path.
Furthermore, since $m=\kappa(G, \sim)$, there are $m$ equivalence classes
to choose from when we assign equivalence classes of $\sim$ to edges
in threads. Thus, for a closed thread of length $l_j$, there are $m^{l_j}$ ways
of assigning equivalence classes of $\sim$ to its edges,
whereas, for an open thread, there are $m^{l_j + 1}$ ways of doing so,
since we also need to assign an equivalence class to
the witnesses.
Since there are $s(\sim)!$ ways to order the equivalence classes
and we know that
$s(\sim)\geq m$, we see that the number of $\sim$ to a sieve graph is
$$
\leq \frac{1}{m!} (2k + m)^{2k} \prod_{j = 1}^{r} m^{l_j + e_j},
$$
where $e_j=0$ if the $j$th thread is closed and $e_j=1$ if it is open.
%
Summing over all $l_j\leq \ell$ and $e_j=0,1$, we obtain
\[\leq \frac{(2 k + m)^{2 k}}{m!} \left(\frac{1+\frac{1}{m}}{1-\frac{1}{m}}
m^{\ell+1}\right)^r\leq (2 k + m)^{2 k} \frac{2^r m^{(\ell+1) r}}{m!}.\]
Each of the $r$ possible threads intersects the horizontal path at one of its $2k + 1$
vertices. Thus there are at most $\binom{2k + r}{r} \leq (2k + 1)^{r}$ ways of choosing how the
threads intersect the horizontal path. 
Consequently, the number of pairs $(G,\sim)$ with given $r$ and $m$ is
\[\leq (2 k + m)^{2 k} (2 k + 1)^r 
\frac{2^r m^{(\ell+1) r}}{m!}.\]
Therefore, the triple sum in \eqref{eq:blug} is at most
\begin{equation}\label{eq:lateran}\begin{aligned}
    2^{2 k} \mathscr{L}^{2 k + m} (2 k + m)^{2 k}
&\sum_{\frac{m}{\ell+1} \leq r\leq m}
2^{r (\ell + 2)}
(2 k + 1)^r
\frac{(2 m^{\ell+1})^r}{m!}
\left(\frac{\log H}{\mathscr{L} H_0}\right)^{r}\\
=
\frac{2^{2 k} \mathscr{L}^{2 k + m}  (2 k + m)^{2 k} 
}{m!} &\sum_{\frac{m}{\ell+1} \leq r\leq m}
\left(\frac{
2^{\ell+3} m^{\ell+1} (2 k + 1) 
  \log H}{\mathscr{L} H_0}\right)^{r}
.\end{aligned}\end{equation}
We are assuming that $4 (2 k + 1) (2 m)^{\ell + 1} \log H/\mathscr{L} H_0\leq 1/2$, and so our bound in \eqref{eq:lateran} is
\[\begin{aligned}&\leq \frac{2 (4 k + 2 m)^{2 k} \mathscr{L}^{2 k + m}}{m!} 
\left(\frac{
2^{\ell+3} m^{\ell+1} (2 k + 1) 
\log H}{\mathscr{L} H_0}\right)^{\frac{m}{\ell+1}}\\
&\leq
2 (4 k + 2 m)^{2 k} \mathscr{L}^{2 k + m} \frac{(2 m)^m}{m!}
\left(\frac{(8 k + 4) 
\log H}{\mathscr{L} H_0}\right)^{\frac{m}{\ell+1}}.
\end{aligned}\]
We finish by using the inequality $m^m/m! \leq \sum_{j=0}^\infty m^j/j! = e^m$.
\end{proof}
\subsection{Bounding the main term}\label{subs:boundmas}
We will now prove an estimation that will later allow us to bound
a main term coming from an application of Prop.~\ref{prop:sievorino}.
The main ideas are much as in \S \ref{subs:contrert}, but we will have to
use a notion of
{\em strong non-redundancy} that is slightly different from
{\em non-redundancy} as defined  in \S \ref{subs:contrert}.
\begin{prop}\label{prop:boundmas}
    Let $\mathbf{N} = (N,2N]\cap \mathbb{Z}$ and
    $\mathbf{k}=\{1,2,\dotsc,2 k\}$, $k\geq 1$.
      Let $\mathbf{P}\subset [H_0,H]$ be a set of primes with
      $\mathscr{L} = \sum_{p\in \mathbf{P}} 1/p \geq 1$, where $H\geq H_0\geq e^2$.
   Let $\ell\leq H_0$. 

      Let $p_1,...,p_{2 k}\in \mathbf{P}$, $\sigma_1,\dotsc,\sigma_{2
        k}\in \{-1,1\}$ and $\mathbf{l}\subset \mathbf{k}$ be given. Define
$\boldsymbol{\beta} = (\beta_1,\dotsc,\beta_{2 k})$ by
      $\beta_i = \sigma_1 p_1 + \dotsc + \sigma_i p_i$.
      Define
      $\mathscr{W}_{\ell,\mathbf{P}}(\boldsymbol{\beta})$
      as before.

      Let $A=\{n\in \mathbb{Z}: p_i|n+\beta_i \; \forall i\in
      \mathbf{l}\}$; assume $A$ is non-empty.
      
Then, for any $m\geq 1$ such that
 $H_0\geq 8 (2 k + 1) (\log H) (4 \mathscr{L})^\ell
  (m + 2 k)^{\ell+1}$, 
\begin{equation}\label{eq:mauric}
  \mathop{\sum_{R\subset \mathscr{W}_{\ell,P}(\boldsymbol{\beta})^\cap}}_{R\cap A \ne \emptyset,\; \omega(\mathfrak{q}(R))\leq m}
\frac{2^{\omega(\mathfrak{q}(R))} }{\mathfrak{q}(R\cap A)}
\ll\frac{2^{\omega(\mathfrak{q}(A))}}{\mathfrak{q}(A)},\end{equation}
where the implied constant is absolute.
  \end{prop}
\begin{proof}
  Let $d$ range over the divisors of $\mathfrak{q}(A)$. Let us consider
  the contribution to the left side of \eqref{eq:mauric} of
  progressions $R$ with $(\mathfrak{q}(R),\mathfrak{q}(A))=d$, $d$ fixed.
  For given $d$, the intersection
  $R\cap A$ determines $R$, provided that it is
  taken together with the conditions
  $(\mathfrak{q}(R),\mathfrak{q}(A))=d$, $R\cap A \ne \emptyset$.

  By definition of $W_{\ell,P}(\boldsymbol{\beta})^\cap$, we can write $R$ as an intersection
  of progressions of the form $P-\beta_i$ with $P\in \mathscr{W}_{\ell,\mathbf{P}}$.
  These progressions, together with $\{p_i\}$, $\{\sigma_i\}$ and
  $\mathbf{l}$, define a valid tuple
  $((G,\sim),\mathbf{l},\vec{\sigma},(p_{[x]}))$, where
  $G$ is a sieve graph, $\sim$ is an equivalence class extending the
  equivalence class on $\mathbf{k}$ defined by
  $i\sim i' \Leftrightarrow p_i = p_{i'}$, and $p_{[x]}$ is such that
  $p_{[i]}=p_i$ for $i\in \mathbf{k}$.
  Since we now care only about the intersection $R\cap A$, 
  we can assume that $(G,\sim)$ is in the set $\mathbf{W}_{k,\ell,\mathbf{l}}'$ of
  {\em strongly non-redundant} pairs $(G,\sim)$, meaning pairs such
  that every thread contains at least one edge (possibly a witness)
  whose equivalence class $[x]$ contains no edge in any other thread
  and no edge in $\mathbf{l}\subset \mathbf{k}$, where, as before,
  we are identifying the edges of the horizontal path with $\mathbf{k}$.
  (If a thread contains only edges equivalent either to edges in other threads
  or to edges in $\mathbf{l}$, we may remove that thread without
  changing $R\cap A$. Repeat.)
  It is clear that $\mathfrak{q}(R\cap A)$ equals $\mathfrak{q}(A)$ times
  the product of $p_{[x]}$ over equivalence classes
  $[x]$ disjoint from $\mathbf{l}$ and not
  contained in $\mathbf{k}$.
  Thus,
  \begin{align*}
  \mathop{\sum_{R\subset W_{\ell,P}(\boldsymbol{\beta})^\cap}}_{R\cap A \ne \emptyset}
  \frac{2^{\omega(\mathfrak{q}(R))}}{\mathfrak{q}(R\cap A)} &\leq
  \frac{2^{\omega(\mathfrak{q}(A))}}{\mathfrak{q}(A)}
  \sum_{d|\mathfrak{q}(A)}
\mathop{  \sum_{\substack{(G,\sim)\in \mathbf{W}_{k,\ell,\mathbf{l}}' \\
\forall i\in \mathbf{l}: [i]\subset \mathbf{k}
 \Leftrightarrow p_i\nmid d\\
 \forall i,j\in \mathbf{k}: i\sim j \Leftrightarrow p_i=p_j}}}_{
  |\{[x]\not\subset \mathbf{k}\}|\leq m}
  \mathop{\mathop{
    \mathop{\sum_{\nu_x=\sigma_x \;\text{for}\; x\in \mathbf{k}}}_{
      \nu_x=\pm 1 \;\text{for}\; x\in E(G)\setminus \mathbf{k}}}_{p_{[x]}\in \mathbf{P}
      \;\text{for}\; [x]\cap \mathbf{k}=\emptyset}}_{
    ((G,\sim),\mathbf{l},\vec{\nu},(p_{[x]}))\; \text{valid}}\;
  \mathop{\prod_{[x]: [x]\cap \mathbf{l}=\emptyset}}_{
    [x]\not\subset \mathbf{k}}
  \frac{2}{p_{[x]}}
  \end{align*}
  where $\vec{\nu} = (\nu_x)_{x\in E(G)}$. Since $d|\mathfrak{q}(A)$ is determined
  by the condition $\forall i\in \mathbf{l}: [i]\subset \mathbf{k}
  \Leftrightarrow d$, we can simply omit both the sum over $d|\mathfrak{q}(A)$
  and that condition.

  We now proceed as in the proof of Lemma \ref{lem:mandrake}.
  Let $(G,\sim)\in \mathbf{W}_{k,\ell,\mathbf{l}}'$ be given. For each thread,
  we may choose an edge whose equivalence class $[x]$ contains no edge
  in any other thread and no element of $\mathbf{l}$. We allow
  the primes associated to the other equivalence classes to range freely
  (except of course for the primes associated to equivalence classes 
  intersecting $\mathbf{l}$; those are fixed already).
  For $y$ the chosen edge in each thread, the sum
  $\sum 1/p_{[y]}$ over allowed values of $p_{[y]}$ is $\leq (\log H)/H_0$,
  just as in Lemma \ref{lem:mandrake}.
  (The only case that works out a little differently is that of
  $y$ such that $[y]$ intersects
  $\mathbf{k}\setminus \mathbf{l}$: then $p_{[y]}$ is already fixed,
 and, evidently, $1/p_{[y]} \leq 1/H_0 \leq (\log H)/H_0$.)
  Hence
  \[
    \mathop{\mathop{
    \mathop{\sum_{\nu_x=\sigma_x \;\text{for}\; x\in \mathbf{k}}}_{
      \nu_x=\pm 1 \;\text{for}\; x\in E(G)\setminus \mathbf{k}}}_{p_{[x]}\in \mathbf{P}
      \;\text{for}\; [x]\cap \mathbf{k}=\emptyset}}_{
    ((G,\sim),\mathbf{l},\vec{\nu},(p_{[x]}))\; \text{valid}}\;
  \mathop{\prod_{[x]: [x]\cap \mathbf{l}=\emptyset}}_{
    [x]\not\subset \mathbf{k}}
  \frac{2}{p_{[x]}}
  \leq 4^{\rho} \mathscr{L}^{s'-r} \left(\frac{\log H}{H_0}\right)^r,
  \]
  where $r$ is the number of threads in $G$, $s'$ is the number of
  equivalence classes $[x]$ of $\sim$ such that $[x]\cap \mathbf{k}=\emptyset$,
  and $\rho$ is the number of edges in threads in $G$.
  Clearly, $s'\leq \rho\leq r (\ell+1)$, and so
  $4^\rho \mathscr{L}^{s'-r} \leq 4^{r (\ell+1)} \mathscr{L}^{r \ell}$.
  It is also clear that $s'\leq |\{[x]\not\subset \mathbf{k}\}|$,
  and so $s'\leq m$.
  
  Now we have to count pairs $(G,\sim)$ with given $r$ and $s'$, much as in
  Proposition \ref{prop:errorterm}. Since the restriction of $\sim$ to
  $\mathbf{k}$ is already given, we just need to count the possible assignments
  of edges in threads to different equivalence classes, as well as, of course,
  possible shapes of $G$ with given $r$ and $s'$.
As in Proposition \ref{prop:errorterm}, each thread intersects the horizontal path at one of its $2 k + 1$
    vertices. Thus,
  the number of pairs $(G,\sim)$ with given $r$ and $s'$ is
  \[\begin{aligned}
  \leq \frac{(2 k+1)^r}{(s')!} &\prod_{j=1}^r \sum_{l_j\leq \ell} ((s' + 2k)^{l_j+1} + (s'+2 k)^{l_j})\leq
  \frac{(2 k+1)^r}{(s')!} \prod_{j=1}^r \frac{1 + \frac{1}{s'+2 k}}{1 - \frac{1}{s'+2 k}}
  (s'+2k)^{\ell+1}\\ &\leq \frac{1}{(s')!}
  \left(2 (2 k + 1) (s' + 2 k)^{\ell+1}\right)^r
  .
  \end{aligned}\]
  for $r\geq 1$, since then $s'+2k\geq r + 2 k \geq 3$. The bound also holds
  for $r=0$.

  We must thus bound
  \begin{equation}\label{eq:amrot}
    \sum_{r=0}^\infty \sum_{s'=r}^{\min(m,r (\ell+1))}
  \frac{1}{(s')!} 4^{r (\ell+1)} \mathscr{L}^{r\ell} \left(\frac{\log H}{H_0}\right)^r
  \left(2 (2 k + 1) (s' + 2 k)^{\ell+1}\right)^r
  .\end{equation}
  Since $\sum_{s'=r}^\infty 1/(s')! \leq e/r!$, we see that
  the expression in \eqref{eq:amrot}
  is at most
  \begin{equation}\label{eq:fnale}
    e \sum_{r=0}^\infty \frac{1}{r'!}
  \left(\frac{4^{\ell+1} \mathscr{L}^{\ell} \log H}{H_0}\right)^r
  \left(2 (2 k + 1) (m + 2 k)^{\ell+1}\right)^r.\end{equation}
  Since we assume $H_0\geq 8 (2 k + 1) (\log H) (4 \mathscr{L})^\ell
  (m + 2 k)^{\ell+1}$, the sum in \eqref{eq:fnale} is $\leq e$,
  and so the double sum in \eqref{eq:amrot} is $\leq e^2$.
\end{proof}

\section{Obtaining cancellation from lone primes}\label{sec:returns}

We will study the effect of restricting our operator $A$ to a set
    $X=X_0\cap Y_\ell$, where $X_0\subset \mathbf{N} = \{N+1,N+2,\dotsc,2 N\}$ is the set of integers in $\mathbf{N}$ having $\leq K \mathscr{L}$ prime factors in
    a set $\mathbf{P}$,
    and $Y_{\ell}$ is as we defined it in \S \ref{subs:sievappr}.
    We will also show how to obtain cancellation in $\Tr (A|_X)^{2 k}$ among walks
    for which there are many steps whose length is different
    from that of all other steps in the walk. (We say that such non-repeated
    lengths are ``lone primes''; recall that the length of any step in the walk
    is a prime in $\mathbf{P}$)

    If we were not restricting $A$ to $X$, we would be showing cancellation
    in $\Tr A^{2 k}$ among walks with many lone primes -- and that task turns out
    to be very easy; in fact, in $\Tr A^{2 k}$, among walks with even a single
    lone prime, cancellation is nearly complete. If we only had to restrict
    $A$ to $X$, and not show cancellation, we would not have a task to carry
    out at all: we would just be getting rid of some possible walks to
    count in the trace. The issue is to restrict $A$ to $X$ and still get
    cancellation in $\Tr (A|_X)^{2 k}$ -- enough cancellation, that is, to show
    that the contribution of walks with many lone primes is small.

    The main result of this section is the following proposition.

    \begin{prop} \label{prop:maincanc}
  Let $\mathbf{N} = (N,2N]\cap \mathbb{Z}$ and
    $\mathbf{k}=\{1,2,\dotsc,2 k\}$, where $k\geq 4$.
    Let $\mathbf{P}\subset [H_0,H]$ be a set of primes.
    Write $\mathscr{L}$ for
    $\sum_{p\in \mathbf{P}} 1/p$, and assume $\mathscr{L}\geq e$.
    Assume as well that $H_0\geq 4 k^2$, $H_0\geq (\log H+2)^4$ and
    $H^{2 k \max(C_0, 12 (K \mathscr{L}+1))}\leq N$, where $C_0$ is an absolute constant, and that $8 \mathscr{L} H_0 H\leq C_1^{2 k}$ for some $C_1$.
    Let 
    $1\leq \ell \leq \min\Big(k,\frac{3 \log H_0}{8 \log \mathscr{L}
      \log(80 \mathscr{L} k)}
    -2\Big)$.

    
      Given $\vec{\sigma}\in \{-1,1\}^{2 k}$ and
  $\vec{p}\in \mathbf{P}^{2 k}$,
  let $\beta_i = \beta_i(\vec{\sigma},\vec{p})$ be defined by
  $\beta_i = \sum_{j=1}^i \sigma_j p_j$ for $0\leq i\leq 2 k$.
Let $\Pi_{\vec{p}}$
  be the partition of $\mathbf{k}$ corresponding
  to the equivalence relation $\sim$ defined by $i\sim j\Leftrightarrow
  p_i=p_j$. We write $p_{[i]}$ to mean $p_{i}$, since $p_i$ depends
only on the equivalence class $[i]$ of $i$.

  Let $A$ be as in \eqref{eq:defA}; let $X_0$ be as above, and let
  $Y_\ell$ be as in \S\ref{subs:sievappr}.
     Define
     \begin{equation}\label{eq:mahedul1}
       \mathscr{S}_1 = \sum_{\mathcal{L} \subset \mathbf{k}} \mathscr{L}^{-|\mathcal{L}| / 2} \sum_{\mathbf{l} \subset \mathbf{k}}
  \mathop{\sum_{
      (\vec{p}, \vec{\sigma}) \in \mathscr{C}(k,\mathcal{L},\mathbf{l})}}_{
      \sigma_1 p_1 + \ldots + \sigma_{2k} p_{2k} = 0}
    \prod_{i\not\in \mathbf{l}} \frac{1}{p_i} \prod_{[i]\in \Pi,
      [i]\not\subset \mathbf{k}\setminus \mathbf{l}} \frac{1}{p_{[i]}},\end{equation}
     \begin{equation}\label{eq:mahedul2}
    \mathscr{S}_2 = 
\frac{1}{N} \mathop{\sum_{\vec{p} \in \mathbf{P}^{2 k}} \sum_{\vec{\sigma} \in \{ \pm 1\}^{2 k}}}_{|\mathbf{S}(\vec{p},\vec{\sigma})|> \frac{2 k}{\log \mathscr{L}}}
   \sum_{\mathbf{l}\subset \mathbf{k}}\;
      \mathop{\sum_{n\in \mathbf{N}}}_{p_i|n+\beta_i \forall i\in \mathbf{l}}
      \prod_{i \in \mathbf{k}\setminus \mathbf{l}} \frac{1}{p_i},\end{equation} 
  where 
  $\mathscr{C}(k,\mathcal{L},\mathbf{l})$ is the set of all pairs
  $(\vec{p},\vec{\sigma})$ with $\vec{p}\in \mathbf{P}^{2 k}$,
  $\vec{\sigma}\in \{-1,1\}^{2 k}$
such that
  \begin{enumerate}
  \item\label{it:singleta} the singletons of $\Pi_{\vec{p}}$ are exactly $\{[i]: i\in \mathcal{L}\}$,
  \item\label{it:wax0y} there is an $n\in \mathbf{N}$ such that $n+\beta_0,n+\beta_1,\dotsc,n+\beta_{2 k}$ are all in $X_0\cap Y_{\ell}$ and $p_i|n+\beta_i$ for every
    $i\in \mathbf{l}\cap (k\setminus \mathcal{L})$,
  \end{enumerate}
    and $\mathbf{S}(\vec{p},\vec{\sigma})$ is the set of all
    $i\in \mathbf{k}$ such that (a) $[i]$ is a singleton, (b) there
    is a $j\not\in \{i,i-1\}$ such that either (i) $p_i|\beta_i-\beta_j$
    and $\beta_j\ne \beta_{i-1}, \beta_i$, or
    (ii) $\beta_i=\beta_j$ and $[j]$ is a singleton.

    Then there exists $\mathscr{X}\subset \mathbf{N}$ with
    $|\mathbf{N}\setminus \mathscr{X}|\ll
N e^{-(K \log K - K + 1) \mathscr{L}}
+ N/\sqrt{H_0}$ such that
    every eigenvalue of $A|_{\mathscr{X}}$ has absolute value
    \[\ll \max\Big(C_1 \mathscr{S}_1^{\frac{1}{2 k}},
    C_1 \mathscr{S}_2^{\frac{1}{2 k}},C_1,\sqrt{\mathscr{L}}\Big),\]
where the implied constant is absolute.
    \end{prop}

    When we say that $n + \beta_u, \ldots, n + \beta_v$ is a walk in $\Gamma|_{Y_\ell}$, we mean not just that $n+\beta_i\in Y_{\ell}$ for $u\leq i\leq v$, but also
    that the {\em edges} between $n+\beta_u$ and $n+\beta_{u+1}$, between
    $n+\beta_{u+1}$ and $n+\beta_{u+2}$, etc., lie in $\Gamma$.
    
In order to prove Proposition \ref{prop:maincanc},
we need to have a detailed understanding of the sets $X_0$ and $Y_\ell$.
We saw in (\S \ref{subs:sievappr}) that we can approximate $Y_\ell$ by the superposition
of a moderate number of arithmetic progressions whose moduli we can control. In \S \ref{subs:cancap},
we will
obtain results on the distribution of $X_0$
in arithmetic progressions. We will also show that one can introduce a
smoothing function into the problem.
We will then be able to prove Proposition \ref{prop:maincanc}.
Proving $\mathscr{S}_1, \mathscr{S}_2\leq R^{2 k}$ for suitable $R$
will take the rest of the paper.

Let us begin by showing that $\mathbf{N} \backslash X$ is small.
The following lemma is completely standard.
\begin{lemma}\label{lem:X0issmall}
  Let $\mathbf{P}\subset [1,H]$ be a set of primes. 
  Let $X_0$ be the set of integers $n\in \mathbf{N} = \{N+1,N+2,\dotsc, 2 N\}$
  having $\geq K \mathscr{L}$ factors in $\mathbf{P}$. Then
  \begin{equation}\label{eq:carego}
    |\mathbf{N}\setminus X_0|\ll N e^{-(K \log K - K + 1) \mathscr{L}},\end{equation}
  where the implied constant is absolute.
\end{lemma}
\begin{proof}
  By \cite[(1.11)]{zbMATH03515559}; alternatively, by
  \cite[Thm.~1]{zbMATH01026275} together with
  the simple bound (resulting from Rankin's trick) in, say, \cite[Lem.~2.1]{zbMATH01203259}.
\end{proof}
    {\em Remark.}
    Since, in our work, $H^{K \mathscr{L}}$ is small compared to $N$, we could
 easily   prove \eqref{eq:carego} ourselves
    in the stronger form
$\leq N  e^{-(K \log K - K + 1) \mathscr{L}} + \text{{\em error term}}$ using
a Chernoff bound
\[\Prob(X>K \mathscr{L}) \leq
\frac{\mathbf{E}\left(e^{X \log K }\right)}{e^{K \mathscr{L} \log K}}
  \leq e^{-(K \log K - K + 1) \mathscr{L}}\]
for $X$ the sum of independent Bernoulli variables $X_p$ with
$\Prob(X_p=1)=1/p$, together with the usual Kubilius model
(Lemma \ref{lem:kubiliusmain} for $\ell=1$ and $\alpha_i=0$).

\begin{lemma}\label{lem:Ylissmall}
  Let $\mathbf{N} = \{N+1,N+2,\dotsc, 2 N\}$.
  Let $\mathbf{P}\subset [H_0,H]$ be a non-empty set of primes,
  where $H\geq H_0> 1$. Let $Y_\ell = Y_{\ell,\mathbf{P}}$ be as defined in
  \S \ref{sec:sievrecur}, with $\ell\leq H_0$. Then
  \begin{equation}\label{eq:rarega}
    |\mathbf{N}\setminus Y_\ell|\leq 3^\ell \mathscr{L}^{\ell} 
    \left(\frac{\log H}{\log H_0} + 1 + \frac{1}{\mathscr{L}}\right) \cdot
    \frac{N}{H_0}
    + 3^{\ell} H^{\ell+1},\end{equation}
  where $\mathscr{L} = \sum_{p\in \mathbf{P}} 1/p$.
\end{lemma}
\begin{proof}
  By definition, $|\mathbf{N}\setminus Y_\ell|\leq S_1 + S_2$, where
\[S_1 = \sum_{l\leq \ell} 
\mathop{\sum_{(\vec{p},\vec{\sigma})\in C_l}}_{\sum_{i=1}^l \sigma_i p_i =0}
|\{n\in \mathbf{N}: p_j|n+\sum_{i=1}^{j-1} \sigma_i p_i\; \forall 1\leq j\leq \ell\}|,\]
\[S_2 = \sum_{l\leq \ell}
\mathop{\mathop{\mathop{\sum_{(\vec{p},\vec{\sigma})\in C_l}}_{
\sum_{i=1}^l \sigma_i p_i \ne 0}}_{\exists p_0|\sum_{i=1}^l \sigma_i p_i \st}}_{p_0\ne p_i \forall 1\leq i\leq 2k}
 |\{n\in \mathbf{N}: p_j|n+\sum_{i=1}^{j-1} \sigma_i p_i\; \forall\, 0\leq j\leq \ell\}|,\]
 where $C_l$ is the set of all $(\vec{p},\vec{\sigma})\in
 \mathbf{P}^l\times \{-1,1\}^l$ for which
 there are no $1\leq i<j<i'\leq \ell$ such that $p_j\ne p_i=p_{i'}$ and no
 $1\leq i <\ell$ such that $p_{i+1}=p_i$ and $\sigma_{i+1}=-\sigma_i$.

 Clearly \[\begin{aligned}S_1 &\leq
 \sum_{l\leq \ell} 
 \mathop{\sum_{(\vec{p},\vec{\sigma})\in C_l}}_{\sum_{i=1}^l \sigma_i p_i =0}
 \left(\frac{N}{\prod_{p\in \{p_1,\dotsc,p_l\}} p} + 1\right)\leq
 \sum_{r\leq \ell} \mathop{\sum_{c_1,\dotsc,c_r\ne 0}}_{|c_1|+\dotsc+|c_r|\leq
   \ell}\; \mathop{\sum_{p_1,\dotsc,p_r\in \mathbf{P}}}_{c_1 p_1 + \dotsb +
   c_r p_r = 0} \left(\frac{N}{\prod_{i=1}^r p_i} + 1\right)\\
 &\leq \frac{N}{H_0}
 \sum_{r\leq \ell} \mathop{\sum_{c_1,\dotsc,c_r\ne 0}}_{|c_1|+\dotsc+|c_r|\leq
   \ell} \sum_{p_1,\dotsc,p_{r-1}\in \mathbf{P}} \frac{1}{\prod_{i=1}^{r-1} p_i} 
 + 3^{\ell} |\mathbf{P}|^{\ell} \leq
 3^\ell \mathscr{L}^{\ell-1} \frac{N}{H_0} + 3^{\ell} H^\ell,\end{aligned}\]
 since $p_1,\dotsc,p_{r-1}$ determine $p_r$ by $c_1 p_1 + \dotsc + c_r p_r = 0$.
 Similarly,
 \[\begin{aligned}S_2 &\leq
 \sum_{r\leq \ell} \mathop{\sum_{c_1,\dotsc,c_r\ne 0}}_{|c_1|+\dotsc+|c_r|\leq
   \ell}\; \mathop{\mathop{\sum_{p_0,p_1,\dotsc,p_r\in \mathbf{P}}}_{
c_1 p_1 + \dotsb + c_r p_r \ne 0}}_{p_0 | c_1 p_1 + \dotsb + c_r p_r}
 \left(\frac{N}{\prod_{i=0}^r p_i} + 1\right)\\
 &\leq \frac{N (\Delta + 1)}{H_0} \sum_{r\leq \ell} \mathop{\sum_{c_1,\dotsc,c_r\ne 0}}_{|c_1|+\dotsc+|c_r|\leq
   \ell} \sum_{p_1,\dotsc,p_r\in \mathbf{P}} \frac{1}{\prod_{i=1}^r p_i}
 +  3^{\ell} |\mathbf{P}|^{\ell+1}\\ &\leq
 3^\ell \mathscr{L}^{\ell} \frac{N (\Delta+1)}{H_0} + 3^{\ell} (H-1)^{\ell+1},\end{aligned}\]
 where $\Delta = \log H/\log H_0$.  (Since
 $0 < |c_1 p_1 + \dotsc + c_r p_r| \leq \ell H \leq H_0^{\Delta+1}$,
 we see that $c_1 p_1 + \dotsc + c_r p_r$ has $\leq \Delta+1$ divisors in
 $\mathbf{P}$.)
\end{proof}

 \subsection{Cancellation over arithmetic progressions}\label{subs:cancap}

Let
$$
f_p(n) = \begin{cases}
  1 - \frac{1}{p} & \text{ if } p | n, \\
  - \frac{1}{p} & \text{ otherwise,}
\end{cases}\;\;\;\;\;\;
g_p(u) = \begin{cases}
  1 - \frac{1}{p} & \text{ if } u=1, \\
  - \frac{1}{p} & \text{ if } u=0.\end{cases}
$$

We recall that, for $r\geq 1$, the
Sobolev norm $|f|_{k,r}$ of a smooth function $f$ is defined as
$$
| f |_{k,r} = \Big ( \sum_{i = 0}^{k} | f^{(i)}|_r^{r} \Big )^{1/r}, 
$$
where, for a measurable function $g$, we let $|g|_{r}$
denote the usual $L^r$ norm $\left(\int_{\mathbb{R}} |g(x)|^{r} dx\right)^{1/r}$.

We begin with the following lemma. In it, we will obtain cancellation
due to factors $f_{p_i}(n+\alpha_i)$, all $p_i$ distinct, with
$n$ ranging on an arithmetic progression. We will take into account factors
of the form $\mathbf{1}_{\omega_{\mathbf{P}}(n) \leq K \mathscr{L}}$, coming from the
restriction to $X_0$, and also some helpful smoothing factors $W_i$ that
will be removed later.

\begin{lemma}\label{lem:recast}
  Let $a + q \mathbb{Z}$ be an arithmetic progression and
  $\mathbf{P}\subset [H_0,H]$ a set of primes, where $H\geq H_0\geq 1$.
  Assume that $\mathscr{L} = \sum_{p\in \mathbf{P}} 1/p \geq 2$.
  Let $N\geq 1$, $K\geq 1$ and $1 \leq k_0 \leq l$ be such that
    $l\leq \sqrt{H_0}$
  and $\log N\geq \max(C_0, 12 (K \mathscr{L}+1)) l \log H$,
  where $C_0$ is as in Lemma \ref{lem:kubiliusprax}.
  Let $(p_1,\dotsc,p_{k_0}) \in \mathbf{P}^{k_0}$ be distinct, with
  $p_i \nmid q$ for all $1\leq i\leq k_0$.
  Let
  $\alpha_1, \ldots, \alpha_{l}$
  be integers.

 Let $\mathcal{P}_j\subset \mathbf{P}$, $1\leq j\leq l$. Assume that
  $p_i \not \in \mathcal{P}_j$
 for all $1 \leq i \leq k_0$, $1 \leq j \leq l$, and also that
  \begin{enumerate}
    \item for each $1\leq j\leq l$, there are at most $D$
  indices $1\leq i\leq k_0$ such that $p_i | \alpha_i- \alpha_j$,
  \item for each $1\leq i\leq k_0$, there are at most $D$ indices
  $1\leq j\leq l$ such that $p_i | \alpha_i - \alpha_j$.
  \end{enumerate}

  Let $W_i:[0,\infty)\to [0,\infty)$ be in $C^\infty$. Assume that
    $W_i(x)=0$ for $x$ in an open neighborhood of $0$ and
      $W_i(x)=1$ for $x\geq 1$.
    Then
    \begin{equation}
      \begin{aligned}\label{eq:kittycat}
  & \sum_{\substack{n \in (N, 2N] \\ n \equiv a \mo q\\
      p | n + \alpha_i \Rightarrow p \not \in \mathcal{P}_i}}
  \prod_{i = 1}^{k_0} 
       f_{p_i}(n + \alpha_i) W_i\Big(\frac{\omega_{\mathbf{P}}(n + \alpha_i)}{\sqrt{\mathscr{L}}}\Big)
      \prod_{i = k_0 + 1}^{l} W_i \Big ( \frac{\omega_{\mathbf{P}}(n + \alpha_i)}{\sqrt{\mathscr{L}}} \Big ) \prod_{i = 1}^{l} \mathbf{1}_{\omega_{\mathbf{P}}(n + \alpha_i) \leq K \mathscr{L}}\\ &\ll \frac{N/q}{\mathscr{L}^{k_0/2}} \prod_{i = 1}^{{k_0}} \frac{1}{p_i} 
      \Big(   D^{k_0} \prod_{i = 1}^{l} \Big ( | W_i |_{D + 2, 1} + (D + 2)!
      (C \mathscr{L})^{D/2+1} e^{- K \sqrt{\mathscr{L}}}\Big)
+     \prod_{i=1}^l |W_i|_\infty 
      \Big ) + N^{\frac{3}{4}}
  \end{aligned}\end{equation}
  where $C$ and the implied constant are absolute.
\end{lemma}

In other words, we gain a factor of roughly $O(\mathscr{L}^{-k_0/2})$
over the trivial bound.
We might be able to obtain a stronger bound by shifting
the integrals we will see in the proof to the left of $\Re u_i = 0$. However, doing so would introduce unwanted combinatorial complications both in the proof of the
Lemma and in the proof of the main theorem.

When we later apply the Lemma, we will always set $D=2$.
\begin{proof}
  For $1\leq i\leq l$, let
  $U_i:[0,\infty)\to \mathbb{R}$ be a smooth function such that $U_i(x)=1$ for
    $x\leq \lfloor K \mathscr{L} \rfloor / \sqrt{\mathscr{L}}$ and $U_i$ is compactly
    supported on $[0, (\lfloor K \mathscr{L} \rfloor +1/2)/ \sqrt{\mathscr{L}}
    ]$.
    We can ensure that, for all $r\in \mathbb{Z}_{>0}$, the $r$th derivative of $U_i$ is bounded by
    $C^r r! \mathscr{L}^{r/ 2}$, $C$ a constant.
    Define $V_i(x) = W_i(x) U_i(x)$.
  By Lemma \ref{lem:kubiliusprax}, the left-side expression in \eqref{eq:kittycat} equals $N$ times
  \begin{equation}\label{eq:covid}\begin{aligned}
    \frac{1}{q} &\mathbb{E} \Big [  \prod_{i = 1}^{k_0} g_{p_i}(Z_{p_i}^{(\alpha_i)}) 
      V_i \Big (\frac{1}{\sqrt{\mathscr{L}}} \sum_{p \in \mathbf{P}} Z_{p}^{(\alpha_i)} \Big ) \prod_{i = k_0 + 1}^{l} V_i \Big ( \frac{1}{\sqrt{\mathscr{L}}} \sum_{p \in \mathbf{P}} Z_{p}^{(\alpha_i)} \Big ) \mathbf{1} \Big ( Z_{p}^{(\alpha_i)} = 0 \ \forall p, i : p \in \mathcal{P}_i \Big ) 
      \Big ] \\ & + O \Big ( \frac{1}{q} 
\mathbb{E} \Big [
  \prod_{1 \leq i \leq k_0} |g_{p_i} (Z_{p_i}^{(\alpha_i)})|
  \Big ] \prod_{i=1}^{l} |W_i|_\infty \cdot e^{- \frac{\log N}{2 \log H}}
+ \frac{H^{l L}}{L!^l N^{1/3}} \Big )
  \end{aligned}\end{equation}
  where $L := \lfloor K \mathscr{L} \rfloor$ and
  where the variables
  $Z_{p}^{(\alpha_i)}$ have the properties in Lem.~\ref{lem:kubiliusmain},
  with $\mathcal{P}_i$ and $q$ as in our statement.
Here, the error terms are easily estimated:
  since all $p_i$ are distinct,
  $$
  \mathbb{E} \Big [ \prod_{i = 1}^{k_0} |g_{p_i}(Z_{p_i}^{(\alpha_i)})| \Big ]  = \prod_{i = 1}^{k_0} \mathbb{E} \Big [ |g_{p_i}(Z_{p_i}^{(\alpha_i)})| \Big ] = \prod_{i = 1}^{k_0} \frac{2}{p_i} \Big (1 - \frac{1}{p_i} \Big ),
  $$
  and, by our assumption $\log N\geq 12 K \mathscr{L} l \log H$,
  we see that $e^{-\frac{\log N}{2 \log H}} \leq e^{-6 K \mathscr{L} l}\leq
  (e \mathscr{L})^{- 6l}$. Moreover, $H^{l L}
  \leq H^{l (K \mathscr{L}+1)} \leq N^{1/12}$, so the last term in 
\eqref{eq:covid} is $O(N^{-1/4})$.

  To evaluate the contribution of the main term, we use an inverse Laplace transform:
  $$
  V_i \Big (\frac{1}{\sqrt{\mathscr{L}}} \sum_{p \in \mathbf{P}} Z_{p}^{(\nu)} \Big ) = \frac{\sqrt{\mathscr{L}}}{2\pi i} \int_{1/\sqrt{\mathscr{L}} - i \infty}^{1/\sqrt{\mathscr{L}} + i \infty} \widetilde{V_i}(\sqrt{\mathscr{L}} u) \exp \Big (- u \sum_{p \in \mathbf{P}} Z_{p}^{(\nu)} \Big ) du.
  $$
  where, as usual, the Laplace transform $\widetilde{V_i}$ is given by
    $$
  \widetilde{V_i}(s) = \int_{0}^{\infty} V_i(x) e^{-s x} dx = \frac{1}{s} \int_{0}^{\infty} V_i'(x) e^{-s x} dx .
  $$
  Since $W_i'(x)\ne 0$ implies that $x<1$ and
  $U_i'(x)\ne 0$ implies $x> \lfloor K \mathscr{L}\rfloor/\sqrt{\mathscr{L}}
  > 1$, 
    $$
  V_i^{(r)}(x) = W_i^{(r)}(x) + U_i^{(r)}(x)  \ , \ r \geq 1
  $$
  and so, by integration by parts,
  \begin{equation} \label{eq:boundsa}
    |\widetilde{V_i}(s)| \leq \frac{1}{|s|} \min \Big ( \frac{1}{|s|^{r}} \cdot \Big ( | W_i^{(r + 1)} |_{1} + | U_i^{(r + 1)}|_{1} \cdot e^{- \frac{
        \lfloor K \mathscr{L}\rfloor}{\sqrt{\mathscr{L}}} \Re s} \Big ) , | W_i' |_1 + | U_i' |_1 \Big )
  \end{equation}
  for any given integer $r \geq 0$ and $\Re s > 0$.

  Thus, our task reduces to understanding the analytic behavior of 
  \begin{equation} \label{eq:main22}
    \mathbb{E} \Big [ \prod_{i = 1}^{k_0} g_{p_i}(Z_{p_i}^{(\alpha_i)}) \exp \Big ( - u_i \sum_{p \in \mathbf{P}} Z_{p}^{(\alpha_i)} \Big ) \prod_{i = k_0 + 1}^{l} \exp \Big ( - u_i \sum_{p \in \mathbf{P}} Z_{p}^{(\alpha_i)} \Big )
       \mathbf{1} \Big ( Z_{p}^{(\alpha_i)} = 0 \ \forall i : p \in \mathcal{P}_i \Big ) 
      \Big ].
  \end{equation}
  Factoring \eqref{eq:main22} according to $p \in \mathbf{P}$, we can write \eqref{eq:main22} as $\mathcal{T}_1 \mathcal{T}_2 $, where
  \[
  \mathcal{T}_1 = \prod_{\substack{p \in \mathbf{P} \\ p \not \in \{p_1, \ldots, p_{k_0}\}}} \mathbb{E} \Big [ \exp \Big ( - \sum_{1 \leq j \leq l} u_j Z_{p}^{(\alpha_j)} \Big ) \mathbf{1} \Big ( Z_{p}^{(\alpha_i)} = 0 \ \forall i : p \in \mathcal{P}_i \Big ) 
    \Big ],\] 
  \[\mathcal{T}_2 = \prod_{i = 1}^{k_0} \mathbb{E} \Big [ g_{p_i}(Z_{p_i}^{(\alpha_i)}) \exp \Big ( - \sum_{1 \leq j \leq l} u_j Z_{p_i}^{(\alpha_j)} \Big )
    \Big ]
  .\]
  We bound $\mathcal{T}_1$ trivially simply using that
  $|\mathcal{T}_1| \leq 1$, which is true because
  $\Re u_j \geq 0$. Letting $\mathcal{A}\subset \mathbb{Z}/p\mathbb{Z}$
  be the set of residues $\{\alpha_j \mo p: 1\leq j\leq l\}$, we see that
    \begin{align*}
    \mathcal{T}_2 &= \prod_{i = 1}^{k_0} \Big ( \frac{1}{p_i}
    \Big (1 - \frac{1}{p_i} \Big)
    \exp \Big (  - \sum_{j\in \mathcal{D}_i} u_j \Big)
    - \frac{1}{p_i^2} \mathop{\sum_{\overline{\alpha}\in \mathcal{A}}}_{
  \overline{\alpha} \ne \alpha_i \mo p}
    \exp\Big(-\mathop{\sum_{1\leq j\leq l}}_{\overline{\alpha}=
      \alpha_j \mo p}
u_j\Big) -
\frac{p_i - |\mathcal{A}|}{p_i^2} \Big)\\
    &= \prod_{i = 1}^{k_0} \Big ( \frac{1}{p_i} \Big (1 - \frac{1}{p_i} \Big ) \Big (\exp \Big (  - \sum_{j \in \mathcal{D}_i} u_j \Big ) - 1 \Big ) + O^* \Big ( \frac{2 l}{p_i^2} \Big ) \Big ),\end{align*}
  where $\mathcal{D}_i$ denotes the set of those $1\leq j \leq l$
  for which $\alpha_j\equiv \alpha_i \mo p_i$.
  Since, for $\Re z\geq 0$, $|\exp(-z)-1| = \left|\int_0^z \exp(-s) ds\right|
  \leq |z|$, we conclude that 
  \[|\mathcal{T}_3|\leq \prod_{i=1}^{k_0} \frac{1}{p_i} \cdot \left(
  \sum_{j\in \mathcal{D}_i} |u_j| + \frac{2}{\sqrt{H_0}}\right).\]
  As $\Re u_i = 1/\sqrt{\mathscr{L}}$, the term $|u_i|$ is at least
  $1/\sqrt{\mathscr{L}}$, and so dominates over $1/\sqrt{H_0}$.
  Combining all these bounds, we can bound the first expected value in
  \eqref{eq:covid} by 
  \begin{align} \label{eq:bound}
    & 3^{k_0} \left(\frac{\sqrt{\mathscr{L}}}{2 \pi}\right)^l
    \iint_{1/\sqrt{\mathscr{L}} - i \infty}^{1/ \sqrt{\mathscr{L}} + i \infty} \prod_{i = 1}^{k_0} \Big ( \frac{1}{p_i} |\widetilde{V_i}(\sqrt{\mathscr{L}} u_i)|
    \sum_{j \in \mathcal{D}_i} |u_j| \Big ) \prod_{i = k_0 + 1}^{l} |\widetilde{V_i}(\sqrt{\mathscr{L}} u_i)| \prod_{i = 1}^{l} |d u_i|.
  \end{align}
  We now change variables $\sqrt{\mathscr{L}} u_i \rightarrow u_i$.
  We obtain that the above is equal to
  \begin{equation} \label{eq:integral}
    \frac{1}{\mathscr{L}^{k_0/2}} \prod_{i = 1}^{k_0} \frac{1}{p_i}\cdot 
    \iint_{1 - i \infty}^{1 +  i \infty} \Big ( \prod_{i = 1}^{k_0} \sum_{j \in \mathcal{D}_i} |u_j|  \Big ) \prod_{i = 1}^{l} |\widetilde{V_i}(u_i)| |d u_i|. 
  \end{equation}
We now write
$$
\prod_{i = 1}^{k_0} \Big ( \sum_{j \in \mathcal{D}_i} |u_j| \Big ) = \sum_{\substack{(j_1, \ldots, j_{k_0}) \in \mathcal{D}_1 \times \ldots \times \mathcal{D}_{k_0}}} \prod_{j = 1}^{l} |u_{i}|^{|\{v : j_v = j\}|}.
$$
By assumption, for every $1\leq j\leq l$, there are at
most $D$ indices $1\leq i \leq k_0$ such that $p_i | \alpha_i - \alpha_j$,
i.e., such that $j\in \mathcal{D}_i$.
Therefore, for any given choice of $(j_1, \ldots, j_{k_0}) \in \mathcal{D}_1 \times \ldots \times \mathcal{D}_{k_0}$, we have $|\{v : j_v = j\}| \leq D$. 

By \eqref{eq:boundsa} and the pointwise bound $|U_i^{(r + 2)}(x)| \leq C^{r+2} (r + 2)! \mathscr{L}^{r/2 + 1}$,
 $$
|\widetilde{V_i}(u)| \leq \frac{1}{|u|^{r+2}}
\Big ( | W_i^{(r + 2)} |_{1} + e (r + 2)! C^{r+2} \mathscr{L}^{r/2 + 1} e^{- K \sqrt{\mathscr{L}}}  \Big ) 
 $$
for $\Re u=1$ and
$r\in \mathbb{Z}_{\geq 0}$ arbitrary, where $C > 0$ is an absolute constant.
Thus,
$$
\int_{1 - i \infty}^{1 + i \infty} |u|^{r} |\widetilde{V_i}(u)| |d u| \leq
| W^{(r+2)} |_1 + e (r + 2)! C^{r + 2} \mathscr{L}^{r/2 + 1} e^{- K \sqrt{\mathscr{L}}}.
$$
Finally, by assumption, $|\mathcal{D}_i| \leq D$ for every $i$. 
Therefore, the expression in \eqref{eq:integral} is
$$
\leq \frac{1}{\mathscr{L}^{k_0/2}} \prod_{i = 1}^{{k_0}} \frac{1}{p_i} \cdot
D^{k_0} \prod_{i = 1}^{l} \Big ( | W_i |_{D + 2, 1} + e (D + 2)! C^{D + 2} \mathscr{L}^{D/2+1} e^{- K \sqrt{\mathscr{L}}} \Big )
$$
and the claim follows. 
\end{proof}

\subsection{Obtaining cancellation in the trace with smoothing}

We must now show that there is indeed enough cancellation in those
terms in $\Tr (A|_X)^{2 k}$ involving many lone primes. Recall that
$X = X_0 \cap Y_{\ell}$. Our main task will be to combine our tools from
\S \ref{sec:sievrecur}, designed to deal with $Y_{\ell}$, with our work from
\ref{subs:cancap}, where we showed how to obtain cancellation
in a smoothed sum when our
variable $n$ ranges on the intersection of $X_0$ and an arithmetic progression.
Here, we will prove a statement that still has smoothing functions in it;
we will then show how to remove the smoothing functions.

For $W:[0,\infty)\to [0,\infty)$ a smoothing function,
define the linear operator $A|_{W,X}$ by
$$
A_{|W,X} = A^{+}|_{W,X} - A^{-}|_{W,X},
$$
where
$$
(A^{+}|_{W,X} g)(n) =  \sum_{\sigma \in \{-1,1\}} \sum_{\substack{p \in \mathbf{P} \\ p | n \\ n, n + \sigma p \in X}} g(n + \sigma p)  W^{1/2} \Big ( \frac{\omega_{\mathbf{P}}(n)}{\sqrt{\mathscr{L}}} \Big ) W^{1/2} \Big ( \frac{\omega_{\mathbf{P}}(n + \sigma p)}{\sqrt{\mathscr{L}}} \Big ),$$
$$
(A^{-}|_{W,X} g)(n) := \sum_{\sigma \in \{-1,1\}} \sum_{\substack{p \in \mathbf{P} \\ n, n + \sigma p \in X}} \frac{g(n + \sigma p)}{p}  W^{1/2} \Big ( \frac{\omega_{\mathbf{P}}(n)}{\sqrt{\mathscr{L}}} \Big )
W^{1/2} \Big ( \frac{\omega_{\mathbf{P}}(n + \sigma p)}{\sqrt{\mathscr{L}}} \Big )
$$

\begin{prop} \label{prop:maintra}
  Let $\mathbf{N} = (N,2N]\cap \mathbb{Z}$ and
    $\mathbf{k}=\{1,2,\dotsc,2 k\}$, where $k\geq 4$.
    Let $\mathbf{P}\subset [H_0,H]$ be a set of primes where $H\geq H_0\geq e$.
    Write $\mathscr{L}$ for
    $\sum_{p\in \mathbf{P}} 1/p$, and assume $\mathscr{L}\geq e$.
    Assume as well that
    $H^{2 k \max(C_0, 12 (K \mathscr{L}+1))}\leq N$
    and $H_0 \geq 4 k^2$, where $C_0$ is an absolute constant.
    Let 
    $1\leq \ell \leq k$ be such that
    $(80 \mathscr{L} k)^{\ell+1}\leq
    (H_0\cdot \mathscr{L}/9k \log H)^{1/2\log \mathscr{L}}$.
    
      Given $\vec{\sigma}\in \{-1,1\}^{2 k}$ and
  $\vec{p}\in \mathbf{P}^{2 k}$,
  let $\beta_i = \beta_i(\vec{\sigma},\vec{p})$ be defined by
  $\beta_i = \sum_{j=1}^i \sigma_j p_j$ for $0\leq i\leq 2 k$.
Let $\Pi_{\vec{p}}$
  be the partition of $\mathbf{k} = \{1,2,\dotsc,2k\}$ corresponding
  to the equivalence relation $\sim$ defined by $i\sim j\Leftrightarrow
  p_i=p_j$.

  Let
  $W:[0,\infty)\to [0, \infty)$ be in $C^\infty$, 
  with $|W|_\infty=1$,
  $W(x)=0$ for $x$ in an open neighborhood of $0$,
  $W(x)=1$ for $x\geq 1$ and
  $|W^\nu|_{4,1}\leq \kappa^{\nu}$ for some $\kappa\geq 1$ and all
  $\nu\geq 1$.
    Then
    \[\begin{aligned}
      \Tr (A|_{W,X_0\cap Y_{\ell}})^{2 k}
  &\ll (C \kappa)^{2k} N \sum_{\mathcal{L} \subset \mathbf{k}} \mathscr{L}^{-|\mathcal{L}| / 2} \sum_{\mathbf{l} \subset \mathbf{k}}
      \mathop{ \sum_{(\vec{p}, \vec{\sigma}) \in
          \mathscr{C}(\mathcal{L})\cap \mathscr{C}'(\mathcal{L},\mathbf{l})}}_{
                \sigma_1 p_1 + \ldots + \sigma_{2k} p_{2k} = 0}
    \prod_{i\not\in \mathbf{l}} \frac{1}{p_i} \prod_{[i]\in \Pi,
      [i]\not\subset \mathbf{k}\setminus \mathbf{l}} \frac{1}{p_{[i]}} 
    \\ & + 
\mathop{\sum_{\vec{p} \in \mathbf{P}^{2 k}} \sum_{\vec{\sigma} \in \{ \pm 1\}^{2 k}}}_{|\mathbf{S}(\vec{p},\vec{\sigma})|> \frac{2 k}{\log \mathscr{L}}}
   \sum_{\mathbf{l}\subset \mathbf{k}}\;
      \mathop{\sum_{n\in \mathbf{N}}}_{p_i|n+\beta_i \forall i\in \mathbf{l}}
      \prod_{i \in \mathbf{k}\setminus \mathbf{l}} \frac{1}{p_i} 
    + N,
  \end{aligned}\]
    where the implied constant and $C$ are absolute,
    \begin{itemize}
      \item
  $\mathscr{C}(\mathcal{L})$ is the set of all $(\vec{p},\vec{\sigma})$
  such that the singletons of $\Pi_{\vec{p}}$ are exactly $[i]$,
  $i\in \mathcal{L}$, \item $\mathscr{C}'(\mathcal{L},\mathbf{l})$ is the
  set of all $(\vec{p},\vec{\sigma})\in \mathbf{P}^{2 k}\times \{\pm 1\}^{2 k}$
  such that there is an $n\in \mathbf{N}$
 for which $n+\beta_0,n+\beta_1,\dotsc,n+\beta_{2 k}$ are all in $X_0\cap Y_{\ell}$ and $p_i|n+\beta_i$ for every
    $i\in \mathbf{l}\cap (\mathbf{k} \setminus \mathcal{L})$,
     \item $\mathbf{S}(\vec{p},\vec{\sigma})$ is the set of all
    $i\in \mathbf{k}$ such that (a) $[i]$ is a singleton, (b) there
    is a $j\not\in \{i,i-1\}$ such that either (i) $p_i|\beta_i-\beta_j$
    and $\beta_j\ne \beta_{i-1}, \beta_i$, or
    (ii) $\beta_i=\beta_j$ and $[j]$ is a singleton.
    \end{itemize}
\end{prop}
\begin{proof}
  Given $\mathcal{S} \subset \mathcal{L}$, let  $\mathscr{P}_0(\mathcal{S})$ denote the set of
  all $(\vec{p}, \vec{\sigma})$ such that
  $\mathbf{S}(\vec{p},\vec{\sigma}) = \mathcal{S}$. Let $X = X_0\cap Y_{\ell}$.
  The trace $\Tr (A|_{W,X})^{2 k}$ equals
  \[\mathop{\sum_{\vec{p} \in \mathbf{P}^{2 k}} \sum_{\vec{\sigma} \in \{ \pm 1\}^{2 k}}}_{\sigma_1 p_1 + \dotsb + \sigma_{2 k} p_{2 k} = 0}\;
  \sum_{\substack{n \in (N, 2N] \\ \forall i : n + \beta_i \in X_0\cap Y_{\ell}}}  \prod_{i = 1}^{2k} f_{p_i}(n + \beta_i) W \Big ( \frac{\omega_{\mathbf{P}}(n + \beta_i)}{\sqrt{\mathscr{L}}} \Big ),\]
  which in turn, can be written as 
    \begin{align}
    &\mathop{\sum_{\mathcal{S}\subset \mathcal{L}\subset \mathbf{k}}}_{
      |\mathcal{S}|\leq s_0}
    \mathop{\mathop{\sum_{\vec{p} \in \mathbf{P}^{2 k}} \sum_{\vec{\sigma} \in \{ \pm 1\}^{2 k}}}_{\sigma_1 p_1 + \dotsb + \sigma_{2 k} p_{2 k} = 0}}_{(\vec{p},\vec{\sigma})\in
      \mathscr{C}(\mathcal{L})\cap \mathscr{P}_0(\mathcal{S})}
    \sum_{\substack{n \in (N, 2N] \\ \forall i : n + \beta_i \in X}}  \prod_{i = 1}^{2k} f_{p_i}(n + \beta_i) W \Big ( \frac{\omega_{\mathbf{P}}(n + \beta_i)}{\sqrt{\mathscr{L}}} \Big ) \label{eq:3trumpets} \\ &+ 
    O^*\Big(\mathop{\sum_{\mathcal{S}\subset \mathbf{k}}}_{
      |\mathcal{S}|> s_0}
    \mathop{\mathop{\sum_{\vec{p} \in \mathbf{P}^{2 k}} \sum_{\vec{\sigma} \in \{ \pm 1\}^{2 k}}}_{\sigma_1 p_1 + \dotsb + \sigma_{2 k} p_{2 k} = 0}}_{(\vec{p},\vec{\sigma})\in \mathscr{P}_0(\mathcal{S})}
    \sum_{\substack{n \in (N, 2N] \\ \forall i : n + \beta_i \in X}}  \prod_{i = 1}^{2k}
      \left|f_{p_i}(n + \beta_i)\right| \Big)\label{eq:lafolia}
  \end{align}
  for any choice of $s_0$. Since $|f_{p_i}(n+\beta_i)| = 1/p_i$ for
  $p_i|n+\beta_i$ and $|f_{p_i}(n+\beta_i)|$ is always at most $1$,
  the expression in \eqref{eq:lafolia} is at most
\[\mathop{\sum_{\vec{p} \in \mathbf{P}^{2 k}} \sum_{\vec{\sigma} \in \{ \pm 1\}^{2 k}}}_{|\mathbf{S}(\vec{p},\vec{\sigma})|> s_0}
   \sum_{\mathbf{l}\subset \mathbf{k}}\;
      \mathop{\sum_{n\in \mathbf{N}}}_{p_i|n+\beta_i \forall i\in \mathbf{l}}
      \prod_{i \in \mathbf{k}\setminus \mathbf{l}} \frac{1}{p_i}.\]  
  Let us focus on the inner sum in \eqref{eq:3trumpets}, namely,
          \begin{equation} \label{eq:ineer}
  \sum_{\substack{n \in (N, 2N] \\ \forall i : n + \beta_i \in X}} \prod_{i = 1}^{2k} f_{p_i}(n + \beta_i) W \Big ( \frac{\omega_{\mathbf{P}}(n + \beta_i)}{\sqrt{\mathscr{L}}} \Big )
        \end{equation}
          where we assume that
$(\vec{p},\vec{\sigma})\in
          \mathscr{C}(\mathcal{L})\cap \mathscr{P}_0(\mathcal{S})$.
          
          We begin by expanding each $f_{p_i}$ with $i\notin \mathcal{L}$.
          Then \eqref{eq:ineer} becomes
  \begin{equation}\label{eq:beltik}
  \sum_{\substack{\mathcal{D} \subset \mathbf{k} \backslash \mathcal{L}}} \prod_{\substack{i \not \in \mathcal{D} \\ i \not \in \mathcal{L}}} \frac{-1}{p_i} \prod_{i \in \mathcal{D}} \Big (1 - \frac{1}{p_i} \Big )
  \sum_{\substack{n\in (N, 2 N]\\\forall i \in \mathcal{D} : p_i | n + \beta_i \\ \forall i \not \in \mathcal{D}\cup \mathcal{L} : p_i \nmid n + \beta_i \\ \forall i : n + \beta_i \in X}}
  \prod_{i \in \mathcal{L}} f_{p_i}(n+\beta_i) \prod_{i = 1}^{2k} W \Big ( \frac{\omega_{\mathbf{P}}(n + \beta_i)}{\sqrt{\mathscr{L}}} \Big ).
  \end{equation}
  Using Proposition \ref{prop:sievorino}, we can rewrite \eqref{eq:beltik} as
  \begin{equation}\label{eq:jorkb}\begin{aligned}
    \sum_{\substack{\mathcal{D} \subset \mathbf{k} \backslash \mathcal{L}}} & \prod_{\substack{i \not \in \mathcal{D} \\ i \not \in \mathcal{L}}} \frac{-1}{p_i} \prod_{i \in \mathcal{D}} \Big (1 - \frac{1}{p_i} \Big ) \sum_{\substack{R \in W_{\ell,\mathbf{P}}(\boldsymbol{\beta})^{\cap} \\ \omega(\mathfrak{q}(R)) \leq m}} c_R
    \sum_{\substack{n\in R\cap (N,2 N]\\\forall i \in \mathcal{D} : p_i | n + \beta_i \\ \forall i \not \in \mathcal{D}\cup \mathcal{L} : p_i \nmid n + \beta_i}}
      \prod_{i \in \mathcal{L}} f_{p_i}(n+\beta_i)
      \prod_{i = 1}^{2k} W \Big ( \frac{\omega_{\mathbf{P}}(n + \beta_i)}{\sqrt{\mathscr{L}}} \Big )  \\ & + O \Big ( 2^m 3^{\ell}
\sum_{\mathbf{l}\subset \mathbf{k}}
      \sum_{\substack{R \in W_{\ell, \mathbf{P}}(\boldsymbol{\beta})^{\cap} \\ m < \omega(\mathfrak{q}(R)) \leq m + \ell}}
      \sum_{\substack{n\in R\cap (N,2 N]\\\forall i \in \mathbf{l} : p_i | n + \beta_i}}
\prod_{i\in \mathbf{k}\setminus \mathbf{l}} \frac{1}{p_i}
            \Big ), 
  \end{aligned}\end{equation}
  for $m$ arbitrary, 
  where, in the error term, we have expanded all $f_{p_i}$, and defined
  $\mathbf{l}$ as the set of those $i\in \mathbf{k}$ for which $p_i|n+\beta_i$.
  (In particular, $\mathcal{D}\subset \mathbf{l}$.)
  We apply Proposition \ref{prop:errorterm}, and see that
the total contribution of this error term to \eqref{eq:3trumpets} is bounded by
\[\begin{aligned}&\leq 2^m 3^{\ell}\cdot N \sum_{m'=m+1}^{m+\ell} (2 e)^{m'} (4 k + 2 m')^{2 k} \mathscr{L}^{2 k + m'}
\left(\frac{(8 k + 4) \log H}{\mathscr{L} H_0}\right)^{\frac{m'}{\ell+1}}\\
&\leq 6^k N \sum_{j=1}^{\ell} (2 e)^{k+j} (8 k)^{2 k} \mathscr{L}^{2k + m + j}
\left( \frac{9 k \log H}{
  \mathscr{L} H_0} \right)^{\frac{m+j}{\ell + 1}}
\\ &\ll (12 e)^k (8 k)^{2 k} \mathscr{L}^{2 k} e^k
\left( \frac{9 k \log H}{
  \mathscr{L} H_0} \right)^{\frac{m}{\ell + 1}}
\ll (80 \mathscr{L} k)^{2 k} \left( \frac{9 k \log H}{
  \mathscr{L} H_0} \right)^{\frac{m}{\ell + 1}} \leq 1
,\end{aligned}\]
provided that $\ell\leq k$ and $m = \lceil k/ \log \mathscr{L}\rceil\leq k$.
 (Any value of $m$ between
$k/ \log \mathscr{L}$ and $k$ would have done about as nicely here; setting
$m$ not much higher than $k/\log \mathscr{L}$ 
will prove useful shortly.)
Here we are using the assumption
$(80 \mathscr{L} k)^{2 (\ell+1) \log \mathscr{L}}\leq \mathscr{L} H_0/9k \log H$
in three
ways:
\begin{enumerate}
  \item to ensure that $2 e \mathscr{L} \cdot (9 k \log H/\mathscr{L} H_0)^{1/(\ell+1)}\leq 1/2$,
so that the geometric series on $j$ converges (to $\leq 2$);
\item to ensure that $H_0\geq 9 k (\log H) (2 k)^{\ell + 1}/\mathscr{L}$,
  so as to be allowed to apply Prop.~\ref{prop:errorterm};
\item to give the final bound $\leq 1$.
  \end{enumerate}
  
Let us now work on the main term, and more precisely on the inner sum
  \begin{equation} \label{eq:inner}
\sum_{\substack{n\in R\cap (N, 2 N]\\\forall i \in \mathcal{D} : p_i | n + \beta_i \\ \forall i \not \in \mathcal{D}\cup \mathcal{L} : p_i \nmid n + \beta_i}} \prod_{i \in \mathcal{L}} f_{p_i}(n+\beta_i)  \prod_{i = 1}^{2k} W \Big ( \frac{\omega_{\mathbf{P}}(n + \beta_i)}{\sqrt{\mathscr{L}}} \Big ).
  \end{equation}
  We let $\mathcal{S}_{R}$ be the set of those indices
$i\in \mathcal{L}\setminus \mathcal{S}$ for which $p_i \nmid \mathfrak{q}(R)$. 
  By assumption, $|\mathcal{S}_{R}| \geq |\mathcal{L}|-|\mathcal{S}| -
  \omega(\mathfrak{q}(R)) \geq |\mathcal{L}| - s_0 - m$.
  
  We now expand all $f_{p_i}$ with $i\not\in \mathcal{S}_R$,
  as we do not intend to obtain cancellation on those terms. We thus rewrite
  \eqref{eq:inner} as
     \begin{equation} \begin{aligned}\label{eq:bobfisch}
       \sum_{\substack{\mathcal{D}_0 \subset \mathcal{L} \backslash \mathcal{S}_{R}}} &
\prod_{\substack{i \not \in \mathcal{D}_0 \cup \mathcal{S}_{R} \\ i \in \mathcal{L}}} \frac{-1}{p_i}
\prod_{\substack{i \in \mathcal{D}_0}} \Big ( 1 - \frac{1}{p_i} \Big ) \\ & \sum_{\substack{n \in R\cap (N, 2N] \\ \forall i \in \mathcal{D} \cup \mathcal{D}_0: p_i | n + \beta_i \\ \forall i \not \in \mathcal{D} \cup \mathcal{D}_0 \cup \mathcal{S}_{R} : p_i \nmid n + \beta_i}} \prod_{i \in \mathcal{S}_{R}} f_{p_i}(n + \beta_i) \prod_{\beta \in \{\beta_1, \ldots, \beta_{2k}\}
  } W^{\nu_{\beta}}\Big ( \frac{\omega_{\mathbf{P}}(n + \beta)}{\sqrt{\mathscr{L}}} \Big )
    \end{aligned}\end{equation}
    where $\nu_{\beta}$ is the number of indices $i$ such that $\beta = \beta_i$.
    
    We now apply Lemma \ref{lem:recast} 
    to obtain cancellation in the inner sum,
    defining $\alpha_i$ to range over all distinct elements of
    $\{\beta_1, \ldots, \beta_{2k}\}$,
    with $\alpha_1,\dotsc,\alpha_{k_0}$, $k_0=|\mathcal{S}_R|$ being the
    elements $\beta_i$ for $i\in \mathcal{S}_R$ (all of which are indeed
    distinct); we let $D=2$, $\mathcal{P}_i = \{p_j: \beta_j = \alpha_i,
    j\notin \mathcal{D}\cup \mathcal{D}_0\cup \mathcal{S}_R\}$
    and $W_i = W^{\nu_i}$, where $\nu_i$ is the number of times that
    $\alpha_i$ appears in $\beta_1,\dotsc,\beta_{2 k}$.
    Since, by assumption, $|W_i|_{4,1}\leq \kappa^{\nu_i}$ and
    $|W_i|_\infty=1$, we obtain that
      \begin{equation}\begin{aligned}\label{eq:thelulz}
        \sum_{\substack{n \in (N, 2N] \cap R \\ \forall i \in \mathcal{D} \cup \mathcal{D}_0 : p_i | n + \beta_i \\ \forall i \not \in \mathcal{D} \cup \mathcal{D}_0 \cup \mathcal{S}_{R} : p_i \nmid n + \beta_i}} & \prod_{i \in \mathcal{S}_R} f_{p_i}(n + \beta_i) \prod_{\beta \in \{\beta_1, \ldots, \beta_{2k}\}} W^{\nu_{\beta}} \Big ( \frac{\omega_{\mathbf{P}}(n + \beta)}{\sqrt{\mathscr{L}}} \Big ) \\ & \ll
          (C \kappa)^{2k} \cdot \frac{N}{\left[\mathfrak{q}(R),
              \prod_{p\in \{p_i: i \in \mathcal{D} \cup \mathcal{D}_0\}}
              p\right]}
       \prod_{p\in \mathcal{S}_R} \frac{1}{p}
                    \cdot \mathscr{L}^{-|\mathcal{S}_R|/2}  + N^{3/4},
      \end{aligned}\end{equation}
      where the implied constant and $C$ are absolute.
      Since $\mathbf{q}(R)\leq H^m$, $\prod_{p\in \mathbf{k}} p \leq H^{2 k}$
      and $\mathscr{L}<H$,
      the main term in \eqref{eq:thelulz} is at least
      $N/(H^{2 k + m} \mathscr{L}^{|\mathcal{S}_R|/2}) \geq N/H^{4 k}$,
      and so the error term in \eqref{eq:thelulz} is dominated by the main
      term.
      We now set $s_0 = 2 k/\log \mathscr{L}$, and, since
      $m = \lceil k/\log \mathscr{L}\rceil$, 
      simplify \eqref{eq:thelulz} further by
      \begin{equation}\label{eq:marino}\mathscr{L}^{-|\mathcal{S}_R|/2} \leq
      \mathscr{L}^{-|\mathcal{L}|/2} \mathscr{L}^{s_0/2+m/2}
      \ll e^{2 k} \mathscr{L}^{-|\mathcal{L}|/2}.\end{equation}
      
      If the condition $n\in R$ is not logically consistent with
      the conditions $p_i|n+\beta_i$ for $i\in \mathcal{D}\cup \mathcal{D}_0$,
      then the bound in \eqref{eq:thelulz} can evidently be replaced by
      $0$. Since the primes in $\mathcal{S}_R$ do not divide $\mathbf{q}(R)$,
      the set of conditions $p_i|n+\beta_i$ for $i\in \mathcal{S}_R$
      (which we have not imposed yet) would be consistent with
      $n\in R$.

      Write $A_{\vec{p},\vec{\sigma},\mathbf{l}}$ for the arithmetic progression
      $\{n: p_i|n+\beta_i\; \forall i\in \mathbf{l}\}$, where
      $\mathbf{l}\subset \mathbf{k}$ and $\beta_i = \sigma_1 p_1 + \dotsc +
      \sigma_i p_i$, as usual. What we obtain from \eqref{eq:thelulz} and
      \eqref{eq:marino} is that
      the expression within the sum $\sum_{\mathcal{D}_0\subset \mathcal{L}\setminus
        \mathcal{S}_R}$ in \eqref{eq:bobfisch} is
      \begin{equation} \label{eq:kaput}\ll (e C K)^{2 k} \mathscr{L}^{-|\mathcal{L}|/2}
\prod_{\substack{i \not \in \mathcal{D}_0 \cup \mathcal{S}_{R} \\ i \in \mathcal{L}}} \frac{1}{p_i}
\cdot \frac{N}{\mathbf{q}(R\cap A_{\vec{p},\vec{\sigma},\mathbf{l}})}\end{equation}
for $\mathbf{l} = \mathcal{D}\cup \mathcal{D}_0\cup \mathcal{S}_R$, provided
that $R\cap A_{\vec{p},\vec{\sigma},\mathbf{l}}\ne \emptyset$; if
$R\cap A_{\vec{p},\vec{\sigma},\mathbf{l}}=\emptyset$, then the expression
in \eqref{eq:bobfisch} is $0$ (and so is \eqref{eq:kaput} since $\mathfrak{q}(\emptyset) = \infty$
by convention). 

Therefore, the contribution of the main term from
\eqref{eq:jorkb} to \eqref{eq:3trumpets} can be bounded by
\begin{align*}O\left((e C K)^{2 k} \mathscr{L}^{-\frac{|\mathcal{L}|}{2}}\right)
  \sum_{\mathcal{L}, \mathbf{l}\subset \mathbf{k}}
   \mathop{\sum_{(\vec{p},\vec{\sigma})\in \mathscr{C}(\mathcal{L})\cap \mathscr{C}'(\mathcal{L},\mathbf{l})}}_{ \sigma_1 p_1 + \ldots + \sigma_{2k} p_{2k} = 0}
    \prod_{i\not\in \mathbf{l}} \frac{1}{p_i} 
    \sum_{\substack{R \in W_{\ell,\mathbf{P}}(\boldsymbol{\beta})^{\cap} \\
        R\cap A_{\vec{p},\vec{\sigma},\mathbf{l}}\ne \emptyset, \omega(\mathfrak{q}(R)) \leq m}} 2^{\omega(\mathbf{q}(R))} \frac{N}{\mathbf{q}(R\cap A_{\vec{p},\vec{\sigma},\mathbf{l}})},
\end{align*}
since $|c_R|\leq 2^{\omega(\mathbf{q}(R))}$.
(If $(\boldsymbol{p}, \boldsymbol{\sigma}) \in \mathscr{C}'(\mathcal{L},\mathbf{l})$
does not hold, then
the inner
sum in \eqref{eq:beltik} is empty.)
We now apply
Proposition \ref{prop:boundmas}, and obtain a bound of
\[  O\left((2 e^2 C K)^{2 k}\right)
\sum_{\mathcal{L}, \mathbf{l}\subset \mathbf{k}}
 \mathscr{L}^{-|\mathcal{L}|/2}
 \mathop{\sum_{(\vec{p},\vec{\sigma})\in \mathscr{C}(\mathcal{L})\cap \mathscr{C}'(\mathcal{L},\mathbf{l})}}_{ \sigma_1 p_1 + \ldots + \sigma_{2k} p_{2k} = 0}
    \prod_{i\not\in \mathbf{l}} \frac{1}{p_i} 
    \frac{1}{\mathbf{q}(A_{\vec{p},\vec{\sigma},\mathbf{l}})} N,
        \]
        by $\omega(\mathfrak{q}(A_{\vec{p},\vec{\sigma},\mathbf{l}}))\leq 2 k$.
        The modulus $\mathbf{q}(A_{\vec{p},\vec{\sigma},\mathbf{l}})$ equals
          $\prod_{[i]\in \Pi,
            [i]\not\subset \mathbf{k}\setminus \mathbf{l}} p_{[i]}$, and
          so we are done.
        \end{proof}

\subsection{Removing a smoothing function}

Our task is now simply to remove a smoothing $W$.
%
Thanks to Prop.~\ref{prop:maintra}, we have a bound on
$\Tr \left(A|_{W,X}\right)^{2 k}$.
Thus, Prop.~\ref{prop:rabmento} applied to $A|_{W,X}$
will give us that there is a small
set $\mathscr{E}\supset \mathbf{N}\setminus X$ such that every eigenvalue of
$(A|_{W,X})|_{\mathbf{N}\setminus \mathscr{E}}
= A|_{W,\mathbf{N}\setminus \mathscr{E}}$ has small absolute value.
It remains to show that
$A|_{\mathbf{N}\setminus \mathscr{E}} - A|_{W,\mathbf{N}\setminus \mathscr{E}}$ has small
norm, i.e., that its eigenvalues all have small absolute value. We will
actually prove the analogous statement for a set
$\mathscr{E} \cup \mathscr{E}'$ slightly larger than
$\mathscr{E}$.

In some sense, our task is trivial: if we let $\mathscr{E}'$ be the
set of all integers $n\in \mathbf{N}$ with $\leq \sqrt{\mathscr{L}}$ prime
factors in $\mathbf{P}$, then $A|_{\mathbf{N}\setminus (\mathscr{E}\cup \mathscr{E}')}$
actually equals $A|_{W,\mathbf{N}\setminus (\mathscr{E}\cup \mathscr{E}')}$.
Therefore, $\|A_{|\mathbf{N} \setminus (\mathscr{E} \cup \mathscr{E}')}\|
= \| A_{|W, \mathbf{N} \setminus (\mathscr{E} \cup \mathscr{E}')} \|
\leq \| A_{|W, \mathbf{N} \setminus \mathscr{E}} \|$,
and so the largest eigenvalue of $A_{|W, \mathbf{N} \setminus \mathscr{E}}$
is controlled. However, $|\mathscr{E} \cup \mathscr{E}'|$ would then be
roughly in the order of $N e^{- (2 \log 2 - 1) \mathscr{L}}$, by
\cite[(1.11)]{zbMATH03515559}, and we would prefer to do
better. 

We may write $A = A^+ - A^-$,
where $A^+ = \Ad_\Gamma$, $A^- = \Ad_{\Gamma'}$ (see \eqref{eq:defA}).
Since $W(\omega_{\mathbf{P}}(n)/\sqrt{\mathscr{L}})$ can be $\ne 1$
only when $\omega_{\mathbf{P}}(n) < \sqrt{\mathscr{L}}$, it is clear
that the operator $A^+|_X - A^+|_{W,X}$ has norm $<\sqrt{\mathscr{L}}$.
It remains to control $A^{-}_{|X} - A^{-}_{|W, X}$. We do so now.





\begin{lemma}\label{lem:rupec}
  Let $\mathbf{N} = \{N+1,N+2,\dotsc,2 N\}$. Let $\mathbf{P}
  \subset [H_0,H]$ be a set of primes,
  and define $\mathscr{L} = \sum_{p\in \mathbf{P}} 1/p$. Assume
  $H_0\geq \max(64 \mathscr{L}^4 k, k^2 \log H)$ and $\mathscr{L}\geq 1$.
   Let $W:[0,\infty)\to [0,\infty)$ satisfy $|W|_\infty=1$ and
     $W(x)=1$ for $x\geq 1/2$. Let $A^-$ and $A^-|_{W,X}$ be as above.

     There is an absolute constant $C$ such that, for any
     $2\leq k\leq 
     \min\left(\frac{\log N/\log H}{\max(C, 24 \sqrt{\mathscr{L}})},
     \sqrt{H_0}\right)$, 
\[     \Tr((A^{-} - A^-_{|W,\mathbf{N}})^{2k}) \ll N,\]
where the implied constant is absolute.
  \end{lemma}
\begin{proof}
  Expanding everything, we find that
  \begin{align*}
    \text{Tr}((A^{-} - A_{|W,\mathbf{N}}^{-})^{2k}) & \leq \sum_{\substack{
 \boldsymbol{p} = (p_1, \ldots, p_{2k}) \in \mathbf{P}^{2k}\\
        \boldsymbol{\sigma} = (\sigma_1, \ldots, \sigma_{2k}) \in \{-1, 1\}^{2k}}}
 \frac{1}{p_1} \dotsb \frac{1}{p_{2k}} \sum_{\substack{n \in \mathbf{N} \\ \forall i\; \omega_{\mathbf{P}}(n + \beta_i) \leq \sqrt{\mathscr{L}}/2}} 1
    \end{align*}
  where $\beta_i = \sigma_1 p_1 + \ldots + \sigma_i p_i$ for all
  $0\leq i\leq 2 k$.

  We separate now the pairs of tuples $(\boldsymbol{\sigma}, \boldsymbol{p})$
  into two sets $\mathcal{S}_1$ and $\mathcal{S}_2$. The set $\mathcal{S}_1$
  consists of such pairs of tuples for which there are at most
  $2 k -\kappa$ distinct values among $\beta_1,\dotsc,\beta_{2 k}$.
  It is clear that, for any $i$ such that there is a $j\leq i$
  with $\beta_i=\beta_j$, the value of $p_i$ is determined by
  $p_1,\dotsc,p_{i-1}$ and the choice of $j$. Hence, the contribution
  of $\mathcal{S}_1$ is bounded by
  \[N\cdot 2^{2 k} \mathscr{L}^{2k-\kappa} \left(\frac{2 k}{H_0}\right)^\kappa.\]
  

Let us bound the contribution of $\mathcal{S}_2$.
Clearly,
\begin{equation}\label{eq:getupl}
\sum_{\substack{n \in \mathbf{N} \\ \omega_{\mathbf{P}}(n + \beta_i) \leq \sqrt{\mathscr{L}}/2}} 1 \leq \exp (  k\cdot \sqrt{\mathscr{L}} ) \sum_{n \in \mathbf{N}} \prod_{1 \leq i \leq 2k} \exp \Big ( - \omega_{\mathbf{P}} (n + \beta_i) \Big ). 
\end{equation}
We now pass to the Kubilius model. Let $C = 2 C_0$, where $C_0$ is as in
Lemma \ref{lem:kubiliusprax}. By Lemma \ref{lem:kubiliusprax},
we can bound the right side of \eqref{eq:getupl} by
$$
2 N e^{k \sqrt{\mathscr{L}}} \mathbb{E} \Big [ \prod_{1 \leq i \leq 2k} \exp \Big ( - \sum_{p \in \mathbf{P}} Z_{p}^{(\beta_i)} \Big ) \Big ] + O(N^{3/4}). 
$$
Using independence, we see that the above expectation is at most
\begin{align*}\prod_{p \in \mathbf{P}} &\Big ( 1 - |\{\beta_i \mo p:\; 1\leq i\leq 2 k\}|\cdot
\frac{(1-
  e^{-1})}{p} \Big )\leq
\prod_{p \in \mathbf{P}} \Big ( 1 - 
\frac{(1-
  e^{-1})}{p} \Big )^{|\{\beta_i \mo p:\; 1\leq i\leq 2 k\}|}.
\end{align*}
Restricting our attention to the $\geq 2k - \kappa$ indices $i$ for which $\beta_i$ is distinct, we see that
$$
|\{ \beta_i \mo p : \; 1 \leq i \leq 2k \}| \geq 2k - \kappa - \sum_{\substack{1 \leq i \neq j \leq k \\ \beta_i \neq \beta_j \\ p | \beta_i - \beta_j}} 1,
$$
so that the Euler product above is bounded above by
\begin{align*}
&\leq \prod_{p \in \mathbf{P}} \Big(1 - \frac{(1-e^{-1})}{p} \Big)^{2 k - \kappa}
\cdot \prod_{p\in \mathbf{P}}
\prod_{\substack{1 \leq i, j \leq 2k \\ \beta_i \neq \beta_j \\ p | \beta_i - \beta_j}}
\Big(1 - \frac{(1-e^{-1})}{p} \Big)^{-1}
.\end{align*}
For any given $1 \leq i, j \leq 2k$ with $\beta_i\ne \beta_j$,
$$
\prod_{\substack{p \in \mathbf{P} \\ p | \beta_i - \beta_j}} \Big ( 1 - \frac{1 - e^{-1}}{p} \Big) = 1 + O \left( \frac{\log H}{H_0} \right) 
$$
since $\beta_i - \beta_j$ has at most $\log H$ prime divisors in $\mathbf{P}$, by
$|\beta_j-\beta_i|< 2 k H \leq H_0 H$. Hence
\[\mathop{\prod_{1\leq i,j\leq 2 k}}_{\beta_i\ne \beta_j} \prod_{\substack{p\in \mathbf{P} \\ p | \beta_i - \beta_j}}
\Big(1 - \frac{(1-e^{-1})}{p} \Big)^{-1} \ll 1\]
since $H_0\geq k^2 \log H$. Therefore,
\[\begin{aligned}
\mathbb{E} &\Big [ \prod_{1 \leq i \leq 2k} \exp \Big ( - \sum_{p \in \mathbf{P}} Z_{p}^{(\beta_i)} \Big ) \Big ] \ll
\prod_{p \in \mathbf{P}} \Big(1 - \frac{(1-e^{-1})}{p} \Big)^{2 k - \kappa}\\
&\leq \exp\Big(- (2k-\kappa) \sum_{p\in \mathbf{P}} \frac{1-e^{-1}}{p}\Big)
= e^{-(1-e^{-1}) (2 k - \kappa) \mathscr{L}}.\end{aligned}\]
We let $\kappa = 2k/5$, and so $(1-e^{-1}) (2 k - \kappa) > k$.

We obtain the overall final bound
\[O\left( e^{- (\mathscr{L} - \sqrt{\mathscr{L}}) k} N +
2^{2 k} \mathscr{L}^{\frac{8}{5} k} \left(\frac{2 k}{H_0}\right)^{2k/5}
N + 
N^{3/4}\right).\]
By $\mathscr{L}\geq 1$, the first term is $\leq N$.
By $H_0\geq 2^{5} \mathscr{L}^4\cdot 2 k$,
the second term is also $\leq N$.
\end{proof}

\begin{proof}[Proof of Prop.~\ref{prop:maincanc}]
  Let $W:[0,\infty)\to [0, \infty)$ be a smooth function with
    $W(x)=0$ in an open neighborhood of $0$, $0\leq W(x)\leq 1$ for
    $0<x<1/2$ and $W(x)=1$ for $x\geq 1/2$. It is immediate that
    $|W^\nu|_{4,1}\leq \kappa^\nu$ for some $\kappa\geq 1$ and all $\nu\geq 1$;
    indeed, the stronger bound $|W^\nu|_{4,1}\ll \nu^4$ holds.
    It is clear that our assumptions are at least as strong as those
    of Prop.~\ref{prop:maintra}. (For instance, our bound on $\ell$
    implies that $H_0 \geq e$.)
    Apply Prop.~\ref{prop:maintra}. We obtain that
    \[\Tr A|_{W,X}^{2 k} \ll (C \kappa)^{2 k} \mathscr{S}_1 N + \mathscr{S}_2 N
    +  N\]
    for $X = X_0\cap Y_\ell$.
    By Prop.~\ref{prop:rabmento} with $A|_{W,X}$ instead of $A$,
    $\epsilon = 1/H_0$ and $\alpha = 2 C_1 (C \kappa \mathscr{S}_1^{1/2 k}
    + \mathscr{S}_2^{1/2k} + 1)$ (where we recall that
    $C_1$ is a constant such that $(8 \mathscr{L} H_0 H)^{1/2k} \leq C_1$),
    we obtain that
there is a set $\mathscr{E}\supset \mathbf{N}\setminus X$
with $|\mathscr{E}|\leq |\mathbf{N}\setminus X| + N/H_0$
such that every eigenvalue of
$A|_{W,\mathbf{N}\setminus \mathscr{E}}$ has absolute value
\[\ll 2 C_1 (C \kappa \mathscr{S}_1^{\frac{1}{2 k}}
+ \mathscr{S}_2^{\frac{1}{2 k}} + 1) \ll C_1
\max\Big(|\mathscr{S}_1|^{\frac{1}{2 k}},
    |\mathscr{S}_2|^{\frac{1}{2 k}},1\Big).\]
By Lemmas \ref{lem:X0issmall} and \ref{lem:Ylissmall},
\[\begin{aligned}
|\mathbf{N}\setminus X|&\ll N e^{-(K \log K - K + 1)\mathscr{L}} +
3^\ell \mathscr{L}^{\ell} 
    \left(\frac{\log H}{\log H_0} + 1 + \frac{1}{\mathscr{L}}\right) \cdot
    \frac{N}{H_0}
    + 3^{\ell} H^{\ell+1}\\ &\leq
    N e^{-(K \log K - K + 1)\mathscr{L}} + H_0^{1/4} \cdot H_0^{1/4}
    \cdot
    \frac{N}{H_0} + N^{1/2}\\ & \ll 
        N e^{-(K \log K - K + 1)\mathscr{L}} + \frac{N}{\sqrt{H_0}},
\end{aligned}\]
where we use assumptions $\ell\leq (\log H_0)/4 \log 3 \mathscr{L}$,
$\log H + 2 \leq H_0^{1/4}$ and $(3 H)^{k+1}\leq N^{1/2}$.

Write $A = A^+ - A^-$,
where $A^+ = \Ad_\Gamma$, $A^- = \Ad_{\Gamma'}$, as in \eqref{eq:defA}.
Since $W(\omega_{\mathbf{P}}(n)/\sqrt{\mathscr{L}})$ can be $\ne 1$
only when $\omega_{\mathbf{P}}(n) < \sqrt{\mathscr{L}}$,
we know that, when we write $A^+|_X - A^+|_{W,X} = (a_{i,j})$
and consider any column index $j$, there are $<\sqrt{\mathscr{L}}$ row
indices $i$ such that $a_{i,j}\ne 0$, and of course $|a_{i,j}|\leq 1$ for all
$i,j$. Hence every eigenvalue of $A^+|_X - A^+|_{W,X}$ has absolute value
$< \sqrt{\mathscr{L}}$.
It remains to show that, for some small $\mathscr{E}'$, the eigenvalues of
$A^+|_{\mathbf{N}\setminus \mathscr{E}'} - A^+|_{W,\mathbf{N}\setminus \mathscr{E}'}$
are also small.

By Lemma \ref{lem:rupec}, $\Tr((A^{-} - A^-_{|W,\mathbf{N}})^{2k}) \ll N$.
Hence, by Prop.~\ref{prop:rabmento} with $A^-$ instead of $A$,
there is a set $\mathscr{E}'\supset \mathbf{N}$
with $|\mathscr{E}'|\leq 4 \mathscr{L} N/H_0$
such that every eigenvalue of
$A^- - A^{-}|_{W,\mathbf{N}\setminus \mathscr{E}'}$ has absolute value
$\leq 2 (H_0 H_1)^{1/2k}$. Clearly $\mathscr{L}/H_0 \ll (\log \log H)/H_0\ll
1/\sqrt{H_0}$, and $(H_0 H_1)^{1/2 k} \leq C_1$.

Restriction does not increase the norm of an operator. We see, then, that
for $\mathscr{X} = \mathbf{N}\setminus (\mathscr{E}\cup \mathscr{E}')$,
every eigenvalue of $A^+|_{\mathscr{X}} - A^+|_{W,\mathscr{X}}$ has absolute value
$<\sqrt{\mathscr{L}}$ and every eigenvalue of $A^{-}|_{\mathscr{X}} -
A^-|_{W,\mathscr{X}}$
has absolute value $\leq C_1$.

We conclude that every eigenvalue of 
$A|_{\mathscr{X}}$ has absolute value \[\ll \max\Big(
C_1 \max\Big(|\mathscr{S}_1|^{\frac{1}{2 k}},
|\mathscr{S}_2|^{\frac{1}{2 k}},1\Big)\Big),\] and that
$|\mathbf{N}\setminus \mathscr{X}|\leq |\mathscr{E}\cup \mathscr{E}'|\ll N e^{-(K \log K - K + 1) \mathscr{L}}
+N/\sqrt{H_0}$.
\end{proof}

{\bf Remark.} An alternative procedure is possible. We can prove that,
for $A_1$, $A_2$ real, symmetric $n$-by-$n$ matrices,
\begin{equation}\label{eq:lecten}
  \Tr((A_1+A_2)^{2 k}) \leq 2^{2k} \text{Tr}(A_1^{2k}) + 2^{2k} \text{Tr}(A_2^{2k})
  .\end{equation}
(One possible proof goes through von Neumann's and H\"older's inequalities.)
We can then apply Prop.~\ref{prop:maintra} and Lemma \ref{lem:rupec}
and use inequality \eqref{eq:lecten} to establish a bound on
$\Tr (A|_{W,X})^{2 k}$. We then use Prop.~\ref{prop:rabmento} once, as a last step.

\section{Prelude to main argument}\label{sec:praeludium}

From now on, our task is to count walks. More precisely: we want to bound
the number of closed walks satisfying certain properties. We can summarize
our work in 
\S \ref{sec:returns} -- and its conclusion, Prop.~\ref{prop:maincanc},
in particular -- as stating that
we saw that we can assume that our walks are
either of the kind we will study in \S \ref{sec:main} -- namely, walks
where few primes appear only once as edge lengths, yet
reappearances of primes are generally well-spaced -- or
of one of a few special kinds. We will now show that walks of those
special kinds contribute little to the final total. In particular,
we will prove that the sum $\mathscr{S}_2$ in Prop.~\ref{prop:maincanc}
is small. We will also show that one of the conditions defining the sums
$\mathscr{S}_1$ can be strengthened and abstracted to some extent.

The arguments
that we will use will prefigure in some ways those in \S \ref{sec:main}.
The prime edge lengths in the walks we will consider will have to satisfy
certain linear conditions; using some very simple geometry of numbers,
we will show that these conditions make the number of possible walks
small. The main difference with \S \ref{sec:main} is that, in the problems
considered here, the choice of linear conditions generally suggests itself,
and the primes that appear as variables to be constrained and those
that determine the constraints are in two sets that are essentially
clear a priori (though the role of these two sets can be reversed).

Let us lay out our setup.
  Let $X\subset \mathbf{N} = \{N+1,N+2,\dotsc, 2N\}$
  and let $\mathbf{P}\subset I=[H_0,H]$ be a set of prime numbers.
  Write $\mathscr{L} = \sum_{p\in \mathbf{P}} 1/p$ and
  $\mathbf{k} = \{1,2,\dotsc, 2 k\}$.

  We will be studying sums over pairs $(\vec{p},\vec{\sigma})$,
  where $\vec{p}\in \mathbf{P}^{2 k}$ and $\vec{\sigma}\in \{-1,1\}^{2 k}$.
We let $\mathbf{L}=\mathbf{L}(\vec{p})$ be the set of indices
    $i\in \mathbf{k}$ such
    that $p_i\ne p_j$ for all $j\in \mathbf{k}$ with $j\ne i$.
    (We call such $p_i$ {\em lone primes}.)
    Define
    $\beta_i = \sigma_1 p_1 + \dotsc + \sigma_i p_i$
    for $i\in \mathbf{k}$, and let $\beta_0 = 0$.

    From now on, a {\em shape} will be a pair $(\sim, \vec{\sigma})$,
    where $\vec{\sigma} \in \{-1,1\}^{2 k}$ and
    $\sim$ is an equivalence
    relation on $\{1,2,\dotsc,2 k\}$.
    We say that the {\em length} of the shape is $2 k$.
    The shape of a pair $(\vec{p},\vec{\sigma})$
    is $(\sim,\vec{\sigma})$, where $i\sim j$
    if and only if $p_i=p_j$. The equivalence relation
    $\sim$ induces a partition $\Pi_{\vec{p}}$ of $\mathbf{k}$. In other
    words, each $\vec{p}$ gives rise to a partition $\Pi_{\vec{p}}$
    of $\mathbf{k}$ such that two $i$, $j$ are in the same equivalence class
    of $\Pi_{\vec{p}}$ if and only if $p_i=p_j$. We will sometimes denote
    $\Pi_{\vec{p}}$ by $\Pi$, omitting the dependence on $\vec{p}$.

\subsection{Lone primes between repeated primes}\label{subs:paradim}

Let us begin by proving an easy bound. While simple, it can be seen
as paradigmatic of the results in this section: there are some divisibility
conditions, caused by repetitions of primes in a vector
$\vec{p} = (p_1,\dotsc,p_{2 k})$, and these conditions impose constraints
on other primes in the vector, thus leading to good bounds on a sum over 
such vectors.

\begin{lemma}\label{lem:primitivo}
  Let $\mathbf{P}\subset [H_0,H]$, $\mathscr{L}$,
  $\mathbf{k}$ and $\mathbf{L}(\vec{p})$ be as above.
       Let $\mathbf{l}, \mathbf{L}\subset \mathbf{k}$ and $\vec{\sigma}\in \{-1,1\}^{2 k}$.
    Write $\mathscr{C}_L(\mathbf{l},\mathbf{L},\vec{\sigma},r)$ for the set of all
    $\vec{p}\in \mathbf{P}^{2 k}$ 
    such that $\mathbf{L}(\vec{p})=\mathbf{L}$ and such that, for some
    \begin{equation}\label{eq:wallenb}1\leq i_1<\jmath_1<i_1'\leq i_2 <\jmath_2<i_2'\leq \dotsc \leq i_r <\jmath_r<i_r'\leq 2 k\end{equation}
    all $1\leq j\leq r$,
    $p_{i_j} = p_{i_j'}$,
    $\jmath_j\in \mathbf{L}(\vec{p})$, and, if
    $i_j,i_j'\in \mathbf{l}$, also
    \begin{equation}\label{eq:babhif}
      p_{i_j}|\sigma_{i_j+1} p_{i_j+1} + \dotsb + \sigma_{i_j'} p_{i_j'}.
      \end{equation}
       For $\vec{p}\in \mathbf{P}^{2 k}$, let the equivalence relation $\sim$ on $\mathbf{k}$ be defined as above, and let
       $\Pi$ be the partition of $\mathbf{k}$ induced by $\sim$. Then,
       for any $r\geq 1$,
\begin{equation} 
  \sum_{\vec{p} \in \mathscr{C}_L(\mathbf{l},\mathbf{L},\vec{\sigma},\vec{i},r)}
    \prod_{i\not\in \mathbf{l}} \frac{1}{p_i} \prod_{[i]\in \Pi,
      [i]\not\subset \mathbf{k}\setminus \mathbf{l}} \frac{1}{p_{[i]}}
    \leq \frac{(8 \mathscr{L} k^2)^k}{
      (k/\sqrt{\mathscr{L}})^{|\mathbf{L}|}}
    \left(\frac{\log H}{H_0}\right)^r
    .\end{equation}
\end{lemma}
Here, as always, we define $p_{[i]}$ to be $p_i$ for any $i$ in the
equivalence class $[i]$.
\begin{proof}
  Fix $i_j$, $i_j'$ and $\jmath_j$ for $1\leq j\leq r$.
  Fix $\sim$ with $[i]$ a singleton iff $i\in \mathbf{L}$,
  and consider all $\vec{p}$ inducing $\sim$,
  and satisfying the conditions in the statement relative to our chosen
  $i_j$, $i_j'$, $\jmath_j$.

  Fix $p_{[i]}$ for all equivalence classes $[i]$ with
  $|[i]|>1$.
  For $1\leq j\leq r$, if $i_j,i_j'\in \mathbf{l}$, then
  condition (\ref{eq:babhif}) holds, and so, once we fix
  $p_{[i]}$ for all singletons $[i]$ with $i_j<i<i_j'$ and $i\ne
  \jmath_j$, we see that the congruence class
  $p_{[\jmath_j]} \mo p_{[i_j]}$ is determined -- that is,
 $p_{[\jmath_j]}$ is
   forced to be in some congruence class $a \pmod{p_{[i_j]}}$.
     Hence, much as in
     \eqref{eq:ramor}, $p_{[\jmath_j]}$ contributes a sum
     \[\mathop{\sum_{p_{[\jmath_j]}\in \mathbf{P}}}_{p_{[\jmath_j]} \equiv a \mo
    p_{[i_j]}} \frac{1}{p_{[\jmath_j]}} \leq \frac{\log H}{H_0}.
  \]
  If, on the other hand, $i_j$ or $i_j'$ is not in $\mathbf{l}$,
  then we can let $p_{[i]}$ vary freely
  for all singletons $[i]$ with $i_j<i<i_j'$. The gain then comes
  elsewhere: if $i_j\not\sim i_{j'}$ for all $j'\ne j$, we have a sum
  \[\sum_{p_{[i_j]} \in \mathbf{P}} \frac{1}{p_{i_j}^2}
  \leq \frac{1}{H_0}\]
  if $i_j, i_j'\not\in \mathbf{l}$, and a sum
  \[\sum_{p_{[i_j]} \in \mathbf{P}} \frac{1}{p_{i_j}} \frac{1}{p_{[i_j]}}
  = \sum_{p_{[i_j]} \in \mathbf{P}} \frac{1}{p_{i_j}^2}
  \leq \frac{1}{H_0}\]
  if exactly one of $i_j$, $i_j'$ is in $\mathbf{l}$.
  More generally, if $[r]$ is an equivalence class and $J\ne\emptyset$ is the
  set of all $1\leq j\leq r$ such that $i_j\in [r]$ and
  $i_j, i_j'$ are not both in $\mathbf{l}$, we see that\footnote{
    It is easy to show that $\sum_{p\geq x} 1/p^r < 1/x^{r-1}$ for
    any $r\geq 2$ and any $x\geq 1$, as follows.
    We can assume that $x$ is a prime.
        Then $\sum_{p\geq x} 1/p^r< 1/x^r + \sum_{\text{$n>x$ odd}} 1/n^r<
        1/x^r + \int_{x}^\infty dt/t^r = (1/x + 1/2(r-1))/x^{r-1}
        \leq 1/x^{r-1}$.}
  \[\sum_{p_{[r]} \in \mathbf{P}} \frac{1}{p_{[r]}^{|J|+1}}
  \leq \frac{1}{H_0^{|J|}}.\]

  Thus, we see that the sum of 
\[\prod_{i\not\in \mathbf{l}} \frac{1}{p_i} \prod_{[i]\in \Pi,
      [i]\not\subset \mathbf{k}\setminus \mathbf{l}} \frac{1}{p_{[i]}}\]
over all $\vec{p}$ we are considering is
\[\leq \left(\frac{\log H}{H_0}\right)^{r} \mathscr{L}^{|\Pi|-r}.\]
Clearly, $|\Pi| \leq k + |\mathbf{L}|/2$.
The number of possible choices for
$i_j$, $i_j'$ and $\jmath_j$ obeying \eqref{eq:wallenb} is $\leq
2^{2 k + r}\leq 2^{3 k}$,
and the number   \footnote{We could obtain a more precise bound by
          working with Stirling numbers of the second kind (for instance)
          in order to bound
          the number of choices of $\sim$ in terms of $|\Pi|$.}
of possible equivalence conditions $\sim$ on $\mathbf{k}$
with a given set of singletons of size $|\mathbf{L}|$
is $\leq (k-|\mathbf{L}|/2)^{2 k- |\mathbf{L}|}\leq k^{2 k - \mathbf{L}}$.
 
We conclude that our total is
\[\leq \frac{2^{3 k} k^{2 k} \mathscr{L}^{k + |\mathbf{L}|/2-r}}{
  k^{|\mathbf{L}|}} \left(\frac{\log H}{H_0}\right)^r.\]
  \end{proof}

    \subsection{Tools}\label{subs:tools}

We will need a very simple lemma, belonging to the most basic kind
of geometry of numbers.
\begin{lemma}\label{lem:bgeonum}
  Let $\mathbf{M}=(b_{i,j})_{1\leq i,j\leq m}$ be a non-singular $m$-by-$m$ matrix with integer entries.
  Assume $|b_{i,j}|\leq C$ for all $1\leq i,j\leq m$.
  Let $\vec{c}\in \mathbb{Z}^m$, and let $r_1,\dotsc,r_m\geq M\geq 1$.
  Let $N_1,\dotsc,N_m$ be real numbers $\geq M$. Then the number of solutions
  $\vec{n}\in \mathbb{Z}^m$ to
  \[r_i | (\mathbf{M} \vec{n} + \vec{c})_i\;\;\forall 1\leq i\leq m\]
  with $N_i\leq n_i\leq 2 N_i$ is at most
  \[\left(\frac{2 C m}{M}\right)^m \prod_{i=1}^m N_i.\]
\end{lemma}
Here the trivial bound is $\prod_{i=1}^m (N_i+1)$.
\begin{proof}
  First, we divide the box $\prod_{i=1}^m [N_i, 2 N_i]$ into
  \[\leq \prod_{i=1}^m \left(\frac{N_i}{M}+1\right) \leq 2^m (\prod_{i=1}^m N_i)/M^m\] boxes
  of sides $\leq M$.
  
  The image of an $m$-by-$m$ box of sides $\leq M$ under the map
  $\vec{n} \mapsto \mathbf{M} \vec{n} + \vec{c}$ is contained in a box
  whose edges are open or half-open interval of length $C m M$.
  Since $r_i\geq M$, inside this last box, there
  are at most $\prod_{i=1}^m C m  = (C m)^m$
  solutions $\vec{m}$ to the equations $r_i|m_i$.
\end{proof}

We will also need a very easy linear-algebra lemma.
\begin{lemma}\label{lem:linearny}
  Let $A$ be an $n$-by-$m$ matrix such that
  \begin{itemize}
    \item every row has at least one non-zero entry
    \item no column has more than $\kappa$ non-zero entries.
  \end{itemize}
Then \[\rank(A)\geq n/\kappa.\]
\end{lemma}
\begin{proof}
  We will construct a set $S$ of columns, starting with $S = \emptyset$. At each step, if there is a row $i$ such that the $i$th entry of every column in $S$ is $0$, include in $S$ the column whose $i$th entry is non-zero. Stop if there is no such row.

  When we stop,
  we must have $\kappa\cdot |S|\geq n$. It is clear that the columns in $S$ are linearly independent, and so
$\rank(A)\geq |S|$.
\end{proof}

    \subsection{Sums with many conditions \texorpdfstring{$p_i|\beta_i-\beta_j$}{pi|betai-betaj}}\label{subs:manycond}
    Let us now bound the sum over some walks set aside by the main result of
    \S \ref{sec:returns}. Again, as in \S \ref{subs:paradim}, we have divisibility
    conditions constraining our variables. 
    The procedure to follow is less evident
    now, as it is not clear which variables $p_i$ we should use to impose divisibility
    conditions and which we should reserve to be constrained by those conditions.
    There is also no obvious one-to-one relation between conditions and constrained
    variables; it is here that we will use the results in \S \ref{subs:tools}.

    \begin{lemma}\label{lem:narwhal}
      Let $\mathbf{P}\subset [H_0,H]$, $\mathscr{L}$,
      $\mathbf{k}$, 
      $\beta_i=\beta_i(\vec{p},\vec{\sigma})$ and $\mathbf{L}(\vec{p})$
      be as set at the beginning of \S \ref{sec:praeludium},
      with       $H_0\geq \max(2,(4 e k \log H)/(\mathscr{L} \log 2))$ and
      $k\geq \sqrt{\mathscr{L}}$.
            Let $\mathbf{S}_0(\vec{p},\vec{\sigma})$ be the set of indices
      $i\in \mathbf{L}(\vec{p})$ such that $p_i|\beta_i-\beta_\jmath$ for some
      $\jmath=\jmath(i)\in \mathbf{k}$ such that
      there is at least one element $j$ of $\mathbf{L}$ for which
      $i<j\leq \jmath$ or $\jmath < j < i$.
      Then, for any $\mathbf{S}_0\subset \mathbf{L}\subset \mathbf{k}$ with
      $|\mathbf{S}_0|\geq 1$,
      \begin{equation}\label{eq:omatre}
        \mathop{\mathop{\mathop{\sum_{\vec{p}}}_{p_i\in \mathbf{P}}\; \mathop{\sum_{\vec{\sigma}}}_{\sigma_i \in \{ \pm 1\}}}_{\mathbf{L}(\vec{p}) = \mathbf{L}}}_{
        \mathbf{S}_0(\vec{p},\vec{\sigma}) = \mathbf{S}_0}
      \sum_{\mathbf{l}\subset \mathbf{k}}
        \prod_{i\not\in \mathbf{l}} \frac{1}{p_i} \prod_{[i]\in \Pi,
      [i]\not\subset \mathbf{k}\setminus \mathbf{l}} \frac{1}{p_{[i]}}
      \leq
      \frac{\mathscr{L}^k (4 k)^{2 k+1}/2}{
        \left(\frac{ 2 \mathscr{L} \log 2}{(2 k)^5 \log H} \cdot H_0\right)^{\frac{|\mathbf{S}_0|-1}{4}}} .
      \end{equation}
    \end{lemma}
    \begin{proof}
      Define the matrix $M=(m_{i,j})_{i,j\in \mathbf{L}}$ by
      $m_{i,j} = \sigma_j$ if $i<j\leq \jmath(i)$ or
      $m_{i,j} = -\sigma_j$ if
      $\jmath< j < i$, and $m_{i,j}=0$ otherwise.
      Write $s = |\mathbf{S}_0|$. We wish to
      show that there are disjoint $I\subset \mathbf{S}_0$,
      $J\subset \mathbf{L}$ such
      that the submatrix $N=M(I,J)$ of $M$ consisting of rows in $I$ and columns
      in $J$ has large rank.

      Let $I$ consist of the first, third, etc., elements of $\mathbf{S}_0$,
      from left to right, except for the last element when
      $|\mathbf{S}_0|$ is odd,
      and let $J=\mathbf{L}\setminus I$. The rank of $N$
      does not change if we multiply the $j$th column by $\sigma_j$.
      It also does not change if we then replace each column of $N$
      but the leftmost one
      by the difference between itself and the column of $N$ immediately
      to its left. Hence we are asking about the rank of the
      $I$-by-$J$ matrix $N'$ such that the row of $N'$ indexed by $i\in I$
      has an entry $-1$ at the least element $j\in J$ such that
      $j>\jmath(i)$ (if there is such an element), an
      entry $1$ at the least element $j\in J$ such that $j>i$
      (if there is such an element), and entries $0$ at all other places.
      By assumption, at least one of those two values $j\in J$ exists,
      and, if both exist, they are distinct:
      if $\jmath(i)> i$, then the least $j\in \mathbf{L}$ with $j>i$
      must exist and satisfy $j\leq \jmath(i)$, and, since $I$ does not contain
      two consecutive elements of $\mathbf{L}$, it is clear that
      $j\in \mathbf{L}\setminus I = J$;
      if $\jmath(i)<i$, then the greatest $j\in \mathbf{L}$ with $j<i$
      must exist and satisfy $j>\jmath(i)$, and, for the same reason as before,
      $j\in J$.
      
      Thus, $N'$ is a matrix where
      every row has at most
      two non-zero entries per row,
      and where at least $|I|$ columns are non-zero
      (namely, for each $i\in I$,
      the column indexed by the least $j\in J$ for which
      $j>i$).
      Applying Lemma \ref{lem:linearny} to the transpose
      $(N')^T$ of $N'$, we conclude that
      \[\rank(N) = \rank(N') = \rank((N')^T) \geq \frac{|I|}{2}
      = \frac{\lfloor |\mathbf{S}_0|/2\rfloor}{2}\geq
      \frac{|\mathbf{S}_0| - 1}{4}
      .\]

      Now let $N''=M(I',J')$ ($I'\subset I$, $J'\subset J$)
      be a non-singular $m$-by-$m$ submatrix of $N$, where
      $m=\rank(N)$. The conditions $p_i|\beta_i-\beta_{\jmath(i)}$
      for $i\in I'$ read as follows:
      \begin{equation}\label{eq:divsys}
        p_i|\sum_{j\in J'} m_{i,j} p_j + c_i,\end{equation}
      where $c_i$ is a linear combination of $p_j$ for
      $j\in \mathbf{k}\setminus J'$,
      and thus does not depend on $p_j$ for $j\in J'
      \subset \mathbf{L}$. Hence, we may apply Lemma
      \ref{lem:bgeonum}, and obtain that, for
      given values of $p_j\in \mathbf{P}$ for $j\in \mathbf{k}\setminus J'$,
      and given $N_j\in [H_0,H]$ for $j\in J'$, the number of solutions
      $(p_j)_{j\in J'}$ to \eqref{eq:divsys} with
      $N_j<p_j\leq 2 N_j$ is at most
      \[\left(\frac{2 m}{H_0}\right)^m \prod_{j\in J'} N_j.\]
      Thus, the sum of $1/\prod_{j\in J'} p_j$ over
      all solutions $(p_j)_{j\in J'}$ to \eqref{eq:divsys} is
      \[\leq \left(\frac{2 m}{H_0} \left\lceil\frac{\log H/H_0}{\log 2}
      \right\rceil\right)^m
\leq \left(\frac{2 m \log H}{H_0 \log 2}\right)^m
      .\]

      The rest is routine. We will sum over all possible shapes
      $(\sim,\vec{\sigma})$ such that the set of singletons of
      $\sim$ is $\mathbf{L}$, over all $\mathbf{l}\subset \mathbf{k}$,
      over all possible choices of
      $\jmath=\jmath(i)$ for $i\in \mathbf{S}_0$,
      and over all walks of shape
      $(\sim,\vec{\sigma})$ fulfilling \eqref{eq:divsys}.
      The expression in
        \eqref{eq:omatre} can be written as 
        \begin{equation}\label{eq:twit}
        \mathop{\sum_{\vec{\sigma}}}_{\sigma_i \in \{ \pm 1\}}
        \mathop{\sum_{\sim}}_{\mathbf{L}(\sim)=\mathbf{L}}
        \sum_{\mathbf{l}\subset \mathbf{k}}\;
        \mathop{\sum_{p_{[i]}\in \mathbf{P}\; \forall [i]\in \Pi}}_{
          \mathbf{S}_0(\vec{p},\vec{\sigma}) = \mathbf{S}_0}
        \prod_{i\not\in \mathbf{l}} \frac{1}{p_i} \prod_{[i]\in \Pi,
      [i]\not\subset \mathbf{k}\setminus \mathbf{l}} \frac{1}{p_{[i]}}
        \end{equation}
        where we let
        $\mathbf{L}(\sim)$ be the set of singletons of $\sim$,
               $\Pi$ be the set of equivalence classes of $\sim$,
        and $\vec{p}$ be the tuple induced by $\sim$ and $p_{[i]}$.
     For given $(\sim,\vec{\sigma})$,
      $\mathbf{l}\subset \mathbf{k}$ and $i\mapsto \jmath(i)$,
     \begin{equation}\label{eq:rarit2}\begin{aligned}
  \mathop{\sum_{p_{[i]}\in \mathbf{P}\; \forall [i]\in \Pi}}_{
          \mathbf{S}_0(\vec{p},\vec{\sigma}) = \mathbf{S}_0}
        \prod_{i\not\in \mathbf{l}} \frac{1}{p_i} \prod_{[i]\in \Pi,
      [i]\not\subset \mathbf{k}\setminus \mathbf{l}} \frac{1}{p_{[i]}} 
\leq \mathscr{L}^{|\Pi|-m} \left(\frac{2 m \log H}{H_0 \log 2}\right)^m
.\end{aligned}\end{equation}
 
        There are $2^{2 k}$ possible choices of $\vec{\sigma}$,
        $2^{2 k}$ possible choices of $\mathbf{l}\subset \mathbf{k}$
        and $\leq (2 k)^{|\mathbf{S}_0|}$ possible choices of $i\mapsto \jmath(i)$.
        The number of choices of $\sim$ with $\mathbf{L}$
        as its set of singletons is
        $\leq k^{2 k - |\mathbf{L}|}$.
        Since $|\Pi|\leq k+ |\mathbf{L}|/2$, we conclude that
       the expression in \eqref{eq:twit} is at most
        \[4^{2 k} k^{2 k - |\mathbf{L}|}
        \mathscr{L}^{k + \frac{|\mathbf{L}|}{2}} (2 k)^{|\mathbf{S}_0|} \left(\frac{2 m \log H}{\mathscr{L} H_0 \log 2}\right)^m.\]
        Since $4 k \log H \leq e^{-1} \mathscr{L} H_0 \log 2$, we see that
        $(4 t \log H/\mathscr{L} H_0 \log 2)^t$ is decreasing on $t$
        for $t\leq 2 k$. We proved at the beginning that $m\geq
        (|\mathbf{S}_0|-1)/4$. Hence
        \[\left(\frac{2 m \log H}{\mathscr{L} H_0 \log 2}\right)^m
        \leq \left(\frac{(|\mathbf{S}_0|-1) \log H}{2 \mathscr{L} H_0 \log 2}
        \right)^{\frac{|\mathbf{S}_0|-1}{4}}
        \leq \left(\frac{2 k\log H}{2 \mathscr{L} H_0 \log 2}
        \right)^{\frac{|\mathbf{S}_0|-1}{4}}
        .\]
    \end{proof}
    


    \begin{lemma}\label{lem:narogu}
      Let $\mathbf{P}\subset [H_0,H]$, $\mathscr{L}$,
      $\mathbf{k}$, $\beta_i=\beta_i(\vec{p},\vec{\sigma})$ and
      $\mathbf{L}(\vec{p})$
      be as set at the beginning of \S \ref{sec:praeludium},
      with $H\geq \max(H_0,e)$, $H_0\geq \max(2 k,
      (\log H)/\mathscr{L})$ and
      $k\geq \sqrt{\mathscr{L}}$.
      Let $\mathbf{S}_1(\vec{p},\vec{\sigma})$
      be the set of indices
      $i\in \mathbf{L}(\vec{p})$ such that $p_i|\beta_i-\beta_\jmath$ for some
      $\jmath=\jmath(i)\in \mathbf{k}$ with
      $\jmath\ne i-1, i$ such that
      \begin{enumerate}
      \item\label{it:putumayo}      there are no elements $j$ of $\mathbf{L}(\vec{p})$ for which
        $i<j\leq \jmath$ or $\jmath < j < i$,
      \item there is at least one prime $p\neq p_i$ 
        such that
        \begin{equation}\label{eq:harleyquinn}
          \mathop{\sum_{j\leq \jmath}}_{p_j=p} \sigma_j \ne 
          \mathop{\sum_{j\leq i}}_{p_j=p} \sigma_j .\end{equation}
      \end{enumerate}
      Then, for any $\mathbf{S}_1\subset \mathbf{L}(\vec{p})\subset \mathbf{k}$,
      \begin{equation}\label{eq:omatre2}\mathop{\mathop{\mathop{\sum_{\vec{p}}}_{p_i\in \mathbf{P}}\; \mathop{\sum_{\vec{\sigma}}}_{\sigma_i \in \{ \pm 1\}}}_{\mathbf{L}(\vec{p},\vec{\sigma}) = \mathbf{L}}}_{
        \mathbf{S}_1(\vec{p},\vec{\sigma}) = \mathbf{S}_1}
        \sum_{\mathbf{l}\subset \mathbf{k}}
        \prod_{i\not\in \mathbf{l}} \frac{1}{p_i} \prod_{[i]\in \Pi,
      [i]\not\subset \mathbf{k}\setminus \mathbf{l}} \frac{1}{p_{[i]}}
            \leq
      \frac{\mathscr{L}^k (4 k)^{2 k} }{
     \left(\frac{\mathscr{L}}{(2 k)^4 \log H}\cdot H_0\right)^{|\mathbf{S}_1|/4}} .
      \end{equation}
    \end{lemma}
    \begin{proof}
      Fix $\vec{\sigma}$, $\mathbf{l}$, $\mathbf{L}\subset \mathbf{k}$
      and a valid $i\mapsto \jmath(i)$.
            Let $\sim$ be an equivalence relation on $\mathbf{k}$ whose
      singletons are $\{j\}$ for $j\in \mathbf{L}$. 
Let $\mathbf{S}_{1,-}$, $\mathbf{S}_{1,+}$ be the sets of elements of
      $\mathbf{S}_1$ with $\jmath(i)<i-1$ or $\jmath(i)>i$,
      respectively. By pigeonhole, either one of these two sets has
      $\geq |\mathbf{S}_1|/2$. We may assume
      without loss of generality that $|\mathbf{S}_{1,+}| \geq |\mathbf{S}_1|/2$.
      Let us see which tuples $\vec{p}\in \mathbf{P}^{2 k}$
      with $p_i=p_j$ iff $i\sim j$ satisfy our conditions.

      For each $i\in \mathbf{S}_{1,+}$, let $S(i)$ be the set consisting
      of every equivalence
      class $[\imath]$ of $\sim$ containing at least one element
      $\imath$ with $i<\imath \leq \jmath$ such that
      \eqref{eq:harleyquinn} holds for $p=p_\imath$. We construct a partition
      $\mathbf{S}_{1,+} = \mathbf{S}_{1,0} \cup \mathbf{S}_{1,1}$ as follows:
      traversing $\mathbf{S}_{1,+}$ from left to right, we include
      $i\in \mathbf{S}_{1,+}$ in $\mathbf{S}_{1,0}$ if
      $S(i)\not\subset \bigcup_{i'\in \mathbf{S}_{1,+}: i'<i} S(i')$,
      and include $i$ in $\mathbf{S}_{1,1}$ otherwise.

      Suppose first that $|\mathbf{S}_{1,0}|\geq |\mathbf{S}_{1,+}|/2$.
       For each
      $i\in \mathbf{S}_{1,0}$, choose some $\iota(i)$ such that
       $[\iota(i)]\in S(i)$ and $[\iota(i)]\not\in S(i')$ for every
       $i'<i$. It is clear that all $\iota(i)$ ($i\in \mathbf{S}_{1,0}$)
       are distinct.
       Let primes $p_j$ for
       $[j]\notin \{[\iota(i)]: i\in \mathbf{S}_{1,0}\}$
       take some values in $\mathbf{P}$. Now let $i$ go through
       $\mathbf{S}_{1,0}$ from left to right, choosing at each
       step a value for $p_{\iota(i)}$. We see each time that
       all $p_{j}$ with $i<j\leq \jmath$ and $j\not\sim \iota(i)$
       have already been chosen, and thus, by \eqref{eq:harleyquinn}
       and $p_i|\beta_i-\beta_\jmath$, where
       $\beta_i-\beta_\jmath = \sum_{i<j\leq \jmath} \sigma_j p_j$,
       the congruence class of $p_{\iota(i)}$ modulo $p_i$
       is determined. (Here we are using the assumption $H_0\geq 2 k$,
       which implies that the inequality \eqref{eq:harleyquinn} holds
       also modulo $p_i$.
       In other words, counting signs, $p_{[\iota(i)]}$ appears a non-zero number of times that is not divisible by $p_i$, since it is greater than $-p_i$ and smaller than $p_i$.)
       
       Again as in \eqref{eq:ramor}, the fact that each $p_{\iota(i)}$ with
       $i\in \mathbf{S}_{1,0}$
      is constrained
      to a congruence class of modulus $\geq H_0$ means that
      the sum of $\prod_{i\in \mathbf{S}_{1,0}} 1/p_{\iota(i)}$
      over allowable $p_{\iota(i)}$ is
      \[ \left(\frac{\log H}{H_0}\right)^{|\mathbf{S}_{1,0}|}.\]
      Hence, much as in \eqref{eq:rarit2},
      \[\mathop{\sum_{p_{[i]}\in \mathbf{P}\; \forall [i]\in \Pi}}_{p_i|\beta_i-\beta_{\jmath(i)}\; \forall i\in \mathbf{S}_{1,0}}
        \prod_{i\not\in \mathbf{l}} \frac{1}{p_i} \prod_{[i]\in \Pi,
      [i]\not\subset \mathbf{k}\setminus \mathbf{l}} \frac{1}{p_{[i]}}      
      \leq
        \mathscr{L}^{|\Pi|-|\mathbf{S}_{1,0}|}
         \left(\frac{\log H}{H_0}\right)^{|\mathbf{S}_{1,0}|},\]
         where $\Pi$ is the set of equivalence classes of $\sim$.
         
         Now suppose that $|\mathbf{S}_{1,1}|\geq |\mathbf{S}_{1,+}|/2$.
         Let $p_{[i]}$ take arbitrary values in $\mathbf{P}$ for
         $i\notin \mathbf{S}_{1,1}$.
         For each $i\in \mathbf{S}_{1,1}$, we choose some 
         $\iota(i)$ such that $[\iota(i)]\in S(i)$.
         We know that
         $i<\iota(i)<\jmath(i)$, and so, by property
         \eqref{it:putumayo} in the statement, $\iota(i) < i'$ for every
         element $i'$ of $\mathbf{S}_{1,1}$ (or of $\mathbf{L}$)
         larger than $i$ (since $\jmath(i)<i'$). By the definition of
         $\mathbf{S}_{1,1}$, we also know that there is some $\iota'(i)<i$
         such that $\iota'(i)\sim \iota(i)$. If possible, we
         choose $\iota'(i)$ in $\mathbf{l}$.
         Thus, letting $i$ go through
         $\mathbf{S}_{1,1}$ from left to right, and choosing $p_i$ at each
         step, we see that, if $\iota(i)$ and $\iota'(i)$ are both
         in $\mathbf{l}$, then
         $p_i$ is constrained by
         $p_{\iota'(i)}=p_{\iota(i)}|n+\beta_{\iota'(i)}$ and
         $p_{\iota(i)}|n+\beta_{\iota(i)}$
         to be in a given congruence class modulo $p_{\iota(i)}$.
         (All values of $p_j$ for all $\iota'(i)<j<\iota(i)$ with
         $j\ne i$ have already been fixed.) In this process, if, for a given
         class $\mathbf{j}$, it happened $s\geq 1$ times that
         $\iota(i)\in \mathbf{j}$,
         but either $\iota(i)$ was not in $\mathbf{l}$ or no
         possible $\iota'(i)$ in $\mathbf{l}$ could be found, then
         $\mathbf{j}$ has at least $s+1$ elements, of which at least
         $s$ are not in $\mathbf{l}$.

         Let $\mathbf{S}' = \{i\in \mathbf{S}_{1,1}: \iota(i), \iota'(i)\in \mathbf{l}\}$.
         We can bound, much as before,
         \[\mathop{\sum_{p_{i}\in \mathbf{S}'}}_{
           p_{\iota(i)}|\beta_{\iota'(i)}-\beta_{\iota(i)}}
         \prod_{i\in \mathbf{S}'} \frac{1}{p_i} \leq
         \left(\frac{\log H}{H_0}\right)^{|\mathbf{S}'|}.\]
         For each equivalence class $\mathbf{j}$ as above
         having at least $s+1$ elements, of which at least $s$ are not
         in $\mathbf{l}$, we have a factor of at most
         $\sum_{p_{\mathbf{j}}\in \mathbf{P}} 1/p_{\mathbf{j}}^{s+1}
         < \frac{1}{H_0^s}$.
      The sum of all $s$ is $|\mathbf{S}_{1,1}|-|\mathbf{S}'|$.     
      Thus, in the end, we obtain
      \[\sum_{p_{[i]}\in \mathbf{P}\; \forall [i]\in \Pi}
              \prod_{i\not\in \mathbf{l}} \frac{1}{p_i} \prod_{[i]\in \Pi,
      [i]\not\subset \mathbf{k}\setminus \mathbf{l}} \frac{1}{p_{[i]}}\leq
        \mathscr{L}^{|\Pi|-|\mathbf{S}_{1,1}|}
        \left(\frac{\log H}{H_0}\right)^{|\mathbf{S}_{1,1}|} .\]
        
        It is time to conclude. Just as in the proof of Lemma \ref{lem:narwhal},
        there are $2^{2 k}$ possibilities for each of $\vec{\sigma}$,
        and $\mathbf{l}\subset \mathbf{k}$,
        at most
        $(2 k)^{|\mathbf{S}_1|}$ possibilities for
        $i\mapsto \jmath(i)$, and
        at most $k^{2 k - |\mathbf{L}|}$ possibilities for
        $\sim$. Recalling that
        $|\Pi|\leq k+|\mathbf{L}|/2$ and $k\geq \sqrt{\mathscr{L}}$,
        we see that our total is 
        \[\leq 4^{2 k} k^{2 k} \mathscr{L}^k (2 k)^{|\mathbf{S}_1|}
        \left(\frac{\log H}{\mathscr{L} H_0}\right)^{|\mathbf{S}_{1,+}|/2}.\]
    \end{proof}

We can now bound the sum $\mathscr{S}_2$ from Prop.~\ref{prop:maincanc}.
\begin{lemma}\label{lem:conclE2}
  Let $N,k \in \mathbb{Z}_{>0}$; write $\mathbf{N} = \{N+1,N+2,\dotsc, 2 N\}$
  and $\mathbf{k} = \{1,2,\dotsc, 2 k\}$.
  Let $\mathbf{P}\subset [H_0,H]$ be a set
  of primes, and write $\mathscr{L} = \sum_{p\in \mathbf{P}} 1/p$.
  Assume that $H^{2 k}\leq N$, $H\geq H_0$,
  $H_0\geq \max(e,(4 e k \log H)/(\mathscr{L} \log 2))$
  and $k\geq \sqrt{\mathscr{L}}$.

  Given $\vec{\sigma}\in \{-1,1\}^{2 k}$ and
  $\vec{p}\in \mathbf{P}^{2 k}$,
  let $\beta_i = \beta_i(\vec{\sigma},\vec{p})$ be defined by
  $\beta_i = \sum_{j=1}^i \sigma_j p_j$. 
  Let $\mathbf{L}(\vec{p})$ be the set of indices
    $i\in \mathbf{k}$ such
  that $p_i\ne p_j$ for all $j\in \mathbf{k}$ with $j\ne i$.
Let $\mathbf{S}(\vec{p},\vec{\sigma})$
be the set of indices $i\in \mathbf{L}(\vec{p})$ for which there exists
$\jmath=\jmath(i)\in \mathbf{k}$, $\jmath\ne i-1,i$,
such that either (i) $p_i|\beta_i-\beta_\jmath$ and $\beta_\jmath\ne
\beta_{i-1},\beta_i$
or (ii) $\beta_i=\beta_\jmath$ and $\jmath \in \mathbf{L}(\vec{p})$.
Then, for $s\geq 1$,
        \begin{equation}\label{eq:artoh}
\mathop{\sum_{\vec{p} \in \mathbf{P}^{2 k}} \sum_{\vec{\sigma} \in \{ \pm 1\}^{2 k}}}_{|\mathbf{S}(\vec{p},\vec{\sigma})|\geq s}
   \sum_{\mathbf{l}\subset \mathbf{k}}\;
      \mathop{\sum_{n\in \mathbf{N}}}_{p_i|n+\beta_i \forall i\in \mathbf{l}}
      \prod_{i \in \mathbf{k}\setminus \mathbf{l}} \frac{1}{p_i} 
      \leq \frac{\mathscr{L}^k (12 k)^{2 k} (4 k + 2)}{
        \left(\frac{2 \mathscr{L} \log 2}{(2 k)^5 \log H} \cdot
        H_0\right)^{\frac{s-1}{8}}}
      \cdot N
      .
        \end{equation} 
\end{lemma}
\begin{proof}
  Consider first $i\in \mathbf{S}(\vec{p},\vec{\sigma})$
  such that $\beta_\jmath \ne \beta_{i-1}, \beta_i$.
  If $\jmath(i)>i$, then, by $\beta_\jmath\ne \beta_i$,
  \eqref{eq:harleyquinn} must hold for some
  prime $p$, and, since $p_i$ is a lone prime, we must have $p\ne p_i$;
  if $\jmath(i)<i$, then, by $\beta_\jmath\ne \beta_{i-1}$
  and again by the fact that $p_i$ is a lone prime,
  we must also have \eqref{eq:harleyquinn} for some $p\ne p_i$.
  Thus $i$ must be either in the set
  $\mathbf{S}_0(\vec{p},\vec{\sigma})$ defined in Lemma \ref{lem:narwhal}
  or in the set $\mathbf{S}_1(\vec{p},\vec{\sigma})$ defined in Lemma
  \ref{lem:narogu}, depending on whether there is an element $j$
  of $\mathbf{L}(\vec{p})$ between $i$ and $\jmath$ (that is,
  $i<j\leq \jmath$ or $\jmath<j<i$).
  
  Consider now $i\in \mathbf{S}(\vec{p},\vec{\sigma})$ such that
  $\beta_\jmath = \beta_i$ for some $\jmath=\jmath(i)\in \mathbf{L}(\vec{p})$ with
  $\jmath\ne i-1,i$.
  If $\jmath>i$, then $i$ is in the set $\mathbf{S}_0(\vec{p},\vec{\sigma})$
  defined in Lemma \ref{lem:narwhal}: take $j = \jmath$.
   If $\jmath(i)<i-1$, then, since
  $\beta_{i-1}-\beta_\jmath = \beta_i-p_i-\beta_\jmath = - p_i \ne 0$,
  we see that \eqref{eq:harleyquinn} must hold for some prime $p\ne p_i$.

  We conclude that $\mathbf{S}(\vec{p},\vec{\sigma})$ is contained
  in the union of the
  sets  $\mathbf{S}_0(\vec{p},\vec{\sigma})$ and
  $\mathbf{S}_1(\vec{p},\vec{\sigma})$ in Lemmas \ref{lem:narwhal}
  and \ref{lem:narogu}, and so, if
  $|\mathbf{S}(\vec{p},\vec{\sigma})|\geq s$, we must have either
  $|\mathbf{S}_0(\vec{p},\vec{\sigma})|\geq (s+1)/2$ or
  $|\mathbf{S}_1(\vec{p},\vec{\sigma})|\geq (s-1)/2$.
    Hence, it is enough to sum
  the bounds from Lemmas \ref{lem:narwhal} and \ref{lem:narogu},
  and multiply them
  by the number of possible choices of
  $\mathbf{L}$ and $\mathbf{S}\subset \mathbf{L}$, namely,
  $3^{2 k}$; finally, we multiply them by $2 N$, since the number of
  elements of an arithmetic progression of modulus
  $m\leq H^{2 k}\leq N$ in $\mathbf{N}$ is $\leq 2 N/m$.
\end{proof}

    \subsection{Disjoint recurrences}

    The sieving procedure in \S \ref{sec:returns} allows us to work
    with paths such that, if $p_{i}=p_{i'}=p$ and $p_j\ne p$
    for some $i<j<i'$, and $p_j|n+\beta_j$ for {\em every}
    $i\leq j \leq i'$, then $i'-i\geq \ell$ for
    a certain large $\ell$. We would like to show that we can
    drop the condition that $p_j|n+\beta_j$ for every $i\leq j\leq i'$
    and obtain a gain if the conclusion does not hold.
    
    \begin{lemma}\label{lem:garmand}
            Let $\mathbf{P}\subset [H_0,H]$, $\mathscr{L}$,
      $\mathbf{k}$, $(p_i,\sigma_i)$ for $i\in \mathbf{k}$ and
      $\beta_i$ for $0\leq i\leq 2 k$
      be as set at the beginning of \S \ref{sec:praeludium},
      with $(\log H)^2\leq H_0\leq H$ and
      $k\geq \sqrt{\mathscr{L}}$.
      For given $\vec{p}$ and $\mathbf{l}\subset \mathbf{k}$,
      let $\mathbf{I}(\vec{p},\mathbf{l})$
      be the set of all pairs $(\vec{i},\vec{i}')$ of tuples
      in $\mathbf{k}^r$ such that
      \begin{enumerate}
        \item $i_1<i_1'\leq i_2 <i_2' \leq \dotsc \leq i_r < i_r'$,
        \item $p_{i_j} = p_{i_j'}$ for all $1\leq j\leq r$,
        \item for every $1\leq j\leq r$, there is a $\jmath=\jmath_j\notin \mathbf{l}$ such that
          $i_j\leq \jmath\leq i_j'$.
      \end{enumerate}
      Then
      \begin{equation}\label{eq:ravon}
        \mathop{\sum_{\vec{p}}}_{p_i\in \mathbf{P}}\;
        \mathop{\sum_{\vec{\sigma}}}_{\sigma_i \in \{ \pm 1\}}
        \sum_{\mathbf{l}\subset \mathbf{k}}
        \sum_{(\vec{i},\vec{i'})\in \mathbf{I}(\vec{p},\mathbf{l})}
        \prod_{i\not\in \mathbf{l}} \frac{1}{p_i} \prod_{[i]\in \Pi,
      [i]\not\subset \mathbf{k}\setminus \mathbf{l}} \frac{1}{p_{[i]}}
            \leq \frac{(20 k)^{2 k} \mathscr{L}^k}{H_0^{r/2}} .
      \end{equation}
    \end{lemma}
    \begin{proof}
      Fix $\vec{\sigma}$, $\mathbf{l}$, $\mathbf{L}\subset \mathbf{k}$,  $\vec{i}$ and $\vec{i}'$.
      Also fix the choice of $\jmath_j$ for each $1\leq j\leq r$
      (should there be more than one choice).
      Let $\sim$ be an equivalence relation on $\mathbf{k}$ whose
      singletons are $\{j\}$ for $j\in \mathbf{L}$.
 Let $\mathbf{r} = \{1\leq j\leq r: i_j, i_j'\in \mathbf{l}\}$,
      and let $\mathbf{r}_0$ be the set of all indices $j\in \mathbf{r}$
      for which $\jmath_j$ is a singleton of $\sim$.
      Define $\mathbf{L}_0= \{ \jmath_j: j\in \mathbf{r}_0\}$.
      
      The idea is that we will get a gain of a factor of at least $\sqrt{H_0}$
      for each $1\leq j\leq r$, not always for the same reason.
      If $j\in \mathfrak{r}_0$, then the reason is straightforward, viz.,
      $p_{\jmath_j}$ is constrained to a congruence class and otherwise
      varies freely. If $j\not\in \mathfrak{r}_0$, we will obtain a gain
      either because $\jmath_j$ is a non-singleton not in
      $\mathbf{l}$ or because $i_j$ is.

      Let $\vec{p}\in \mathbf{P}^{2 k}$ vary among tuples such that
      (a) for $i,j\in \mathbf{k}$, $p_i=p_j$ iff $i\sim j$, 
      (b) $(\vec{i},\vec{i'})\in \mathbf{I}(\vec{p},\mathbf{l})$.
      Once we choose the values of all primes $p_i$
      with $i\in \mathbf{k}\setminus \mathbf{L}_0$,
      the value of each prime $p_i$ with $i\in \mathbf{L}_0$
      has to fall into a fixed
      congruence class modulo $p_{i_j}$, where $j$ is the
      index in $\mathbf{r}_0$ such that $i=\jmath_j$.
      (All other $p_i$ with $i_j<i<i_j'$ have already been chosen.)

      Clearly,
      \[
        \prod_{i\not\in \mathbf{l}} \frac{1}{p_i} \prod_{[i]\in \Pi,
      [i]\not\subset \mathbf{k}\setminus \mathbf{l}} \frac{1}{p_{[i]}} = 
      \mathop{\prod_{[i]\in \Pi}}_{|[i]|>1}
\frac{1}{p_{[i]}^{\min(|[i]|,1+|[i]\cap (\mathbf{k}\setminus \mathbf{l})|)}} 
\mathop{\prod_{[i]\in \Pi}}_{|[i]|=1} \frac{1}{p_i}.\]
      We know that $\sum_{p\in \mathbf{P}} 1/p^r < 1/H_0^{r-1}$.
      If $|[i]|> 1$, then
      $\min(|[i]|,1+|[i]\cap (\mathbf{k}\setminus \mathbf{l})|)\geq
      1 + |[i]\cap (\mathbf{k}\setminus \mathbf{l})|/2$.
      Hence, as $p_{[i]}$ ranges over $\mathbf{P}$ for each
      non-singleton $[i]$, the sum of $\prod_{[i]\in \Pi: |[i]|>1}
      1/p_{[i]}^{\min(|[i]|,1+|[i]\cap (\mathbf{k}\setminus \mathbf{l})|)}$
      is at most
      \begin{equation}\label{eq:carre1}
        \frac{\mathscr{L}^{\{[i]\in \Pi: |[i]|>1, [i]\subset \mathbf{l}\}}}{
        H_0^{|\mathbf{k}\setminus (\mathbf{l}\cup \mathbf{L})|/2}}.\end{equation}
      Once more as in \eqref{eq:ramor}, the fact that each $p_i$ with
      $i\in \mathbf{L}_0$ is constrained
      to a congruence class of modulus $\geq H_0$ means that,
      as $p_{[i]}$ then ranges over $\mathbf{P}$ for each singleton
      $[i]=\{i\}$, the sum of $\prod_{[i]\in \Pi: |[i]|=1}
      1/p_{[i]}$ is at most
      \begin{equation}\label{eq:carre2}
        \mathscr{L}^{\{[i]\in \Pi: |[i]|=1, i\notin \mathbf{L}_0\}}
      \left( \frac{\log H}{H_0}\right)^{\{[i]\in \Pi: |[i]|=1, i\in \mathbf{L}_0\}}
      .\end{equation}

      Let us do our accounting. Assume first that
      $i_j\in \mathbf{l}$ for
      every $2\leq j\leq r$ for which $i_{j-1}'=i_j$.
      Then every $1\leq j\leq r$ not in $\mathbf{r}_0$ corresponds
      to a distinct $i\in \mathbf{k}\setminus (\mathbf{l}\cup
      \mathbf{L})$: either $i_j, i_j'\in \mathbf{l}$, in which case
      we choose $i=\jmath_j$, or one of $i_j, i_j'$ is not in
      $\mathbf{l}$, in which case we set $i$ equal to it (or to
      either of the two, if neither is in $\mathbf{l}$).
      For $j\in \mathbf{r}_0$, we get an element of $\mathbf{L}_0$,
      and thus a factor of $(\log H)/H_0 \leq 1/\sqrt{H_0}$.
      Thus, multiplying the expressions in \eqref{eq:carre1} and
      \eqref{eq:carre2}, we obtain at most
      \begin{equation}\label{eq:bobor}
        \frac{\mathscr{L}^{|\Pi|}}{H_0^{r/2}}.\end{equation}

      If there is some $2\leq j\leq r$ for which $i_{j-1}'=i_j$ and
      $i_j\not\in \mathbf{l}$, we gain a factor of $\sqrt{H_0}$ from it
      in \eqref{eq:carre1}, and can remove it, recurring to the case
      of the tuple $(\vec{\hat{i}},\vec{\hat{i}}')$,
      where $\vec{\hat{i}}$ is $\vec{i}$ with $i_j$ removed and
      $\vec{\hat{i}}'$ is $\vec{i}'$ with $i_{j-1}'$ removed, so that
      we work with
      \[i_1<i_1'\leq \dotsc \leq i_{j-2}<i_{j-2}'\leq i_{j-1}<i_{j}'\leq
      i_{j+1}<i_{j+1}'\leq \dotsc <i_r'.\]
      (Note that $p_{i_{j-1}} = p_{i_j'}$.) Repeat as needed.
      In the end, we obtain \eqref{eq:bobor} as our bound again.

      We finish by counting our number of choices
      of
      $\vec{\sigma}$, $\mathbf{l}$, $\mathbf{L}\subset \mathbf{k}$,
      $\vec{i}$ and $\vec{i}'$, as well as the number of possible
      equivalence relations $\sim$. Each of $\vec{\sigma}$,
      $\mathbf{l}$ and $\mathbf{L}$ contributes a factor of at most
      $2^{2 k}$. The pair $(\vec{i},\vec{i}')$ contributes a factor
      of at most $\sqrt{6}^{2 k}$, because of the way its entries are ordered.
      The number of choices of $\sim$ with
      $\mathbf{L}$ as its set of singletons is at most $k^{2 k-|\mathbf{L}|}$.
      Thus we obtain a total of
      \[\leq 96^k \sum_{\mathbf{L}\subset k} k^{2 k- |\mathbf{L}|}
      \frac{\mathscr{L}^{k+|\mathbf{L}|/2}}{H_0^{r/2}}
      \leq 96^k\cdot 2^{2 k} \cdot k^{2 k} \frac{\mathscr{L}^k}{H_0^{r/2}}.\]
    \end{proof}

Thanks to Lemmas \ref{lem:primitivo} and \ref{lem:garmand},
we can now rewrite the sum
$\mathscr{S}_1$ from Prop.~\ref{prop:maincanc}.
As always, we understand that $\vec{p}\in \mathbf{P}^{2 k}$
induces an equivalence relation on $\mathbf{k} = \{1,2,\dotsc,2 k\}$ by
$i\sim j \Leftrightarrow p_i=p_j$, and we write $[i]$ for the equivalence
class of a given $i\in \mathbf{k}$. A pair $(\vec{p},\vec{\sigma})$
with $\sigma\in \{-1,1\}^{2 k}$ induces a word
$w(\vec{p},\vec{\sigma}) = x_{[1]}^{\sigma_1} x_{[2]}^{\sigma_2} \dotsb x_{[2 k]}^{\sigma_{2 k}}$. We write $w_{i,j}(\vec{p},\vec{\sigma})$ for the subword
$x_{[i]}^{\sigma_i} x_{[i+1]}^{\sigma_{i+1}} \dotsb x_{[j]}^{\sigma_{j}}$. 
As is usual, we say that a word is {\em trivial} if it reduces to the identity.
\begin{prop}\label{prop:betterreduc}
  Let $k\in \mathbb{Z}_{>0}$, and write $\mathbf{k} = \{1,2,\dotsc,2 k\}$.
    Let $\mathbf{P}\subset [H_0,H]$ be a set of primes.
    Write $\mathscr{L}$ for
    $\sum_{p\in \mathbf{P}} 1/p$, and assume $\mathscr{L}\geq e$.
    Assume that $k\geq \sqrt{\mathscr{L}}$ and $(\log H)^2\leq H_0\leq H$.

        Given $\vec{\sigma}\in \{-1,1\}^{2 k}$ and
  $\vec{p}\in \mathbf{P}^{2 k}$,
  let $\beta_i = \beta_i(\vec{\sigma},\vec{p})$ be defined by
  $\beta_i = \sum_{j=0}^i \sigma_j p_j$ for $0\leq i\leq 2 k$.
  Let $K\geq 1$, $\ell\geq 1$.
  Let $\Pi_{\vec{p}}$
  be the partition of $\mathbf{k}$ corresponding
  to the equivalence relation $\sim$ defined by
  $i\sim j\Leftrightarrow p_i=p_j$; write $p_{[i]}$ to mean $p_{i}$.
  Define $\mathscr{S}_1$ be as in \eqref{eq:mahedul1}. Then,
  for any $\kappa>2k/\ell$,
    

  \[\mathscr{S}_1 = \mathscr{S}_1' + O\left(
\frac{(20 k)^{2 k} \mathscr{L}^k}{H_0^{\kappa/4-k/2\ell}}
  \right),\]
  where \begin{equation}\label{eq:barbud}
   \mathscr{S}_1' =  \sum_{\mathcal{L} \subset \mathbf{k}} \mathscr{L}^{-|\mathcal{L}| / 2} \sum_{\mathbf{l} \subset \mathbf{k}}
  \mathop{\sum_{
      (\vec{p}, \vec{\sigma}) \in \mathscr{C}_0(k,\mathcal{L},\mathbf{l})}}_{
      \sigma_1 p_1 + \ldots + \sigma_{2k} p_{2k} = 0}
    \prod_{i\not\in \mathbf{l}} \frac{1}{p_i} \prod_{[i]\in \Pi_{\vec{p}},
      [i]\not\subset \mathbf{k}\setminus \mathbf{l}} \frac{1}{p_{[i]}}
   ,\end{equation}
  and
  $\mathscr{C}_0(k,\mathcal{L},\mathbf{l})$ is the set of all pairs
  $(\vec{p},\vec{\sigma})$ with $\vec{p}\in \mathbf{P}^{2 k}$,
  $\vec{\sigma}\in \{-1,1\}^{2 k}$
  such that
  \begin{enumerate}[(i)]
  \item\label{it:bellaciao1}
    there is an $n\in \mathbf{N}$ such that
    $p_i|n+\beta_i$ for every $i\in \mathbf{l}\cap
    (\mathbf{k}\setminus \mathcal{L})$ and
    $n+\beta_i$ has $\leq K \mathscr{L}$
prime divisors for every $1\leq i\leq 2 k$,
  \item\label{it:bellaciao2} there are no
   \begin{equation}\label{eq:remeist}
        1\leq i_1<\jmath_1<i_1'\leq i_2 <\jmath_2<i_2'\leq \dotsc \leq i_\kappa <\jmath_\kappa<i_\kappa'\leq 2 k\end{equation}
        such that
        $i_j \sim i_j'$, $i_j \not\sim \jmath_j$,
        and
$w_{i_j+1,i_j'-1}(\vec{p},\vec{\sigma})$ 
        is non-trivial for all $1\leq j\leq \kappa$,
  \item\label{it:bellaciao3} the singletons of $\Pi_{\vec{p}}$ are exactly $\{[i]: i\in \mathcal{L}\}$,
  \end{enumerate}
\end{prop}
\begin{proof}

Condition \eqref{it:wax0y} defining $\mathscr{C}(k,\mathcal{L},\mathbf{l})$
in Prop.~\ref{prop:maincanc} implies conditions \eqref{it:bellaciao1}
here, and of course condition \eqref{it:singleta}
there is identical to condition \eqref{it:bellaciao3} here. Thus,
what remains is just to bound the contribution to $\mathscr{S}_1$ of pairs
$(\vec{p},\vec{\sigma})$ obeying the conditions in Prop.~\ref{prop:maincanc}
but not condition \eqref{it:bellaciao2} here; that is, pairs for which
there are in fact 
\begin{equation}\label{eq:radeau} 1\leq i_1<\jmath_1<i_1'\leq
\dotsc \leq i_\kappa <\jmath_\kappa<i_\kappa'\leq 2 k\end{equation}
as in \eqref{it:bellaciao2}.

If $1\leq j\leq \kappa$ is such that
$n+\beta_i\in \mathbf{l}\cap (\mathbf{k}\setminus \mathcal{L})$ for
every $i$ with $i_j\leq i\leq i_j'$, then, by condition
\eqref{it:wax0y}, $p_i|n+\beta_i$ for every $i_j\leq i\leq i_j'$, and
$n+\beta_{i_j}\in Y_{\ell}$ for some $n\in \mathbf{N}$; hence, by
the definition of $Y_{\ell}$ and the fact that
$p_{i_j} = p_{i_j'} \ne p_{\jmath_j}$, we must have
$i_j'-i_j>\ell$. Since, by \eqref{eq:radeau}, the sum of
$i_j'-i_j$ over all $j$ is $<2 k$, we see that
there are fewer than $2k/\ell$ values of $j$ with the property being
considered.

The contribution of pairs $(\vec{p},\vec{\sigma})$
for which there are $\geq r_1$ indices
$1\leq j\leq \kappa$ for which there is an index
$\jmath_j'$ with $i_j\leq \jmath_j'\leq i_j'$
and $\jmath_j'\notin \mathbf{l}$ was already bounded in Lemma
\ref{lem:garmand}. The contribution of pairs for which there
are $\geq r_2$ indices $j$ for which there is a $\jmath_j'$ with
$i_j\leq \jmath_j'\leq i_j'$
and $\jmath_j'\in \mathcal{L}$ was bounded in Lemma \ref{lem:primitivo}.
Since either $r_1$ or $r_2$ must be larger than
$(\kappa-2k/\ell)/2$, we conclude that the contribution to
$\mathscr{S}_1$ of the pairs we are discussing is
\[\leq \frac{(8 \mathscr{L} k^2)^k}{
      (k/\sqrt{\mathscr{L}})^{|\mathbf{L}|}}
\left(\frac{\log H}{H_0}\right)^{\kappa/2-k/\ell} +
 \frac{(20 k)^{2 k} \mathscr{L}^k}{H_0^{\kappa/4-k/2\ell}} \leq
\frac{2 (20 k)^{2 k} \mathscr{L}^k}{H_0^{\kappa/4-k/2\ell}}.\]
\end{proof}

\section{Counting walks with few lone primes and few disjoint recurrences}\label{sec:main}
%

We have reduced the task of bounding the trace $\Tr (A|_X)^{2 k}$ to that
of bounding the quantity $\mathscr{S}_1'$ defined in \eqref{eq:barbud}.
It is now time to bound $\mathscr{S}_1'$.

\subsection{Main idea}\label{subs:mainidea}
It follows immediately from
condition (\ref{it:bellaciao1}) in Prop.~\ref{prop:betterreduc} that
for $i,i'\in \mathbf{l}$ such that $p_i=p_{i'}=p$,
    \begin{equation}\label{eq:manchego}p|\sum_{j=i_1}^{i_2-1} \sigma_j p_j .
      \end{equation}    
    To show
    that the divisibility conditions \eqref{eq:manchego}
    restrict our possible closed walks greatly, we should show
    that enough of them are independent, or rather that, when that is not the
    case, either we have a rare pair $(\vec{p},\vec{\sigma})$, or
    there are many indices $i$ that are neither in $\mathbf{l}$ nor
    in $\mathcal{L}$, thus making the contribution of our pair
    to $\mathscr{S}_1'$ small.
    ``Independence'' here means, in essence, ``linear independence''
    (of the right sides of divisibility relations of
    the form \eqref{eq:manchego}).
        
    Recall that the {\em shape} of a walk given by a pair $(\vec{p},\vec{\sigma})$
    is defined to be $(\sim,\vec{\sigma})$, where $i\sim j$
    if and only if $p_i=p_j$. We write $\prod=\prod_\sim$ for the set of
    equivalence classes of $\sim$.
    For example,    
    if $p_{[1]}, p_{[2]}, p_{[3]}, p_{[6]}$ are distinct primes such that
    $2 p_{[1]} + p_{[3]} - p_{[6]} = 0$, then the pair $(\vec{p},\vec{\sigma})$ with
    \[\vec{p} = (p_1,p_2,p_3,p_4,p_5,p_6) =
    (p_{[1]},p_{[2]},p_{[3]},p_{[1]},p_{[2]},p_{[6]})\] and
    $\vec{\sigma} = (1,-1,1,1,1,-1)$ induces the
    walk
    \begin{align*}\label{eq:bachflut}
      n & \rightarrow n + p_{[1]} \rightarrow n +
      p_{[1]} - p_{[2]} \rightarrow n + p_{[1]} - p_{[2]} +
      p_{[3]} \rightarrow n + p_{[1]} - p_{[2]} + p_{[3]} +
      p_{[1]} \\ & \rightarrow n + 2 p_{[1]} - p_{[2]} + p_{[3]}
      + p_{[2]} \rightarrow n + 2 p_{[1]} + p_{[3]} - p_{[6]} =  n,
    \end{align*}
    which has shape $(\sim, \vec{\sigma})$, where $\vec{\sigma}$ is as above
    and $\sim$ is the equivalence relation that partitions $\{1,\ldots, 6\}$ into the following equivalence classes:
    \begin{equation} \label{eq:equivclass}
   \{ \{ 1,4 \} , \{2,5 \} , \{ 3 \}, \{ 6 \} \}. 
    \end{equation}
    Now our notation makes sense: $p_{[i]}$ stands for the prime $p_j$ for any $j$
    in the equivalence class $[i]$ of $i$; for example, $p_{[1]} = p_1 = p_5$.
If $1$ and $4$ are in $\mathbf{l}$, then
\[p_{[1]}| p_{[1]} -p_{[2]} + p_{[3]};\]
if $2$ and $5$ are in $\mathbf{l}$, then
\[p_{[2]}| -p_{[2]} + p_{[3]} + p_{[1]}\]
also holds.

A shape $(\sim, \vec{\sigma})$ induces a word $w$ in the free group
generated by the letters $x_{[i]}$, $[i]\in \prod_\sim$, in the natural way:
\[w = x_{[1]}^{\sigma_1} x_{[2]}^{\sigma_2} \dotsb x_{[2 k]}^{\sigma_{2 k}}.\]
Let $w'$ be the reduction of the word $w$. We define the {\em reduced shape}
$(\sim',\vec{\sigma}')$ to be the shape inducing $w'$. We could also define
$\vec{p}'$ analogously, viz., to consist of those entries $p_i$ of $\vec{p}$
such that the letter $x_{[i]}$ is still present after reduction.

\subsection{Counting and ranks}
From now on, we will study closed walks given by pairs $(\vec{p},\vec{\sigma})$
in a set $\mathscr{C}_0(k,\mathcal{L},\mathbf{l})$ as defined in
Prop.~\ref{prop:betterreduc}.
        
    Our plan is to color some equivalence classes blue
    and some others red in such a way that, when
    we look at divisibility
    conditions (\ref{eq:manchego}) for $i_1\sim i_2$ such that
    the equivalence class $I=[i_1]=[i_2]$ is blue,
    the linear
    combinations $\sum_{i_1<j<i_2: \text{$[j]$ red}} \sigma_j x_{[j]}$
    (where $x_{[j]}$ are formal variables)
    span a linear subspace of $\mathbb{R}\left[\{x_J\}_{\text{$J$ red}}\right]$ of high dimension -- say, dimension $r$ (called the {\em rank}).

    We will then be able to show that the total contribution to $\mathscr{S}_1'$
    of walks of shape $(\sim,\vec{\sigma})$ is small,
    due to the existence of $r$ independent divisibility conditions
   that the primes occurring at red steps in the walk must obey.


   
Thanks to Lemma \ref{lem:bgeonum}, we can bound well
the
total contribution to $\mathscr{S}_1'$ of all walks
of a shape $(\sim,\vec{\sigma})$ that we can color in such a way that
the rank $r$ is large.
\begin{lemma}\label{lem:weepnth}
  Let $\sim$ be an equivalence relation on
  $\mathbf{k}=\{1,2,\dotsc,2k\}$. 
  Let $\mathbf{red}$, $\mathbf{blue}$ be disjoint subsets of
  the set  $\Pi$ of equivalence classes of $\sim$.
  
   

  Let $\vec{\sigma} \in \{-1,1\}^{2 k}$.
  Let $x_J$ be a formal variable for each $J\in \mathbf{red}$. Define
  \begin{equation}\label{eq:defi0}
    v(i) = \mathop{\sum_{j<i}}_{[j]\in \mathbf{ red}}
    \sigma_j x_{[j]}\;\;\;\;\text{for $1\leq i\leq 2 k$.}
    \end{equation}
  Assume
  $\Span_{\mathbb{R}}
  \{v(i_2)-v(i_1): [i_1]=[i_2]\in \mathbf{blue}\}$
  has dimension $\geq r\geq 0$.
  Let $\mathbf{l}\subset \mathbf{k}$ be given.
  Given $p_{[i]}\in \mathbf{P}$ for each $[i]\in \Pi$, write
  $\beta_i = \sum_{j=1}^i \sigma_j p_{[j]}$ for $0\leq i\leq 2 k$.
  
    
  Let $\mathbf{P}$ be a set of primes in $[H_0,H]$ with
$H\geq H_0\geq 1$ and    $\mathscr{L}=\sum_{p\in \mathbf{P}} 1/p\geq 1$. Then
  \begin{equation}\label{eq:ultraseven}
    \mathop{\sum_{\{p_{[i]}\}_{[i] \in \Pi}, p_{[i]}\in \mathbf{P}}}_{
      i_1,i_2\in \mathbf{l} \wedge [i_1]=[i_2]\in \textbf{blue}\Rightarrow p_{i_1}|\beta_{i_2}-\beta_{i_1}}
    \prod_{i\notin \mathbf{l}} \frac{1}{p_{[i]}}
    \prod_{[i]\in \Pi, [i]\not\subset \mathbf{k}\setminus \mathbf{l}}
    \frac{1}{p_{[i]}}\leq
    \left(\frac{4 k r \log H}{H_0 \mathscr{L} \log 2}\right)^r
     \mathscr{L}^{|\Pi|}.
     \end{equation}
\end{lemma}
Here we adopt the convention that $r^r = 1$ for $r=0$.
\begin{proof}
  Let $r'$ be the dimension of the space spanned by
  $\{v(i_2)-v(i_1)\}_{i_1,i_2\in \mathbf{l}, [i_1],[i_2]\in \textbf{blue}}$. It is simple to show that
  $r'$ is at least $r-\sum_{[i]\in \mathbf{blue}} e_{[i]}$,
  where $e_{[i]} = \min(|[i]\setminus ([i]\cap \mathbf{l})|,|[i]|-1)$.
  First, if $i\in \mathbf{l}$ for every $i\in \mathbf{k}$
  with $[i]\in \textbf{blue}$ (that is, if you wish, all blue indices
  are ``lit'') then $r'=r$ and $e_{[i]}=0$
  for every $[i]\in \textbf{blue}$, and so
  $r'\geq r-\sum_{[i]\in \mathbf{blue}} e_{[i]}$ clearly holds.
  If we remove an index $i$ from $\mathbf{l}$ (we put out the light in $i$),
  and $[i]\in \textbf{blue}$,
  we are reducing the
  rank $r'$ by at most $1$, and $e_{[i]}$ increases by $1$, unless
  $i$ was the only element of $\mathbf{l}$ in its class $[i]$
  (i.e., $|[i]\cap \mathbf{l}|=1$).
  In that last case, $e_{[i]}$ does not increase when $i$ is taken out from $\mathbf{l}$, but $r'$ also does
  not decrease, as there were no two distinct $i_1,i_2\in \mathbf{l}$ with $[i_1]=[i_2]=[i]$.
  By induction, we conclude that
  $r'\geq r-\sum_{[i]\in \mathbf{blue}} e_{[i]}$.
  
 As we know, a matrix of rank $r'$ must have a non-singular $r'$-by-$r'$ submatrix.
 Hence, we can find subsets
 $\mathbf{I}\subset \{(i_1,i_2)\in \mathbf{l}\times \mathbf{l}:
 [i_1]=[i_2]\in \mathbf{blue}\}$,
  $\mathbf{J}\subset \mathbf{red}$ with $|\mathbf{I}|=|\mathbf{J}|=r'$
    such that the matrix
  $\mathbf{M} = (a_{(i_1,i_2),\mathbf{j}})_{(i_1,i_2) \in \mathbf{I}, \mathbf{j}\in \mathbf{J}}$
  given by
  \[a_{(i_1,i_2),\mathbf{j}} = \mathop{\sum_{i_1\leq j< i_2}}_{[j] = \mathbf{j}} \sigma_j\]
  is non-singular.
  Allow the variables $p_{[i]}$ for
  $[i]\in \Pi\setminus \mathbf{J}$ to take any values in $\mathbf{P}$;
  let us examine what constraints there are then
  on $p_{[j]}$, $[j]\in \mathbf{J}$, and what the contribution of
  $\prod_{[j]\in \mathbf{J}} 1/p_{[j]}^{e_{[j]}+1}$
  to the left side of \eqref{eq:ultraseven} then is.
(Notice that we can rewrite the expression
  $\prod_{i\notin \mathbf{l}} 1/p_{[i]}
    \prod_{[i]\in \Pi, [i]\not\subset \mathbf{k}\setminus \mathbf{l}}
    1/p_{[i]}$ on the left side of \eqref{eq:ultraseven} as
    $\prod_{[i]\in \Pi} 1/p^{e_{[i]}+1}$.)
    

  We split the range of each $p_{[j]}$, $[j]\in \mathbf{J}$, into 
  dyadic intervals $N_{[j]}\leq p_{[j]} \leq 2 N_{[j]}$.
  By Lemma \ref{lem:bgeonum},
  given a choice of such an interval for every $[j]\in J$,
  there are
  \[\left(\frac{4 k r'}{H_0}\right)^{r'} \prod_{[j]\in \mathbf{J}} N_{[j]}\]
  possibilities for the tuple $(p_{[j]})_{[j]\in \mathbf{J}}$. Each such tuple
  contributes a factor of less than $\prod_{[j]\in \mathbf{J}} 1/N_{[j]}$,
  and so their total contribution is 
  $\leq (4 k r'/H_0)^{r'}$, or, over the whole range,
  \[\leq \left(\frac{4 k r'}{H_0}\right)^{r'} \left\lceil
  \frac{\log H/H_0}{\log 2}\right\rceil^{r'}\leq
\left(\frac{4 k r' \log H}{H_0 \log 2}\right)^{r'} ,\]
since we can assume without loss of generality that $H_0\geq 2$.

  Now we let all the variables $p_{[i]}$ for $[i]\in \mathbf{\Pi}
  \setminus \mathbf{J}$ range freely in $\mathbf{P}$,
    ignoring any unused divisibility conditions.
  Since $e_{[i]}\geq 0$,
  the contribution of each factor $\sum_{p_{[i]}\in \mathbf{P}} 1/p_{[i]}^{e_{[i]}+1}$
  is $\leq \mathscr{L}$, trivially.
  For $[i]\in \mathbf{red}\setminus \mathbf{J}$ or
  for $[i]\in \mathbf{blue}$ with $e_{[i]}=0$, we are content
  with this estimate.
  To bound $\sum_{p\in \mathbf{P}} 1/p_{[i]}^{e_{[i]}+1}$ for
  $[i]\in \mathbf{blue}$ with $e_{[i]}>0$, we apply
  the bound $\sum_{p\in \mathbf{P}} 1/p^\alpha \leq 1/H_0^{\alpha-1}$ (see the first 
  footnote to the proof of Prop.~\ref{lem:primitivo}).

We obtain, all in all, a bound of
  \[\frac{\mathscr{L}^{|\Pi\setminus \mathbf{J}|-
      |\{[i]\in \mathbf{blue}:  e_{[i]}>0\}|}
  }{H_0^{\sum_{[i]\in \mathbf{blue}} e_{[i]}}}
  \left(\frac{4 k r' \log H}{H_0 \log 2}\right)^{r'} \leq
\mathscr{L}^{|\Pi|}
\left(\frac{4 k r' \log H}{\mathscr{L} \log 2}\right)^{r'}
\frac{1}{H_0^{r' + \sum_{[i]\in \mathbf{blue}} e_{[i]}}}
.\]

  Recalling that
  $r-\sum_{[i]\in \mathbf{blue}} e_{[i]} \leq r'\leq r$,
  we arrive at the bound \eqref{eq:ultraseven}.
  (Obviously, $\mathscr{L} \leq \sum_{2\leq m\leq H} 1/m < \log H$,
    and so $4 kr'\log H/(\mathscr{L} \log 2) > 1$.)
\end{proof}

\subsection{From linear algebra to graph theory}


Let $\sim$ be an equivalence relation on $\textbf{k}=\{1,2,\dotsc,2k\}$,
and let $\vec{\sigma}\in \{-1,1\}^{2 k}$.
As we said
in \S \ref{subs:mainidea}, we write $(\sim',\vec{\sigma}')$ for the reduced shape associated with
$(\sim,\vec{\sigma})$: $(\sim',\vec{\sigma}')$
is the shape of the reduction $w'$ of the word $w$ associated to
$(\sim,\vec{\sigma})$. For $[i]$ an equivalence class of $\sim$,
if the letter $x_{[i]}$ in $w$ does not appear in
$w'$, we color $[i]$ yellow.

 We define $\mathscr{G}=\mathscr{G}_{(\sim,\vec{\sigma})}$
(the {\em graph induced by $(\sim,\vec{\sigma})$}) to be an
undirected graph having the non-yellow equivalence classes of $\sim$ as vertices,
and an edge between two distinct vertices $v_1$, $v_2$ iff there are
$i_1,i_2\in \{1,2,\dotsc, 2 k\}$ with $v_1 = [i_1]$,
$v_2=[i_2]$ such that every equivalence class $[j]$ with
$i_1<j<i_2$ or $i_2<j<i_1$ is yellow.
For example, a shape $(\sim, \vec{\sigma})$ with $\sim$ having
equivalence classes as in \eqref{eq:equivclass}
induces the following graph: 

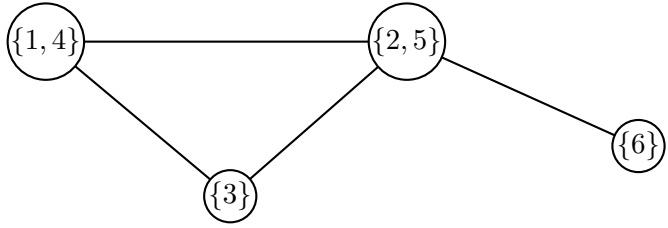
\tikzstyle{every node}=[circle, draw, fill=white!50,
  inner sep=0pt, minimum width=4pt]
\begin{center}
\begin{tikzpicture}[thick,scale=0.8]
  \draw {
    (0:0) node {$\{1,4\}$} -- (0:6) node {$\{2,5\}$}
    (0:6) -- (-40:4) node {$\{3\}$}
    (0:0) -- (-40:4)
    (0:6) -- (-10:10) node {$\{6\}$}
  };
\end{tikzpicture}
\end{center}
Let us give a second, related example, with some reduction
and yellow equivalence classes. Say
we have $(\sim,\vec{\sigma})$ with
$\vec{\sigma} = (1,-1,1,-1,1,-1,1,1,1,-1)$ and with
$\sim$ having equivalence classes
$\{\{1,8\},\{2,9\},\{3\},\{4,7\},\{5,6,10\}\}$
Then the word induced by $(\sim,\vec{\sigma})$ is
\begin{equation}\label{eq:cafrune}
  w = x_{[1]} x_{[2]}^{-1} x_{[3]} x_{[4]}^{-1} x_{[5]} x_{[5]}^{-1} x_{[4]} x_{[1]} x_{[2]} x_{[5]}^{-1},\end{equation}
which has reduction
\[w' = x_{[1]} x_{[2]}^{-1} x_{[3]} x_{[1]} x_{[2]} x_{[5]}^{-1}.\]
Hence, the equivalence class $\{4,7\}$ is colored yellow, and the
graph $\mathscr{G}_{(\sim,\vec{\sigma})}$ is
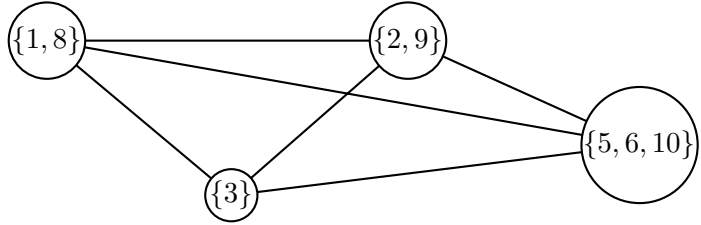
\tikzstyle{every node}=[circle, draw, fill=white!50,
  inner sep=0pt, minimum width=4pt]
\begin{center}
\begin{tikzpicture}[thick,scale=0.8]
  \draw {
    (0:0) node {$\{1,8\}$} -- (0:6) node {$\{2,9\}$}
    (0:6) -- (-40:4) node {$\{3\}$}
    (0:0) -- (-40:4)
    (0:6) -- (-10:10) node {$\{5,6,10\}$}
    (0:0) -- (-10:10)
    (-40:4) -- (-10:10)
  };
\end{tikzpicture}
\end{center}
It is clear that $\mathscr{G}_{(\sim,\vec{\sigma})}$ is always connected.

Since the vertices of $\mathscr{G}_{(\sim,\vec{\sigma})}$ are the non-yellow
equivalence classes of $\sim$, they are in one-to-one correspondence
with the equivalence classes of $\sim'$. As the example above shows,
while two vertices connected by an edge in $\mathscr{G}_{(\sim',\vec{\sigma}')}$
are always connected by an edge in $\mathscr{G}_{(\sim,\vec{\sigma})}$, the
converse need not hold; that is,
$\mathscr{G}_{(\sim',\vec{\sigma}')}$ may be a proper
subgraph of $\mathscr{G}_{(\sim,\vec{\sigma})}$.
The set of vertices of the two graphs is the same.
A coloring of the vertices of $\mathscr{G}_{(\sim,\vec{\sigma})}$ induces
a coloring of the vertices of $\mathscr{G}_{(\sim',\vec{\sigma}')}$, and thus
a coloring of the indices $1\leq i\leq 2 k'$ of the reduced shape.

Given a subset $V'\subset V$ of the set of vertices $V$ of a graph
$\mathscr{G}=(V,E)$, we write $\mathscr{G}|_{V'}$ for the graph
$(V',E')$ with $E' = \{(v_1,v_2)\in E: v_1,v_2\in V'\}$.

\begin{lemma}\label{lem:lintree}
  Let $(\sim,\vec{\sigma})$ be a shape of length $2 k$.
  Let $\mathscr{G}=\mathscr{G}_{(\sim,\vec{\sigma})}$.
  Let $\mathbf{red}$, $\mathbf{blue}$ be disjoint subsets
  of vertices of $\mathscr{G}$
  such that $\mathscr{G}|_{\mathbf{blue}}$ is connected.
  
    Let $x_{[j]}$ be a formal variable for each
  $[j] \in \mathbf{red}$, and
  $$
    v(i) = \sum_{\substack{j < i \\ [j] \in \mathbf{red}}} \sigma_j x_{[j]}
    \;\;\;\;\text{for $1\leq i\leq 2 k$.}
  $$

  Then the space $V$ spanned by the vectors 
  \begin{equation}\label{eq:sillycat}v(i_2) - v(i_1)
    ,\;\;\;\;\;\;\;i_1,i_2: [i_1]=[i_2]\in \mathbf{blue}
    \end{equation}
    equals the space $W$ spanned by the vectors
    \begin{equation}\label{eq:neatcat}
      v(i_2) - v(i_1),\;\;\;\;\;\;\;
      i_1,i_2 \in \bigcup \mathbf{blue}.
    \end{equation}
\end{lemma}
\begin{proof}
  Any connected graph contains a vertex that can be removed without making
  the graph disconnected. (The proof is very simple: if the graph
  is a tree, we remove an edge containing a leaf; if it is not, we remove
  an edge contained in a cycle.)
  We can
  thus proceed by induction: assume the result is true if we remove an element
  $I$ from $\mathbf{blue}$, where $I$ is such that
  $\mathscr{G}|_{\mathbf{blue}'}$ is connected for 
  $\mathbf{blue}' = \mathbf{blue}\setminus \{I\}$.

  Since $\mathscr{G}|_{\mathbf{blue}}$ is connected, there must be
  $I'\in \mathbf{blue}'$ such that there is an edge between $I'$ and
  $I$ in $\mathscr{G}|_{\mathbf{blue}}$. In other words, there are
  $i_0\in I$, $i_0'\in I'$ such that there is no non-yellow
  index $j$ between $i_0$ and $i_0'$.
  
  We have $V = V' + \mathbb{R} v_1 + \dots \mathbb{R} v_l$, where
  $V'$ is defined as $V$ is in (\ref{eq:sillycat}), but with
  $\mathbf{blue}'$ in place of $\mathbf{blue}$,
  and
  \[v_r = v(i_0) - v(i_r)\]
for $1\leq r\leq l$, where $i_0,i_1,\dotsc,i_l$ are the elements of $I$.
By the inductive hypothesis, $V'$ equals the space $W'$ spanned
by the vectors in (\ref{eq:neatcat}), but with $\mathbf{blue}'$ instead of
$\mathbf{blue}$. Since $[i_0], [i_0'] \in \mathbf{blue}$ and there
are no non-yellow indices between $i_0$ and $i_0'$,
we see that
$v(i_0) = v(i_0')$.
Thus, $v_r$ does not change if $i_0$ is
replaced by $i_0'$. Hence
\[V = V' + \mathbb{R} (v(i_0')-v(i_1)) + \dots +\mathbb{R} (v(i_0')-v(i_l)).\]
We know that
  $v(i') - v(i_0') \in W' = V' \subset V$ for all $i' \in \bigcup \mathbf{blue}'$.
It follows that
$(v(i') - v(i_0')) + (v(i_0')-v(i_r)) = v(i') - v(i_r)$ is in $V$
for every $i'\in \bigcup \mathbf{blue}'$ and every $0\leq r\leq l$.
By the assumption $V' = W'$, we also know that $v(i') - v(j') \in V$ for all $i',j' \in \bigcup \mathbf{blue}'$.
We conclude that $W\subset V$.
Since $V \subset W$ by definition, it follows that $V=W$.
\end{proof}

\begin{prop}\label{prop:rankbound}
  Let $(\sim,\vec{\sigma})$ be a shape of length $2 k$,
  and let $(\sim',\vec{\sigma}')$ be the reduced shape
  (of length $2 k'$, say) associated thereto. Let $\kappa\in \mathbb{Z}_{>0}$.
  Assume that there are no
  \begin{equation}\label{eq:mendels}
    1\leq i_1<\jmath_1<i_1'\leq i_2 <\jmath_2<i_2'\leq \dotsc \leq
  i_\kappa <\jmath_\kappa<i_\kappa'\leq 2 k'\end{equation}
      such that $i_j\sim' i_j'$ and $i_j\not\sim' \jmath_j$
      for all $1\leq j\leq \kappa$.

  Let $\mathscr{G}=\mathscr{G}_{(\sim,\vec{\sigma})}$, and partition the
  set of vertices of $\sim$
  into sets $\mathbf{red}$, $\mathbf{blue}$
  such that $\mathscr{G}|_{\mathbf{blue}}$ is connected. 
    Let $x_{[j]}$ be a formal variable for each
  $[j] \in \mathbf{red}$, and
  $$
    v(i) = \sum_{\substack{j < i \\ [j] \in \mathbf{red}}} \sigma_j x_{[j]}
    \;\;\;\;\text{for $1\leq i\leq 2 k$.}
  $$
  

Then the space spanned by the vectors
  \begin{equation}\label{eq:silliestcat}v(i_2) - v(i_1)
    ,\;\;\;\;\;\;\;i_1,i_2:[i_1]=[i_2]\in \mathbf{blue}
  \end{equation}
  has dimension at least
  \begin{equation}\label{eq:lowborank}
    \frac{s}{\kappa} - 1,\end{equation}
  where $s$ is the number of indices $1\leq j< 2 k'$ in the reduced shape
  such that $j$ is colored blue and $j+1$ is colored red.
\end{prop}
\begin{proof}
Let us start with a complicated-looking definition of something very simple.
  Write \[1\leq i_{1,-}\leq i_{1,+}<i_{2,-}\leq i_{2,+} < \dotsc <
  i_{s',-}\leq i_{s',+}\leq 2 k'\]
  for the indices such that (a) for all 
  $i_{r,-}\leq j\leq i_{r,+}$, $1\leq r\leq s$, the index $j$
  in the reduced shape is colored blue,
  (b) for all $i_{r,+}<j<i_{r+1,-}$, $1\leq r\leq s-1$,
  and also for $j<i_{1,-}$ and $j>i_{s',+}$,
  the index $j$ in the reduced shape is colored red,
  (c) $i_{r+1,-}-i_{r,+}>1$ for all $1\leq r\leq s-1$ (i.e., the
  red intervals are non-empty, except possibly for the ones at the ends).
  It is clear that $s'$ equals either $s$ or $s+1$.
  
  We know from Lemma \ref{lem:lintree} that the space spanned
  by the vectors in (\ref{eq:silliestcat}) equals the
  space spanned by the vectors in (\ref{eq:neatcat}), which,
  in turn, equals the space $W$ spanned by $v(i_{r+1,-})-v(i_{r,+})$ for
  $1\leq r\leq s'-1$.

  A variable $x_{[j]}$, $[j]\in \mathbf{red}$,
  can appear in $v(i_{r+1,-})-v(i_{r,+})$ only for
  $<\kappa$ values of $r$, as otherwise we would obtain
  a succession of the kind \eqref{eq:mendels} forbidden in the statement
  (with indices $i_1, i_1'=i_2, i_2'=i_3,\dotsc$ equal to the elements
  of $[j]$).
  Essentially for the same reason, 
  $v(i_{r+1,-})-v(i_{r,+})$ can be the zero
  vector only 
  for $< \kappa$ values of $r$: since the word corresponding to
  $(\sigma',\sim')$ is reduced and $i_{r+1,-}-i_{r,+}>1$,
  the subword $w_r$ corresponding to the interval $i_{r,+}<j<i_{r+1,-}$
  is reduced and non-empty; thus, if $v(i_{r+1,-})-v(i_{r,+})=0$, $w_r$ must contain
  appearances of $x$ and $x^{-1}$ (in that order or the inverse order)
  for some letter $x$, and an appearance of some other letter $y$ in between;
  in other words, there are $i_{r,+}<i_r<j_r<i_r'<i_{r+1,-}$ such that
  $i_r\sim' i_r'$ and $i_r\not\sim' j_r$, and so we see that we cannot
  obtain this conclusion $\kappa$ times or more, or else we would obtain
 a succession as in \eqref{eq:mendels}.
  
  The conditions
  of Lemma \ref{lem:linearny} are thus fulfilled (with $\kappa-1$ instead
  of $\kappa$) for the matrix
  $A=(a_{r,J})$ with rows indexed by those $r\in \{1,2,\dotsc, s'-1\}$ for which
  $v(i_{r+1,-})-v(i_{r,+})\ne 0$, columns indexed
  by equivalence classes $\mathbf{j}\in\mathbf{red}$, and
  \begin{equation}\label{eq:erwor}a_{r,[j]}  = \sum_{\substack{\iota(i_{r,+}) < \jmath < \iota(i_{r + 1, -}) \\ [\jmath] = \mathbf{j}}} \sigma_{\jmath} =
  \sum_{\substack{i_{r,+} < j < i_{r + 1, -} \\ [j]' = \mathbf{j}}} \sigma_{j}'
  ,\end{equation}
  where we are writing $\iota(i)$ for the index $1\leq \iota\leq 2 k$
  in the non-reduced shape corresponding to an index $1\leq i\leq 2k'$ in the
  reduced shape. 
  (The letters in a word $w$ lying between two letters that
  become adjacent in the reduction must, obviously, reduce to $0$, and
  so the indices $\jmath$ that do not have a corresponding index $j$ in the
  reduced shape make a total contribution of $0$ to the first sum in
  \eqref{eq:erwor}.)
  
  By Lemma \ref{lem:linearny},
  the rank of $A$ is
  \[\geq \frac{s'-1-(\kappa-1)}{(\kappa-1)} \geq
  \frac{s-1}{\kappa-1} - 1 \geq \frac{s}{\kappa}-1,\]
  where the last inequality holds if $s\geq \kappa$; if
  $s<\kappa$, then $\rank(A)\geq 0 > s/\kappa-1$.
  
  It is easy to see that the rank of $A$ also equals the dimension of $W$.

\end{proof}


Our task is then to choose a subset $\mathbf{blue}$
of the set of vertices $V$ of our graph $\mathscr{G}$
in such a way
that $\mathscr{G}|_{\textbf{blue}}$ is connected and
the quantity $s$ in Prop.~\ref{prop:rankbound} is as large as possible.

Let us draw an arrow from $[i]'$ to $[i+1]'$ for each
$1\leq i< 2 k'$, and also an arrow from $[2k']'$ to $[1]'$. (We draw each
arrow only once, i.e., we avoid multiple arrows.
We are drawing an arrow from
$[2k']'$ to $[1]'$ to make sure that there is at least one arrow going into
each vertex.) If we wish to be formal, we may say that we are superimposing
a directed graph on an undirected one.

For $S\subset V$, we define
the {\em out-boundary}
$\vec{\partial} S$ to be the set of all $w\in V\setminus S$ such that there is some
arrow going from some element of $S$ to $w$.
It is clear that $s\geq |\vec{\partial} S|-1$. Thus,
our question reduces to the following: when can we choose a subset
$\mathbf{blue}\subset V$ such that $\mathscr{G}|_{\textbf{blue}}$ is connected
and
$\vec{\partial} \mathbf{blue}$ is large?

\subsection{Connected sets and boundaries}

A {\em spanning tree} of an undirected graph $\mathbf{G}=(V,E)$ is a subgraph of
$\mathbf{G}$ that (a) is a tree, (b)
has the same vertex set $V$ as $\mathbf{G}$.
The following is a standard result in graph theory, discovered several
times with sometimes different conditions in the 1980s (see \cite[\S 1]{
  zbMATH04214027}).
Recall that our graphs are undirected by default, and that we
do not allow loops or multiple edges.
\begin{prop}[\cite{zbMATH04214027}; vd.\ also \cite{zbMATH03756505},
    \cite{zbMATH03857148}, \cite{zbMATH04134065}\footnote{In the case where the degree of
      every vertex is $3$, the result is in \cite{zbMATH03756505}
      and (essentially) also in \cite{zbMATH03857148}.}]\label{prop:storer}
  Let $\mathbf{G}$ be a connected graph with $n$ vertices, all of
  degree $\geq 3$. Then $\mathbf{G}$ has a spanning tree
  with $\geq n/4 + 2$ leaves.
\end{prop}

We will actually find it convenient to use a more general statement.
\begin{corollary}[vd.\ also \cite{zbMATH06124268}, \cite{zbMATH06406586},
\cite{zbMATH06347738},
\cite{zbMATH06347739}]\label{cor:maxleaf}
  Let $\mathbf{G}$ be a connected graph such that at least $n$ of its vertices
  have degree $\geq 3$. Then $\mathbf{G}$ has a spanning tree
  with $\geq n/4+2$ leaves.
\end{corollary}
Getting the best constant $+2$ is unimportant for our purposes, but we
might as well.
\begin{proof}
  We will reduce matters to the case where Prop.~\ref{prop:storer} applies,
  namely, that of no vertices of degree $1$ or $2$.
  
  Assume first that there are at least two distinct vertices
  $v_1$, $v_2$ of degree $1$. Let $\mathbf{G}'$ be the graph obtained
  by identifying them, forming a new vertex $v$.
  We can assume, recursively, that
  $\mathbf{G}'$ has a spanning tree
  with $\geq n/4+2$ leaves, since $\mathbf{G}'$ has fewer vertices of degree
  $1$ than $\mathbf{G}'$. If the spanning tree contains $v$
  as a leaf, it is valid as a spanning tree of $\mathbf{G}$. If it
  contains $v$ as an internal vertex (i.e., not a leaf), we separate
  $v$ into $v_1$ and $v_2$ (thus increasing the number of leaves by $2$),
  and find that we have two trees, covering all vertices of $\mathbf{G}$;
  there is some edge of $\mathbf{G}$ connecting them, and we may add
  it at a cost of at most $2$ leaves. Thus we obtain a spanning tree of
  $\mathbf{G}$ with $\geq n/4+2$ leaves in any event.

  Assume now that $\mathbf{G}$ has exactly one vertex $v$ of degree $1$.
  Make an extra copy\footnote{The argument in this paragraph was kindly contributed by
    Brendan McKay.} $\mathbf{G}'$ of $\mathbf{G}$, and give the name
  $v'$ to its copy of $v$. Identify $v$ and $v'$ to make a new graph
  $\mathbf{H}$, with $w$ being the new vertex formed by $v$ and $v'$.
  We apply the statement recursively, and obtain that
  $\mathbf{H}$ has a spanning tree with at least $2n/4 + 2$ leaves.
  At least one of the two halves of $\mathbf{H}$ has at least $n/4+1$
  of those leaves. The vertex $w$ cannot be one of those leaves, since
  the tree then would not be connected. Hence, when we split $w$ into
  $v$ and $v'$ again, we obtain that each half of the spanning tree
  gains one leaf, and so the rich half (say, that of $\mathbf{G}$)
  now has $n/4 + 2$ leaves.

  We have reduced matters to the case where there are no vertices of degree
  $1$. If there are any two vertices $w_1$, $w_2$ of degree $2$ connected
  by an edge, we can replace them and their edges ($\{v_1,w_1\}$,
  $\{w_1,w_2\}$, $\{w_2,v_2\}$) by a new vertex $w$ and edges
  $\{v_1,w\}$, $\{w,v_2\}$. If $w$ is a vertex of degree $2$
  connected to vertices $v_1$, $v_2$ without an edge between them,
  we can remove $w$ and create an edge $\{v_1,v_2\}$. Thus,
  we can assume we are in a situation where there is a vertex $w$ of
  degree $2$ and edges $\{v_1,w\}$, $\{w,v_2\}$, $\{v_1,v_2\}$,
  where $v_1$, $v_2$ are of degree $\geq 3$.

  If $v_1$ and $v_2$ both have degree $>3$, we may simply remove
  $w$ and the edges containing it, and apply the statement recursively.
  Suppose, then, that one of
  $v_1$, $v_2$ has degree $3$; say it is $v_1$. Denote the third neighbor
  of $v_1$ by $u$. Then remove the vertices $v_1$ and $w$ and all edges
  containing them, and create an edge between $u$ and $v_2$ if there is not
  one already; the number of vertices of degree
  $\geq 3$ goes down by at most $3$, the number of vertices of
  degree $2$ does not increase, and the total number of vertices
  decreases. Thus, applying our statement recursively,  we obtain
  a spanning tree with at least $(n-3)/4+2$ leaves. If the edge
  $\{u,v_2\}$ is in the spanning tree, we replace it
  by $\{u,v_1\}$, $\{v_1,v_2\}$ and $\{v_1,w\}$, thereby gaining a leaf.
  If $\{u,v_2\}$ is not in the spanning tree, add
  the edges $\{v_1,v_2\}$ and $\{v_2,w\}$, thereby gaining $2$ leaves
  and losing at most $1$. Thus we obtain a spanning tree with
  $\geq (n-3)/4+2+1>n/4+2$ leaves.
\end{proof}

Even in directed graphs, we do not allow loops or multiple edges (arrows);
at the same time, we allow that there may be arrows $(v,w)$, $(w,v)$
going in opposite directions.
The {\em in-degree} of $w$ is the number of arrows $(v,w)$ going into
  $w$.

\begin{lemma}\label{lem:palomas}
  Let $G$ be a directed graph 
  such that every vertex has positive in-degree.
Let $S$ be a set of $m$ vertices of $G$.
Then there is a subset $S'\subset S$ with
$|S'|\geq m/3$ such that, for every $w\in S'$, there is 
an arrow $(v,w)$ from some vertex $v$ not in $S'$ to $w$.
\end{lemma}
\begin{proof}
  We may remove arrows until the in-degree of every vertex is exactly $1$.
  Then $G$ is a union of disjoint cycles.
  If a cycle is contained entirely in $S$, we number its vertices in order,
  starting at an arbitrary vertex, and include in $S'$ the second, fourth,
  etc., elements. If some but not all vertices in a cycle are in $S$,
  the vertices in the cycle that are in $S$ fall into disjoint subsets of the
  form $\{v_1,\dotsc,v_r\}$, where there is an arrow from some $v$ not in $S$
  to $v_1$, and arrow from $v_i$ to $v_{i+1}$ for each $1\leq i\leq r-1$;
  we include $v_1, v_3,\dotsc$ in $S'$. If no vertices in a cycle are in $S$,
  then of course we do not include the cycle's vertices in $S'$.
\end{proof}
  
We can now prove a general result suited to our needs.\footnote{The proof, brief as it is, owes much to F. Petrov's answers to two related questions in MathOverflow (Questions 362168 and 381364).}

\begin{prop}\label{prop:iolence}
  Let $\mathbf{G}=(V,E)$ be a connected, undirected graph such that at least
  $n$ of its vertices have degree $>2$. Let there be arrows $(v,w)$
  for some pairs of distinct vertices $v,w$, in such a way that
  the in-degree of every vertex of $\mathbf{G}$ is positive.

  Then there is a subset $V'\subset V$ such that $\mathbf{G}|_{V'}$
  is connected and the out-boundary $\vec{\partial} V'$ has
  $\geq n/12 + 4/3$ elements.
\end{prop}
\begin{proof}
  By Cor.~\ref{cor:maxleaf}, $\mathbf{G}$ has a spanning tree with
  $\geq n/4+2$ leaves. Apply Lemma \ref{lem:palomas} with $S$ equal
  to the set of leaves, and define $V' = V\setminus S'$; since $V'$
  contains all inner nodes in the tree, it is connected.
\end{proof}

    We already know that our graph $\mathscr{G}$ is connected, and that
    the in-degree of each vertex of $\mathscr{G}$ is $>1$.
   (We defined $\mathscr{G}$ including only
    non-yellow equivalence classes in its set of vertices with this purpose in
    mind.)
The question
is how many of the vertices of $\mathscr{G}$ have degree $>2$, or rather what
happens when few do.

\subsection{Degrees and freedom}
The fact that the degree of a vertex is a local condition allows us to
show that vertices of degree $\leq 2$ constrain our choices. To be
more precise: if an equivalence class contains more than one element
(that is, if it corresponds to a non-lone prime) and also has degree $\leq 2$
as a vertex of $\mathscr{G}_{(\sim,\vec{\sigma})}$, then, as we read our
word from left to right, when we come to a letter $x$ in our equivalence class,
we know what follows it: one out of $\leq 2$ possible letters (or a third
possibility: $x$ itself).

Recall the definition of $\mathscr{G}_{(\sim,\vec{\sigma})}$,
and, in particular, that, if we know the set $\mathbf{n}\subset \mathbf{k}$
of indices in non-yellow equivalence classes
of $\sim$, then the restriction of $\sim$ to $\mathbf{n}$
in fact determines $\mathscr{G}_{(\sim,\vec{\sigma})}$.
We may thus write $\mathscr{G}_{\mathbf{n},\sim}$ for the graph determined by
a subset $\mathbf{n}\subset \mathbf{k}$ and
an equivalence relation $\sim$ on $\mathbf{n}$.
\begin{lemma}\label{lem:adamant}
  Let $\mathbf{n} \subset \mathbf{k} = \{1,2,\dotsc,2k\}$.
  Let $S_{k,\mathbf{n},\kappa}(\rho)$ be the set of
  equivalence relations $\sim$ on $\mathbf{n}$
  such that (a) $\mathscr{G}_{\mathbf{n},\sim}$ has $\leq \rho$ vertices of
  degree $>2$, and (b) for every equivalence class $\mathbf{i}$ corresponding
  to a vertex of degree $>2$, there are at most $\kappa$ elements
  $i'\in \mathbf{i}$ such that the following element $i$ of $\mathbf{n}$ is not in $\mathbf{i}$.
  
  Then
  \[\left|S_{k,\mathbf{n},\kappa}(\rho)\right|\leq 5^{|\mathbf{n}|}
  (2 k)^{(\kappa-1)\rho + 2}
  .\]
\end{lemma}
When, speaking of elements of $\mathbf{n}$,
we say ``following'', ``preceding'' or ``consecutive'', we skip over elements
of $\mathbf{k}\setminus \mathbf{n}$; that is, the element of $\mathbf{n}$
following $i$ is the smallest index $i'\in\mathbf{n}$
such that $i'>i$.
\begin{proof}
  We will show that we can code an element of
  $S_{k,\mathbf{n},\kappa}(\rho)$
  as a string $\vec{s}$ on $5$ letters with indices in $\mathbf{n}$, supplemented by
  some additional information at each of at most $\rho+2$ indices.

  Let us be given an element $\sim$ of $S_{k,\mathbf{n},\kappa}(\rho)$.
  We let $V_{\leq 2}$ be the set of equivalence classes whose associated
  vertices in $\mathscr{G}_{\mathbf{y},\sim}$ have degree $\leq 2$,
  and let $V_{>2}$ be the set of the other equivalence classes.
  Now let us go through the indices
  $i\in \mathbf{n}$ from left to right, and define
  $\vec{s}$ as follows so as to describe $\sim$.
  
  If the equivalence class $[i]$
  is new -- that is, if there is no $j\in \mathbf{n}$ with $j<i$ such that $j\in [i]$ --
  we let $s_i = *$. Assume that that is not the case. Let $i'$ be the
  element of $\mathbf{n}$ immediately preceding $i$.
  If $[i]=[i']$, we let $s_i=0$.
    If $[i']\in V_{\leq 2}$ and $i$ is in an equivalence class $\ne [i']$ that
  has already been seen next to $[i']$ -- that is, if there exist two
  consecutive elements $j$, $j'$ (in either order: $j<j'$ or $j'<j$)
  of $\mathbf{n}$ with $j\in [i]$, $j'\in [i']$ and
  $j,j'\leq i'$ -- then $[i]$ can be one of at most two equivalence classes,
  and what is more, they are
  already known to the person who will be reconstructing $\sim$
  by reading $\vec{s}$ from left to right. We let $s_i=1$ or $s_i=2$
  depending on which of those two equivalence classes we mean (named in order
  of appearance, say).
    In all remaining cases, we let
  $s_i = .$ (a dot).

  We should now count the number of indices $i$ such that $s_i = .$,
  as in that case, and only in that case, we should tell our reader
  which equivalence class $[i]$ we mean after all. (In all other cases,
  the class $[i]$ is determined by $s_i$ and by the part of $\sim$
  that has already been reconstructed, that is, the restriction of
  $\sim$ to $j\in \mathbf{n}$ with $j<i$.)
  For $[i']\in V_{\leq 2}$, it can happen at most once
  (that is, for at most one element of $i'\in [i']$) that $s_i\ne 0,1,2$
  for the index $i\in \mathbf{n}$ following $i'$,
  unless $1\in [i']$, in which case it can happen twice. (Someone who
  already has a neighbor and will end with a total of at most two neighbors
  can meet a new neighbor at most once.) For $[i']\in V_{> 2}$,
  by assumption, it can happen at most $\kappa$ times that $s_i\ne 0$.
  Hence, the total number of indices $i\in \mathbf{n}$ 
  with $s_i\in \{*,.\}$ is at most $\kappa |V_{>2}| + |V_{\leq 2}| + 1 + 1$,
  where the last $+1$ comes from the first index $i$ in
  $\mathbf{n}$ (as it has no index $i'$
  in $\mathbf{n}$ preceding it). It is clear that
  the total number of indices $i$ with $s_i=*$ is $|V_{>2}|+|V_{\leq 2}|$.
  Hence, the number of indices $i$ with $s_i=.$ is
  \[\leq \kappa |V_{>2}| + |V_{\leq 2}| + 2 - (|V_{>2}|+|V_{\leq 2}|) =
  (\kappa-1) |V_{>2}| + 2 \leq (\kappa - 1) \rho + 2.\]
\end{proof}
{\bf Remark.} The above proof can be seen as a simple kind of what we
will call a ``writer-reader'' argument: if a writer can code an object
by assigning it one of $M$ possible strings,
and a reader can identify the object given the string, then there were
at most $M$ possible objects. We will see a more complex argument of the
same general kind below.

\begin{prop}\label{prop:mainbound}
  Let $\rho\geq 1$.
 Let $k\in \mathbb{Z}_{>0}$, and write $\mathbf{k} = \{1,2,\dotsc,2 k\}$.
    Let $\mathbf{P}\subset [H_0,H]$ be a set of primes.
    Write $\mathscr{L}$ for $\sum_{p\in \mathbf{P}} 1/p$. Assume $\mathscr{L}\geq 1$ and let $K\geq 1$.
  
Let $\kappa \in \mathbb{Z}_{>0}$. Define
             \begin{equation}\label{eq:rafcar}\mathscr{S}_{1,\leq \rho}' =
 \sum_{\mathcal{L} \subset \mathbf{k}} \mathscr{L}^{-|\mathcal{L}| / 2} \sum_{\mathbf{l} \subset \mathbf{k}}
  \mathop{\sum_{
      (\vec{p}, \vec{\sigma}) \in \mathscr{C}_0(k,\mathcal{L},\mathbf{l})}}_{
    \sim|_{\mathbf{n}(\sim,\vec{\sigma})} \in S_{k,\mathbf{n},\kappa}(\rho)}
    \prod_{i\not\in \mathbf{l}} \frac{1}{p_i} \prod_{[i]\in \Pi,
      [i]\not\subset \mathbf{k}\setminus \mathbf{l}} \frac{1}{p_{[i]}},\end{equation}
    where $\sim$ is the equivalence relation on $\mathbf{k}$ induced by
    $\vec{p}$, $\Pi$ is the associated partition of $\mathbf{k}$,
 $\mathbf{n}(\sim,\vec{\sigma})$ is the set of indices in non-yellow equivalence
    classes of $\sim$,
    $\mathscr{C}_0(k,\mathcal{L},\mathbf{l})$ is as in
    Prop.~\ref{prop:betterreduc} and $S_{k,\mathbf{n},\kappa}(\rho)$ is as in
    Lemma \ref{lem:adamant}.
    Assume that $((\kappa-1) \rho + 2)\leq 2 k/(\log 2 k)$.
    Then
    \[\mathscr{S}_{1,\leq \rho}'\leq 
    (6400 (K + 2))^k \mathscr{L}^k. \]
\end{prop}
We could in fact make do with the usually weaker assumption
$((\kappa-1) \rho+2) \leq \frac{2 k}{\log K \mathscr{L}}$ by introducing one of the
ideas we are about to see into the proof of Lemma \ref{lem:adamant}.
\begin{proof}
  Let $\mathbf{l},\mathbf{n}\subset \mathbf{k}$, $\vec{\sigma} \in
  \{-1,1\}^{2 k}$. We let $\sim|_{\mathbf{n}}$ range over the elements
  of $S_{k,\mathbf{n},\kappa}(\rho)$. We let $p_{[i]}$ vary in $\mathbf{P}$ for
  $[i]$ in the set $\overline{\mathbf{n}}$ of non-yellow equivalence classes,
  and see that
  \begin{equation}\label{eq:kafrun}\sum_{p_{[i]}\in \mathbf{P}: [i]\in \overline{\mathbf{n}}}
  \mathop{\prod_{i\not\in \mathbf{l}}}_{[i]\in \overline{\mathbf{n}}}
  \frac{1}{p_i} \mathop{\prod_{[i]\in \overline{\mathbf{n}}}}_{
    [i]\not\subset \mathbf{k}\setminus \mathbf{l}} \frac{1}{p_{[i]}}
\leq \sum_{p_{[i]}\in \mathbf{P}: [i]\in \overline{\mathbf{n}}}
  \prod_{[i]\in \overline{\mathbf{n}}}
  \frac{1}{p_{[i]}} \leq \mathscr{L}^{|\overline{\mathbf{n}}|}
  \leq \mathscr{L}^{\frac{|\mathbf{n}|}{2}}
  \mathscr{L}^{\frac{|\mathcal{L}_\sim|}{2}},
  \end{equation}
  where $\mathcal{L}_\sim$ is the set of singletons of $\sim$.

  Now let us look at the complement $\mathbf{k}\setminus \mathbf{n}$,
  which will be partitioned into yellow equivalence classes.
  For any $(\sim,\vec{\sigma})$ having $\mathbf{k}\setminus \mathbf{n}$ as its set of yellow indices,
  the restriction
$w|_{\mathbf{k}\setminus \mathbf{n}}$
  of the associated word $w$ to the indices
  $i\in\mathbf{k}\setminus \mathbf{n}$ must have trivial reduction.
  (For instance, for $w$ as in \eqref{eq:cafrune}, the restriction of $w$
  to $\mathbf{k}\setminus \mathbf{n}$ reads as $x_{[4]}^{-1} x_{[4]}$,
  which has trivial reduction.) Hence, $w|_{\mathbf{k}\setminus \mathbf{n}}$
  induces a balanced sequence of parentheses\footnote{That is, a sequence such
    as $()(())$, but not as in $())(()$. The word $x_1 x_1^{-1} x_1^{-1} x_2^{-1} x_2 x_1$ would induce $()(())$.}
    of length $2 m$,
  where $m = |\mathbf{k}\setminus \mathbf{n}|/2$.
  What is more, for $i\in \mathbf{k}\setminus \mathbf{n}$,
  if we know that balanced sequence of parentheses, and know that
  $i$ corresponds to a closing parenthesis $)$, we know that $i$ must
  correspond to the equivalence class of the corresponding opening parenthesis
  $($. Thus, we need to specify $[i]$ (and $p_{[i]}$) only for $i$ in the
  set $\mathbf{o}\subset \mathbf{k}\setminus \mathbf{n}$ of indices of
  opening parentheses.

  Again, we read $\sim$ from left to right.
  If $i$ is the leftmost index in $\mathbf{o}$ within its equivalence class
  $[i]$, then $i$ defines $[i]$, and its contribution is simply a factor of
  $\leq \sum_{p\in \mathbf{P}} 1/p = \mathscr{L}$. (We may mark $i$ by an
  asterisk, as in the proof of Lemma \ref{lem:adamant}.)
  If $i\not\in \mathbf{l}$, then we see that $i$ contributes
  its own new term $1/p_i$ to \eqref{eq:rafcar};
  thus, as $p_i$ varies
  across primes $p_{j}$ that have appeared for some $j<i$, it
  contributes at most $\sum_{p\in \mathbf{P}} 1/p = \mathscr{L}$, and
  then of course $p_i$ determines the equivalence class $[j]$ in which
  $i$ lies. (Alternatively, we could simply specify the equivalence class $[j]$
  and then gain a factor of $1/p \leq 1/H_0$, thus obtaining a factor
  $\leq 2k/H_0$ instead of $\mathscr{L}$.)
If $i\in \mathbf{l}$ with  $j\not\in \mathbf{l}$
 for all $j\in [i]$ with $j<i$ (a case that we should also mark -- with a $\#$, say), then we are in the same situation; again, we obtain
 a factor of $\leq \mathscr{L}$ (or $\leq 2k/H_0$).

  Lastly, let us examine the case of $i\in \mathbf{l}$ such that
  there is some $j\in [i]$ with $j<i$ and $j\in \mathbf{l}$, keeping in mind that
  we have not specified $[i]$ yet. By condition
  (\ref{it:bellaciao1}) in Prop.~\ref{prop:betterreduc},
  there is an $n\in \mathbf{N}$ such that $n_{i-1} =
  n+\sigma_1 p_{[1]} + \dotsc + \sigma_{i-1} p_{[i-1]}$ has $\leq K \mathscr{L}$
  prime divisors in $\mathbf{P}$, and moreover, since $i\in \mathbf{l}$,
  we know that $p_{[i]}|n_{i-1} + \sigma_{[i]} p_{[i]}$, and so
  $p_{[j]}=p_{[i]}|n_{i-1}$. Hence
  \begin{equation}\label{eq:crement}
    p_{[j]}|\sigma_j p_{[j]} + \dotsc + \sigma_{i-1} p_{[i-1]}.\end{equation}
  Let us call an equivalence class $[j]$ a {\em plausible candidate}
  if it contains a $j$ such that \eqref{eq:crement} holds and $j$ is in
  $\mathbf{l}$. It is clear that being or not a plausible candidate
  does not depend on the choice of $n$. A moment's thought shows that,
  for any $n$ as above, and any plausible candidate $[j]$,
  the prime $p_{[j]}$ is actually forced to divide $n_{i-1}$.
  Since we know that $n_{i-1}$ has $\leq K \mathscr{L}$
  prime divisors in $\mathbf{P}$, it follows that there
  are $\leq K \mathscr{L}$ plausible candiates.
  Hence,  specifying the equivalence class $[j]=[i]$ costs us
  a factor of $\leq K \mathscr{L}$.

  \begin{danger}
    Since the argument above is slightly subtle, it may be worthwhile to
    restate it more informally. We may think of a ``writer'' who wishes
    to communicate $p_1,\dotsc,p_{2 k}$ to a ``reader''.\footnote{A first
      draft had a mahout and an elephant instead of a writer and a reader; the second named author managed to convince the first one
      that the terminology might be considered eccentric by
      editors. The term {\em elephant} had the advantage of reminding
      the reader (of this paper) that the ``reader'' has the gift of a good
     memory, and may remember everything that happened since
      the beginning of the communication.}
      At a given point
    $i$, the reader already knows $p_1,\dotsc,p_{i-1}$ and
    $\sigma_1,\dotsc,\sigma_{i-1}$ as well as $\mathbf{l}$, and it is
    the writer's task to specify $p_i$ to the reader using very few bits.
    In the case we are currently considering, the writer cannot
    simply tell the reader for which equivalence class
    $[j]$ with $j<i$ it is the case that $i\in [j]$, i.e., $p_i=p_{[j]}$:
    there
    are, in principle, $i-1$ choices of $[j]$, and so it
    is too costly to specify one directly. (Telegrams
    are expensive, and $i-1$ can be
    in the order of $k$.) However, the choice can be restricted to
    the equivalence classes $[j]$ that are {\em plausible candidates}.
    The list of plausible candidates (that is, classes $[j]$ satisfying
    \eqref{eq:crement}) is something which the writer and the reader
    can determine independently, without communicating, using
    simply their knowledge of $p_1,\dotsc,p_{i-1}$,
    $\sigma_1,\dotsc,\sigma_{i-1}$ and $\mathbf{l}$.
    Then the writer
    simply has to specify to the reader that what is meant is the $[j]$
    that is the
    11th plausible candidate in the list, say. It is the case that there
    are $\leq K \mathscr{L}$ plausible candidates: we know as much
    from the existence of an $n$ satisfying condition
    (\ref{it:bellaciao1}) in Prop.~\ref{prop:betterreduc}, by the
    argument above. Hence, the writer is specifying an integer between
    $1$ and $\lfloor K \mathscr{L}\rfloor$, and that much is acceptable,
    as it costs us a factor of $K\mathscr{L}$, which is generally
    much less than $k$.
  \end{danger}
  
  The number of balanced sequences of parentheses of length
  $2 m$ is the {\em Catalan number}
  $C_m = \frac{1}{m+1} \binom{2 m}{m}$; we will just use the coarse bound
  $C_m<2^{2 m}$. Marking some openings by $*$ or $\#$, and specifying one
  of $\leq K \mathscr{L}$ classes in others, costs us a factor of
  $\leq (K \mathscr{L}+2)^m \leq ((K+2) \mathscr{L})^m$.

  We thus see that the total contribution any given choice of
  $\mathbf{l},\mathbf{n}\subset \mathbf{k}$, $\vec{\sigma}\in \{-1,1\}^{2 k}$
  and $\sim|_\mathbf{n} \in S_{k,\mathbf{n},\kappa}(\rho)$ to
  $\mathscr{S}_{1,\leq \rho}'$ is
  \[\leq \mathscr{L}^{\frac{|\mathbf{n}|}{2}}
  \mathscr{L}^{\frac{|\mathcal{L}_\sim|}{2}} \cdot 2^{2 m} ((K+2) \mathscr{L})^m
 \leq (4 (K+2))^{k -\frac{|\mathbf{n}|}{2}} \mathscr{L}^{k+\frac{|\mathcal{L}_\sim|}{2}}.\]
 Summing over all $\mathbf{l}$, $\mathbf{n}$, $\mathcal{L}\subset \mathbf{n}$
 and $\vec{\sigma}$,
 and using the bound on $|S_{k,\mathbf{n},\kappa}(\rho)|$ from Lemma \ref{lem:adamant}
 and the assumptions $((\kappa-1) \rho + 2)\leq 2 k/(\log 2 k)$ and $K\geq 1$,
 we obtain that
 \[\begin{aligned}\mathscr{S}_{1,\leq \rho}'
 &\leq 6^{2 k} (2 k)^{(\kappa-1) \rho+2}
 \mathscr{L}^k \sum_{\mathbf{n}\subset \mathbf{k}} 5^{|\mathbf{n}|}
 (4 (K+2))^{k -\frac{|\mathbf{n}|}{2}}
 \\&\leq (2 k)^{(\kappa-1) \rho+2} (144 (K+2) \mathscr{L})^k
 \left(1 + \frac{5}{\sqrt{4 (K+ 2)}}\right)^{2 k}
 \leq e^{2 k} (860 (K + 2) \mathscr{L})^k
 .\end{aligned}\]
\end{proof}

We can now use the tools we developed before to bound
the other terms in $\mathscr{S}_1'$.
\begin{prop}\label{prop:aground}
  Let $k\geq 20$ and $\mathbf{k} = \{1,2,\dotsc,2 k\}$.
    Let $\mathbf{P}\subset [H_0,H]$ be a set of primes.
    Write $\mathscr{L}$ for $\sum_{p\in \mathbf{P}} 1/p$. Assume
    $\mathscr{L}\geq 1$, $(4 k^2 \log H)/(\mathscr{L} \log 2)\leq
    H_0^{1/3}$
    and $\log H_0 \leq 2 k$.
        Let $\kappa$ be such that
    $\kappa (\kappa-1) \leq 2 (\log H_0)/45 (\log 2 k)^2$. Let $K\geq 1$.

        Define $\mathscr{S}_1'$ as in \eqref{eq:barbud}. Then
       \[\mathscr{S}_1'\leq  (6500 (K + 2))^k \mathscr{L}^k. \]
\end{prop}
It may be worth noting that we will not be using the condition
$\sigma_1 p_1 + \dotsc + \sigma_{2 k} p_{2 k}=0$ in \eqref{eq:barbud} at all.
Using it would result in a minute improvement in the application of
Prop.~\ref{prop:rankbound}: we would add $1$ to our lower bound on $s$,
and so we could drop the assumption $\log H_0\leq 2 k$.
\begin{proof}
  Clearly, for any $\rho$, $\mathscr{S}_1' \leq \mathscr{S}_{1,\leq \rho}'
  + \mathscr{S}_{1,>\rho}'$, where $\mathscr{S}_{1,\leq,\rho}'$ is as in
  \eqref{eq:rafcar} and
  \begin{equation}\label{eq:kitty}
  \mathscr{S}_{1,>\rho}' =
 \sum_{\mathcal{L} \subset \mathbf{k}} \mathscr{L}^{-|\mathcal{L}| / 2} \sum_{\mathbf{l} \subset \mathbf{k}}
  \mathop{\sum_{
      (\vec{p}, \vec{\sigma}) \in \mathscr{C}_0(k,\mathcal{L},\mathbf{l})}}_{(*)}
    \prod_{i\not\in \mathbf{l}} \frac{1}{p_i} \prod_{[i]\in \Pi,
      [i]\not\subset \mathbf{k}\setminus \mathbf{l}} \frac{1}{p_{[i]}},\end{equation}
with notation as in the statement of Prop.~\ref{prop:mainbound},
and with condition (*) being that $\mathscr{G}_{(\sim,\vec{\sigma})}$
have $>\rho$ vertices of degree $>2$.
By Prop.~\ref{prop:iolence}
with $\mathbf{G} = \mathscr{G}_{(\sim,\vec{\sigma})}$, we can choose
a set of vertices $\text{blue}$ of 
$\mathscr{G}_{(\sim,\vec{\sigma})}$ such that the number
$s$ defined\footnote{We could have defined $s$ cyclically (that is,
  so as to count $j= 2 k$ if $2k$ is blue and $1$ is red) and thus save
  a term $+1$ here. We did not bother.}
in Prop.~\ref{prop:rankbound} satisfies
$s> \rho/12 + 1/3$. Then, by Prop.~\ref{prop:rankbound},
the space spanned by the vectors in \eqref{eq:silliestcat}
has dimension $r > (\rho/12+1/3)/\kappa - 1$.

We can clearly replace the sum over $(\vec{p},\vec{\sigma})$ in
\eqref{eq:kitty} by a sum over $(\sim,\vec{\sigma})$, followed
by a sum in which $p_{[i]}$ ranges over $\mathbf{P}$ for each
equivalence class $[i]$ of $\sim$. By Lemma \ref{lem:weepnth},
that sum will be
\[\leq \left(\frac{4 k r \log H}{H_0 \mathscr{L} \log 2}\right)^r
\mathscr{L}^{|\Pi|} 
\leq \left(\frac{4 k^2 \log H}{H_0 \mathscr{L} \log 2}\right)^{
  \frac{\rho}{12 \kappa} - 1}
\mathscr{L}^{k + |\mathcal{L}|/2}.\]
The number of possible choices of $\mathcal{L}$, $\mathbf{l}$,
$\sim$ and $\vec{\sigma}$ is $\leq 8^{2 k}\cdot k^{2 k}$, since
the number of partitions of a set of $2 k$ elements is $\leq k^{2 k}$ for $k>1$.
Hence
\begin{equation}\label{eq:adrob}
  \mathscr{S}_{1,>\rho}'\leq (8 k)^{2 k}
\left(\frac{4 k^2 \log H}{H_0 \mathscr{L} \log 2}\right)^{
  \frac{\rho}{12 \kappa} - 1} \mathscr{L}^k \leq
\frac{(8 k)^{2 k}}{H_0^{\frac{2}{3}
    \left(  \frac{\rho}{12 \kappa} - 1\right)}} \mathscr{L}^k .\end{equation}
Since we must satisfy the condition $((\kappa-1) \rho+2) \leq 2k/\log 2 k$
in the statement of Prop.~\ref{prop:mainbound}, we might as well set
$\rho = (2k/(\log 2 k) - 2)/(\kappa-1)$. Then, since $k\geq 20$,
\[\frac{\rho}{12 \kappa} =
\frac{\frac{2 k}{\log 2 k} - 2}{12 \kappa (\kappa-1)} 
\geq \frac{4}{5} \frac{2 k/\log 2 k}{12 \kappa (\kappa-1)}.
\]
Thus, by our assumptions,
\[\frac{2}{3} \left(\frac{\rho}{12 \kappa} - 1\right) \log H_0
\geq \left(\frac{2}{45} \frac{2 k/\log 2 k}{\kappa (\kappa-1)}
- \frac{2}{3}\right) \log H_0
\geq 2 k \log 2 k - \frac{4}{3} k
\]
and so we see from \eqref{eq:adrob} that
\[  \mathscr{S}_{1,>\rho}'\leq 4^{2 k} e^{\frac{4}{3} k} \mathscr{L}^k.\]
By Prop.~\ref{prop:mainbound},
$\mathscr{S}_{1,\leq \rho}'\leq (6400 (K+2))^k \mathscr{L}^k$, and so we
are done.
\end{proof}

\subsection{Parameter choices. Conclusion}

\begin{proof}[Proof of Main Theorem]
  We will apply Prop.~\ref{prop:maincanc},
  which bounds the eigenvalues of $A|_{\mathscr{X}}$
  in terms of $\mathscr{S}_1$ and $\mathscr{S}_2$.
  By Prop.~\ref{prop:betterreduc},
  \[\mathscr{S}_1 = \mathscr{S}_1' + O\left(
  \frac{e^{2 k \log k}}{H_0^{\kappa/8}} \cdot
  (400 \mathscr{L})^k\right) = \mathscr{S}_1' +
  O(400 \mathscr{L})^k 
  \]
  provided that $\kappa\geq 4 k/\ell$,
  $\kappa \log H_0 \geq 16 k \log k$, and the conditions
  in Prop.~\ref{prop:betterreduc} (namely, $\mathscr{L}\geq e$,
  $k\geq \mathscr{L}$,
  $(\log H)^2\leq H_0\leq H$) hold.
  We already showed in Prop.~\ref{prop:aground} that
       \[\mathscr{S}_1'\leq O(K \mathscr{L})^k, \]
       provided that the assumptions in the statement there also hold:
       $k\geq 20$, $\frac{4 k^2 \log H}{\mathscr{L} \log 2} \leq H_0^{1/3}$,
       $\log H_0\leq 2 k$ and
       $\kappa (\kappa-1) \leq 2 (\log H_0)/45 (\log 2 k)^2$.
       Lastly, by Lemma \ref{lem:conclE2},
       for $s = 2k/\log \mathscr{L}$, and assuming
       $H^{2 k}\leq N$, $k\geq \sqrt{\mathscr{L}}$,
       $H_0\geq e$ 
       and $H_0\geq (2 k)^6 (\log H)^{6/5}/(2 \mathscr{L} \log 2)^{6/5}$
       (the last two of which are implied easily by one of the assumptions we have just made),
       \[\mathscr{S}_2 \leq
\frac{(12 k)^{2 k} (4 k + 2)}{
        \left(\frac{2 \mathscr{L} \log 2}{(2 k)^5 \log H} \cdot
        H_0\right)^{\frac{s-1}{8}}} \mathscr{L}^k
\leq \frac{(14 k)^{2 k}}{H_0^{\frac{s-1}{48}}} \mathscr{L}^k.
       \]
       Assume $\log H_0 \geq 48 (\log \mathscr{L}) (\log 2 k)$;
       then $(s/48) \log H_0 \geq 2k \log 2 k$, and so, by
       $\log H_0\leq 2 k$,
       \[\mathscr{S}_2 \leq 7^{2 k} H_0^{\frac{1}{48}} \mathscr{L}^k
       \leq 7^{2 k} e^{\frac{2k}{48}} \mathscr{L}^k
       \leq 8^{2 k} \mathscr{L}^k.\]

       We now make the assumptions in the statement of Prop.~\ref{prop:maincanc},
       other than the ones that follow from the ones we've already made:
       $\mathscr{L}\geq e$,
       $H_0\geq(\log H+2)^4$, $H^{2 k \max(C_0,12 (K \mathscr{L}+1))}\leq N$
       and $H =  O(1)^{k}$ (for an implied constant of our choosing). Let
       \[\ell =
       \left\lfloor
       \frac{\log H_0}{4 \log \mathscr{L} \log(80 \mathscr{L} k)}
       \right\rfloor
       - 2.\]
       Then, by Prop.~\ref{prop:maincanc},
       there is an $\mathscr{X}\subset \mathbf{N}$ with
    $|\mathbf{N}\setminus \mathscr{X}|\leq
N e^{-(K \log K - K + 1) \mathscr{L}} + N/\sqrt{H_0}$ such that
       every eigenvalue of $A|_{\mathscr{X}}$ has absolute value
    \[\ll \max\Big( \mathscr{S}_1^{\frac{1}{2 k}},
    \mathscr{S}_2^{\frac{1}{2 k}},1,\sqrt{\mathscr{L}}\Big)
    \ll \sqrt{K \mathscr{L}}.\]

    It remains to choose $\kappa$ and $k$, and show that our
    assumptions are consistent. Since the only upper constraint on
    $\kappa$ is $\kappa (\kappa-1) \leq 2 (\log H_0)/45 (\log 2 k)^2$,
    we may we well set
    $\kappa = \lfloor \sqrt{(2/45) \log H_0}/\log 2 k\rfloor$.
    Since we must satisfy $k\gg \log H$, we set $k = c_0 \log H$
    for $c_0$ a constant small enough that
    $2 c_0 C_0 \leq \mathscr{L}$ and
    $24 c_0 (1+1/\mathscr{L}) \leq 1$, so that
    $H^{2 k \max(C_0,12 (K \mathscr{L}+1))}\leq N$ holds.

    We must make sure that $\kappa \geq 4k/\ell$.
    (The condition $\kappa \log H_0 \geq 16 k \log k$ will follow
    immediately, since $\ell< (\log H_0)/4 \log k$.)
    In other words, we require that
    $4 c_0 \log H\leq \kappa \ell$. Since $c_0\leq 1/24$, it is enough
    to require
    \[\log H \leq 6
    \left\lfloor \frac{\sqrt{\frac{2}{45} \log H_0}}{\log 2 k}\right\rfloor
    \left(\left\lfloor \frac{\log H_0}{4 \log \mathscr{L} \log 80 \mathscr{L}
      k}\right\rfloor - 2\right).\]
    Since $k = c_0 \log H$ and $\mathscr{L}\leq \sum_{p\leq H} 1/p \leq
    \log \log H + O(1)$, it is clearly enough to require that
    $(\log H) (\log \log H)^3 \leq (\log H_0)^{3/2}$, at least if we
    assume, as we may, that $H$ is larger than a constant.

    The other conditions now follow: $k\geq \mathscr{L}$,
    $H_0\geq (\log H)^2$, $k\geq 20$,
$(4 k^2 \log H)/(\mathscr{L} \log 2) \leq H_0^{1/3}$,
    $\log H_0 \leq 2 k$,  $\log H_0 \geq 48 (\log \mathscr{L}) (\log 2 k)$,
    $H_0\geq (\log H + 2)^4$ all hold for
    $H$ larger than a constant, given our chosen value for $k$ and our lower bound on $\log H_0$.

    We thus obtain the statement of the main theorem, with a bound
\[|\mathbf{N}\setminus \mathscr{X}|\ll
    N e^{-(K \log K - K + 1) \mathscr{L}} + \frac{N}{\sqrt{H_0}}.\]
    To obtain the bound
    $|\mathbf{N}\setminus \mathscr{X}|\ll
    N e^{-K \mathscr{L} \log K} + N/\sqrt{H_0}$,
    we simply replace $K$ by $e K$, and note that
    $e K \log e K - e K + 1 = e K \log K +1 > K \log K$.
    The implied constant in \eqref{eq:mainthm} is then multiplied by
    $\sqrt{e}$, and nothing else changes.

\end{proof}

{\bf Remark.} It may seem curious that we never used the assumption
that the walks we are counting are {\em closed} walks.
That assumption could have saved us a factor of no more than $H$
(meaning a factor of $H^{1/2k} = \sqrt{e}$ in the final result),
and would in fact have left part of the main term unaffected:
trivial walks (that is, walks whose reduction is the walk of length $0$)
are {\em ipso facto} closed, and so are walks where each prime $p$ appearing
as an edge appears exactly twice, once as $p$ and once as $-p$.


\section{Consequences}\label{sec:conseq}

\subsection{Immediate corollaries}
Corollaries \ref{cor:maic} and \ref{cor:kukuruc}
follows easily from the main theorem.

\begin{proof}[Proof of Corollary \ref{cor:maic}]
  Clearly, we may assume that $C\geq 1$. We may also assume that
  $\log H \leq \sqrt{(\log N)/16 C \mathscr{L}}$, since
  the contribution of the set $\mathbf{P}'$ of primes $p$ between
  $\sqrt{(\log N)/16 C \mathscr{L}}$ and $\sqrt{(\log N)/\mathscr{L}}$
  is easy to bound: by Cauchy-Schwarz,
  the second sum on the left of \eqref{eq:corncob} is bounded
  by
  \[
  \sum_{p\in \mathbf{P}'} \frac{1}{p}
  \sum_{\sigma = \pm 1} \sum_{n\in \mathbf{N}} |f(n) \overline{g(n+\sigma p)}| \leq
  \sum_{p\in\mathbf{P}'}
  \frac{1}{p}
  \cdot 2 N |f|_2 |g|_2 \ll_C \frac{N}{\sqrt{\log N}} |f|_2 |g|_2
  \leq \frac{N}{\sqrt{\log N}} |f|_2 |g|_2.\]
  whereas the first sum on the left of \eqref{eq:corncob}
  is bounded by
  \[\sum_{n\in \mathbf{N}} \sum_{\sigma=\pm 1} \sum_{p\in \mathbf{P}'} \frac{
    |f(n)|^2 + |g(n+\sigma p)|^2}{2} \leq
  \sum_{n\in \mathbf{N}} |f(n)|^2 \omega_{\mathbf{P}'}(n) +
  \sum_{n\in \mathbf{N}} |g(n)|^2 \omega_{\mathbf{P}'}(n),\]
  and, again by Cauchy-Schwarz, for $h=f$ and $h=g$,
  \[\sum_{n\in \mathbf{N}} |h(n)|^2 \omega_{\mathbf{P}'}(n) \leq
  \sqrt{\sum_{n\in \mathbf{N}} |h(n)|^4}
  \sqrt{\sum_{p\in \mathbf{P}'} \omega_{\mathbf{P}'}(n)^2}\ll_C
  \sqrt{N} |h|_4^2 \cdot \frac{\sqrt{N}}{\sqrt[4]{\log N}}
\leq \frac{e^{2 C \mathscr{L}} N}{\sqrt[4]{\log N}} \ll_C N.\]
  
  Thus, we may take $K = 16 C$. We may assume that
  $(\log H_0)/\mathscr{L} \geq 40 C$, as a moment's thought suffices to show
  that otherwise $H_0$ and $H$ have to be bounded in terms of $C$, and
  then what we need to prove would be trivial.
  
  Apply the main theorem.
  We obtain a subset $\mathscr{X}\subset \mathbf{N}$ such that
  \[|\mathbf{N}\setminus \mathscr{X}|\leq
  N e^{-K \mathscr{L} \log K} + \frac{N}{\sqrt{H_0}} \leq N e^{-16 C \mathscr{L} \log 16 C}
  + N e^{-20 C \mathscr{L}} \leq 2 N e^{-20 C \mathscr{L}}\]
  and 
  $|\langle f, A|_{\mathscr{X}} f\rangle| =
O(\sqrt{K \mathscr{L}}) = O(\sqrt{C \mathscr{L}})$, or, in other words,
  \[\frac{1}{N\mathscr{L}} \left|\sum_{n\in \mathscr{X}} \sum_{\sigma=\pm 1}
  \mathop{\sum_{p\in \mathbf{P},\; p|n}}_{n+\sigma p\in \mathscr{X}} f(n)
  \overline{g(n+\sigma p)} -
  \sum_{n\in \mathscr{X}} \sum_{\sigma=\pm 1}
  \mathop{\sum_{p\in \mathbf{P}}}_{n+\sigma p\in \mathscr{X}}
  \frac{f(n) \overline{g(n+\sigma p)}}{p}\right| =
      O\left(\sqrt{\frac{C}{\mathscr{L}}}\right).\]

      By Cauchy-Schwarz,
\begin{equation}\label{eq:guern}\begin{aligned}
\left|\mathop{\sum_{n\in\mathbf{N}}\; \sum_{\sigma = \pm 1}
  \sum_{p\in \mathbf{P}}}_{(n\not\in \mathscr{X}\vee
  n + \sigma p\not\in \mathscr{X}) \wedge n+ \sigma p\in \mathbf{N}}
\frac{f(n) \overline{g(n + \sigma p)}}{p}\right| &\leq
2 \sum_{p\in \mathbf{P}} \frac{1}{p} \cdot  \sqrt{N} |f|_2 \sqrt{
  \sum_{n\in \mathbf{N}\setminus \mathscr{X}} |g(n)|^2} \\
&+ 2 \sum_{p\in \mathbf{P}} \frac{1}{p} \cdot \sqrt{N} |g|_2 \sqrt{
  \sum_{n\in \mathbf{N}\setminus \mathscr{X}} |f(n)|^2}.\end{aligned}
  \end{equation}
Recall that $|f|_2, |g|_2\leq 1$. By Cauchy-Schwarz once more,
\[  \Big(\sum_{n\in \mathbf{N}\setminus \mathscr{X}} |f(n)|^2\Big)^2 \leq
|\mathbf{N}\setminus \mathscr{X}| \cdot
\sum_{n\in \mathbf{N}\setminus \mathscr{X}} |f(n)|^4 \]
and similarly for $g$ in place of $f$. Hence, the left side
of \eqref{eq:guern} is $\leq 2 |\mathbf{N}\setminus \mathscr{X}|^{1/4} N^{3/4} (|f|_4 + |g|_4)
\leq 4 e^{- 4 C \mathscr{L}} N \ll N/\sqrt{\mathscr{L}}$.

Again by Cauchy-Schwarz (twice),
\begin{equation}\label{eq:badcat}\Big|\mathop{\mathop{\sum_{n}
      \sum_{\sigma = \pm 1}
      \sum_{p\in \mathbf{P},\, p|n}}_{n, n+\sigma p\in \mathbf{N}}}_{n\not\in \mathscr{X}\vee
    n+\sigma p\not\in \mathscr{X}} f(n) \overline{f(n+\sigma p)}\Big|\end{equation}
  is bounded by
\[\begin{aligned}
  &\leq 2 \sum_{p\in \mathbf{P}}
  \Big(\sqrt{\mathop{\sum_{n\in \mathbf{N}\setminus \mathscr{X}}}_{p|n} |f(n)|^2}
  \sqrt{\mathop{\sum_{n\in \mathbf{N}}}_{p|n} |g(n)|^2} +
  \sqrt{\mathop{\sum_{n\in \mathbf{N}\setminus \mathscr{X}}}_{p|n} |g(n)|^2}
  \sqrt{\mathop{\sum_{n\in \mathbf{N}}}_{p|n} |f(n)|^2}\Big)\\
  &\leq 2\sqrt{\sum_{p\in \mathbf{P}}
    \mathop{\sum_{n\in \mathbf{N}\setminus \mathscr{X}}}_{p|n} |f(n)|^2}
  \sqrt{\sum_{p\in \mathbf{P}}
    \mathop{\sum_{n\in \mathbf{N}}}_{p|n} |g(n)|^2} +
   2\sqrt{\sum_{p\in \mathbf{P}}
    \mathop{\sum_{n\in \mathbf{N}\setminus \mathscr{X}}}_{p|n} |g(n)|^2}
  \sqrt{\sum_{p\in \mathbf{P}}
    \mathop{\sum_{n\in \mathbf{N}}}_{p|n} |f(n)|^2}.\end{aligned}\]
Yet again by Cauchy-Schwarz,
\[\sum_{p\in \mathbf{P}}
    \mathop{\sum_{n\in \mathbf{N}}}_{p|n} |f(n)|^2
    \leq \sqrt{\sum_{n\in \mathbf{N}} \omega_{\mathbf{P}}(n)^2 \cdot
      \sum_{n\in \mathbf{N}} |f(n)|^4} \ll \mathscr{L} e^{2 C \mathscr{L}} N,\]
    \[\sum_{p\in \mathbf{P}}
    \mathop{\sum_{n\in \mathbf{N}\setminus \mathscr{X}}}_{p|n} |f(n)|^2
    \leq \sqrt{\sum_{n\in \mathbf{N}\setminus \mathscr{X}} \omega_{\mathbf{P}}(n)^2
      \cdot
      \sum_{n\in \mathbf{N}} |f(n)|^4} \ll e^{2 C \mathscr{L}} \sqrt{N}
    \cdot
    \sqrt{\sum_{n\in \mathbf{N}\setminus \mathscr{X}} \omega_{\mathbf{P}}(n)^2},\]
      and similarly for $g$ in place of $f$. One last time by
      Cauchy-Schwarz,
      \[\sum_{n\in \mathbf{N}\setminus \mathscr{X}} \omega_{\mathbf{P}}(n)^2 \leq
      \sqrt{|\mathbf{N}\setminus \mathscr{X}|\cdot
        \sum_{n\in \mathbf{N}} \omega_{\mathbf{P}}(n)^4} \ll \mathscr{L}^2
      e^{- 10 C \mathscr{L}}.\]
      Hence, the expression in \eqref{eq:badcat} is
      \[\ll \sqrt{ \mathscr{L} e^{2 C \mathscr{L}} N \cdot
        e^{2 C \mathscr{L}} \mathscr{L} e^{- 5 C \mathscr{L}} N}
      \ll \mathscr{L} e^{-C \mathscr{L}/2} N  \ll N/\mathscr{L}.\]
\end{proof}

\begin{proof}[Proof of Corollary \ref{cor:kukuruc}]
  We may assume that $C\geq 1$. Much as in the proof of Cor.~\ref{cor:maic},
  we may assume that $\log H \leq \sqrt{(\log N)/8 C \mathscr{L}}$
  and $(\log H_0)/\mathscr{L} \geq 20 C$.
  Apply the main theorem with $K = 8 C$.
  We obtain a subset $\mathscr{X}\subset \mathbf{N}$ such that
  \[|\mathbf{N}\setminus \mathscr{X}| \leq
  N e^{-K \mathscr{L} \log K} + \frac{N}{\sqrt{H_0}} \leq
  N e^{-8 C \mathscr{L} \log 8 C} + N e^{-10 C \mathscr{L}}\leq 2 N e^{-10 C \mathscr{L}}\]
 and
  $|A|_{\mathscr{X}} f|_2^2 \ll \mathscr{L}$, or, in other words,
  \[\frac{1}{N} \sum_{n\in \mathscr{X}}\left|
  \sum_{p \in \mathbf{P},\, p|n} \mathop{\sum_{\sigma=\pm 1}}_{n+\sigma p\in \mathscr{X}}
  f(n + \sigma p) - \sum_{p \in \mathbf{P}} \mathop{\sum_{\sigma=\pm 1}}_{n+\sigma p\in \mathscr{X}}
  \frac{f(n + \sigma p)}{p}\right|^2 = O\left( \mathscr{L}\right),\]
  We proceed essentially as in the proof of Cor.~\ref{cor:maic}.
  By Cauchy-Schwarz, for $\sigma = \pm 1$,
  \[\begin{aligned}\sum_{n\in \mathbf{N}\setminus \mathscr{X}}
  &\Big|\mathop{\sum_{p\in \mathbf{P}}}_{n+ \sigma p\in \mathbf{N}}\!\!
  \frac{f(n+\sigma p)}{p}\Big|^2\leq \sum_{p\in \mathbf{P}} \frac{1}{p}\cdot
\sum_{p\in \mathbf{P}} 
  \mathop{\sum_{n\in \mathbf{N}\setminus \mathscr{X}}}_{n+\sigma p \in \mathbf{N}}
  \frac{|f(n+\sigma p)|^2}{p}\\
  &\leq \left(\sum_{p\in \mathbf{P}}
  \frac{1}{p}\right)^2 \sqrt{|\mathbf{N}\setminus \mathscr{X}| \cdot
    \sum_{n\in \mathbf{N}} |f(n)|^4}
\leq \mathscr{L}^2 e^{-5 C \mathscr{L}} e^{2 C \mathscr{L}} \ll 1.\end{aligned}\]
By H\"{o}lder,
  \[\begin{aligned}\sum_{n\in \mathbf{N}\setminus \mathscr{X}}
 \Big|\mathop{\sum_{p\in \mathbf{P},\, p|n}}_{n+ \sigma p\in \mathbf{N}}
 f(n+\sigma p)\Big|^2\leq |\mathbf{N}\setminus \mathscr{X}|^{1/3}
 \Big(\sum_{n\in \mathbf{N}}
 \Big|\mathop{\sum_{p\in \mathbf{P},\, p|n}}_{n+ \sigma p\in \mathbf{N}}
 f(n+\sigma p)\Big|^3\Big)^{2/3},\end{aligned}\]
  and again by H\"{o}lder (twice),
  \[\begin{aligned}
  \sum_{n\in \mathbf{N}}
 &\Big|\mathop{\sum_{p\in \mathbf{P},\, p|n}}_{n+ \sigma p\in \mathbf{N}}
 f(n+\sigma p)\Big|^3 \leq
 \sum_{n\in \mathbf{N}} \omega_\mathbf{P}(n)^2
 \mathop{\sum_{p\in \mathbf{P},\, p|n}}_{n+ \sigma p\in \mathbf{N}}
 |f(n+\sigma p)|^3
\\ &\leq \sum_{n\in \mathbf{N}} |f(n)|^3
\sum_{p\in \mathbf{P},\, p|n}\!\!\!\omega_\mathbf{P}(n-\sigma p)^2
\leq \Big(\sum_{n\in \mathbf{N}} |f(n)|^4\Big)^{\frac{3}{4}}
\Big(\sum_{n\in \mathbf{N}}
\Big(\sum_{p\in \mathbf{P},\, p|n} \!\!\!\omega_\mathbf{P}(n-\sigma p)^2\Big)^4
\Big)^{\frac{1}{4}}\\
&\leq N^{3/4} |f|_4^3
\Big(\sum_{p_1,\dotsc,p_4\in \mathbf{P}}
\mathop{\sum_{n\in \mathbf{N}}}_{p_1|n,\dotsc,p_4|n} \omega_\mathbf{P}(n-\sigma p_1)^2
\dotsb \omega_\mathbf{P}(n-\sigma p_4)^2\Big)^{1/4}
\ll N e^{3 C \mathscr{L}} \mathscr{L}^{O(1)}.
  \end{aligned}\] 
 Hence
 $\left|(A(f|_{\mathscr{X}}))|_{\mathbf{N}\setminus \mathscr{X}}\right|_2^2\ll \mathscr{L} + e^{-\frac{10 C}{3} \mathscr{L}}
 e^{2 C \mathscr{L}} \mathscr{L}^{O(1)} \ll \mathscr{L}$.

 We must still bound $|A (f|_{\mathbf{N}\setminus \mathscr{X}})|_2^2$.
 By Cauchy-Schwarz,
  \[\begin{aligned}\sum_{n\in \mathbf{N}}\;
  \Big|\!\!\!\!\!
  \mathop{\sum_{p\in \mathbf{P}}}_{n+ \sigma p\in \mathbf{N}\setminus \mathscr{X}}
  \frac{f(n+\sigma p)}{p}\Big|^2&\leq \sum_{p\in \mathbf{P}} \frac{1}{p}\cdot
\sum_{p\in \mathbf{P}}\!\!\!
\mathop{\sum_{n\in \mathbf{N}}}_{n+\sigma p \in  \mathbf{N}\setminus \mathscr{X}}
\frac{|f(n+\sigma p)|^2}{p}
\leq \Big(\sum_{p\in \mathbf{P}} \frac{1}{p}\Big)^2 \sum_{n\in \mathbf{N}
  \setminus \mathscr{X}} |f(n)|^2\\
&\leq \mathscr{L}^2 \sqrt{|\mathbf{N}\setminus \mathscr{X}|\cdot 
    \sum_{n\in \mathbf{N}} |f(n)|^4}
\leq \mathscr{L}^2 e^{- 5 C \mathscr{L}} e^{2 C \mathscr{L}} \ll 1
  \end{aligned}\]
  and, by Cauchy-Schwarz and H\"older,
  \[\begin{aligned}&\sum_{n\in \mathbf{N}}
  \Big|\!\!\mathop{\sum_{p\in \mathbf{P},\, p|n}}_{n+ \sigma p\in \mathbf{N}
\setminus \mathscr{X}}
  f(n+\sigma p)\Big|^2\leq
  \sum_{n\in \mathbf{N}} \omega_{\mathbf{P}}(n)
  \mathop{\sum_{p\in \mathbf{P},\, p|n}}_{n+ \sigma p\in \mathbf{N}
\setminus \mathscr{X}}
   |f(n+\sigma p)|^2\\
   &\leq
   \sum_{n\in \mathbf{N}\setminus \mathscr{X}} |f(n)|^2 
   \sum_{p\in \mathbf{P},\, p|n} \omega_{\mathbf{P}}(n-\sigma p)
   \leq \sqrt{ \sum_{n\in \mathbf{N}} |f(n)|^4\cdot
   \sum_{n\in \mathbf{N}\setminus \mathscr{X}} 
   \Big(\sum_{p\in \mathbf{P},\, p|n} \omega_{\mathbf{P}}(n-\sigma p)\Big)^2}\\
   &\leq N^{1/2} |f|_4^2 \cdot |\mathbf{N}\setminus \mathscr{X}|^{1/3}\cdot
   \Big(\sum_{n\in \mathbf{N}}
   \Big(\sum_{p\in \mathbf{P},\, p|n} \omega_{\mathbf{P}}(n-\sigma p)\Big)^6
   \Big)^{1/6}
   \ll N e^{2 C \mathscr{L}} e^{-\frac{10 C}{3} \mathscr{L}}
   \mathscr{L}^{O(1)} \ll 1.
   \end{aligned}\]
\end{proof}

 \begin{proof}[Proof of Corollary \ref{cor:sansH0}]
Let $H_j = \exp((\log H)^{(3/4)^{j-1}})$, and $1\leq j\leq j_0$,
     where $j_0$ is the maximal integer such that
     $\log H_{{j_0}+1} \geq (\log H_{j_0})^{2/3} (\log \log H_{j_0})^2$
     and $H_{j_0+1} \geq e^{2\sqrt{K \mathscr{L}} \log K}$. (Set $j_0=0$ if there
     is no such integer.)
     It is clear that
     $j_0\ll \log \log \log H$.
     (It is here that we use the condition $H\geq 16 > \exp(\exp(1))$.)
     
  For $1\leq j\leq j_0$, we write $A_j$ for the operator defined as in 
  \eqref{eq:alterno}, only with
  $\mathbf{P}_j = \mathbf{P}\cap (H_{j+1},H_j]$ instead
  of $\mathbf{P}$.
  Let $\mathscr{L}_j = \sum_{p\in \mathbf{P}_j} 1/p$.
  If $\mathscr{L}_j\geq e$, we apply the main theorem
  with $\mathscr{L}_j$ instead of $\mathscr{L}$, $[H_{j+1},H_j]$ instead of
  $[H_0,H]$  and $K_j = K \mathscr{L}/\mathscr{L}_j\geq K$ instead of $K$.
    We obtain that the
  $L^2\to L^2$ norm of $A_j|_{\mathscr{X}_j}$ is $\ll \sqrt{K_j \mathscr{L}_j}
  \ll \sqrt{K \mathscr{L}}$, where $\mathscr{X}_j\subset \mathbf{N}$
  satisfies \[\left|\mathbf{N}\setminus \mathscr{X}_j\right|
  \ll \left(e^{-K_j \mathscr{L}_j \log K_j} + 
  \frac{1}{\sqrt{H_{j+1}}}\right) N
  = e^{-K \mathscr{L} \log \frac{K \mathscr{L}}{\mathscr{L}_j}} N +
  \frac{N}{\sqrt{H_{j+1}}}
  .\]
  If $\mathscr{L}_j<e$, we let instead $K_j =
  e^{3/2} \sqrt{K \mathscr{L}}/\mathscr{L}_j$ and use
  the trivial bound $\ll K_j \mathscr{L}_j \ll \sqrt{K \mathscr{L}}$
  on the $L^2\to L^2$ norm of $A_j|_{\mathscr{X}_j}$,
  where $\mathscr{X}_j$ is the set of all
  $n\in \mathbf{N}$ with $\leq K_j \mathscr{L}_j$ divisors in
  $\mathbf{P}_j$. (If $\mathscr{L}_j=0$, we simply omit the index $j$ from
  consideration.)
  Then, by Lemma \ref{lem:X0issmall},
  \begin{equation}\label{eq:boxo}\begin{aligned}
      \left|\mathbf{N}\setminus \mathscr{X}_j\right|
    &\ll e^{-(K_j \log K_j - K_j + 1) \mathscr{L}_j} N
    \leq e^{- K_j \mathscr{L}_j \log \frac{\sqrt{e K \mathscr{L}}}{\mathscr{L}_j}} N\\
    &\leq e^{- e^{3/2} \sqrt{K \mathscr{L}} \log \sqrt{\frac{K \mathscr{L}}{\mathscr{L}_j}}} N
     \leq e^{- \sqrt{K \mathscr{L}} \log \frac{K \mathscr{L}}{\mathscr{L}_j}} N
     ,\end{aligned}\end{equation}
  since, for $\mathscr{L}_j<e$, we have
$\sqrt{e K \mathscr{L}}/\mathscr{L}_j >
  \sqrt{K \mathscr{L}/\mathscr{L}_j}$.
  
  Finally, we define $A_{j_0+1}$ as in 
  \eqref{eq:alterno}, but with
  $\mathbf{P}_{j_0+1} = \mathbf{P}\cap (1,H_{j_0+1}]$ instead of
     $\mathbf{P}$. Let $\mathscr{L}_{j_0+1} = \sum_{p\in \mathbf{P}_{j_0+1}} 1/p$.
     We may assume that $H_{j_0+1}\geq 2$, as otherwise $\mathbf{P}_{j_0+1}$
     is empty and we may omit the case $j=j_0+1$. Then
     \[\mathscr{L}_{j_0+1} \leq \log \log H_{j_0+1} + O(1) \leq
     \frac{4}{3} (\log 2 \sqrt{K \mathscr{L}} + \log \log K) + O(1),\]
     and so we see that there is an absolute constant $C\geq 1$ such that
     \[C \sqrt{K \mathscr{L}} \geq \mathscr{L}_{j_0+1}\;\;\;\;\;\text{and}
     \;\;\;\;\;
     e C \sqrt{K \mathscr{L}}/\mathscr{L}_{j_0+1}\geq K^{1/e}.\] Set
     $K_{j_0+1} = e C \sqrt{K \mathscr{L}}/ \mathscr{L}_{j_0+1}\geq e$.
     (We note that $e\cdot (\log K_{j_0+1} - 1) \geq
     e\cdot \log K^{1/e}=\log K$.)
     Again by the trivial bound, the $L^2\to L^2$ norm of
     $A_{j_0+1}|_{\mathscr{X}_{j_0+1}}$ is
     $\ll K_{j_0+1} \mathscr{L}_{j_0+1}\ll \sqrt{K \mathscr{L}}$,
    and \eqref{eq:boxo} holds for $j=j_0+1$.
We also see that
    \[\left|\mathbf{N}\setminus \mathscr{X}_{j_0+1}\right|
    \ll e^{-(K_{j_0+1} \log K_{j_0+1} - K_{j_0+1} + 1) \mathscr{L}_{j_0+1}} N \leq
    e^{-e C \sqrt{K \mathscr{L}} N
      \log \frac{C \sqrt{K \mathscr{L}}}{\mathscr{L}_{j_0+1}}} N
    \leq e^{-\sqrt{K \mathscr{L}} \log K} N
    .  \]
    
    Let $\mathscr{X} = \bigcup_{1\leq j\leq j_0+1} \mathscr{X}_j$,
     Clearly
     \[A|_\mathscr{X} = \sum_{j=1}^{j_0+1} A_j|_{\mathscr{X}}.\]
     The absolute values of the eigenvalues of $A_j|_{\mathscr{X}}$ are bounded
     by the $L^2\to L^2$ norm of $A_j|_{\mathscr{X}_j}$.
     (This statement is true in general
     for projections $\pi$ and real symmetric (or normal) operators $A$:
     for any $v$ (defined over $\mathbb{R}$ or $\mathbb{C}$)
     $|\langle \pi A \pi v,v\rangle| = |\langle A \pi v, \pi v\rangle|
     \leq |A|_{L^2\to L^2}
     |\pi v|_2^2 \leq |A|_{L^2\to L^2}
     |v|_2^2$, and so, by the spectral theorem,
     $|\pi A \pi|_{L^2\to L^2} \leq |A|_{L^2\to L^2}$.)
     Hence, by the triangle inequality and our bounds
     on $L^2\to L^2$ norms, every eigenvalue of
 $A|_{\mathscr{X}}$ has absolute value at most
     \begin{equation}\label{eq:warf}
       \sum_{j=1}^{j_0+1} O(\sqrt{K\mathscr{L}})
     = O(\sqrt{K \mathscr{L}} \log \log \log H).\end{equation}

     It remains to bound $|\mathbf{N}\setminus \mathscr{X}|$.
     Evidently, $|\mathbf{N}\setminus \mathscr{X}|\leq
     \sum_{j=1}^{j_0+1} |\mathbf{N}\setminus \mathscr{X}_j|$.
     Since the sequence $H_1, H_2,\dotsc$ decreases faster than geometrically,
     it is clear that, if $j_0>0$,
     \begin{equation}\label{eq:benjy}\sum_{j=1}^{j_0} \frac{1}{\sqrt{H_{j+1}}} \ll
     \frac{1}{\sqrt{H_{j_0+1}}} \leq e^{-\sqrt{K \mathscr{L}} \log K}.\end{equation}
     The contribution of the other terms to
     $|\mathbf{N}\setminus \mathscr{X}|$ is at most
    $O(N)$ times
        \[\begin{aligned}
e^{- \sqrt{K \mathscr{L}} \log K} +
\sum_{j=1}^{j_0}
e^{- \sqrt{K \mathscr{L}} \log \frac{K \mathscr{L}}{\mathscr{L}_j}}.
\end{aligned}\]
Since
\[\sum_{j=1}^{j_0}
e^{- \sqrt{K \mathscr{L}} \log \frac{\mathscr{L}}{\mathscr{L}_j} 
} \leq
\sum_{j=1}^{j_0}
e^{- \log \frac{\mathscr{L}}{\mathscr{L}_j}} =
\sum_{j=1}^{j_0} \frac{\mathscr{L}_j}{\mathscr{L}} = 1,\]
we conclude that
$|\mathbf{N}\setminus \mathscr{X}|\ll e^{-\sqrt{K \mathscr{L}} \log K}$.

 \end{proof}
 {\bf Remark.}
 It is possible to remove the unpleasant factor of $\log \log \log H$
 from the bound \eqref{eq:slorse} under some circumstances.
 If $\mathscr{L}\gg \log \log H$, or if we aim at a bound of
 $O(\sqrt{K \log \log H})$, we can use the fact that, in the proof
 of \eqref{cor:sansH0}, $\mathscr{L}_j\ll (3/4)^j \log \log H$, and so
 we set $K_j = c K ((3/4)^j \log \log H)/\mathscr{L}_j$ for
 $\mathscr{L}_j\geq e$ and
 $K_j = c \sqrt{K ((3/4)^j \log \log H}/\mathscr{L}_j$ for
 $\mathscr{L}_j<e$. Then the left side of \eqref{eq:warf}
 is replaced by a geometric series, and thus the factor of
 $\log \log \log H$ disappears. On the other hand, the upper bound
 on $|\mathbf{N}\setminus \mathscr{X}|$ increases.

 As it happens, we do not use Cor.~\ref{cor:sansH0} in what follows; the
 condition on $H_0$ and $H$ in the main theorem or Cor.~\ref{cor:maic}
 will turn out not to be unduly restrictive.
    \subsection{Sums without divisibility conditions}\label{subs:divcond}

     We shall now see how to 
     bound sums such as the second double sum in
     the statement (\ref{eq:corncob}) of Corollary \ref{cor:maic}.
We will follow closely \cite[\S 5.5]{HelfUbis},
which is essentially an exposition of \cite{MR3569059}.
As a first example, we will work out the double sum we need so
as to prove Corollaries \ref{cor:newlogchowla} and
\ref{cor:newalmostchowla}; that task
was already undertaken by Tao in \cite{MR3569059}
(Lemma 2.5, Prop.~2.6 and Lemmas 3.4--3.5), but we have to carry it
out ourselves, as we need actual bounds.

\begin{lemma}\label{lem:circmet}
  Let $\mathbf{Q}$ be a set of integers $1\leq q\leq H$.
  Let $\{v_h\}_{1\leq h\leq H}$, $v_h\in \mathbb{C}$
and $\{w_h\}_{1\leq h\leq 2 H}$, $w_h\in \mathbb{C}$ be given.
  Write $\mathscr{Q} = \sum_{q\in \mathbf{Q}} 1/q$.
  Then, for any $\epsilon>0$,
  \[\sum_{q\in \mathbf{Q}} \sum_{h=1}^H \frac{v_h \overline{w_{h+q}}}{q} \ll
  \left(\epsilon |v|_2 |w|_2 + |w|_1 \int_{\mathfrak{M}_\epsilon}
  \left|\sum_{h=1}^H v_h e(h\alpha)\right|\, d\alpha
  \right)\cdot \mathscr{Q} 
  ,\]
  where $\mathfrak{M}_{\epsilon} = \left\{\alpha\in [0,1]: |\sum_{q\in \mathbf{Q}}
  e(q \alpha)/q|> \epsilon \mathscr{Q}\right\}$.
\end{lemma}
Here we write
$|v|_2^2 = \sum_{h=1}^H |v_h|^2$,
$|w|_2^2 = \sum_{h=1}^{2 H} |w_h|^2$.
\begin{proof}
  The first step is typical for the circle method:
  \[\begin{aligned}
  \sum_{q\in \mathbf{Q}} \sum_{h=1}^H \frac{v_{h} \overline{w_{h+q}}}{q}
  &= \sum_{m=1}^{2 H} \sum_{q\in \mathbf{Q}} \sum_{h=1}^H
    \frac{v_h \overline{w_m}}{q}
    \int_0^1 e((m-h-q) \alpha) d\alpha\\
    &= \int_0^1 \overline{W_{2 H}(\alpha)}
    V_H(\alpha) Q(\alpha) d\alpha,
  \end{aligned}\]
  where $V_M(\alpha) =\sum_{m=1}^M v_m e(m \alpha)$,
  $W_M(\alpha) =\sum_{m=1}^M w_m e(m \alpha)$
  and
  $Q(\alpha) = \sum_{q\in \mathbf{Q}} e(q \alpha)/q$. Now we cut the integral into ``major arcs'' $\mathfrak{M}_\epsilon$
  and ``minor arcs''
  $\mathfrak{m}_\epsilon = [0,1)\setminus \mathfrak{M}_\epsilon$.
  Clearly
    \[\begin{aligned}
    \left|\int_{\mathfrak{M}_\epsilon}
    \overline{W_{2 H}(\alpha)} V_H(\alpha) Q(\alpha) d\alpha\right|
    &\leq
(\max_\alpha |V_{H}(\alpha)| |Q(\alpha)|)\cdot
  \int_{\mathfrak{M}_\epsilon} |W_{2 H}(\alpha)|\, d\alpha \\ &\leq
  |v|_1 \mathscr{Q}  \int_{\mathfrak{M}_\epsilon} |W_{2 H}(\alpha)|\, d\alpha,
\end{aligned}\]
whereas
\[\begin{aligned}
\left|\int_{\mathfrak{m}_\epsilon}
    \overline{W_{2 H}(\alpha)} V_H(\alpha) Q(\alpha) d\alpha\right|
    &\leq \epsilon \mathscr{Q} \int_0^1 |W_{2 H}(\alpha)| |V_H(\alpha)| d\alpha\\
    &\leq
    \epsilon \mathscr{Q}
    \sqrt{\int_0^1 |W_{2 H}(\alpha)|^2 d\alpha \cdot
      \int_0^1 |V_H(\alpha)|^2 d\alpha} \leq
    \epsilon \mathscr{Q} |v|_2 |w|_2.\end{aligned}\]
\end{proof}

\begin{lemma}\label{lem:notrestr}
    Let $\mathbf{Q}$, $\mathscr{Q}$ and $\mathfrak{M}_\epsilon$ be as in Lemma
    \ref{lem:circmet}.
    Assume $\mathbf{Q}$ is a set of primes contained in $[H/2,H]$
    with $H\geq 4$ and $\mathscr{Q} \geq \delta/\log H$, $\delta>0$.
Then, for any $\epsilon>0$,
\[\left|\mathfrak{M}_\epsilon\right| \ll \frac{1}{(\epsilon \delta)^4 H}.\]
\end{lemma}
The proof follows a suggestion of B. Green's given in a footnote to \cite[Lemma~3.5]{MR3569059}.
\begin{proof}
  Let $Q(\alpha) = \sum_{p\in \mathbf{Q}} e(p\alpha)/p$. Then
  \[\begin{aligned}
  \int_0^1 \left|Q(\alpha)\right|^4 d\alpha &=
  \int_0^1 \left|Q(\alpha)^2\right|^2 d\alpha \\ &= \int_0^1
  \left|\sum_{|m|\leq H} \left(\mathop{\sum_{p,q\in \mathbf{Q}}}_{q-p=m} \frac{1}{p q}
  \right) e( m \alpha)\right|^2 d\alpha
  = \sum_{|m|\leq H} \left|\mathop{\sum_{p,q\in \mathbf{Q}}}_{q-p=m} \frac{1}{p q}
  \right|^2
 .\end{aligned}\]
  For $m = 0$, the inner sum is $\ll 1/H \log H$ by the prime number theorem.
  For $m\ne 0$, an upper-bound sieve, such as
  Brun's sieve (vd., e.g., \cite[Thm.~6.9]{MR2647984} or
  \cite[Thm.~6.2.5]{zbMATH02239783}) yields
  \[\mathop{\sum_{p,q\in \mathbf{Q}}}_{q-p=m} \frac{1}{p q}\ll
  \frac{1}{H (\log H)^2} \prod_{p|m} \left(1 + \frac{1}{p}\right).\]
  It is easy to show that $\sum_{m=1}^H
  \prod_{p|m} \left(1 + \frac{1}{p}\right)^2 \ll H$. Hence
  \[\int_0^1 \left|Q(\alpha)\right|^4 d\alpha \ll
  \frac{1}{H^2 (\log H)^2} + \frac{1}{H (\log H)^4} \ll
  \frac{1}{H (\log H)^4},\]
  from where we obtain our bound on $|\mathfrak{M}_{\epsilon}|$.
  \end{proof}

\begin{prop}\label{prop:ohnhum}
  Let $\mathbf{N} = \{N+1,N+2,\dotsc 2 N\}$.
  Let $\mathbf{Q}$ be a non-empty set of primes in $[H/2, H]$, $H\geq 4$.
  Write $\mathscr{Q} = \sum_{p\in \mathbf{Q}} 1/p$, $\delta = \mathscr{Q} \log H$.
  Let $f_1,f_2:\mathbb{Z}\to \mathbb{C}$ with $|f_1|_\infty, |f_2|_\infty \leq 1$.
    
  Then, for any $\epsilon >0$,
  \[\begin{aligned}
\frac{1}{\mathscr{Q} N}
  \sum_{p\in \mathbf{Q}} \sum_{n\in \mathbf{N}} \frac{f_1(n) \overline{f_2(n+p)}}{p}
  &\ll 
   \frac{1}{(\epsilon \delta)^4} \max_{\alpha\in
    [0,1]}
  \Big( \frac{1}{N H}
  \int_N^{2 N} \left|\sum_{x<m\leq x + 2 H} f_2(m) e(m \alpha)\right| dx
  + \frac{H}{N}\Big)
  \\&+ \epsilon |f_1|_2 |f_2|_2 + H/N.\end{aligned}\]
\end{prop}
Recall we define the $\ell^2$ norm of a function $f$ on
$\mathbf{N}$ by $|f|_2^2 = (1/N) \sum_{n\in \mathbf{N}} |f(n)|^2$.
\begin{proof}
  If we displace $\mathbf{N}$ a little, the sum on $n\in \mathbf{N}$
  barely changes:
  \[\sum_{n\in \mathbf{N}} f_1(n) \overline{f_2(n+p)} = O(h) +
  \sum_{n\in \mathbf{N}+h} f_1(n) \overline{f_2(n+p)} = O(h) +
  \sum_{n\in \mathbf{N}} f_1(n+h) f_2(n+h+p).\]
  Summing over all $1\leq h\leq H$, we obtain
  \[\sum_{n\in \mathbf{N}} f_1(n) \overline{f_2(n+p)} = O(H) + \frac{1}{H}
  \sum_{h=1}^H \sum_{n\in \mathbf{N}}  f_1(n+h) \overline{f_2(n+h+p)},\]
  and so, summing over $p\in \mathbf{Q}$, we see that
  \[\sum_{p\in \mathbf{Q}}
  \sum_{n\in \mathbf{N}} \frac{f_1(n) \overline{f_2(n+p)}}{p}
  = O(\mathscr{Q} H) + \frac{1}{H}
  \sum_{n\in \mathbf{N}} \sum_{p\in \mathbf{Q}} \frac{1}{p}
  \sum_{h=1}^H f_1(n+h) \overline{f_2(n+h+p)}.\]
At the cost of another term $O(\mathscr{Q} H)$, we
will assume that $f_1$ and $f_2$ are supported on
$\mathbf{N}$.

  We apply Lemma \ref{lem:circmet} with $v_h = f_1(n+h)$ and
  $w_h =f_2(n+h)$ for $n\in \mathbf{N}$. We obtain that
\[\begin{aligned}
  \sum_{n\in \mathbf{N}} \sum_{p\in \mathbf{Q}} \frac{1}{p}
  \sum_{h=1}^H f_1(n+h) &\overline{f_2(n+h+p)}
  \ll 
  \epsilon \mathscr{Q} \sum_{n\in \mathbf{N}} 
 \sqrt{\sum_{h=1}^H |f_1(n+h)|^2}\sqrt{\sum_{h=1}^{2 H} |f_2(n+h)|^2}\\
 &+ \mathscr{Q} \sum_{n\in \mathbf{N}}  \left(\sum_{h=1}^{H} |f_1(n+h)|\right)  \int_{\mathfrak{M}_\epsilon}
  \left|\sum_{h=1}^{2 H} f_2(n+h) e(h\alpha)\right|\, d\alpha
  .\end{aligned}\]
By Cauchy-Schwarz,
\[\sum_{n\in \mathbf{N}} 
\sqrt{\sum_{h=1}^H |f_1(n+h)|^2}\sqrt{\sum_{h=1}^{2 H} |f_2(n+h)|^2} \leq
\sqrt{H N |f_1|_2^2 \cdot 2 H N |f_2|_2^2} \ll H N |f_1|_2 |f_2|_2,\]
and, by $|f_1|_\infty\leq 1$,
\[\begin{aligned}
\sum_{n\in \mathbf{N}}  \left(\sum_{h=1}^{H} |f_1(n+h)|\right)  
\left|\sum_{h=1}^{2 H} f_2(n+h) e(h\alpha)\right|\, d\alpha \ll
H 
  \sum_{n\in \mathbf{N}} 
\left|\sum_{h=1}^{2 H} f_2(n+h) e(h\alpha)\right|.\end{aligned}\]
By the bound on $|\mathfrak{M}_\epsilon|$ in Lemma \ref{lem:notrestr},
\[\begin{aligned}
\int_{\mathfrak{M}_\epsilon}
\sum_{n\in \mathbf{N}}  \left|\sum_{h=1}^{2 H} f_2(n+h) e(h\alpha)\right|\, d\alpha
&\leq |\mathfrak{M}_\epsilon| \max_{\alpha \in \mathfrak{M}_\epsilon}
\sum_{n\in \mathbf{N}}  \left|\sum_{h=1}^{2 H} f_2(n+h) e(h\alpha)\right|\\
&\leq \frac{1}{(\epsilon \delta)^4 H} \max_{\alpha \in \mathfrak{M}_\epsilon}
\sum_{n\in \mathbf{N}}  \left|\sum_{h=1}^{2 H} f_2(n+h) e(h\alpha)\right|.
\end{aligned}\]
We conclude that
\[\begin{aligned}\frac{1}{H}
  \sum_{n\in \mathbf{N}} \sum_{p\in \mathbf{Q}} \frac{1}{p}
  \sum_{h=1}^H f_1(n+h) \overline{f_2(n+h+p)} &\ll
  \frac{\mathscr{Q}}{(\epsilon \delta)^4 H} 
   \max_{\alpha \in \mathfrak{M}_\epsilon}
  \sum_{n\in \mathbf{N}} 
     \left|\sum_{h=1}^{2 H} f_2(n+h) e(h\alpha)\right|\\
   &+ \epsilon \mathscr{Q} N |f_1|_2 |f_2|_2 
     .\end{aligned}\]
We add a term $(\mathscr{Q}/(\epsilon \delta)^4 H) H^2$ so as to 
drop the assumption
that $f_2$ is supported on $\mathbf{N}$, and replace the sum over
$n\in \mathbf{N}$ by an integral.
\end{proof}

\begin{corollary}\label{cor:notmarat}
  Let $\mathbf{N} = \{N+1,N+2,\dotsc 2 N\}$.
  Let $\mathbf{P}$ be a set of primes in $[H_0, H]$, $H_0\geq 5$,
  $2H_0\leq H\leq N/\log 2 N$.
  Write $\mathscr{L} = \sum_{p\in \mathbf{P}} 1/p$. Let
  $\lambda$ be the Liouville function. Let $f:\mathbf{N}\to \mathbb{C}$
  satisfy $|f|_\infty\leq 1$.

  Then
  \begin{equation}\label{eq:argodio}
    \sum_{p\in \mathbf{P}} \sum_{n\in \mathbf{N}} \frac{f(n) \lambda(n+p)}{p}  
  \ll \eta^{1/5} |f|_2^{4/5} \mathscr{L}^{1/5} (\log \Delta)^{4/5} N + H \mathscr{L}
  \end{equation}
  for $\eta = (\log \log H_0)/\log H_0 + 1/\log^{1/700} N$ and
  $\Delta = (\log H)/\log H_0$.
\end{corollary}
The exponent $1/700$ here comes from \cite[Thm.~1.3]{MR3435814}; it
can almost certainly be greatly improved.\footnote{Thm.~4.9 in
  \cite{HelfUbis} gives an exponent of $1/45-\epsilon$ instead, but it also
gives a lower power of $\log H_0$.} Note (\ref{eq:argodio}) is better
than trivial only if $\mathscr{L}/\log \Delta > \eta^{1/4}$.
\begin{proof}
  Let $H_j = 2^j H_0$ for $1\leq j\leq j_0 = \log (H/H_0)/\log 2 + 1$,
  $\mathbf{Q}_j = \mathbf{P}\cap [H_j/2,H_j)$,
    $\mathscr{Q}_j = \sum_{p\in \mathbf{Q}_j} 1/p$, $\delta_j
    = \mathscr{Q}_j \log H_j$.

    Let $j$ range over all values in $\{1,2,\dotsc j_0\}$ such that
    $\mathbf{Q}_j\ne \emptyset$.
    By \cite[Thm.~1.3]{MR3435814},
    \begin{equation}\label{eq:mrt}\max_{\alpha\in
      [0,1]} \int_N^{2 N} \left|\sum_{x<m\leq x + 2 H_j} \lambda(m) e(m \alpha)\right| dx
    \ll \left(\frac{\log \log 2 H_j}{\log 2 H_j} + \frac{1}{\log^{1/700} N}\right)
    N H_j
    .\end{equation}
    Applying Prop.~\ref{prop:ohnhum} with $f=\lambda$, $H_j$ instead of $H$
    and $\mathbf{Q}_j$ instead of $\mathbf{Q}$,
    we obtain that
    \begin{equation}\label{eq:eliez}\sum_{p\in \mathbf{Q}_j} \sum_{n\in \mathbf{N}}
    \frac{\lambda(n) \lambda(n+p)}{p}  \ll
    \left(\epsilon_j |f|_2 +  \frac{\eta}{(\epsilon_j \delta_j)^4}
    \right) \frac{\delta_j N}{\log H_j} + H \mathscr{Q}_j\end{equation}
    for $\delta_j = \mathscr{Q}_j \log H_j$, where
    $\mathscr{Q}_j = \sum_{p\in \mathbf{Q}_j} 1/p$. 
    We let $\epsilon_j = \eta^{1/5}/|f_2|^{1/5} \delta_j^{4/5}$, so
    that the right side of (\ref{eq:eliez}) minus $H \mathscr{Q}_j$ becomes
    $2 \eta^{1/5} \delta_j^{1/5} N |f_2|^{4/5}/ \log H_j$.
    By H\"older's inequality,
    \[\sum_{j=1}^{j_0} \frac{\delta_j^{1/5}}{\log H_j} \leq
    \left(\sum_{j=1}^{j_0} \frac{\delta_j}{\log H_j}\right)^{1/5}
    \left(\sum_{j=1}^{j_0} \frac{1}{\log H_j}\right)^{4/5} \ll
    \mathscr{L}^{1/5}\cdot \left(\sum_{j=1}^{j_0} \frac{1}{j+\log H_0}\right)^{4/5}.\]
    Since $H\geq 2 H_0$, it is easy to see that $\sum_{j=1}^{j_0} 1/(j+
    \log H_0) \ll \log \Delta$.
\end{proof}

\subsection{Proofs of Corollaries \ref{cor:cruxio}, \ref{cor:newlogchowla},
  \ref{cor:newalmostchowla} and \ref{cor:almav}}

We obtain Cor.~\ref{cor:cruxio} as a straightforward consequence of
Corollary \ref{cor:maic}. There are two simple ideas in the proof --
scaling $F_1(\Omega(n))$ and $F_2(\Omega(n))$ so that they have $\ell_2$-norm
about $1$, and using a congruence trick to eliminate the sign $\sigma$
from Cor.~\ref{cor:maic}.
     \begin{proof}[Proof of Cor.~\ref{cor:cruxio}]
       We can assume $S_1$ and $S_2$ are non-empty.
       For any $p\in \mathbf{P}$,
       \[\begin{aligned}
     \sum_{\frac{N}{p}<n\leq \frac{2 N}{p}} F_1(\Omega(n))
     \overline{F_2(\Omega(n+1))} &=
     \sum_{\frac{N}{p}<n\leq \frac{2 N}{p}} F_1(\Omega(p n) - 1)
     \overline{F_2(\Omega(p n+ p)-1)}\\  &=
     \sum_{n\in \mathbf{N}, p|n}
     F_1(\Omega(n) - 1)
     \overline{F_2(\Omega(n+ p)-1)}.\end{aligned}\] 
       Now, for $a,j\in \mathbb{Z}/3 \mathbb{Z}$, we may define
       \[f_{1,a}(n) = \begin{cases} F_1(\Omega(n)-1) &\text{if
           $n\equiv a \mo 3$,}\\ 0 &\text{otherwise,}\end{cases}\]
       \[
       f_{2,a,j}(n) = \begin{cases} F_2(\Omega(n)-1) &\text{if
           $n\equiv a+j \mo 3$,}\\ 0 &\text{otherwise.}\end{cases}
       \]
       Since every event of the form $\Omega(n) = k$ (any $k$)
       happens for a proportion $\ll 1/\sqrt{\log \log N}$
       of elements of $\mathbf{N}$, we see that, for $i=1,2$,
       $f_{i,a}$ has $\ell^2$-norm $O(\sqrt{s_i/\sqrt{\log \log N}})$.
       Let $J$ be the set consisting of all $j\in \{1,2 \mo 3\}$
       such that $\mathscr{L}_j\geq e$, where
       $\mathbf{P}_j = \{p\in \mathbf{P}: p\equiv j \mo 3\}$ and
       $\mathscr{L}_j = \sum_{p\in \mathbf{P}_j} 1/p$.
       For each $a\in \mathbb{Z}/3\mathbb{Z}$
and $j\in J$, apply
       Cor.~\ref{cor:maic} with
       $f = K_1 f_{1,a}$,
       $g = K_2 f_{2,a,j}$ and $C=1$,
       where $K_1 = \min(1/|f_{1,a}|_2,e^{\mathscr{L}_j})$ and
       $K_2 = \min(1/|f_{2,a,j}|_2,e^{\mathscr{L}_j})$.
       (We are choosing the scaling factors $K_1$, $K_2$
       so that the assumption
       $|f|_4, |g|_4\leq e^{\mathscr{L}_j}$ in Cor.~\ref{cor:maic} holds.)
       The condition $32 C +4 \leq (\log H_0)/\mathscr{L}$ (i.e.,
       $\log H_0 \geq 36 \mathscr{L}$) in Cor.~\ref{cor:maic} holds
       if, as we may assume, $H$ is larger than a constant.
       We obtain that
       \begin{equation}\label{eq:garlux}  \frac{1}{N}
       \left|\sum_{n\in \mathbf{N}} \sum_{p\in \mathbf{P}_j, p|n} f_{1,a}(n)
\overline{f_{2,a,j}(n\pm p)} - \sum_{n\in \mathbf{N}} \sum_{p\in \mathbf{P}}
\frac{f_{1,a}(n) \overline{f_{2,a,j}(n\pm p)}}{p}
\right|
       \end{equation} is
       \[\ll \frac{
       \sqrt{\mathscr{L}_j}}{K_1 K_2} \ll
    \prod_{i=1}^2 \max\left(\frac{\sqrt{s_i}}{\sqrt[4]{\log \log N}}
       ,e^{-\mathscr{L}_j}\right)\cdot
              \sqrt{\mathscr{L}_j}
       \]
It is clear that, for $p\in \mathbf{P}_j$,
$f_{1,a}(n)$ and $g_{1,a,j}(n-p)$ cannot both be non-zero:
$n\equiv a \mo 3$ implies $n- p \equiv a - j \not\equiv a +j \mo 3$.
Hence
\[\begin{aligned}  \sum_{n\in \mathbf{N}} \sum_{p\in \mathbf{P}_j, p|n}
     F_1(\Omega(n) - 1)
     \overline{F_2(\Omega(n+ p)-1)} &=
     \sum_{a\in \mathbb{Z}/3 \mathbb{Z}}
     \sum_{n\in \mathbf{N}} \sum_{p\in \mathbf{P}_j, p|n}
     f_{1,a}(n)
     \overline{f_{2,a,j}(n+p)}\\ &=
     \sum_{a\in \mathbb{Z}/3 \mathbb{Z}}
     \sum_{n\in \mathbf{N}} \sum_{p\in \mathbf{P}_j, p|n}
     f_{1,a}(n)
     \overline{f_{2,a,j}(n\pm p)}.\end{aligned}\]
We may thus apply our bound on \eqref{eq:garlux} to conclude that,
for $j\in J$,
\begin{equation}\label{eq:lichess}\begin{aligned} &\sum_{p\in \mathbf{P}_j}
\sum_{\frac{N}{p}<n\leq \frac{2 N}{p}} F_1(\Omega(n))
\overline{F_2(\Omega(n+1))} =
\sum_{a\in \mathbb{Z}/3 \mathbb{Z}}
\sum_{n\in \mathbf{N}} \sum_{p\in \mathbf{P}_j, p|n}
     f_{1,a}(n) \overline{f_{2,a,j}(n\pm p)}\\
     &= \sum_{a\in \mathbb{Z}/3 \mathbb{Z}} \left(
      \sum_{n\in \mathbf{N}} \sum_{p\in \mathbf{P}_j}
\frac{f_{1,a}(n) \overline{f_{2,a,j}(n\pm p)}}{p}
+ O\left(
    \prod_{i=1}^2 \max\left(\frac{\sqrt{s_i}}{\sqrt[4]{\log \log N}}
       ,e^{-\mathscr{L}_j}\right)\cdot \sqrt{\mathscr{L}_j}\right)\right)\\
     &= 
\sum_{n\in \mathbf{N}} \sum_{p\in \mathbf{P}_j}
      \frac{F_1(\Omega(n) - 1)\overline{F_2(\Omega(n+ p)-1)}}{p}
      + O\left(
      \prod_{i=1}^2 \max\left(\frac{\sqrt{s_i}
      \sqrt[4]{\mathscr{L}}}{\sqrt[4]{\log \log N}}
       ,e^{-\mathscr{L}_j} \sqrt{\mathscr{L}_j}\right)\right)
      .\end{aligned}\end{equation}
We then remember to divide by $\mathscr{L}$.

For any $p\in \mathbf{P}$,
\[\sum_{\frac{N}{p}<n\leq \frac{2 N}{p}} F_1(\Omega(n))
\overline{F_2(\Omega(n+1))} \;\;\;\;\;\;\text{and}\;\;\;\;\;\;
      \sum_{n\in \mathbf{N}}
      \frac{F_1(\Omega(n) - 1)\overline{F_2(\Omega(n+ p)-1)}}{p}\]
      are both $O(\sqrt{s_1 s_2}/\sqrt{\log \log N})\cdot N/p$, by
      Cauchy-Schwarz and the bound $\ll 1/\sqrt{\log \log N}$ on
      the proportion of $n\in \mathbf{N}$ with given $\Omega(n)$.
      Hence, the contribution to both the left and the right side
      of \eqref{eq:startrek} from $p\in \mathbf{P}\setminus
      \bigcup_{j\in J} \mathbf{P}_j$ is
      $O(\sqrt{s_1 s_2}/\sqrt{\log \log N})/\mathscr{L}$,
      and we could bound the contribution from all $p\in \mathbf{P}_j$ by
      $O(\sqrt{s_1 s_2}/\sqrt{\log \log N}) \mathscr{L}_j/\mathscr{L}$.

      To obtain the bound $O(1/\sqrt{\mathscr{L}})$, simply note that
      $\sqrt{s_i} \sqrt[4]{\mathscr{L}}/\sqrt[4]{\log \log N}
      \ll \sqrt[4]{\mathscr{L}}$ and $\sqrt[4]{\mathscr{L}}\geq 1
      \geq e^{-\mathscr{L}_j} \sqrt{\mathscr{L}_j}$ (for
      $\mathscr{L}_j\geq e$).
      Alternatively, apply the bound
      $O\Big(\frac{\sqrt{s_1 s_2}}{\sqrt{\log \log N}}\Big)
      \frac{\mathscr{L}_j}{\mathscr{L}}$
      we just obtained when $j\in J$ is such that
      $\mathscr{L}_j\leq \log \log \log N$; for $j\in J$
      such that $\mathscr{L}_j> \log \log \log N$,
      we use the bound from
      \eqref{eq:lichess}, together with
      \[e^{-\mathscr{L}_j} \sqrt{\mathscr{L}_j}
      \leq \frac{\sqrt{\log \log \log N}}{\log \log N} \ll
      \frac{\sqrt{s_i} \sqrt[4]{\mathscr{L}}}{
        \sqrt[4]{\log \log N}}.\]
     \end{proof}

     Let us see how to go from a double sum of the kind estimated in
     Cor.~\ref{cor:cruxio} to a weighted average.
     The procedure, which is very simple, is already implicit in
     \cite{MR3569059}.
          \begin{lemma}\label{lem:radaro}
       Let $f_1,f_2:\mathbb{Z}_{>0}\to \mathbb{C}$ with
       $|f_1|_\infty, |f_2|_\infty \leq 1$.
       Let $\mathbf{P}$ be a set of primes $p\leq H$. Write
       $\mathscr{L} = \sum_{p\in \mathbf{P}} 1/p$.
       Let
       \begin{equation}\label{eq:reqZ}Z(T) = \frac{1}{T \mathscr{L}}
       \sum_{p\in \mathbf{P}}
 \sum_{\frac{T}{p} < n\leq \frac{2 T}{p}}
 f_1(n) f_2(n+1).\end{equation}
 Then, for any $w>1$,
 \[ \sum_{\frac{x}{w} <n\leq x}  \frac{f_1(n) f_2(n+1)}{n} =
 \int_{x/w}^x Z(t) \frac{dt}{t} + O\left(\frac{\log H}{\mathscr{L}}\right),\]
 where the implied constant is absolute.
     \end{lemma}
     \begin{proof}
       It is easy to see that, for any $p$,
 \[\begin{aligned}
  \int_{x/w}^x \frac{1}{t^2} &\sum_{\frac{t}{p} < n\leq \frac{2 t}{p}}
  f_1(n) f_2(n+1) dt =
  \sum_{\frac{x}{w p} <n\leq \frac{2 x}{p}}
  f_1(n) f_2(n+1)
  \int_{\max\left(\frac{n p}{2},\frac{x}{w}\right)}^{\min\left(n p,x\right)} \frac{dt}{t^2}
\\  &=
  \sum_{\frac{x}{w p} <n\leq \frac{2 x}{p}}
  \frac{f_1(n) f_2(n+1)}{n p} + O\left(\frac{1}{p}\right)
=\frac{1}{p} \sum_{\frac{x}{w} <n\leq x}  \frac{f_1(n) f_2(n+1)}{n}
  + O\left(\frac{\log p}{p}\right),
 \end{aligned}\]
        since $\sum_{y/p<n\leq y} 1/n = O(\log p)$.
 Hence \[\begin{aligned}
 \sum_{\frac{x}{w} <n\leq x}  \frac{f_1(n) f_2(n+1)}{n} &=
 \frac{1}{\mathscr{L}} \cdot \sum_{p\in \mathbf{P}} \frac{1}{p}
 \sum_{\frac{x}{w} <n\leq x}  \frac{f_1(n) f_2(n+1)}{n}\\
&=
 \int_{x/w}^x \frac{1}{t^2 \mathscr{L}}
\sum_{p\in \mathbf{P}}
 \sum_{\frac{t}{p} < n\leq \frac{2 t}{p}}
 f_1(n) f_2(n+1) dt + \frac{1}{\mathscr{L}} \sum_{p\in \mathbf{P}}
 \frac{\log p}{p}.\end{aligned}\]
     \end{proof}

          \begin{proof}[Proof of Cor.~\ref{cor:newlogchowla}]
      We assume without loss of generality that
  $w\leq \exp(\sqrt{\log x})$. Let $H_0 = \exp((\log w)^{2/3})$ and
  $H = H_0 \exp((\log w)^{8/9})$ (say).
  We easily verify some of the conditions in the statement of the main
  theorem:
  (a) $\log H_0 \leq (\log H)^{3/4}$, (b)  $\log H_0 \geq (\log H)^{2/3+1/15}$ for
  $w$ larger than a constant, (c) $\log H \leq \sqrt{(\log N)/\mathscr{L}}$ for
  $N\geq x/w$ and $x$ larger than a constant.

  Let $\mathbf{P}$ be the set of all primes in $[H_0,H]$.
  By Lemma \ref{lem:radaro} with $f_1 = f_2 = \lambda$,
    \[ \sum_{\frac{x}{w} <n\leq x}  \frac{\lambda(n) \lambda(n+1)}{n} =
\int_{x/w}^x Z(t) \frac{dt}{t} + O\left(\frac{\log H}{\mathscr{L}}\right),\]
where $Z(t)$ is as in \eqref{eq:reqZ}.
By Corollary \ref{cor:cruxio} with $F_1(n) = F_2(n) = (-1)^n$
and Corollary \ref{cor:notmarat},
\[\begin{aligned} Z(N) &= \frac{1}{N\mathscr{L}}
\sum_{n\in \mathbf{N}} \lambda(n) \sum_{p\in \mathbf{P}} \frac{\lambda(n+p)}{p}
+ O\left(\frac{1}{\sqrt{\mathscr{L}}}\right)\\
&\ll \frac{\eta_T^{1/5} (\log \Delta)^{4/5}}{\mathscr{L}^{4/5}} +
\frac{1}{\sqrt{\mathscr{L}}} \ll \eta_T^{1/5} +
\frac{1}{\sqrt{\mathscr{L}}}\ll
\frac{1}{\sqrt{\mathscr{L}}},
\end{aligned}\]
for integers $N\geq \sqrt{x}$ (say),
where $\eta_t = (\log \log H_0)/\log H_0 + 1/\log^{1/700} t$ and 
$\Delta = (\log H)/\log H_0$.
Thus $Z(t) \ll 1/\sqrt{\mathscr{L}}$
for $t\geq \sqrt{x}$.

It is clear that $\mathscr{L} = \log 
\Delta + O(1) \gg \log \log w$. Hence
\[ \sum_{\frac{x}{w} <n\leq x}  \frac{\lambda(n) \lambda(n+1)}{n} =
\int_{x/w}^x \frac{1}{\sqrt{\log \log w}} \frac{dT}{T} +
O\left(\frac{\log H}{\log \log w}\right)
= O\left(\frac{\log w}{\sqrt{\log \log w}}\right)
.\]
     \end{proof}

          Again we will need a simple lemma. This one is implicit in
          \cite{zbMATH07141311};
          it reduces a result ``at almost all scales'' to the estimation
          of a double sum.
               \begin{lemma}\label{lem:cocot}
       Let $f_1,f_2:\mathbb{Z}_{>0}\to \mathbb{C}$ with
       $|f_1|_\infty, |f_2|_\infty \leq 1$.
       Let $\mathbf{P}$ be a set of primes $p\leq H$. Write
       $\mathscr{L} = \sum_{p\in \mathbf{P}} 1/p$.
       Let $S(x) = (1/x) \sum_{x<n\leq 2 x} f_1(n) f_2(n+1)$, and let
       \begin{equation}\label{eq:reqZabs}Z^\circ(T) = \frac{1}{T \mathscr{L}}
         \sum_{p\in \mathbf{P}}
         \left|\sum_{\frac{T}{p} < n\leq \frac{2 T}{p}}
         f_1(n) f_2(n+1)\right| =
         \frac{1}{\mathscr{L}}  \sum_{p\in \mathbf{P}} \left|
         \frac{1}{p} S\left(\frac{T}{p}\right)\right|
         .\end{equation}
       Then, for any $w>1$,
       \[\int_{x/w}^x |S(t)| \frac{{\rm d}t}{t} =
\int_{x/w}^x |Z^\circ(T)| \frac{{\rm d}T}{T} + O(\log H).
       \]
     \end{lemma}
     \begin{proof}
       By definition of $Z^\circ(T)$ and $S(T)$,
       \[\begin{aligned}
       \int_{x/w}^x |Z^\circ(T)| \frac{{\rm d}T}{T} &=
       \frac{1}{\mathscr{L}} \sum_{p\in \mathbf{P}}
       \frac{1}{p} \int_{x/w}^x \left|S\left(\frac{T}{p}\right)
       \right| \frac{{\rm d}T}{T} =
       \frac{1}{\mathscr{L}} \sum_{p\in \mathbf{P}}
       \frac{1}{p} \int_{x/w p}^{x/p} \left|S(t)\right|
       \frac{{\rm d}t}{t} 
       \\
       &= \int_{x/w H}^x \frac{1}{\mathscr{L}}
       \left(\sum_{\frac{x}{w t} \leq p \leq \frac{x}{t}} \frac{1}{p}\right) |S(t)|
       \frac{{\rm d}t}{t} .
       \end{aligned}\]
       The sum $\sum_{x/w t\leq p\leq x/t} 1/p$ equals $\mathscr{L}$
       for $x/w\leq t\leq x/H$. Hence
       \[\begin{aligned}
       \int_{x/w}^x |Z^\circ(T)| \frac{{\rm d}t}{T} &=
       \int_{x/w}^{x/H} 
       |S(t)| \frac{{\rm d}t}{t} + \int_{x/w H}^{x/w} O(|S(t)|) \frac{{\rm d}t}{t} +
       \int_{x/H}^x O(|S(t)|) \frac{{\rm d}t}{t}\\
&=
       \int_{x/w}^{x/H} 
       |S(t)| \frac{{\rm d}t}{t} + O(\log H),
       \end{aligned}\]
       since $|S(t)|\leq 2$ for all $t$.
     \end{proof}
     
   \begin{proof}[Proof of Corollary \ref{cor:newalmostchowla}]
       Assume without loss of generality that
       $w\leq \exp(\sqrt{\log x})$. Let $H_0 =\exp( (\log w)^{2/3})$ and
       $H = H_0 \exp((\log w)^{8/9})$, say.

       By Lemma \ref{lem:cocot} with $f_1 = f_2 = \lambda$
       and $\mathbf{P}$ equal to the set of all primes in $[H_0,H]$,
          \[
        \frac{1}{\log w}  \int_{x/w}^x |S(t)| \frac{{\rm d}t}{t} =
\frac{1}{\log w} \int_{x/w}^x |Z^\circ(T)| \frac{{\rm d}T}{T} + O\left(\frac{\log H}{\log w}\right),\]
     where $Z^\circ$ is as in \eqref{eq:reqZabs}. Now,
     $Z^\circ(T) = (\mathscr{L}^+/\mathscr{L}) Z^+(T) -
     (\mathscr{L}^-/\mathscr{L}) Z^-(T)$ for
     \[Z^\sigma(T) = \frac{1}{T \mathscr{L}^+}
         \sum_{p\in \mathbf{P}^\sigma(T)}
         \sum_{\frac{T}{p} < n\leq \frac{2 T}{p}}
         \lambda(n) \lambda(n+1),\]
         where
         $\mathbf{P}^+(T)$ (respectively, $\mathbf{P}^-(T)$) be the set of all primes $p\in [H_0,H]$ such that $S(x/p)\geq 0$ (respectively, $S(x/p)<0$), and
         $\mathscr{L}^\sigma = \mathscr{L}^\sigma(T) = \sum_{p\in \mathbf{P}^\sigma(T)} 1/p$ for $\sigma = -,+$.
         
         Just as in the proof of Cor.~\ref{cor:newlogchowla},
         we apply Cor.~\ref{cor:cruxio} with $F_1(n) = F_2(n) =
         (-1)^n$ and Corollary \ref{cor:notmarat}, and obtain that,
         for $\sigma = +,-$ such that $\mathscr{L}^\sigma \geq 1$,
         \[Z^\sigma(T) = O\left( \frac{1}{\sqrt{\mathscr{L}^\sigma}}\right).\]
         If $\mathscr{L}^\sigma < 1$, we use the trivial bound
         $Z^\sigma(T) \ll 1 < 1/\sqrt{\mathscr{L}^{\sigma}}$.
         Hence, in any event,
         \[Z^\circ(T) \ll \frac{\sqrt{\mathscr{L}^+}}{\mathscr{L}}
         + \frac{\sqrt{\mathscr{L}^-}}{\mathscr{L}}
         \ll \frac{1}{\sqrt{\mathscr{L}}} \ll
         \frac{1}{\sqrt{\log \log w}}.\]
         We conclude that
         \[
         \frac{1}{\log w}  \int_{x/w}^x |S(t)| \frac{{\rm d}t}{t} \ll
         \frac{1}{\sqrt{\log \log w}} + \frac{\log H}{\log w} \ll
         \frac{1}{\sqrt{\log \log w}}.\]
\end{proof}
     
   \begin{proof}[Proof of Cor.~\ref{cor:almav}]
     Proceed exactly as in the proof of Corollary \ref{cor:newalmostchowla},
     only now
     \[Z^\sigma(T) = \frac{1}{T \mathscr{L}^+}
         \sum_{p\in \mathbf{P}^\sigma(T)}
         \sum_{\frac{T}{p} < n\leq \frac{2 T}{p}}
         f(n) \lambda(n+1)\]
         for $f$ defined by $f(n)=1$ for $\Omega(n)=k$ and $f(n)=0$ for
         $\Omega(n)\ne k$.  Apply Cor.~\ref{cor:cruxio} (with
         $S_1 = \{k\}$ and $S_2 = \mathbb{Z}_{\geq 0}$) and 
         Cor.~\ref{cor:notmarat}. The dominant error term in the bound from
         Cor.~\ref{cor:cruxio} is now
         \[\frac{\sqrt{s_1 s_2}}{\sqrt{\log \log N}} \frac{1}{\sqrt{\mathscr{L}}}
         = \frac{1}{(\log \log N)^{\frac{1}{4}}} \frac{1}{\sqrt{\mathscr{L}}}
         \gg \frac{\epsilon^{-1/2}}{(\log \log N)^{3/4}}.\]
   \end{proof}

{\bf Note.} The recent literature, starting in
\cite{MR3488742} and including \cite{MR3569059} and
\cite{zbMATH07141311}, tends to state results for $\lambda$ in greater
generality, including a broader class of multiplicative functions.
(The relatively large term $1/(\log N)^{1/700}$ in (\ref{eq:mrt}) seems
to be in part an artifact of this approach.) Here we have chosen to state
results for $\lambda$ alone for the sake of simplicity. Both the main result
and Corollary \ref{cor:maic} are fully general -- they do not even require
multiplicativity -- and, as we have just shown, deriving results such as
Corollary \ref{cor:newlogchowla} and \ref{cor:newalmostchowla}
from them is straightforward. In particular, we can replace
$\lambda(n) \lambda(n+1)$ in Cor.~\ref{cor:newlogchowla} and \ref{cor:newalmostchowla} by $\lambda(m_1 n + a_1) \lambda(m_2 n + a_2)$ with
$(m_i,a_i)=1$ and $(m_1,a_1) \ne \pm (m_2,a_2)$; it is
essentially enough to consider
$\lambda(n) \chi(n)$, with $n$ a Dirichlet character, instead of
$\lambda(n)$.

   \subsection{Proof of Corollary \ref{cor:lobster}}\label{subs:antoro}

   We will now prove Cor.~\ref{cor:lobster}. It will be enough
   to estimate the double sum on the right side of \eqref{eq:startrek},
   which, in this case, will equal a linear combination of sums of the form
\begin{equation}\label{eq:gina}\sum_{n\in \mathbf{N}}
  \mathop{\mathop{\sum_{p\in \mathbf{P}}}_{\Omega(n)=k_1}}_{\Omega(n)=k_2} \frac{1}{p},\end{equation}
where $k_1,k_2\in \mathbb{Z}_{>0}$.
We will estimate \eqref{eq:gina} by a fairly straightforward circle-method
approach. As is usual,
that approach will demand some estimates on exponential sums;
they will be provided by a mean-value theorem of a kind that is
essentially contained in \cite{MR3488742}, together with classical bounds.

Given $k$, let
$$
\pi_k(N, \alpha) = \sum_{\substack{n \in \mathbf{N} \\ \Omega(n) = k}} e(n \alpha)
$$
and $\pi_k(N) := \pi_k(N, 0)$. 
The following approximation is a minor sharpening of the main theorem in
\cite{MR952323}.

\begin{lemma} \label{le:majorarc}
  Let $k \geq 1$ be an integer. 
  Suppose that $q \leq \log N$ and $|\beta| \leq (\log N) / N$. Then, 
  $$
  \pi_k \Big ( N, \frac{a}{q} + \beta \Big ) = \frac{f(N, k; q)}{\varphi(q)} \int_{N}^{2N} e(\beta x) dx + O \Big ( \frac{N (1 + N |\beta|)}{\log N}
  |\log\log 2 q|^5 \Big )
  $$
  where, for
  $r = (k-1)/\log \log N$,
  \begin{align}\label{eq:deffNk}
    f(N, k; q) & = \left(F(r) r^{\Omega(q)-\omega(q)} (r-1)^{\omega(q)}
    + O\left(\frac{6^{\Omega(q)}}{k}\right)\right)
    \frac{(\log\log N)^{k - 1}}{(k - 1)! \log N},
  \end{align}
   with $F$ an analytic function such that $F(0) = F(1) = 1$.
\end{lemma}
We adopt the convention that $0^0 = 1$. Since the function $F(z)$
is independent of $q$ and $\beta$, it is the function determined by Sathe
\cite{zbMATH03079672} and Selberg \cite{zbMATH03091940}
for $q=1$, $\beta=0$.
\begin{proof}
  We can assume without loss of generality that $k \leq 2.9 \log\log N$;
  if $k > 2.9 \log\log N$, then, by a standard large-deviation bound
  (see, e.g., \cite[Lem.~2.2]{zbMATH01203259} and
  \cite[Ch.~0, Thm.~08]{zbMATH00193573}
  $|\pi_k(N,\alpha)|\leq \pi_k(N)\ll N (\log N)^{1-2.9 \log 2} \log \log N \ll
  N/\log N$,
and so \eqref{le:majorarc} holds trivially.

Let $T(x; z, a/q) = \sum_{n \leq x} z^{\Omega(n)} e(n a / q)$.  Clearly
  \begin{align*}
  \pi_k(N, \alpha) & = \frac{1}{2\pi i} \oint_{|z| = 1} \sum_{n \in \mathbf{N}} e(\beta n) e(a n / q) z^{\Omega(n)} \frac{dz}{z^{k + 1}} \\ & = \frac{1}{2\pi i} \oint_{|z| = 1} \Big ( \int_{x=N^+}^{x=2N^+} e(\beta x) d T(x; z, a/q) \Big ) \frac{dz}{z^{k + 1}}.
  \end{align*}
  By \cite[Lemme 1]{MR952323}
  (a special case of \cite{zbMATH03794177}) with $t=1$,
  \begin{align*}
    T(x;z,a/q)
    = \frac{x F(z)}{\varphi(q)} z^{\Omega(q) - \omega(q) + 1} (z - 1)^{\omega(q)} (\log x)^{z -1} + O(x (\log x)^{\Re z - 2} |\log \log 2 q|^{5}). 
  \end{align*}
  Integrating by parts we therefore find
  \begin{align*}
  \pi_k(N, \frac{a}{q} + \beta) & = \frac{1}{\varphi(q)} \frac{1}{2 \pi i} \oint_{|z| = 1} \!\! F(z) z^{\Omega(q) - \omega(q) + 1} (z - 1)^{\omega(q)}\int_{N}^{2N} e(\beta x) (x (\log x)^{z - 1})' dx \frac{dz}{z^{k + 1}} \\ & \ \ \ \ + O \left( \frac{N ( 1 + N |\beta|)}{\log N} \cdot |\log\log 2 q|^5 \right)
  \end{align*}
  Here of course $(x (\log x)^{z-1})' = (\log x)^{z-1} + (z-1)\cdot (\log x)^{z-2}$.
  It is clear that the term $(z-1)\cdot (\log x)^{z-2}$ can contribute at most
  $O(2^{\omega(q)}/\phi(q))\cdot  N/\log N = O(N/\log N)$ to
  $\pi_k(N,a/q+\beta)$. We take out $(z-1)\cdot (\log x)^{z-2}$, and rewrite the
  remaining main term as
  \begin{equation}\label{eq:anmiat}
    \frac{1}{\varphi(q)} \int_{N}^{2N} e(\beta x) f(x,k; q) dx
  \end{equation}
  where
  $$
  f(x, k; q) = \frac{1}{2\pi i} \oint_{|z| = 1} F(z) z^{\Omega(q) - \omega(q) + 1} (z - 1)^{\omega(q)} (\log x)^{z - 1}  \frac{dz}{z^{k + 1}}.
  $$
    We notice that
  $$
  \frac{\partial}{\partial x} f(x, k; q) \ll \frac{(\log x)^{-1}}{x}.
  $$
  Therefore
  $$
  \frac{1}{\varphi(q)} \int_{N}^{2N} e(\beta x) f(x, k; q) dx = \frac{f(N, k; q)}{\varphi(q)} \int_{N}^{2N} e(\beta x) dx +
  O\left( \frac{N}{\varphi(q) \log N} \right).
  $$
  We must finally obtain an estimate for $f(N, k; q)$. We will use what
  amounts to the Selberg-Delange method (cf. \cite[II.5]{zbMATH06471876}).
  Let $r = (k-1)/\log\log N$. Define
\[ G(z) = F(z) z^{\Omega(q) - \omega(q)} (z - 1)^{\omega(q)}.\]
  We shift the contour integration to a circle of radius $r$ around $0$. Write
  \begin{align*}
  \frac{1}{2\pi i}\int_{|z| = r} G(z) (\log N)^{z - 1} \frac{dz}{z^{k}}
  & = \frac{G(r)}{2\pi i} \int_{|z| = r} (\log N)^{z - 1} \frac{dz}{z^k} \\ & +
  \frac{1}{2\pi i} \int_{|z| = r} \Big ( G(z) - G(r) - G'(r) (z - r) \Big ) (\log N)^{z - 1} \frac{dz}{z^{k}}.
  \end{align*}
  where we use the fact that
  $$
  \frac{1}{2\pi i} \int_{|z| = r} (z - r) (\log N)^{z - 1} \frac{dz}{z^{k}} = 0
  $$
  owing to the choice of $r$.
  The main term here is given by
  \begin{equation}\label{eq:kaniwa}\begin{aligned}
  \frac{G(r)}{2\pi i} \int_{|z| = r} (\log N)^{z - 1} \frac{dz}{z^k}
  &= \frac{G(r)}{\log N} \frac{1}{2\pi i} 
  \int_{|z| = r} \left(1 + \dotsc + \frac{(z \log \log N)^{k-1}}{(k-1)!} + \dotsc\right) \frac{dz}{z^k}\\
  &= \frac{G(r)}{\log N}
  \frac{(\log \log N)^{k-1}}{(k-1)!} \asymp
  \frac{G(r)}{\log N} \cdot \frac{1}{\sqrt{k}} \left(\frac{e}{r}\right)^{k-1},
  \end{aligned}\end{equation}
  by Stirling's formula.
  Furthermore, 
  \[\begin{aligned}
   \Big | G(z) - G(r) - G'(r) (z - r) \Big | &\ll
  (\max_{|z|\leq r} |G''(r)|) \cdot |z-r|^2
  \ll \Omega(q)^2 |z|^{\Omega(q)-\omega(q)} |2 z|^{\omega(q)} |z-r|^2\\
  &\ll \Omega(q)^2 2^{\omega(q)} |z|^{\Omega(q)} |z-r|^2\ll
  6^{\Omega(q)} |z-r|^2
    \end{aligned}\]
  since $|z| = r < k/\log \log N \leq 2.9< 3$ by assumption.
  Thus,
  \[\begin{aligned}\Big|\frac{1}{2\pi i} 
  &\int_{|z| = r} \Big ( G(z) - G(r) - G'(r) (z - r) \Big ) (\log N)^{z - 1} \frac{dz}{z^{k}}\Big| \\ &\ll 6^{\Omega(q)} r
  \int_{-\pi}^\pi |1 - e^{i \theta}|^2 \Big|
  (\log N)^{r e^{i \theta} - 1}\Big| \frac{d\theta}{r^k}
  \ll \frac{6^{\Omega(q)}}{r^{k-1}} (\log N)^{r-1}
  \int_{-\pi}^\pi \theta^2 (\log N)^{r (\cos \theta-1)} d\theta
  \end{aligned}\]
  It is clear that the exponent in $\exp((\log \log N) r (\cos \theta-1))$
  is non-positive and bounded from below by a constant 
  for $|\theta|\ll 1/\sqrt{r \log \log N} = 1/\sqrt{k-1}$, and decreases faster than exponentially
  as $|\theta|$ increases beyond that range. Hence
  \[\int_{-\pi}^\pi \theta^2 (\log N)^{r (\cos \theta-1)} d\theta
  \ll \frac{1}{k^{3/2}}.\]
  Our error term is thus
  \[\ll 
\frac{6^{\Omega(q)}}{r^{k-1}} \frac{(\log N)^{r- 1}}{k^{3/2}} =
\frac{6^{\Omega(q)}}{r^{k-1}} \frac{e^{k-1}}{k^{3/2} \log N}.\]
Comparing this term to our main term in \eqref{eq:kaniwa}, we conclude that
\[f(N, k; q) = \left(G(r) + O\left(\frac{6^{\Omega(q)}}{k}\right)\right)
\frac{1}{\log N} \frac{(\log \log N)^{k-1}}{(k-1)!}\]
and we obtain our result.
\end{proof}

We will also need the following bound for exponential sums. We will not
need the strongest bounds of its kind available, but we will require a
``log-free'' bound, that is, a bound
free of unnecessary powers of $\log x$.
\begin{lemma}\label{lem:vinograd}
  Let $x\geq 1$, $\alpha\in \mathbb{R}/\mathbb{Z}$.
  Let $q\geq 1$ and $a$ be coprime and satisfy
  $|\alpha-a/q|\leq 1/q^2$. Then, for any $\epsilon>0$,
  \begin{equation}\label{eq:hickoc}\left|\sum_{p\leq x} e(\alpha p)\right| \ll_{\epsilon}
    \frac{x}{\log x} \left(
    \frac{1}{q^{1/2 - \epsilon}} +
    \frac{1}{(\log x)^{1/\epsilon}}\right)
  \end{equation}
\end{lemma}
\begin{proof}
  If $q\geq (\log x)^{A}$ for $A = 4/\epsilon$, it is enough to apply a Vinogradov-type bound
  resulting from Vaughan's identity (see \cite{MR0498434} e.g.,
  \cite[\S 25]{MR0217022}):
  \[\left|\sum_{n\leq x} \Lambda(n) e(\alpha n)\right| \ll
  \left(\frac{x}{\sqrt{q}} + \sqrt{x q} + x^{4/5}\right) (\log x)^4,\]
  together with summation by parts.

  Assume from now on that $q<(\log x)^A$. Given $Q$, we can always find
  an approximation $a'/q'$ to $\alpha$ with $q'\leq Q$ and
  $|\alpha-a'/q'|\leq 1/q' Q$. We can set $Q=x/(\log x)^A$,
  and, if $q'\geq (\log x)^A$, we may use $a'/q'$ instead of $a/q$,
  reducing matters to the case just considered. If $q'<(\log x)^A$,
  then, since we may assume $(\log x)^A \leq \sqrt{x}$, we 
  have $Q>q$, and so, since $|a/q-a'/q'|\geq 1/q q'$ when
  $a/q$, $a'/q'$ are distinct, we see that $a/q$, $a'/q'$ are identical.
  Hence, $|\alpha -a/q|\leq 1/q Q$.

  If we are satisfied with an ineffective result, we
  may use Siegel-Walfisz and summation by parts.
  Otherwise, we may use a log-free result such as
  \cite[Thm.~3]{MR2607306},\footnote{
    Bound on $\sum_{p\leq x} e(\alpha p)$ are of course central to
    approaches to the ternary Goldbach problem by means of the circle method.
    One of the bounds \cite[Thm. 2b]{MR2104806} in Vinogradov's 1947 monograph
    was already almost log-free; it does not quite give
    us \eqref{eq:hickoc}, as it has a power $(\log x)^{\epsilon}$
    in the numerator, but it would suffice for our application
    (in the proof of Prop.~\ref{prop:com}).
   See \cite[\S 10.3]{Helfbook} for a discussion of log-free bounds.
    }    
    together with integration by parts.
\end{proof}

The following lemmas are of a very classical kind.
\begin{lemma}\label{lem:aub1}
  Let $X\geq 1$, $\exp((\log X)^{2/3+\epsilon}) \leq P \leq Q\leq X$
  for some $\epsilon>0$.
  Then, for any $t$ with $|t|\leq X$ and any character $\chi$ of
  modulus $\leq (\log X)^{4/3}$,
  \[\Big|\sum_{P<p\leq Q} \frac{\chi(p)}{p^{\frac{1}{2} + i t}}\Big| =
  \sqrt{Q}\cdot \left(\frac{O(1/\log Q)}{1 + |t|} +
  \frac{O_\epsilon(1)}{\exp((\log X)^{\epsilon/2})}\right),\]
    where the term $O(1)/(1+|t|)$ is present only if $\chi$ is principal.
\end{lemma}
We could allow the modulus of $\chi$ to be $\ll (\log X)^A$, at the cost
of effectivity, but, in our application, we will not need to.
\begin{proof}[Sketch of proof]
  By Perron's formula (truncated; see, e.g.,
  \cite[\S 17, Lemma]{MR0217022} or \cite[Thm.~5.2]{MR2378655}),
  for any $\kappa>1/2$,
  \[\begin{aligned}
  \sum_{P<p\leq Q} \frac{\chi(p)}{p^{\frac{1}{2} + i t}}
  &= \frac{1}{2\pi i} \int_{\kappa - i T}^{\kappa + i T}
  \log L(s+1/2+it,\chi) \cdot \frac{Q^s-P^s}{s} ds\\
  &+ O_\kappa\left(\frac{Q^{1/2} \log Q}{T}\right) + O\left(
  \frac{1}{P^{1/4}}\right).
  \end{aligned}\]
  We choose a large $T$ (say, $T = Q^2$) and move the contour of
  integration to the left of $\Re s = 1/2$, but so that
  $s+1/2$ stays within the Vinogradov-Korobov zero-free region\footnote{While the Vinogradov-Korobov zero-free region for Dirichlet $L$-functions has undoubtedly been known for long, it seems
    hard to find a proof in the standard references (as has
    been remarked elsewhere; vd., e.g, \cite{zbMATH06192538}, footnote 2).
    The zero-free region in 
    \cite[Ch. VIII, Satz 6.2]{zbMATH03131972} is somewhat narrower
    but would be sufficient for our purposes, though
    the exponent $2/3$ in the statement of
    Lemma \ref{lem:aub1} would have to be changed to $3/4$.
    Fortunately, there are statements more general than
    what we need in the literature (\cite{zbMATH04187945}; cf.
    \cite{zbMATH03645196} apud \cite{zbMATH05232007}).}
  for $L(s,\chi)$.
  
  For $\chi$ principal, $\log L(s+1/2+i t,\chi)$ has a branch point at
  $s=1/2-it$,
  and thus the contour must go along the horizontal line
  $\Im s = - it$ from the left, around $s=1$ and then back to the left
  (``truncated Hankel contour''); see
  \cite[p. 754]{zbMATH02640400} for an early instance of this procedure,
  applied precisely to $\log \zeta(s)$, which is really the function in question
  here.
The result is a main term of size $O((\sqrt{Q}/\log Q)/(1+|t|))$.
  
  We use bounds on $L(s,\chi)$ and on the number of zeros of $L(s,\chi)$
  as in \cite[Thm.~11.4]{MR2378655} (say) to obtain our bound on
  the remaining terms.
  The effect of a hypothetical Landau-Siegel zero $\beta$ is
  negligible (namely,
  $O_{\epsilon'}(1)/\exp((\log X)^{\epsilon'})$ for any $0<\epsilon'<\epsilon$),
  since we have the effective bound $\beta < 1 - c/q^{1/2} (\log q)^2$
  (\cite[Cor.~11.12]{MR2378655}, \cite[Thm.~5.28(2)]{MR2061214}),
  where $q\leq (\log X)^{4/3}$ of the modulus of $\chi$, and we know
  that $P\geq \exp((\log X)^{2/3+\epsilon})$.
    The Vinogradov-Korobov zero-free region is broad enough that,
  thanks to the same lower bound $P\geq \exp((\log X)^{2/3+\epsilon})$,
  the other zeros contribute at most $Q\cdot O(\exp((\log X)^{-\epsilon/2}))$.
\end{proof}

\begin{lemma}\label{lem:aub2}
   Let $N\geq 1$.
   Then, for any $t$ with $|t|\leq N$, any $\alpha \in \mathbb{R}/\mathbb{Z}$
   and any character $\chi$ of
  modulus $\leq (\log N)^{2-\varepsilon}$, $0<\varepsilon\leq 1$,
  \[\Big|\sum_{N<n\leq 2 N} \frac{e^{2\pi i \alpha \Omega(n)} \chi(n)}{
    n^{1/2+ i t}}\Big| =
  \sqrt{N} \left(\frac{O(1)}{1+|t|} +
\frac{O_\varepsilon(1)}{\exp((\log N)^{\varepsilon/3})}
  \right).\]
\end{lemma}
We would actually do just fine with $\varepsilon=1$.
\begin{proof}[Sketch of proof]
  Define
  \[Z_{\alpha,\chi}(s) = \sum_{n\geq 1} \frac{e^{2\pi i \alpha \Omega(n)} \chi(n)}{n^s} = \prod_p \left(1 - \frac{e^{2\pi i \alpha} \chi(p)}{p^s}\right)^{-1}.\]
  Then, for $\Re s > 1$,
  \[\log Z_{\alpha,\chi}(s)=-\sum_p\sum_k
\frac{e(k\alpha)\chi(p^k)}{kp^{ks}}
  -e(\alpha)\sum_p\sum_k \frac{\chi(p^k)}{kp^{ks}} - G(s),\]
  where \[G(s) =
  \sum_p\sum_{k=1}^\infty \frac{\chi(p^k)}{kp^{ks}}(e(k\alpha)-e(\alpha)).\]
  Since the term $k=1$ here vanishes, we see that $G(s)$ can be
  extended analytically to $\Re s>1/2$, and is bounded on any region
  $\Re s\geq \sigma_0$ with $\sigma_0>1/2$.

  Hence
  \[Z_{\alpha,\chi}(s)=\exp(e(\alpha)\log Z_{0,\chi}(s)+G(s))=L(s,\chi)^{e(\alpha)}\exp(G(s)).
  \]
  We then proceed much as in the proof of Lemma \ref{lem:aub2}.
  Again, if $\chi$ is principal, there is a branch point at $s=1/2-it$,
  and so we have to use a truncated Hankel contour again; see
  \cite[Ch.~II.5, \S 3]{zbMATH06471876}.
  Just as before, there is a term $O_{\varepsilon'}(1)/\exp((\log N)^{\varepsilon'/2})$
  ($0<\varepsilon'<\varepsilon$ arbitrary) coming from a hypothetical Landau-Siegel
  zero. The contribution of all the other zeroes is even smaller, whether
  we use the classical zero-free region or the Vinogradov-Korobov zero-free
  region as above.
\end{proof}

We can now prove a mean-value theorem -- a variant of the stronger
results in \cite{MR3488742}. 

\begin{lemma} \label{lem:mr}
  Let $0<\epsilon<1/3$. 
  Let $\chi$ be a character to modulus $q \leq \log N$.
  Let $\alpha \in \mathbb{R}$. Then, for any $W,H\geq 1$ with $\log H \leq (\log N)^{2/3}$ and any $A\geq 1$,
  $$
  \int_{W \leq |t| \leq N / H} \Big | \sum_{n \in \mathbf{N}} \frac{e^{2\pi i \alpha \Omega(n)} \chi(n)}{n^{1/2 + it}} \Big |^2 dt
 \ll_{\epsilon} \frac{N}{W} +
  \frac{N}{(\log H)^{1-\epsilon}} + \frac{N}{(\log N)^{\frac{1}{3}-\epsilon}}.
  $$
\end{lemma}
  The statement is of a kind essentially covered by
  \cite{MR3488742}.
\begin{proof}
  We set the parameters
  $$
  P_1 = \exp(\log^{\epsilon} H),\, Q_1 = H\;\;\; \text{ and }\;\;\; P_2 = \exp(\log^{2/3 + \epsilon} N),\, Q_2 = N^{1/3}. 
  $$
  Let $\Delta = \log H$ and
  $$
  P_{j,\Delta}(s) = \sum_{\substack{e^{j / \Delta} < p \leq e^{(j + 1) / \Delta} \\ P_1 \leq p \leq Q_1}} \frac{\chi(p)}{p^{s}}.$$
  We let
  $J_0 = \Delta \log P_1$, $J_1 = \Delta \log Q_1$, and
  split the range of integration $t$ into two sets:
  $$
  \mathcal{T}_1 := \{ W \leq |t| \leq N / H : |P_{j, \Delta}(\tfrac 12 + it)| \leq \frac{e^{\frac{j}{2 \Delta}}}{V} \text{ for all }
  J_0 \leq j \leq J_1 \}\;\;\;\text{with}\; V = \log^{10} N, 
  $$ 
  and the complement $\mathcal{T}_2$. 
  By \cite[Lemma 12]{MR3488742}, for some $J_0\leq j\leq J_1$,
  $$\begin{aligned}
  \int_{\substack{W \leq |t| \leq \frac{N}{H} \\ t \in \mathcal{T}_1}} \Big  | \sum_{n \in \mathbf{N}} \frac{e^{2\pi i \alpha \Omega(n)} \chi(n)}{n^{1/2 + it}} \Big |^2 dt &\ll
  J_1 \Delta \log H 
  \int_{\substack{W \leq |t| \leq \frac{N}{H} \\ t \in \mathcal{T}_1}} |P_{j, \Delta}(\tfrac 12 + it) R_{j, \Delta}(\tfrac 12 + it)|^2 dt \\ &+
  N \left( \frac{1}{\Delta} + \frac{\log P_1}{\log Q_1}\right)
 , \end{aligned}$$
  where
  $$
  R_{j,\Delta}(\tfrac 12 + it) = \sum_{\substack{N e^{-j / \Delta} \leq m \leq 2 N e^{-j / \Delta}}} \frac{e^{2 \pi i \alpha \Omega(m)} \chi(m)}{m^{1/2 + it}} \cdot \frac{1}{\#\{ P \leq p \leq Q : q | m\} + 1}. 
  $$
  (When applying \cite[Lemma 12]{MR3488742},
  we use the fact that
  $e^{2\pi i \alpha \Omega(m p)} = e^{2\pi i} e^{2\pi i \alpha \Omega(m)}$ for any $p$, $m$ with
  $p\nmid m$.)
  We now simply use the definition of $\mathcal{T}_1$ to apply a point-wise bound on $P_{j,\Delta}$ and a standard mean-value theorem (such as, e.g.,
  \cite[Thm.~9.1]{MR2061214}) to estimate the remaining integral. We obtain a bound of the form
  $$
  \ll J_1 \Delta \log H \cdot
  \frac{e^{j/\Delta}}{V^2} \left(\frac{N}{H} +
  \frac{N}{e^{j/\Delta}}\right) 
  + N \left(\frac{1}{\log H} + \frac{1}{(\log H)^{1-\epsilon}}\right)
  \ll
  \frac{N}{(\log H)^{1-\epsilon}},$$
  which is good enough.

  It therefore remains to bound the contribution of $\mathcal{T}_2$.
  By Lemma \ref{lem:aub2},
  $$
  \sum_{n \in \mathbf{N}} \frac{e^{2\pi i \alpha \Omega(n)} \chi(n)}{n^{1/2 + it}} \ll \frac{\sqrt{N}}{1 + |t|}
  $$
  in the range $W \leq |t| \leq \exp((\log N)^{1/3})$.
  We let $W' = \exp((\log N)^{\epsilon/3})$, and see that then
    $$
    \int_{W \leq |t| \leq W'} \Big | \sum_{n \in \mathbf{N}} \frac{e^{2\pi i \alpha \Omega(n)} \chi(n)}{n^{1/2 + it}} \Big |^2 dt \ll \frac{N}{W}. 
    $$
    On the remaining part of the integral,
    we apply again \cite[Lemma 12]{MR3488742}, obtaining,
    $$  \begin{aligned}
      \int_{\substack{W' \leq |t| \leq \frac{N}{H} \\ t \in \mathcal{T}_2}} \Big  | \sum_{n \in \mathbf{N}} \frac{e^{2\pi i \alpha \Omega(n)} \chi(n)}{n^{1/2 + it}} \Big |^2 dt &\ll \Delta (\log Q_2)^2
      \int_{\substack{W' \leq |t| \leq \frac{N}{H} \\ t \in \mathcal{T}_2}} |P_{j, \Delta}(\tfrac 12 + it) R_{j, \Delta}(\tfrac 12 + it)|^2 dt
      \\ &+ N \left(\frac{1}{\Delta} + \frac{\log P_2}{\log Q_2}\right)
\end{aligned}$$
  for some $\Delta \log P_2 \leq j \leq \Delta \log Q_2$.
  We split the integral further into the part with $|t| \leq N e^{-j / \Delta}$ and the part with $N e^{-j / \Delta} \leq |t| \leq N / H$. On the first part,
  we apply Lemma \ref{lem:aub1}:
  \begin{equation}\label{eq:efobou}
  |P_{j,\Delta}(\tfrac 12 + it)| \ll \frac{e^{\frac{j}{2 \Delta}}}{W'},
  \end{equation}
  followed by the mean-value theorem. This yields a total bound of the form
  $$
  \ll \Delta (\log Q_2)^2 \cdot \frac{e^{j/\Delta}}{(W')^2} \frac{N}{e^{j/\Delta}}
  \ll
  \frac{N (\log N)^3}{(W')^2}
  $$
  which is better than we need.

  On the part of the integral with $|t| > N e^{-j / \Delta}$,
  we still have the bound \eqref{eq:efobou}.
  We also notice that $\mathcal{T}_2$ is in fact a tiny set: by \cite[Lemma 8]{MR3488742} (with $P = e^{j/\Delta}$ and
  $a_p = \sqrt{p/P}$),
  \[
  |\mathcal{T}_2| \ll
(N/H)^{2 \frac{\log V}{\log P_2}} V^2 e^{2 \frac{\log N}{\log P_2} \log \log N}
  \ll_{\epsilon'} N^{\epsilon'}\]
  for any $\epsilon'>0$.
    Therefore, we can bound the remaining integral by
  $$
  \frac{e^{j / \Delta}}{W'} \sum_{t \in \mathcal{T}_2^{\star}} |R_{j, \Delta}(\tfrac 12 + it)|^2
  $$
  where $\mathcal{T}_2^{\star}$ is a discrete set of $\ll N^{\epsilon}$ points. We now appeal to \cite[Thm.~9.6]{MR2061214} (due to Montgomery), and obtain,
  setting $\epsilon'= 1/3$,
  \[ \sum_{t \in \mathcal{T}_2^{\star}} |R_{j, \Delta}(\tfrac 12 + it)|^2 \ll
  \Big ( \frac{N}{e^{j / \Delta}} + N^{1/2 + \epsilon'} \Big ) (\log N)^2 \ll
 \frac{N}{e^{j / \Delta}} (\log N)^2
  ,\]
  and thus a bound of $\ll N (\log N)^2/W'$, which is certainly sufficient.
  Finally,
  \[  N \left( \frac{1}{\Delta} + \frac{\log P_1}{\log Q_1}\right) \ll
  \frac{N}{(\log N)^{1/3-\epsilon}}
  .\]
\end{proof}

We come to the result we will need in order
to prove Corollary \ref{cor:lobster}. Cf.
\cite[Th\'eo. 1]{zbMATH06786127},
which gives an estimate of the same shape, though only as an upper bound.

\begin{prop} \label{prop:com}
  Let $N\geq 1$.
Let  $\exp(\log^{\varepsilon} N) \leq H_0 \leq H
  \leq \exp((\log N)^{2/3})$, where $0<\varepsilon<1$.
  Let $\mathbf{P}$ be the set of primes in $[H_0, H]$. 
  Uniformly on $k, \ell \geq 1$, 
  $$
  \sum_{p \in \mathbf{P}} \sum_{\substack{n \in \mathbf{N} \\ \Omega(n) = k \\ \Omega(n + p) = \ell}} \frac{1}{p} = \sum_{p \in \mathbf{P}} \frac{1}{p} \cdot \frac{\pi_k(N) \pi_{\ell}(N)}{N} \cdot \left( 1 +
  O\left(\frac{|k-\log \log N|^2 +  |\ell-\log \log N|^2 + O_\varepsilon(1)}{(\log \log N)^2}\right)\right)
  %
  $$
\end{prop}

\begin{proof}

  If $k$ or $\ell$ is not in the interval $I_{\kappa} = [(1 - \kappa) \log\log N, (1 + \kappa) \log\log N]$ for some $\kappa > 0$ then the result follows from a sieve upper bound, with implied constants depending on $\kappa$. We can thus
  assume that $k,\ell$ belong to $I_{\kappa}$ for an arbitrarily small, but fixed, $\kappa > 0$. 

  Breaking the outer sum into dyadic intervals, we see that it suffices to determine the asymptotic behavior of
  \begin{equation} \label{eq:circle}
  \sum_{H'< p \leq 2H'} \sum_{\substack{n \in \mathbf{N} \\ \Omega(n) = k \\ \Omega(n + p) = \ell}} 1
  \end{equation}
  with $\exp(\log^{\varepsilon} N) \leq H' \leq \exp(\sqrt{\log N})$.
  (Since $\log \log H - \log \log H_0 \gg \log \log N$, we need not worry
  about the primes in one dyadic interval not entirely contained
  in $[H_0,H]$.)

  We now follow the circle method: we rewrite \eqref{eq:circle} as
  \begin{equation} \label{eq:realcircle}
  \int_{0}^{1} \Big ( \sum_{\substack{n \in \mathbf{N} \\ \Omega(n) = k}} e(n \alpha) \Big ) \Big ( \sum_{\substack{m \in \mathbf{N} \\ \Omega(m) = \ell}} e(-m \alpha) \Big ) \Big ( \sum_{H' < p \leq 2H'} e(p \alpha) \Big ) d \alpha. 
  \end{equation}
  Given $\eta > 0$, we define the major arcs
  $$
  \mathcal{M}_{\eta, W} := \Big \{ \alpha \in [0,1] : \Big | \alpha - \frac{a}{q} \Big | \leq \eta \text{ for some } (a,q) = 1,\; q \leq W \Big \}
    $$
  and the corresponding minor arcs $\mathfrak{m}_{\eta, W} = [0,1] \backslash \mathfrak{M}_{\eta, W}$.
  We will keep $W \leq \log N$ throughout; ultimately, we will choose $W$ to be a small power of $\log N$. 
  We let $\eta = (\log N)^{\delta}/N$, with $\delta>0$ small and to be chosen
  later. 
  We split the integral in \eqref{eq:realcircle} into an integral over $\mathfrak{M}_{\eta, W}$ and an integral over $\mathfrak{m}_{\eta, W}$. The integral over $\mathfrak{M}_{\eta, W}$ is easy to compute.
  Indeed, using Lemma \ref{le:majorarc} and integration by parts (on the short sum over primes $H'<p\leq 2H'$), we find that the integral over the major arcs is equal to
  \begin{align}\label{eq:ramodo}
    \sum_{q \leq W} \sum_{(a,q) = 1}
    \int_{|\beta| \leq \eta} \pi_k \Big (N, \frac{a}{q} + \beta \Big ) \pi_{\ell} \Big (N, - \frac{a}{q} - \beta \Big  ) \Big ( \sum_{H' < p \leq 2 H'} e(p a / q) \Big ) d \alpha
  \end{align}
  plus a negligible error term.
  By the theorem that one may call Page's, Landau-Page, or ``effective Siegel-Walfisz'' \cite[Cor.~11.17]{MR2378655},
  \[\sum_{H' < p \leq 2 H'} e\left(\frac{p a}{q}\right) = 
  |\{H'<p\leq 2 H'\}|\cdot
  \left(\frac{1}{\phi(q)}\mathop{\sum_{0\leq r<q}}_{(r,q)=1}
  e\left(\frac{r a}{q}\right)
  + \exp\left(- c \frac{\sqrt{\log H'}}{(\log \log H')^2}\right)\right).\]
  for some $c>0$.
  The Ramanujan sum
  $\sum_{0\leq r<q: (r,q)=1} e(r a/q)$ of course equals $\mu(q)$.
  We now appeal to Lemma \ref{le:majorarc}, and conclude that
  the expression in \eqref{eq:ramodo} equals
  \[\begin{aligned}
  (N + O(\eta^{-1})) &|\{H'<p\leq 2 H'\}|
    \sum_{\substack{q \leq W \\ (a,q) = 1}} \frac{f(N, k; q) f(N, \ell; q) \mu(q)}{\varphi(q)^3}\\
    &+ O\left(\frac{(N \eta)^2 W^2 (\log \log N)^5}{\log N}
    \frac{N H'}{\log H'}\right)
    + O\left(
    \exp\left(- c \frac{\sqrt{\log H'}}{(\log \log H')^2}\right)
    \frac{N H'}{\log H'}\right)
    .\end{aligned}\]
    Our main term comes from $q = 1$, whereas the terms with $q \geq 2$
  go in their entirety to the error term: looking at the definition
  \eqref{eq:deffNk}, we see that each prime factor of $q$ gives us
  an additional
  factor of $O(\varepsilon_k(N) \cdot \varepsilon_{\ell}(N))$, where
  $$
  \varepsilon_v(N) = \frac{|v - \log\log N|}{\log\log N}. 
  $$
  We thus obtain that the contribution of the major arcs is
  \[N |\{H' <p\leq 2 H'\}| \left(f(N,k;q) f(N,\ell;q) +
  O\left(\frac{|k-\log \log N| |\ell-\log \log N| + O_\varepsilon(1)}{(\log \log N)^2}\right)
  \right)\]
  provided that $(N\eta) W \leq (\log N)^{2/5}$ (say).
  Of course $|k-\log \log N| |\ell -\log \log N|$ is bounded by
  $(|k-\log \log N|^2 + |\ell-\log \log N|^2)/2$.

  It remains to bound the contribution of the minor arcs
  $\mathfrak{m}_{\eta, W}$. 
  We split $\mathfrak{m}_{\eta, W}$ further, into a union of
  $\mathfrak{m}_{\eta, W} \cap
  \mathfrak{M}_{H'^{-1/2}, W}$ and $\mathfrak{m}_{H'^{-1/2}, W}$.
  For $\alpha \in \mathfrak{m}_{H'^{-1/2}, W}$
  (``very minor arcs''), we let $Q = \sqrt{H'}$ and find
  a diophantine approximation $a/q$ to $\alpha$ with $q\leq Q$
  and $|\alpha - a/q|\leq 1/q Q$. We see that
  $q>W$ by definition of $\mathfrak{m}_{H'^{-1/2}, W}$.
  Hence, by Lemma \ref{lem:vinograd}, 
  $$
  \sum_{H' \leq p \leq 2H'} e(\alpha p) \ll \frac{H'}{\log H'} \cdot \frac{1}{W^{9/20}}
  $$
  provided that $W \leq H'$. 
  Therefore the contribution of $\alpha \in \mathfrak{m}_{H'^{-1/2}}$ to \eqref{eq:realcircle} is bounded by
  $
  \frac{N H'}{\log H'}\cdot \frac{1}{W^{9/20}},
  $
  which is entirely sufficient.

  Thus it remains to handle the part of the integral over $\alpha \in \mathfrak{m}_{\eta} \cap \mathfrak{M}_{H'^{-1/2}}$ (``middle arcs''). By Cauchy-Schwarz,
  this part of the integral is bounded by
  $$
  \frac{H'}{\log H'} \sup_{k \geq 1} \sum_{q \leq W} \sum_{(a,q) = 1}
  \int_{\eta \leq |\beta| \leq H'^{-1/2}} \Big | \sum_{\substack{n \in \mathbf{N} \\ \Omega(n) = k}} e(n a / q) e(n \beta) \Big |^2  d \beta. 
  $$
  We can detect the condition $\omega(n) = k$ using the formula
  $$
  \int_{0}^{1} e^{2\pi i \alpha \Omega(n)} e^{-2\pi i \alpha k} d\alpha.
  $$
  Therefore, again by Cauchy-Schwarz, it suffices to bound
  \begin{equation}\label{eq:wombat}
    \frac{H'}{\log H'} \sup_{\alpha \in [0,1]}
      \sum_{q \leq W} \sum_{(a,q) = 1}
  \sum_{\substack{q \leq W \\ (a,q) = 1}} \int_{\eta \leq |\beta| \leq H'^{-1/2}} \Big | \sum_{\substack{n \in \mathbf{N}}} e^{2\pi i \alpha \Omega(n)} e(n a / q) e(n \beta) \Big |^2  d \beta,
  \end{equation}
    the advantage being that the function $e^{2\pi i \alpha \Omega(n)}$ is now multiplicative.

    We can also go from additive to multiplicative characters, in what is
    essentially a standard way. By \cite[Corollary 5.3]{MR3909235}
    (second inequality),
    the expression within $\sup_{\alpha\in [0,1]}$
    in \eqref{eq:wombat} is at most
    \begin{equation} \label{eq:boundds}
      W^2
  \sup_{\substack{q \leq W \\ \chi \mo q}}
  \int_{\rho \eta N \leq |t| \leq H'^{-1/2} N / \rho} \Big | \sum_{n \in \mathbf{N}} \frac{e^{2\pi i \alpha \Omega(n)} \chi(n)}{n^{1/2 + it}} \Big |^2 dt +
  O\left(\rho^2 + \frac{1}{(\eta N)^2}\right) N
    \end{equation}
    for $\rho\in [H'^{-1/2}, 1]$ arbitrary.
    By Lemma \ref{lem:mr}, the expression within $\sup$ in
    \eqref{eq:boundds} is 
    \[\ll_{\epsilon} \frac{1}{\rho \eta} +
    \frac{N}{(\log H')^{1-\epsilon}} +
    \frac{N}{(\log N)^{1/3-\epsilon}}\]
    for any $0<\epsilon<1/3$. We set $\rho = 1/(\eta N)^{1/3}$,
    and see that the whole expression in \eqref{eq:wombat} is
    \[\ll_\epsilon
    \frac{N H' W^2}{\log H'} \left(\frac{1}{(\log H')^{1-\epsilon}} +
    \frac{1}{(\log N)^{1/3-\epsilon}} +  \frac{1}{(\eta N)^{2/3}}\right).\]
    We now let $\epsilon=1/6$, $\eta = (\log N)^{1/6}/N$ and
$W = (\log N)^{1/27}$. We obtain a total bound
    $$
    \frac{N H'}{\log H'} \left(\frac{1}{(\log H')^{5/6}} +
      \frac{1}{(\log N)^{1/27}}\right)
      $$
      on the contribution of the middle arcs.
\end{proof}
\begin{corollary} \label{cor:como}
    Let $N\geq 1$. Let
  $\exp(\log^{\varepsilon} N) \leq H_0 \leq H
  \leq \exp((\log N)^{2/3})$, where $0<\varepsilon<1$.
  Let $\mathbf{P}$ be the set of primes in $[H_0, H]$;
  write $\mathscr{L} = \sum_{p\in \mathbf{P}} 1/p$.
  Let $I_1, I_2\in \mathbb{Z}_{>0}$ be intervals. Then
  \[
  \sum_{p \in \mathbf{P}} \sum_{\substack{n \in \mathbf{N} \\ \Omega(n) \in I_1 \\ \Omega(n + p)\in I_2}} \frac{1}{p}
  = 
  O\left(\frac{\mathscr{L} N}{\log \log N}\right)
  +O_\varepsilon\left(\frac{\mathscr{L} N}{(\log \log N)^2}\right)
  .\]
  %
\end{corollary}
\begin{proof}
  By Proposition \ref{prop:com},
  \begin{equation}\label{eq:mallory}\begin{aligned}
  &\sum_{p \in \mathbf{P}} \sum_{\substack{n \in \mathbf{N} \\ \Omega(n) \in I_1 \\ \Omega(n + p)\in I_2}} \frac{1}{p}
  = \frac{\mathscr{L}}{N} \sum_{k_1\in I_1} \sum_{k_2\in I_2}
  \pi_{k_1}(N) \pi_{k_2}(N) (-1)^{k_1} (-1)^{k_2}
  \\
  &+ \frac{\mathscr{L}/N}{(\log \log N)^2}
  \sum_{k_1\in I_1} 
  \pi_{k_1}(N)
  \sum_{k_2\in I_2} \pi_{k_2}(N) (O(|k_1-\log \log N|^2 + |k_2-\log \log N|^2)+
  O_\varepsilon(1)).
  \end{aligned}\end{equation}
  Since $(\Omega(n)-\log \log N)/\sqrt{\log \log N}$ has a limit distribution
  centered around the origin with fast decay (Erd\H{o}s-Kac, generalized
  to $\Omega(n)$; it is enough to apply Lem.~\ref{le:majorarc} with
  $q=1$ and $\beta=0$, and use a crude bound to bound the case of $\Omega(n)$
  very large),
  \[\sum_{k_i\in I_i}  |k_1-\log \log N|^2 \pi_{k_i}(N) \ll \log \log N \cdot N.\]
  Hence, the error term in \eqref{eq:mallory} is
  $O(\mathscr{L} N/\log \log N) + O_\varepsilon(\mathscr{L} N/(\log \log N)^2)$.
  
  We should now estimate
  $\sum_{k\in I_i} (-1)^k \pi_k(N)$.
  Again, we use Lem.~\ref{le:majorarc} with
  $q=1$ and $\beta=0$, and obtain
  \[\sum_{k\in I_i} (-1)^k \pi_k(N) = N\cdot \sum_{k\in I_1} (-1)^k
  F\left(\frac{k-1}{\log \log N}\right)
  \frac{(\log \log N)^{k-1}}{(k-1)! \log N}
  + O\left(\frac{N}{\log \log N}\right).\]
  Since $F$ is analytic,
  $|F((k-1)/\log \log N) - F(k/\log \log N)|\ll 1/\sqrt{\log \log N}$
  for $k\ll \log \log N$. It is clear that
  $k\mapsto (\log \log N)^{k-1}/((k-1)! \log N)$ is increasing for
  $k\leq \log \log N$ and decreasing for $k> \log \log N$. Hence
  \[\left|
  \sum_{k\in I_1} (-1)^k   \frac{(\log \log N)^{k-1}}{(k-1)! \log N}\right|
  \leq 2 \max_{k\in I_1} \left|
  \frac{(\log \log N)^{k-1}}{(k-1)! \log N}\right| \ll
  \frac{1}{\sqrt{\log \log N}}.\]
  Hence, by partial summation,
  \[\sum_{k\in I_1} (-1)^k
  F\left(\frac{k-1}{\log \log N}\right)
  \frac{(\log \log N)^{k-1}}{(k-1)! \log N}
  = O\left(\frac{1}{\sqrt{\log \log N}}\right).\]  
  \end{proof}

We can now finally prove our last result.
\begin{proof}[Proof of Cor.~\ref{cor:lobster}]
       Assume without loss of generality that
       $w\leq \exp(\sqrt{\log x})$. Let $H_0 = \exp((\log w)^{2/3})$ and
       $H = H_0 \exp((\log w)^{8/9})$, say, just as before.
       Let $\mathbf{P}$ be the set of all primes in $[H_0,H]$.

       Let $x/w\leq T\leq x$.
       By Cor.~\ref{cor:cruxio} with
       $F_i(k)=(-1)^k$ for $k\in I_i$ and $F_i(k)=0$ for $k\not\in I_i$,
       \[\frac{1}{T \mathscr{L}}
       \sum_{p\in \mathbf{P}} \sum_{\frac{T}{p}<n\leq \frac{2 T}{p}}
       \lambda(n) \lambda(n+1)
       = \frac{1}{T \mathscr{L}} \sum_{p\in \mathbf{P}}
       \mathop{ \mathop{\sum_{T<n\leq 2 T}}_{\Omega(n)+1\in I_1}}_{
         \Omega(n+p)+1\in I_2} \frac{\lambda(n) \lambda(n+p)}{p}
       + O_\varepsilon\left(\frac{\sqrt{s_1 s_2}}{\log \log N}\right),
       \]
       since $\mathscr{L} \gg \log \log w \gg \varepsilon \log \log N$.
       By Corollary \ref{cor:como},
      \[ \frac{1}{T \mathscr{L}} \sum_{p\in \mathbf{P}}
       \mathop{ \mathop{\sum_{T<n\leq 2 T}}_{\Omega(n)+1\in I_1}}_{
         \Omega(n+p)+1\in I_2} \frac{\lambda(n) \lambda(n+p)}{p} =
       O\left(\frac{1}{\log \log N}\right)
       +        O_\varepsilon\left(\frac{1}{(\log \log N)^2}\right)
       .\]

       Hence, by Lemma \ref{lem:radaro},
       \[\begin{aligned}
       \mathop{\mathop{\sum_{\frac{x}{w}<n\leq x}}_{
           \Omega(n) \in I_1}}_{\Omega(n+p)\in I_2}
       \frac{\lambda(n) \lambda(n+1)}{n} &= 
       \int_{x/w}^x
       \frac{1}{T \mathscr{L}}
       \sum_{p\in \mathbf{P}} \sum_{\frac{T}{p}<n\leq \frac{2 T}{p}}
       \lambda(n) \lambda(n+1)
       \frac{dT}{T} + O\left(\frac{\log H}{\mathscr{L}}\right)
       \\ &=
       (\log w)\cdot
       O_\varepsilon\left(\frac{\sqrt{s_1 s_2}}{\log \log N}\right)
             ,\end{aligned}\]
       since $\mathscr{L}\gg \log \log N$ and $H\ll w$.
       \end{proof}
     
\section{Concluding remarks}
\subsection{Edges of composite length}
In general, in a graph of degree $d$, it may be the case that the largest
eigenvalue of the adjacency matrix other than the one corresponding to
constant eigenfunctions is $O(\sqrt{d})$, but one cannot hope to do better
(simply because there are $|V| d$ closed paths of length $2$;
see, e.g., \cite[Claim~2.8]{zbMATH05302790}). If we are to obtain a result such as
Corollary \ref{cor:maic} with a bound better than
$O(1/\sqrt{\mathscr{L}})$,
we ought to consider graphs whose average degree is substantially larger than
$\mathscr{L}$.

A first step could be to study a graph
$\Gamma$ with
$\mathbf{N}$ as its set of vertices, and edges $\{n,n+d\}$ for $d|n$
with $d\in (H_0,H]$ a product of two primes $p_1,p_2\geq H_0$.
  Then the average degree would already increase enough that
  we would presumably obtain asymptotics for $\Omega(n)=k$, $\Omega(n+1)=k'$
  for any $k,k' = \log \log x + O(\sqrt{\log \log x})$ (almost everywhere or in the sense of logarithmic averaging, as usual). 
  
Going further, one could consider a graph $\Gamma$ with
$\mathbf{N}$ as its set of vertices, and edges $\{n,n+d\}$ for
$d|n$, $d$ in a subset $D$ of the set of all integers in $(H_0,H]$
  without prime factors $<H_0$. It may be helpful to define $D$
  so that each of its elements has a somewhat below-average number of prime
  factors -- say, $\leq (\log H)/3 \log H_0$. The effect would be to
  bias consecutive integers $d_i$, $d_{i+1}$ in a walk $n, n\pm d_1, n\pm d_1
  \pm d_2,\dotsc$ towards having small gcd $(d_i,d_{i+1})$. In spite of this
  intentional bias, there would be
   no shortage of technical difficulties down the road, but the
  potential reward seems worthwhile: we might hope to bound
  $(1/\log w) \sum_{x/w<n\leq x} \lambda(n) \lambda(n+1)/n$ by a small
  power $(\log w)^{-c}$, $c>0$ (cf.\  Cor.~\ref{cor:newlogchowla})
  and, similarly, prove that $(1/\log w) \int_{x/w}^x |S(t)| dt/t
  =O(1/(\log w)^c)$ instead of (\ref{eq:newalmostchowla}).

  On the other hand, if we do not bias $D$ towards having a below-average
  number of prime factors, we can hope for better bounds. If we could actually
  prove expansion for a graph with edges $\{n,n+d\}$, $d|n$, with
  few restrictions on $d$, then we could in principle hope for
  a bound of $O(1/\sqrt{\log H})$. Then one might be able to prove that
  $\lambda(n+1)/n$ averages to $0$ at almost all scales
  when $n$ ranges over integers with $\leq 100$ prime factors, say, or possibly even $\leq 2$ prime factors. It is
  unclear whether that goal is realistic at this stage, but it is certainly
  something to have an eye on.

  \subsection{Non-backtracking paths}
  We have worked with an operator $A = \Ad_{\Gamma} - \Ad_{\Gamma'}$ defined
  in terms of adjacency operators $\Ad_{\Gamma}$, $\Ad_{\Gamma'}$. It may be
  possible and even convenient to work with non-backtracking operators
  instead. For a recent result obtained in this way and whose proof
  seems to have some elements in common with ours, see
  \cite{bordenave}. If one is to treat edges of composite length, what
  may make sense is to forbid two consecutive edge lengths
  $d_i$, $d_{i+1}$ from having large gcd.


  \subsection{Further applications: elliptic curves}\label{subs:eva}
  The list of corollaries and implicatons
  in the introduction is not of course exhaustive.
  Chowla's conjecture, on which we focus, has plenty of applications itself.
  One such application is to determining the average root number of
  families of elliptic curves. For each family $\mathscr{E}$,
  we need to know Chowla's
  conjecture for one polynomial $M_\mathscr{E}$ describing its places
  of multiplicative reduction. Going through the proof in
  \cite{Hell} (vd. also \cite{zbMATH07053663}), one can see that it behaves
  well when we modify the average to be a logarithmic average, or an
  average in an interval $N<n\leq 2 N$. Thus, we obtain, for example,
  that, for
  \[\mathscr{E}(t):((t-1)t+4) y^2 = x^3 - 3 x + ((t-1) t + 2),\]
  which is a family with $M_\mathscr{E} = x (x - 1)$,
  the curve $\mathscr{E}(n)$ (which is indeed an elliptic curve for all
  $n\ne 0,1$) has,
  in the limit,
  $0$ as the logarithmic average of its root number, and also
  has average root number $0$ over $N<n\leq 2 N$ at almost all scales
  (in the sense of Cor.~\ref{cor:newalmostchowla}).
  Either zero-average result
  would also follow from \cite{MR3569059} or
  \cite{zbMATH07141311}, but our work results in better bounds,
  naturally.

  \subsection{Further applications: the Erd\H{o}s discrepancy problem}

  It was shown by Borwein, Choi and Coons \cite{BCC} that there exists a completely multiplicative
  function $f : \mathbb{N} \rightarrow \{-1, 1\}$ such that,
  as $x\to \infty$,
  $$
  \Big | \sum_{n \leq x} f(n) \Big | \asymp \log x.
  $$
  It has been proposed (see, e.g., \cite{KMPT})
  that this is the slowest growth rate that any completely multiplicative
  function taking values in $\{-1, 1\}$ can have.
To be precise: the conjecture is that, for any completely
  multiplicative function $f : \mathbb{N} \rightarrow \{-1, 1\}$, there exist infinitely many $x$
  such that
  \begin{equation} \label{edpconj}
  \Big | \sum_{n \leq x} f(n) \Big | \gg \log x. 
  \end{equation}
  This conjecture can be viewed as a quantitative refinement of the Erd\"os discrepancy problem, recently settled
  by Tao \cite{TaoEDP}.

  Our work should allow one to prove \eqref{edpconj} with $\log x$ replaced by $\sqrt{\log\log x}$. To carry this task out, one needs to extend our Corollary \ref{cor:newlogchowla} to general multiplicative functions and follow the strategy laid out in \cite[\S 3]{TaoEDP}. 
  
  \subsection{Higher-degree analogues. Tensors and spectral norms.}
  \label{subs:adan}

  There are other structures whose possible expansion properties
  appear in the study of Chowla's
  conjecture in higher degree. There are plenty of reasons why one would
  like to have Chowla's conjecture in higher degree, even if logarithmically
  weighted. We gave one such reason
in \S \ref{subs:eva}. Another is that
  a logarithmically
  weighted Chowla's conjecture for arbitrary degree would imply
a logarithmically weighted version of Sarnak's conjecture
\cite{zbMATH06919988}.  

  One natural approach would be to aim to show, analogously to
  Cor.~\ref{cor:maic}, that, for $\mathbf{N}=\{N+1,\dotsc,2 N\}$
  and a set of primes $\mathbf{P}$ satisfying conditions analogous to those
  in the main theorem, and given positive integers $a_1<\dotsc<a_k$,
  \begin{equation}\label{eq:misirlu}
    \frac{1}{N \mathscr{L}}
    \left|\sum_{n\in \mathbf{N}} \mathop{\sum_{p\in \mathbf{P}}}_{p|n}
    f_1(n+a_1 p) \dotsb f_k(n+a_k p) -
    \sum_{n\in \mathbf{N}} \sum_{p\in \mathbf{P}}
    \frac{f_1(n+a_1 p) \dotsb f_k(n+a_k p)}{p}\right| \ll_{k,C} \frac{1}{\sqrt{\mathscr{L}}}
  \end{equation}
  for any $f_1,\dotsc,f_k:\mathbf{N}\to \mathbb{C}$ satisfying
  (a) $|f_i|_{r_i}\leq 1$ for some exponents $r_i$ of our choice with
  $1/r_1+\dotsc+1/r_k\leq 1$, (b) some bounds of the type $|f|_{s_i}\leq e^{C \mathscr{L}}$ for some higher $s_i>r_i$. (Here the conditions (b) are weak conditions
  analogous to the conditions on $|f|_4$, $|g|_4$ in Cor.~\ref{cor:maic}; they are meant to control
  the contribution of a few problematic elements of $\mathbf{N}$.)
  
  It is reasonable to
  guess that \eqref{eq:misirlu} holds. What is more, one may make
  a guess that implies \eqref{eq:misirlu} in the same way that
  the Main Theorem implies Cor.~\ref{cor:maic}: we may define a
  {\em tensor}, or multilinear operator,
  \[A = \frac{1}{N} \sum_{n\in \mathscr{X}}
  \left(\mathop{\sum_{p\in \mathbf{P}}}_{p|n}
  e_{n+a_1 p} \otimes \dotsb \otimes e_{n+a_k p} -
  \sum_{p\in \mathbf{P}}
  \frac{e_{n+a_1 p} \otimes \dotsb \otimes e_{n+a_k p}}{p}\right),\]
  where $\mathscr{X}$ is a subset of $\mathbf{N}$ with
  $|\mathbf{N}\setminus \mathscr{X}|$ small; then one may conjecture that
  $A$ has small spectral norm, as an operator on
  $\ell^{r_1}\times \dotsb \times \ell^{r_k}$ for some choice of $r_i$
  with $1/r_1 + \dotsb + 1/r_k \leq 1$. Of course, we can also write
  $A$ in terms of a hypergraph (or rather two),
  much as we defined $A$ in \eqref{eq:defA}
  in terms of two graphs $\Gamma$, $\Gamma'$.
  
  The tensor $A$ can be replaced by its symmetrization; we would
  still be able to obtain \eqref{eq:misirlu}, much as in the proof
  of Cor.~\ref{cor:cruxio}. There is a literature on spectral
  norms of symmetric tensors, much of it computational in nature
  (see, e.g., \cite{zbMATH07212229}; see also
  \cite{zbMATH00750658}, which in some sense prefigures our definition \eqref{eq:defA}).
 One obstacle is that there is
  no obvious analogue of the trace: we can bound the norm of a
  symmetric linear operator
  by $(\Tr A^{2 \ell})^{1/2\ell}$, $\ell\geq 1$ arbitrary, but there is
  no clear analogue for the spectral norm of a tensor, in general.
  We could choose to work with $r_1=r_2=2$ and $r_3=\dotsc=r_k=\infty$,
  and then the spectral norm of $A$ would equal the maximum
  of the spectral norm of the bilinear operator
  $A(\cdot,\cdot,v_3,\dotsc,v_k)$
  over
  $v_3,\dotsc,v_k$ with $|v_i|_\infty\leq 1$ for $i\geq 3$ -- and of course
  the norm of that (symmetric) bilinear operator can be bounded by
  a trace of the $2\ell$th power of the corresponding matrix (call it $A'$).
  The trace $\Tr (A')^{2 \ell}$
  is then a complicated linear combination of products of entries of
  $v_3,\dotsc,v_k$, and it is not clear how to proceed further.

  There may be an entirely different strategy towards higher-order
  Chowla in the sense of the logarithmic average. For instance, one might
  consider the graph
  having $\{1,2,\dotsc,N\}$ as its set of vertices and edges of the form
  $\{n,n/p+1\}$ for $p|n$; the idea would be to write
  $\lambda(n+a_1 p) \dotsb \lambda(n+a_k p) =
  \lambda(n') \lambda(n'/p+a_2-a_1)\dotsb \lambda(n'/p+a_k-a_1)$
  for $n'=n+a_1 p$, and use induction on $k$.
It is unclear whether an approach such as the one
  here would work for that graph, as we quickly obtain divisibility relations
  with non-linear terms; the graph is also very much non-local,
  and so an analogue of \S \ref{sec:eigentrace} would not seem feasible.

  At any rate, we would still have the problem of estimating the second
  double sum in \eqref{eq:misirlu}. That problem is open as of the time
  of writing, but might
  be within reach,
  due to the broad range we are allowing for $H$,
  namely, $H\leq \exp((\log x)^{1/2-\epsilon})$;
   see \cite{MRTTZ}, which proves
  the {\em Fourier uniformity}
  of $\lambda$ in the polynomial-phase case for $H\geq \exp((\log N)^{5/8+\epsilon})$.
  
\bibliographystyle{alpha}
\bibliography{trace}
\end{document}